\documentclass[11pt, a4paper,reqno]{amsart}
\pdfoutput=1
\usepackage{tikz,enumitem,rotating,latexsym,bm,stmaryrd,caption,}
\usetikzlibrary{positioning,intersections,decorations}
  \usepackage{colortbl,breqn,mathdots}
\usetikzlibrary{shapes}
\usetikzlibrary{arrows}  
\definecolor{ao(english)}{rgb}{0.0, 0.5, 0.0}
\def\dover#1{\underline{\underline{#1}}}

 \usetikzlibrary{decorations.pathmorphing}
\usetikzlibrary{snakes}
\tikzset{snake it/.style={decorate, decoration=snake}}
\tikzset{zigzag/.style={decorate, decoration=zigzag}}
\usepackage{float}
\definecolor{darkgreen}{rgb}{0.0, 0.5, 0.0}

\usetikzlibrary{decorations.pathmorphing}

	\definecolor{eng}{rgb}{0.0, 0.5, 0.0}
\definecolor{apple}{rgb}{0.55, 0.71, 0.0}
\definecolor{cadmium}{rgb}{0.0, 0.42, 0.24}
\definecolor{darkspringgreen}{rgb}{0.09, 0.45, 0.27}
\definecolor{amethyst}{rgb}{0.6, 0.4, 0.8}
\definecolor{ao}{rgb}{0.0, 0.0, 1.0}
\definecolor{atomictangerine}{rgb}{1.0, 0.6, 0.4}
\definecolor{carmine}{rgb}{0.59, 0.0, 0.09}
\newcommand{\mtau}{{\mathscr{P}^\ctau_{(W,P)}}}
\newcommand{\mtausmall}{{\mathscr{P}_{(W^\ctau,P^\ctau)}}}
 \newcommand{\blor}{
 \begin{tikzpicture}
\draw[orange,->](0,0) to(0.2,0);
\draw[blue,->](0,0) to(-0.2,0);
  \end{tikzpicture}}
 \newcommand{\orbl}{
 \begin{tikzpicture}
\draw[blue,->](0,0) to(0.2,0);
\draw[orange,->](0,0) to(-0.2,0);
  \end{tikzpicture}}
\newcommand{\blbl}{
 \begin{tikzpicture}
\draw[blue,->](0,0) to(0.2,0);
\draw[blue,->](0,0) to(-0.2,0);
  \end{tikzpicture}}
\newcommand{\oror}{
 \begin{tikzpicture}
\draw[orange,->](0,0) to(0.2,0);
\draw[orange,->](0,0) to(-0.2,0);
  \end{tikzpicture}}

\definecolor{toggle}{rgb}{1.0, 0.94, 0.96}

\usepackage[normalem]{ulem}
 \newcommand{\plus}{ }
 \newcommand{\pDelta}{\Delta}
 \newcommand{\dom}{{\sf dom}}
  \newcommand{\cod}{{\sf cod}}
  \newcommand{\fuck}{{1'}}
  \newcommand{\fuckfuck}{{1''}}

 \newcommand{\gsigma}{{{{\color{gray}\sigma}}}}

 \newcommand{\cpi}{{{{\color{orange}\pi}}}}
 \newcommand{\csigma}{{{{\color{magenta}\sigma}}}}
 \newcommand{\ctau}{{{{\color{cyan}\tau}}}}
\newcommand{\csigmaw}{{{{\color{magenta}\sigma}}}}
\newcommand{\ctauw}{{{{\color{cyan}\tau}}}}

 \newcommand{\al}{{{{\color{magenta} \sigma}}}}
 \newcommand{\crho}{{{{ \color{darkgreen}\rho}}}}

 \newcommand{\diag}{{\sf e}}

\newcommand{\exx}{{b_\al }}
\newcommand{\deltax}{{{\color{magenta}\delta_\alpha}}}
\newcommand{\SSTX}{{{\color{black}\sf T}}}

\newcommand{\eps}{ \varepsilon}

\newcommand{\fork}{{\sf fork}}

\newcommand{\isit}{{i}}
\newcommand{\isitone}{\al(i+1)}
\newcommand{\Shl}{\widehat{\mathfrak{S}}_{h\ell}}
\newcommand{\w}{{\underline{w}}}
\newcommand{\vvv}{{\underline{v} }}
\newcommand{\uuu}{{\underline{u} }}
\newcommand{\y}{{\underline{y}}}
\newcommand{\x}{{\underline{x}}}

\newcommand{\unvvv}{{{z}}}
\newcommand{\grade}{v}
\newcommand{\dgrm}{\mathcal{D}}
\newcommand{\dgrmdeg}{\mathcal{D}^{\rm deg}}
\newcommand{\dgrmf}{\mathcal{D}_{\rm f}}
\newcommand{\dgrmBS}{\mathcal{D}_{\rm BS}}
\newcommand{\dgrmBSdeg}{\mathcal{D}_{\rm BS}^{\rm deg}}
\newcommand{\dgrmBSdegsum}{\mathcal{D}_{\rm BS}^{{\rm deg},\oplus}}
\newcommand{\dgrmBSF}{\mathcal{D}_{\rm BS}^F}
\newcommand{\dgrmF}{\mathcal{D}^F}
\newcommand{\dgrmBSsumshift}{\mathcal{D}_{\rm BS}^{\oplus,(-)}}
\newcommand{\dgrmstd}{\mathcal{D}_{\rm std}}
\newcommand{\dgrmBSstd}{\mathcal{D}_{{\rm BS},{\rm std}}}
\newcommand{\dgrmBSpast}{\mathcal{D}_{{\rm BS},p|\ast}}
\newcommand{\dgrmBSpastsumshift}{\mathcal{D}_{{\rm BS},p|\ast}^{\oplus,\langle - \rangle}}
\newcommand{\dgrmpast}{\mathcal{D}_{p|\ast}}

\newcommand{\Wf}{W_{\rm f}}
\newcommand{\W}{W}
\newcommand{\Wp}{\W_p}
\newcommand{\Ssf}{S_{{\rm f}}}
\newcommand{\Ss}{S}
\newcommand{\Ssp}{\Ss_p}
\newcommand{\Sspexpr}{\expr{\Ss}_p}
\newcommand{\Sspone}{\Ss_{p \cup 1}}
\newcommand{\sh}{s_{\rm h}}
\newcommand{\linkexpr}{\expr{\Ss}_{p|1}}
\newcommand{\Wpcosets}{\prescript{p}{}{\W}}
\newcommand{\Alc}{\text{\bf Alc}}
 \newcommand{\Pdiptwo}{{\sf M}_{i,i+2}}
 \newcommand{\Pdipj}{{{\sf M}_{i,j}}}

 \newcommand{\nodelabel}{1'}

\def\down{\vee}
\def\up{\wedge}

\usepackage{standalone}

\usepackage{etex}
 
\tikzset{
  variable line width/.style={
    every variable line width/.append style={#1},
    to path={%
      \pgfextra{%
        \draw[every variable line width/.try,line width=\pgfkeysvalueof{/tikz/thickness}] (\tikztostart) -- (\tikztotarget);
      }%
      (\tikztotarget)
    },
  },
  thickness/.initial=0.6pt,
  every variable line width/.style={line cap=round, line join=round},
}

\usepackage{todonotes}
\newcommand{\warning}[1]{\todo[color=red!75]{#1}}

\usepackage{tikz}
\usetikzlibrary{matrix,  intersections, calc, decorations.pathreplacing} 

\usepackage{tikz-cd}

\newlength{\superthick}
\newlength{\cornerradius}
\tikzstyle{corner}=[rounded corners=\cornerradius]
\tikzstyle{dot}=[circle, inner sep=0pt, minimum size=4.8pt]
\tikzstyle{string}=[line width=\superthick]
\tikzstyle{std}=[string,dash pattern=on 0.9pt off 0.9pt]
\definecolor{realcyan}{rgb}{0,1,1}

 \usepackage{amsmath,amsthm,amsfonts,amssymb,mathrsfs,pb-diagram}
\usepackage[
bookmarks=true,colorlinks=true,linktoc=page,citecolor=darkgreen,linkcolor=darkgreen,urlcolor=darkgreen]{hyperref}
\usepackage{caption}
\usepackage{lipsum,wasysym}
\usepackage{mathtools}
\usepackage[a4paper,margin=1in]{geometry}
\usepackage{cleveref}
 \usepackage{amsmath}
\mathchardef\mhyphen="2D
\usepackage{color}
\usepackage{xcolor}
\usepackage{ifthen}
\usepackage{sidecap}   
\definecolor{mediumblue}{rgb}{0.0, 0.0, 0.8}

\newcommand{\Res}{{\rm Res}}
\newcommand{\Rem}{{\rm Rem}}
\newcommand{\Add}{{\rm Add}}
\newcommand{\Ind}{{\rm Ind}}
\newcommand{\hstar}{\mathfrak{h}^*}
\newcommand\mptn[2]{\mathscr{P}^{#1}_{#2}}
\renewcommand{\geq}{\geqslant}
\renewcommand{\leq}{\leqslant}
 \newcommand{\Q}{{\mathbb Q}}

    \tikzset{wei/.style={black,double=white,thick,double
distance=1.5pt}}

\newcommand{\fA}{\mathfrak{A}}
\newcommand{\fB}{\mathfrak{B}}
\newcommand{\fC}{\mathfrak{C}}
\newcommand{\fD}{\mathfrak{D}}
\newcommand{\fE}{\mathfrak{E}}
\newcommand{\fF}{\mathfrak{F}}
\newcommand{\fG}{\mathfrak{G}}
\newcommand{\fH}{\mathfrak{H}}
\newcommand{\fI}{\mathfrak{I}}
\newcommand{\fJ}{\mathfrak{J}}
\newcommand{\fK}{\mathfrak{K}}
\newcommand{\fL}{\mathfrak{L}}
\newcommand{\fM}{\mathfrak{M}}
\newcommand{\fN}{\mathfrak{N}}
\newcommand{\fO}{\mathfrak{O}}
\newcommand{\fP}{\mathfrak{P}}
\newcommand{\fQ}{\mathfrak{Q}}
\newcommand{\fR}{\mathfrak{R}}
\newcommand{\fS}{\mathfrak{S}}
\newcommand{\fT}{\mathfrak{T}}
\newcommand{\fU}{\mathfrak{U}}
\newcommand{\fV}{\mathfrak{V}}
\newcommand{\fW}{\mathfrak{W}}
\newcommand{\fX}{\mathfrak{X}}
\newcommand{\fY}{\mathfrak{Y}}
\newcommand{\fZ}{\mathfrak{Z}}
\newcommand{\fa}{\mathfrak{a}}
\newcommand{\fb}{\mathfrak{b}}
\newcommand{\fc}{\mathfrak{c}}
\newcommand{\fd}{\mathfrak{d}}
\newcommand{\fe}{\mathfrak{e}}
\newcommand{\ff}{\mathfrak{f}}
\newcommand{\ffg}{\mathfrak{g}}
\newcommand{\fh}{\mathfrak{h}}
\newcommand{\ffi}{\mathfrak{i}}
\newcommand{\fj}{\mathfrak{j}}
\newcommand{\fk}{\mathfrak{k}}
\newcommand{\fl}{\mathfrak{l}}
\newcommand{\fm}{\mathfrak{m}}
\newcommand{\fn}{\mathfrak{n}}
\newcommand{\fo}{\mathfrak{o}}
\newcommand{\fp}{\mathfrak{p}}
\newcommand{\fq}{\mathfrak{q}}
\newcommand{\fr}{\mathfrak{r}}
\newcommand{\fs}{\mathfrak{s}}
\newcommand{\ft}{\mathfrak{t}}
\newcommand{\fu}{\mathfrak{u}}
\newcommand{\fv}{\mathfrak{v}}
\newcommand{\fw}{\mathfrak{w}}
\newcommand{\fx}{\mathfrak{x}}
\newcommand{\fy}{\mathfrak{y}}
\newcommand{\fz}{\mathfrak{z}}

\newcommand{\sA}{\mathscr{A}}
\newcommand{\sB}{\mathscr{B}}
\newcommand{\sC}{\mathscr{C}}
\newcommand{\sD}{\mathscr{D}}
\newcommand{\sE}{\mathscr{E}}
\newcommand{\sF}{\mathscr{F}}
\newcommand{\sG}{\mathscr{G}}
\newcommand{\sH}{\mathscr{H}}
\newcommand{\sI}{\mathscr{I}}
\newcommand{\sJ}{\mathscr{J}}
\newcommand{\sK}{\mathscr{K}}
\newcommand{\sL}{\mathscr{L}}
\newcommand{\sM}{\mathscr{M}}
\newcommand{\sN}{\mathscr{N}}
\newcommand{\sO}{\mathscr{O}}
\newcommand{\sP}{\mathscr{P}}
\newcommand{\sQ}{\mathscr{Q}}
\newcommand{\sR}{\mathscr{R}}
\newcommand{\sS}{\mathscr{S}}
\newcommand{\sT}{\mathscr{T}}
\newcommand{\sU}{\mathscr{U}}
\newcommand{\sV}{\mathscr{V}}
\newcommand{\sW}{\mathscr{W}}
\newcommand{\sX}{\mathscr{X}}
\newcommand{\sY}{\mathscr{Y}}
\newcommand{\sZ}{\mathscr{Z}}

\newcommand{\hd}{\operatorname{hd}}
\newcommand{\cA}{\mathcal{A}}
\newcommand{\cB}{\mathcal{B}}
\newcommand{\cC}{\mathcal{C}}
\newcommand{\cD}{\mathcal{D}}
\newcommand{\cE}{{\rm M}}
\newcommand{\cF}{\mathcal{F}}
\newcommand{\cG}{\mathcal{G}}
\newcommand{\cH}{\mathcal{H}}
\newcommand{\cI}{\mathcal{I}}
\newcommand{\cJ}{\mathcal{J}}
\newcommand{\cK}{\mathcal{K}}
\newcommand{\cL}{\mathcal{L}}
\newcommand{\cLR}{\mathcal{LR}}
\newcommand{\cM}{\mathcal{M}}
\newcommand{\cN}{\mathcal{N}}
\newcommand{\cO}{\mathcal{O}}
\newcommand{\cP}{\mathcal{P}}
\newcommand{\cQ}{\mathcal{Q}}
\newcommand{\cR}{\mathcal{R}}
\newcommand{\cS}{\mathcal{S}}
\newcommand{\cT}{\mathcal{T}}
\newcommand{\cU}{\mathcal{U}}
\newcommand{\cV}{\mathcal{V}}
\newcommand{\cW}{\mathcal{W}}
\newcommand{\cX}{\mathcal{X}}
\newcommand{\cY}{\mathcal{Y}}
\newcommand{\cZ}{\mathcal{Z}}

\newcommand{\Z}{\mathbb{Z}}
\newcommand{\R}{\mathbb{R}}
\newcommand{\N}{\mathbb{N}}
\newcommand{\C}{\mathbb{C}}

\DeclareMathOperator{\reg}{reg}
\DeclareMathOperator{\image}{im}

 \newcommand{\alphar}{{{\color{magenta}\boldsymbol \alpha}}}
 \newcommand{\bet}{{{\color{cyan}\boldsymbol \tau}}}
  \newcommand{\gam}{{{{\color{orange!95!black}\boldsymbol\pi}}}}

 \newcommand{\betar}{{{\color{green!70!black}\boldsymbol \beta}}}

 \tikzset{wei2/.style={gray,double=gray!15!white,
distance=1pt}}

\allowdisplaybreaks
\numberwithin{equation}{section}
\usepackage{scalefnt}

\newtheorem{thm}{Theorem}[section]
\newtheorem{cor}[thm]{Corollary}

\newtheorem{lem}[thm]{Lemma}
\newtheorem{prop}[thm]{Proposition}

\newtheorem*{prop*}{Proposition}
\newtheorem*{thmA}{Theorem A}

\newtheorem*{thmB*}{Theorem B}
\newtheorem*{thmC*}{Theorem C}\newtheorem*{thm*}{Theorem D}
\newtheorem*{cor*}{Corollary}

\newtheorem*{conj*}{Conjecture A}

\newtheorem*{conj1*}{Conjecture B}
\newtheorem*{Acknowledgements*}{Acknowledgements}

\theoremstyle{rmk}

\theoremstyle{defn}
\newtheorem{rmk}[thm]{Remark}

\newtheorem{defn}[thm]{Definition}
\newtheorem{eg}[thm]{Example}

\newcommand{\great}{>}
\newcommand{\less}{<}
\newcommand{\greatoreq}{\geq}
\newcommand{\lessoreq}{\leq}
\newcommand{\codeg}{\mathrm{codeg}}
\newcommand{\triv}{\mathrm{triv}}
\newcommand{\id}{\mathrm{id}}

\newcommand{\TL}{\mathrm{TL}}
\newcommand{\rad}{\mathrm{rad}}
\newcommand{\res}{\mathrm{res}}
\newcommand{\ik}{{k}}
\newcommand{\Std}{{\rm Std}}
\newcommand{\SStd}{{\rm Path}}
\renewcommand{\det}{{\rm det}}
\newcommand{\epsilonLIRONdontchange}{\epsilon}
\newcommand{\Shape}{\operatorname{Shape}} 
\newcommand{\Path}{{\rm Path}}

\newcommand{\NPath}{{\rm NPath}}
\newcommand{\BPath}{{\rm BPath}}

\newcommand{\CPath}{{\rm Path}_{\underline{w}}}
\newcommand{\la}{\lambda}
\newcommand{\I}{i}
\newcommand{\J}{j}
\newcommand{\K}{k}
\newcommand{\M}{m}
\renewcommand{\L}{l}
\def\Ca{\mathcal C}
\newcommand{\Lead}{\operatorname{Lead}}
\newcommand{\SSTS}{\mathsf{S}}
\newcommand{\tSSTT}{\overline{\mathsf{T}}} 
\newcommand{\tla}{\overline{\x}}
\newcommand{\tmu}{\overline{\y }}
\newcommand{\SSTT}{\mathsf{T}}  
\newcommand{\SSTP}{\mathsf{P}}  
\newcommand{\SSTU}{\mathsf{U}}  
\newcommand{\SSTV}{\mathsf{V}}  
\newcommand{\SSTQ}{\mathsf{Q}}  
\newcommand{\SSTR}{\mathsf{R}}  
\newcommand{\sts}{\mathsf{s}}  
\newcommand{\stt}{\mathsf{t}}  
\newcommand{\stu}{\mathsf{u}}  
\newcommand{\stv}{\mathsf{v}}  
\newcommand{\ZZ}{{\mathbb Z}}
\newcommand{\NN}{{\mathbb N}}
\newcommand{\g}{\ell}
 
\newcommand{\CC}{{\mathbb{C}}}
\newcommand{\RR}{{\mathbb R}}
\newcommand{\Hyp}{{\mathbb E}_{\alpha,me}}
\newcommand{\length}{{t}}
\DeclareMathOperator\noedge{\:\rlap{\hspace*{0.25em}/}\text{---}\:}
\DeclareMathOperator{\Hom}{Hom}
\def\Mod{\textbf{-Mod}}
\let\<=\langle
\let\>=\rangle

\newcommand\Dec[1][A]{\mathbf{D}_{#1}(t)}
\newcommand\Cart[1][A]{\mathbf{C}_{#1}(t)}

\newcommand{\north}{top }
\newcommand{\northT}{{\sf T}}
\newcommand{\south}{bottom } 
\newcommand{\southT}{{\sf B}}

\newcommand\mydots{\makebox[1em][c]{\color{cyan}.\hfil\color{magenta}.\hfil\color{cyan}.\hfil\color{magenta}.}}
\tikzset{
ultra thin/.style= {line width=0.05pt},
very thin/.style=  {line width=0.2pt},
thin/.style=       {line width=0.1pt},
semithick/.style=  {line width=0.6pt},
thick/.style=      {line width=0.8pt},
very thick/.style= {line width=1.2pt},
ultra thick/.style={line width=1.6pt}
}

\crefname{defn}{Definition}{Definitions}
\crefname{thm}{Theorem}{Theorems}
\crefname{prop}{Proposition}{Propositions}
\crefname{lem}{Lemma}{Lemmas}
\crefname{cor}{Corollary}{Corollaries}
\crefname{conj}{Conjecture}{Conjectures}
\crefname{section}{Section}{Sections}
\crefname{subsection}{Subsection}{Subsections}
\crefname{eg}{Example}{Examples}
\crefname{figure}{Figure}{Figures}
\crefname{rem}{Remark}{Remarks}
\crefname{rmk}{Remark}{Remarks}
\crefname{equation}{equation}{equation}

\Crefname{defn}{Definition}{Definitions}
\Crefname{thm}{Theorem}{Theorems}
\Crefname{prop}{Proposition}{Propositions}
\Crefname{lem}{Lemma}{Lemmas}
\Crefname{cor}{Corollary}{Corollaries}
\Crefname{conj}{Conjecture}{Conjectures}
\Crefname{section}{Section}{Sections}
\Crefname{subsection}{Subsection}{Subsections}
\Crefname{eg}{Example}{Examples}
\Crefname{figure}{Figure}{Figures}
\Crefname{rem}{Remark}{Remarks}
\Crefname{rmk}{Remark}{Remarks}

  \newcommand{\Mull}{{\rm M}}
    
\usepackage[hang,flushmargin]{footmisc}
 
\newcommand\Dim[2][t]{\text{\rm Dim}_{#1}#2}

\newcommand\REMOVETHESE[2]{{{{\mathsf{M}}_{#1}^{#2}	}}}
\newcommand\ADDTHIS[2]{{{{\mathsf{P}}_{#1}^{#2}}}}

 \newcommand{\Spotzero}{\SSTS_{0,\al}}
  \newcommand{\Spotone}{\SSTS_{1,\al}}
   \newcommand{\Spottwo}{\SSTS_{2,\al}}
    \newcommand{\Spotthree}{\SSTS_{3,\al}}
 \newcommand{\Spotq}{\SSTS_{q,\al}}
  \newcommand{\Spotb}{\SSTS_{\exx,\al}}
  \newcommand{\Spotqplus}{\SSTS_{q+1,\al}}

 \newcommand{\Forkq}{{\sf F}_{q, \al}}
  \newcommand{\Forkqplus}{{\sf F}_{q+1, \al}}
 \newcommand{\Forkone}{{\sf F}_{1, \al}}
 \newcommand{\Forkzero}{{\sf F}_{0, \al}} 
  \newcommand{\Forktwo}{{\sf F}_{2, \al}}
    \newcommand{\Forkthree}{{\sf F}_{3, \al}}
  \newcommand{\Forkb}{{\sf F}_{\exx, \al}}

 \newcommand{\TForkq}{{\sf F}_{q,\al }}
  \newcommand{\TForkqplus}{{\sf F}_{q+1,\al}}
 \newcommand{\TForkone}{{\sf F}_{1,\al }}
 \newcommand{\TForkzero}{{\sf F}_{0,\al }} 
  \newcommand{\TForktwo}{{\sf F}_{2,\al }}

\renewcommand{\labelitemi}{$\circ $}

\def\Item{\item\abovedisplayskip=0pt\abovedisplayshortskip=5pt~\vspace*{-\baselineskip}} 

 \author{Chris Bowman}
       \address{Department of Mathematics, 
University of York, Heslington, York,  UK}
\email{Chris.Bowman-Scargill@york.ac.uk}
  
 \author{Maud De Visscher}
	\address{Department of Mathematics, City, University of London,   London, UK}
\email{Maud.DeVisscher@city.ac.uk}

    \author{Niamh Farrell}
\address{
Leibniz Universit\"at Hannover,
Welfengarten 1, D-30167 Hannover, Germany
 }

\email{nifarrel@tcd.ie}

		\author{Amit  Hazi}
       \address{Department of Mathematics, 
University of York, Heslington, York,  UK}
\email{Amit.Hazi@york.ac.uk}

 \author{Emily Norton}
\address{
School of Mathematics, Statistics and Actuarial Science, University of Kent, Canterbury, UK  }
 \email{E.Norton@kent.ac.uk}

\newcommand\NF[1]{{\color{violet}{#1}}}
 
\begin{document}

 \title[Oriented Temperley--Lieb algebras  and combinatorial  Kazhdan--Lusztig theory  ] {Oriented Temperley--Lieb algebras   
  \\ and combinatorial  Kazhdan--Lusztig theory  
 } 
   \maketitle 
  
\vspace{-0.3cm}
 \begin{abstract}
   We define oriented Temperley--Lieb algebras for Hermitian symmetric spaces. 
   This allows us to  explain the existence of closed combinatorial formulae for the  Kazhdan--Lusztig polynomials for these spaces. 
 \end{abstract}

\section{Introduction} 

To each  parabolic Coxeter system,  $(W,P)$,
  we have  an associated family of ``anti-spherical Kazhdan--Lusztig polynomials",  $n_{\la,\nu}(q)$, indexed by pairs of cosets for $P\leq W$.  
  These polynomials are 
some of the most important combinatorial objects in 
  Lie theory  and representation theory
    and they can be computed  (at least in theory) via a recursive, non-positive formula.  
  Deodhar 
 proposed a (non-recursive!) combinatorial  
 approach to studying  these polynomials in \cite{MR1065215}.

  Libedinsky--Williamson categorified the anti-spherical  Kazhdan--Lusztig polynomials by interpreting them  as composition factor  multiplicities of simple modules within standard modules for  the {\em anti-spherical Hecke category}, $\mathcal{H}_{(W,P)}$ \cite{antiLW}.      
   In more detail:  we first fix a reduced word $\underline{\mu}$ for each $\mu \in {^PW}$,
{\color{black}for $\la  \in {^PW}$ the  
standard $\mathcal{H}_{(W,P)}$-module  $\Delta(\la)$ has {\em light leaves basis}}   enumerated by    
$\cup_{\mu\in {^PW}}\Path(\la,\underline{\mu})$
 the set of all paths (or ``Bruhat strolls")   in the coset graph  for $(W,P)$ which  terminate at $\la$ such that the  steps in these paths are ``coloured by" $\underline{\mu}$. 
 This basis is graded according to the  degree statistic for the underlying paths, we record this in the matrix  
 \[
\Delta^{(W,P)}:=( \Delta_{\la,\underline{\mu}}(q))_{\la,\mu\in {{^P}W}}
\qquad \qquad
\Delta_{\la,\underline{\mu}}(q)=\sum_{\SSTS \in {\rm Path}(\la,\underline{\mu})}\!\!\!\!\!\! q^{\deg(\SSTS)}
\]
which is a  (square) lower uni-triangular matrix.  
This matrix can be factorised {\em uniquely} as a product of  
lower uni-triangular matrices 
\[N^{(W,P)}:=(n_{\la,\nu}(q))_{\la,\nu\in {{^P}W}}
\qquad 
B^{(W,P)}:=(b_{\nu,\underline{\mu}}(q))_{\nu,\mu\in {{^P}W}} 
\]
such that $n_{\la,\nu}(q)\in q\ZZ[q]$ 
for $\la \neq \nu$ 
and 
$b_{\nu,\underline{\mu}}(q)\in  \ZZ[q+q^{-1}]$. 
The polynomial  $n_{\la,\nu}(q)$ is the anti-spherical  Kazhdan--Lusztig polynomial for $\la \leq \nu \in {^PW}$.  Over the complex field,    the
polynomial  $n_{\la,\nu}(q)$ counts the graded composition factor multiplicity    $[\Delta(\la):L(\nu)]$ and the 
polynomials  $b_{\nu,\underline{\mu}}(q)$ describe the graded character 
of the simple module $L(\nu)$ \cite{MR3245013,antiLW}.  
This provides an {\em innately  positive} interpretation of the  coefficients of the polynomial  $n_{\la,\nu}(q)$ and thus proves the famous Kazhdan--Lusztig positivity conjecture \cite{MR3245013} and its anti-spherical counterpart \cite{antiLW}. 

Libedinsky--Williamson proposed that this extra $\mathcal{H}_{(W,P)}$-structure should provide new insight toward Deodhar's goal of a counting formula for the Kazhdan--Lusztig polynomials.   They    
ask in \cite[Problem 1.2]{MR4166587} whether it is possible to construct
  an explicit  
   basis of ``canonical light  leaves" for 
a $\ZZ$-module $\mathbb N_{\la,\underline\nu}$ 
whose graded rank   is equal to    $n_{\la,\nu}(q)$.  
Each canonical light leaf basis element of degree $k$ would  
then be a   generator of some composition factor  of $\Delta(\la)$ isomorphic to  $L(\nu)\langle k \rangle $.    
   We solve this problem in the case of  
Hermitian symmetric pairs (see \cref{coxeterlabelD2}) by introducing an oriented Temperley--Lieb algebra of type $(W,P)$  for all Hermitian symmetric pairs $(W,P)$.

\begin{thmA}\label{theoremA}
Let $(W,P)$ be a Hermitian symmetric pair. 
 For all $\la,  \nu  \in {^PW}$, the space 
 $\mathbb N_{\la,\underline\nu}$ has basis indexed  by the set of 
  ``standard" basis elements in the anti-spherical module for the oriented Temperley--Lieb algebra
  of type $(W,P)$.
 These elements can be described in a closed combinatorial (non-iterative) fashion.  
 Moreover, this construction is entirely independent of the choice of a reduced word $\underline{\nu}$. 
 \end{thmA}

    \begin{figure}[ht!]
\[\begin{tikzpicture}[scale=0.525]

\draw[very thick]( -8,3)--++(180:0.8)
coordinate(hi);
\draw[very thick]( -11,3)--++(0:0.8)
coordinate(hi2);
\draw[very thick,densely dotted] (hi) -- (hi2); 
\draw[very thick]( -8,3)--(-4,3); 
\draw[very thick]( -1,3)--++(180:0.8)
coordinate(hi);
\draw[very thick]( -4,3)--++(0:0.8)
coordinate(hi2);
\draw[very thick,densely dotted] (hi) -- (hi2); 
\draw[very thick]( -11,3)--(-13,3); 
\draw[very thick]( -1,3)--(1,3);

\draw[very thick, fill=white ] ( -13,3) coordinate  (hi) circle (6pt) ; 

\path(hi)--++(-90:0.62) node {\scalefont{0.9} $1$}; 

\draw[very thick, fill=white ] ( -11,3) coordinate (hi) circle (6pt); 

\path(hi)--++(-90:0.62) node {\scalefont{0.9} $2$}; 

\draw[very thick,  fill=white] ( -8,3) coordinate (hi) circle (6pt);  

\path(hi)--++(-90:0.62) node {\scalefont{0.9} $k\!-\!1$}; 

\draw[very thick, fill=magenta] ( -6,3) coordinate (hi) circle (6pt); 
 
 \path(hi)--++(-90:0.62) node {\scalefont{0.9} $k$}; 
 
\draw[very thick, fill=white ] ( -4,3) coordinate (hi) circle (6pt); 

 \path(hi)--++(-90:0.62) node {\scalefont{0.9} $k\!+\!1$}; 

\draw[very thick, fill=white ] ( -1,3) coordinate (hi) circle (6pt); 

 \path(hi)--++(-90:0.62) node {\scalefont{0.9} $n\!-\!1$}; 

\draw[very thick, fill=white ] ( 1,3) coordinate (hi) circle (6pt); 

 \path(hi)--++(-90:0.62) node {\scalefont{0.9} $n\!  $}; 
 \end{tikzpicture}\]
%
%
%
%
 \[
 \begin{minipage}{6cm}
 \begin{tikzpicture}[scale=0.525]
 \draw[very thick]( -1,3)--++(180:0.8); 
\draw[very thick]( -1,3)--++(180:0.8)
coordinate(hi);
\draw[very thick]( -4,3)--++(0:0.8)
coordinate(hi2);
\draw[very thick,densely dotted] (hi) -- (hi2); 
 \draw[very thick]( -1,3)--(1,3); 

 \draw[very thick] ( -6,3)--( -4,3); 

 \draw[very thick] ( -6,3.1)--( -8,3.1); 
 \draw[very thick] ( -6,2.9)--( -8,2.9);

\draw[very thick, fill=magenta ] ( -8,3) coordinate (hi) circle (6pt); 

 \path(hi)--++(-90:0.62) node {\scalefont{0.9} $\fuck$}; 

 \path(hi)--++(0:1) node {\scalefont{2}$\mathbf >$};

\draw[very thick, fill=white ] ( -6,3) coordinate (hi) circle (6pt); 

 \path(hi)--++(-90:0.62) node {\scalefont{0.9} $2$};

\draw[very thick, fill=white ] ( -4,3) coordinate (hi) circle (6pt); 

 \path(hi)--++(-90:0.62) node {\scalefont{0.9} $3$}; 

\draw[very thick, fill=white ] ( -1,3) coordinate (hi) circle (6pt); 

 \path(hi)--++(-90:0.62) node {\scalefont{0.9} $n\!-\!1$}; 

\draw[very thick, fill=white ] ( 1,3) coordinate (hi) circle (6pt); 

 \path(hi)--++(-90:0.62) node {\scalefont{0.9} $n   $}; 
 \end{tikzpicture}
   \end{minipage}\qquad \quad
 \begin{minipage}{6cm}
 \begin{tikzpicture}[scale=0.525]
 
\draw[very thick]( -1,3)--++(180:0.8); 
\draw[very thick]( -1,3)--++(180:0.8)
coordinate(hi);
\draw[very thick]( -4,3)--++(0:0.8)
coordinate(hi2);
\draw[very thick,densely dotted] (hi) -- (hi2); 
 \draw[very thick]( -1,3)--(1,3); 

 \draw[very thick] ( -6,3)--( -4,3); 

 \draw[very thick] ( -6,3.1)--( -8,3.1); 
 \draw[very thick] ( -6,2.9)--( -8,2.9);

\draw[very thick, fill=white] ( -8,3) coordinate (hi) circle (6pt); 

 \path(hi)--++(0:1) node {\scalefont{2}$\mathbf <$};

 \path(hi)--++(-90:0.62) node {\scalefont{0.9} $\fuck$};

\draw[very thick, fill=white ] ( -6,3) coordinate (hi) circle (6pt); 

 \path(hi)--++(-90:0.62) node {\scalefont{0.9} $2$};

\draw[very thick, fill=white ] ( -4,3) coordinate (hi) circle (6pt); 

 \path(hi)--++(-90:0.62) node {\scalefont{0.9} $3$}; 

\draw[very thick, fill=white ] ( -1,3) coordinate (hi) circle (6pt); 

 \path(hi)--++(-90:0.62) node {\scalefont{0.9} $n\!-\!1$}; 

\draw[very thick, fill=magenta  ] ( 1,3) coordinate (hi) circle (6pt); 

 \path(hi)--++(-90:0.62) node {\scalefont{0.9} $n    $}; 
 \end{tikzpicture}
 \end{minipage}
\]
\vspace{-0.1cm} 
 \[
  \begin{minipage}{8.25cm}
\begin{tikzpicture}[scale=0.525]
\clip(-9.5,0.5) rectangle  (3,5.5); 
\draw[white](-9.5,0.5) rectangle  (3,5.5); 
\draw[very thick]( -1,3)--++(180:0.8)
coordinate(hi);
\draw[very thick]( -4,3)--++(0:0.8)
coordinate(hi2);
\draw[very thick,densely dotted] (hi) -- (hi2); 
 \draw[very thick]( -1,3)--(1,3); 

 \draw[very thick](-8,4.5)--(-6,3);
 \draw[very thick](-8,1.5)--(-6,3);
  \draw[very thick](-4,3)--(-6,3);

\draw[very thick, fill=magenta ] (-8,4.5) coordinate (hi) circle (6pt); 

 \path(hi)--++(180:0.6) node {\scalefont{0.9}$ \fuckfuck $};

\draw[very thick, fill=white ] (-8,1.5) coordinate (hi) circle (6pt); 

 \path(hi)--++(180:0.6) node {\scalefont{0.9}$1$};

\draw[very thick, fill=white ] ( -6,3) coordinate (hi) circle (6pt); 

 \path(hi)--++(-90:0.62) node {\scalefont{0.9} $2$};

\draw[very thick, fill=white ] ( -4,3) coordinate (hi) circle (6pt); 

 \path(hi)--++(-90:0.62) node {\scalefont{0.9} $3$}; 

\draw[very thick, fill=white ] ( -1,3) coordinate (hi) circle (6pt); 

 \path(hi)--++(-90:0.62) node {\scalefont{0.9} $n\!-\!1$}; 

\draw[very thick, fill=white ] ( 1,3) coordinate (hi) circle (6pt); 

 \path(hi)--++(-90:0.62) node {\scalefont{0.9} $n   $}; 
 \end{tikzpicture}
 \end{minipage} 
   \begin{minipage}{7cm}
\begin{tikzpicture}[scale=0.525]
\clip(-9.5,0.5) rectangle  (3,5.5); 
\draw[white](-9.5,0.5) rectangle  (3,5.5); 
\draw[very thick]( -1,3)--++(180:0.8)
coordinate(hi);
\draw[very thick]( -4,3)--++(0:0.8)
coordinate(hi2);
\draw[very thick,densely dotted] (hi) -- (hi2); 
 \draw[very thick]( -1,3)--(1,3); 

 \draw[very thick](-8,4.5)--(-6,3);
 \draw[very thick](-8,1.5)--(-6,3);
  \draw[very thick](-4,3)--(-6,3);

\draw[very thick, fill=white ] (-8,4.5) coordinate (hi) circle (6pt); 

 \path(hi)--++(180:0.6) node {\scalefont{0.9}$ \fuckfuck    $};

\draw[very thick, fill=white ] (-8,1.5) coordinate (hi) circle (6pt); 

 \path(hi)--++(180:0.6) node {\scalefont{0.9}$1$};

\draw[very thick, fill=white ] ( -6,3) coordinate (hi) circle (6pt); 

 \path(hi)--++(-90:0.62) node {\scalefont{0.9} $2$};

\draw[very thick, fill=white ] ( -4,3) coordinate (hi) circle (6pt); 

 \path(hi)--++(-90:0.62) node {\scalefont{0.9} $3$}; 

\draw[very thick, fill=white ] ( -1,3) coordinate (hi) circle (6pt); 

 \path(hi)--++(-90:0.62) node {\scalefont{0.9} $n\!-\!1$}; 

\draw[very thick, fill=magenta] ( 1,3) coordinate (hi) circle (6pt); 

 \path(hi)--++(-90:0.62) node {\scalefont{0.9} $n   $}; 
 \end{tikzpicture}
 \end{minipage} 
 \]

 \vspace{-0.4cm}
\[ \qquad\qquad \begin{minipage}{8cm}
\begin{tikzpicture}[scale=0.525]
\draw[very thick]( -1,3)--++(180:0.8)
coordinate(hi);
\draw[very thick]( -3,3)--++(0:0.8)
coordinate(hi2);
\draw[very thick] (hi) -- (hi2); 
 \draw[very thick]( -1,3)--(1,3); 

   \draw[very thick](-3,3)--(-7,3);

   \draw[very thick]( -3,5)--++(-90:2); 

\draw[very thick, fill=magenta] ( -7,3) coordinate (hi) circle (6pt); 

 \path(hi)--++(-90:0.62) node {\scalefont{0.9}$1  $};

\draw[very thick, fill=white ] ( -5,3) coordinate (hi) circle (6pt); 

 \path(hi)--++(-90:0.62) node {\scalefont{0.9}$2$};

\draw[very thick, fill=white ] ( -3,3) coordinate (hi) circle (6pt); 

 \path(hi)--++(-90:0.62) node {\scalefont{0.9}$3$};

\draw[very thick, fill=white ] ( -3,5) coordinate (hi) circle (6pt); 

 \path(hi)--++(0:0.6) node {\scalefont{0.9}$6$};

\draw[very thick, fill=white ] ( -1,3) coordinate (hi) circle (6pt); 

 \path(hi)--++(-90:0.62) node {\scalefont{0.9}$4$}; 

\draw[very thick, fill=white ] ( 1,3) coordinate (hi) circle (6pt); 

 \path(hi)--++(-90:0.62) node {\scalefont{0.9}$5  $};

 \end{tikzpicture}
 \end{minipage}\!\!\!\!\!\!\!\!\!\! \begin{minipage}{8.25cm}
\begin{tikzpicture}[scale=0.525]
\draw[very thick]( -1,3)--++(180:0.8)
coordinate(hi);
\draw[very thick]( -3,3)--++(0:0.8)
coordinate(hi2);
\draw[very thick] (hi) -- (hi2); 
 \draw[very thick]( -1,3)--(3,3); 

   \draw[very thick](-3,3)--(-7,3);

   \draw[very thick]( -1,5)--++(-90:2); 

\draw[very thick, fill=magenta] ( -7,3) coordinate (hi) circle (6pt); 

 \path(hi)--++(-90:0.62) node {\scalefont{0.9}$1  $};

\draw[very thick, fill=white ] ( -5,3) coordinate (hi) circle (6pt); 

 \path(hi)--++(-90:0.62) node {\scalefont{0.9}$2$};

\draw[very thick, fill=white ] ( -3,3) coordinate (hi) circle (6pt); 

 \path(hi)--++(-90:0.62) node {\scalefont{0.9}$3$};

\draw[very thick, fill=white ] ( -1,5) coordinate (hi) circle (6pt); 

 \path(hi)--++(0:0.6) node {\scalefont{0.9}$7$};

\draw[very thick, fill=white ] ( -1,3) coordinate (hi) circle (6pt); 

 \path(hi)--++(-90:0.62) node {\scalefont{0.9}$4$}; 

\draw[very thick, fill=white ] ( 1,3) coordinate (hi) circle (6pt); 

 \path(hi)--++(-90:0.62) node {\scalefont{0.9}$5  $};

 \draw[very thick, fill=white ] ( 3,3) coordinate (hi) circle (6pt); 

 \path(hi)--++(-90:0.62) node {\scalefont{0.9}$6  $};

 \end{tikzpicture}
 \end{minipage} 
 \]
 
 \caption{ 
 Enumeration of nodes in the parabolic Dynkin diagram of the Hermitian symmetric pairs. 
 Namely, types $(A_{n },   A_{k-1} \times A_{n-k } )$, 
 $(C_n, A_{n-1})  $ and $(B_n, B_{n-1}) $, 
  $(D_{n+1} , A_{n})  $ and $(D_{n+1},  D_{n }) $ and $(E_6 , D_5 )  $ and $(E_7 , E_6) $ respectively.  
  The single node not belonging to the parabolic is highlighted in pink in each case.  
 }
\label{coxeterlabelD2}
\end{figure}

The use of Temperley--Lieb style combinatorics  for  calculating  Kazhdan--Lusztig polynomials goes back to work of Brundan and Stroppel  in type  
$(A_n, A_k\times A_{n-k-1})$ \cite{MR2918294,MR2600694,MR2781018,MR2955190} and Cox and De Visscher in type  $(D_n, A_{n-1})$ \cite{MR2813567}.  
 In this paper, we generalise these ideas to all Hermitian symmetric pairs 
 and lift the combinatorics  to a higher structural level; we do this by interpreting these
 polynomials as graded-dimensions of ``anti-spherical modules" for oriented Temperley--Lieb algebras of type $(W,P)$.
   The definition of these algebras and their  anti-spherical modules 
  is simple and uniformly given in terms of the underlying root system, see \cref{maindef,maindef2}.  We also relate these newly defined oriented Temperley--Lieb algebras of type $(W,P)$ to the generalised Temperley--Lieb algebras of type $W$ introduced by Fan and Graham. In Section 4, we give a proof of Theorem A for types $(D_n, D_{n-1})$, $(B_n, B_{n-1})$ and the exceptional types $(E_6, D_5)$ and $(E_7, E_6)$.  From Section 5 onwards we focus solely on the remaining types, namely $(W,P)=(A_n, A_k\times A_{n-k-1})$, $(D_n, A_{n-1})$ and $(C_n, A_{n-1})$. 
 In \cref{results1} we prove that the oriented  Temperley--Lieb algebras admit a diagrammatic 
visualisation, and use this in \cref{sectiongrading} to understand the graded structure of the algebra 
by way of  {\em closed combinatorial formulas}.  
In \cref{results2} we  apply these ideas to the anti-spherical module and hence prove Theorem A for the remaining types.

Further rewards of our approach will be harvested in the companion paper \cite{compan}.   In   \cite{compan}, we establish isomorphisms interrelating  Hecke categories   and use these isomorphisms in order  to construct the basic algebras of these Hecke categories   and prove that they   are  standard Koszul  
  (this uses the results of this paper in order to deduce   the required 
graded vector space dimension counts) and to prove that the $p$-Kazhdan--Lusztig polynomials are entirely independent of the prime $p\geq 0$.

\begin{Acknowledgements*}
 The first   and fourth  authors are grateful for funding from EPSRC  grant EP/V00090X/1 and the Royal Commission for the Exhibition of 1851, respectively.
 The authors are thankful to 
  Jon Brundan and Catharina Stroppel for interesting and informative conversations and to   the MFO workshop ``Character Theory and Categorification" for providing an excellent environment for collaboration.   

We are also  very thankful to the anonymous referee for their incredibly   detailed comments, which improved the exposition of the paper a great deal.
\end{Acknowledgements*}

\section{Kazhdan--Lusztig polynomials and Deodhar's defect}\label{reduced-defn}
  Let $(W, S_W)$ be a Coxeter system:  $W$ is the group generated by the finite set $S_W$ subject to the relations $(st)^{m_{st}} = 1$ for  
 $s,t\in S_W$, $ {m_{st}}\in \NN\cup\{\infty\}$ 
 satisfying ${m_{st}}= {m_{ts}}$, and ${m_{st}}=1$ if and only if  $s= t$.  
   Let $\ell : W \to \mathbb{N}$  be the corresponding length function. 
  Consider $S_P \subseteq S_W$ a  subset and $(P, S_P)$
 its corresponding Coxeter system. We say that $P$ is the parabolic subgroup corresponding to  $S_P\subseteq S_W$.
Let  $^ PW \subseteq W$ denote a set of minimal length coset representatives in $P\backslash W$. 
 For $\w=s_{i_1}s_{i_2}\cdots  s_{i_\ell}$ an expression in the generators $s_{i_j} \in S_W$   for $0\leq j \leq \ell$, we define a subexpression  of $\underline{w}$ to be an expression of the form  $\w^{\underline{k}}  :=s_{i_1}^{k_1}s_{i_2}^{k_2}\cdots s_{i_\ell}^{k_\ell}$ where
$\underline{k}=(k_1,k_2,\dots ,k_\ell)\in\{0,1\}^\ell$. 
 We let $\leq $ denote the (strong)  Bruhat order on ${^PW}$: namely $y\leq w$ if for some  
reduced expression $\w$ for $w$, there exists a reduced expression  $\y$ for $y$ such that  $\y$ is a subexpression of $\w$.   

We define a directed graph $\mathcal{G}_{(W,P)}$ with vertex set ${^PW}$ and edges defined as follows. For $\la, \mu\in {^PW}$ we have an edge $\la \rightarrow \mu$ if $\mu = \la s_i > \la$ for some $s_i\in S_W$. (Note that this is the Hasse diagram of the poset $({^PW}, \leq_r)$ where $\leq_r$ denotes the (weak) right Bruhat order). Examples are given in \cref{Bruhatexample-KL,howTworks}. 

The identity element $1 \in W$ is the minimal coset representative of the identity coset $P$, and for convenience, we will denote it by $\varnothing$ instead (the empty word in the generators).

\begin{figure}[ht!]
 \[
 \begin{tikzpicture} [yscale=0.51,xscale=-0.51]

\draw[magenta, line width=3] (2,0.5)--(2,3) coordinate (hi);

\draw[cyan, line width=3] (hi)--(6,5.5) coordinate (hi2);
\draw[darkgreen, line width=3] (hi)--(-2,5.5) coordinate (hi3);
\draw[cyan, line width=3] (hi3)--(2,8) coordinate (hi4);
\draw[darkgreen, line width=3] (hi2)--(2,8) coordinate (hi4);
  \draw[magenta, line width=3] (hi4)--++(90:2.5) coordinate (hi7);

\path (hi7) --++(180:2)coordinate  (origin); 
 \draw[fill=white,rounded corners](origin) rectangle ++(4,1) 
 node [midway] {\scalefont{0.9}$\color{magenta}s_2\color{darkgreen}s_3  \color{cyan}s_1 \color{magenta}s_2 $};

\path (0,0)  coordinate (origin) ; 
\draw[fill=white,rounded corners](origin) rectangle ++(4,1)  node [midway] {$\varnothing$};

\path (0,2.5) coordinate (origin) ; 
\draw[fill=white,rounded corners](origin) rectangle ++(4,1) 
 node [midway] {\scalefont{0.9}$\color{magenta}s_2 $};

\path (-4,5) coordinate (origin); 
 
\draw[fill=white,rounded corners](origin) rectangle ++(4,1) 
 node [midway] {\scalefont{0.9}$\color{magenta}s_2\color{darkgreen}s_3 $};

\path (4,5) coordinate (origin); 
 
\draw[fill=white,rounded corners](origin) rectangle ++(4,1) 
 node [midway] {\scalefont{0.9}$\color{magenta}s_2\color{cyan}s_1 $};

\path (0,7.5) coordinate (origin); 
 
\draw[fill=white,rounded corners](origin) rectangle ++(4,1) 
 node [midway] {\scalefont{0.9}$\color{magenta}s_2\color{darkgreen}s_3 \color{cyan}s_1 $};

\draw[very thick]( -7+1.8,3-1)--(-4+1.8,3-1); 
 
\draw[very thick, fill=darkgreen] ( -7+1.8,3-1) coordinate circle (7pt);  

\draw[very thick, fill=magenta] ( -5.5+1.8,3-1) coordinate circle (7pt); 
 
\draw[very thick, fill=cyan] ( -4+1.8,3-1) coordinate circle (7pt); 

%

\draw[very thick, rounded corners] (-6.5+1.8,2.5-1) rectangle (-7.5+1.8,3.5-1); 
\draw[very thick, rounded corners] (-4.5+1.8,2.5-1) rectangle (-3.5+1.8,3.5-1);

\end{tikzpicture}
\quad
\begin{tikzpicture} [yscale=0.51,xscale=-0.51]

\draw[magenta, line width=3] (2,0.5)--(2,3) coordinate (hi);

\draw[cyan, line width=3] (hi)--(6,5.5) coordinate (hi2);
\draw[darkgreen, line width=3] (hi)--(-2,5.5) coordinate (hi3);
\draw[cyan, line width=3] (hi3)--(2,8) coordinate (hi4);
\draw[darkgreen, line width=3] (hi2)--(2,8) coordinate (hi4);
\draw[orange, line width=3] (hi3)--(-6,8) coordinate (hi5);
\draw[cyan, line width=3] (hi5)--(-2,10.5) coordinate (hi6);
\draw[orange, line width=3] (hi4)--(-2,10.5) coordinate (hi6);
\draw[magenta, line width=3] (hi4)--(6,10.5) coordinate (hi7);

\draw[magenta, line width=3] (hi6)--(2,13) coordinate (hi8);
\draw[orange, line width=3] (hi7)--(2,13) coordinate (hi8);
\draw[darkgreen, line width=3] (hi8)--(2,15) ;

\path (0,0)  coordinate (origin) ; 
\draw[fill=white,rounded corners](origin) rectangle ++(4,1)  node [midway] {$\varnothing$};

\path (0,2.5) coordinate (origin) ; 
\draw[fill=white,rounded corners](origin) rectangle ++(4,1) 
 node [midway] {\scalefont{0.9}$\color{magenta}s_2 $};

\path (-4,5) coordinate (origin); 
 
\draw[fill=white,rounded corners](origin) rectangle ++(4,1) 
 node [midway] {\scalefont{0.9}$\color{magenta}s_2\color{darkgreen}s_3 $};

\path (4,5) coordinate (origin); 
 
\draw[fill=white,rounded corners](origin) rectangle ++(4,1) 
 node [midway] {\scalefont{0.9}$\color{magenta}s_2\color{cyan}s_1 $};

\path (0,7.5) coordinate (origin); 
 
\draw[fill=white,rounded corners](origin) rectangle ++(4,1) 
 node [midway] {\scalefont{0.9}$\color{magenta}s_2\color{darkgreen}s_3 \color{cyan}s_1 $};

\path (-8,7.5) coordinate (origin); 
 
\draw[fill=white,rounded corners](origin) rectangle ++(4,1) 
 node [midway] {\scalefont{0.9}$\color{magenta}s_2\color{darkgreen}s_3 \color{orange}s_4 $};

\path (-4,10) coordinate (origin); 
 \draw[fill=white,rounded corners](origin) rectangle ++(4,1) 
 node [midway] {\scalefont{0.9}$\color{magenta}s_2\color{darkgreen}s_3\color{orange}s_4  \color{cyan}s_1 $};

\path (4,10) coordinate (origin); 
 \draw[fill=white,rounded corners](origin) rectangle ++(4,1) 
 node [midway] {\scalefont{0.9}$\color{magenta}s_2\color{darkgreen}s_3  \color{cyan}s_1 \color{magenta}s_2 $};

\path (0,12.5) coordinate (origin);

\draw[fill=white,rounded corners](origin) rectangle ++(4,1) 
 node [midway] {\scalefont{0.9}$\color{magenta}s_2\color{darkgreen}s_3\color{orange}s_4  \color{cyan}s_1 \color{magenta}s_2$};

\path (0,15) coordinate (origin);

\draw[fill=white,rounded corners](origin) rectangle ++(4,1) 
 node [midway] {\scalefont{0.9}$\color{magenta}s_2\color{darkgreen}s_3\color{orange}s_4  \color{cyan}s_1 \color{magenta}s_2\color{darkgreen}s_3$};

%
%
%
%
%
%

\draw[very thick]( -8.5+1.8,3-1)--(-4+1.8,3-1); 
 
\draw[very thick, fill=darkgreen] ( -7+1.8,3-1) coordinate circle (7pt);  

\draw[very thick, fill=magenta] ( -5.5+1.8,3-1) coordinate circle (7pt); 
 
\draw[very thick, fill=cyan] ( -4+1.8,3-1) coordinate circle (7pt); 

%
 \draw[very thick, fill=orange] ( -8.5+1.8,3-1) coordinate circle (7pt);

\draw[very thick, rounded corners] (-6.5+1.8,2.5-1) rectangle (-9+1.8,3.5-1); 
\draw[very thick, rounded corners] (-4.5+1.8,2.5-1) rectangle (-3.5+1.8,3.5-1);

\end{tikzpicture}
\]

\caption{The  graph   $\mathcal{G}_{(W,P)}$ for  $(W,P)=(A_3,A_1 \times A_1  )$ and $(A_4,A_1 \times A_2  )$ respectively. (We haven't drawn the direction on the edges but all arrows are pointing upwards). } 
\label{Bruhatexample-KL}
\end{figure} 
 
We now define $\widehat{\mathcal{G}}_{(W,P)}$  to be the directed graph having the same set of vertices as $\mathcal{G}_{(W,P)}$ but replacing each edge in $\mathcal{G}_{(W,P)}$  between $\lambda$ and $\lambda s_i$ by four directed edges 
\[\lambda \xrightarrow{ \ s_i \ }\lambda ,  \quad \lambda s_i  \xrightarrow{ \ s_i \ }\lambda s_i,
\quad  \lambda \xrightarrow{ \ s_i \ }\lambda s_i , \quad   \lambda s_i \xrightarrow{ \ s_i \ }\lambda \]
and, in order to keep the notation to a minimum, we will simply  label the edge
 by the subscript of the reflection (not the reflection itself). 
  We assign a degree to each edge in $\widehat{\mathcal{G}}_{(W,P)}$ by setting 
\[{\rm deg}(\lambda \xrightarrow{ \ i \ } \lambda s_i) = {\rm deg}(\lambda s_i  \xrightarrow{ \ i \ } \lambda) = 0  
\qquad {\rm deg}(\lambda \xrightarrow{ \ i \ } \lambda) = \left\{ \begin{array}{ll} 1 & \mbox{if $\lambda s_i > \lambda$}\\ -1 & \mbox{if $\lambda s_i < \lambda$}\end{array}\right.\]
 Given   a path   (or ``Bruhat stroll") on $\widehat{\mathcal{G}}_{(W,P)}$
\[\SSTT \, : \, \lambda_1 \xrightarrow{i_1} \lambda_2 \xrightarrow{i_2} \lambda_3 \xrightarrow{i_3} \ldots \xrightarrow{i_{k-1}} \lambda_k,\]
we say that the degree  ${\rm deg}(\SSTT)$ is the sum of the degrees of each edge in $\SSTT$. 
(The degree is also sometimes known as the ``Deodhar defect".)  
We also define the {\sf weight of $\SSTT$}, denoted by $\omega(\SSTT)$ to be the expression 
\[\omega(\SSTT) :=  {s_{i_1}s_{i_2}s_{i_3} \ldots s_{i_{k-1}}}.\]

We write ${\rm Path}_{(W,P)}$ for the set of all paths on $\widehat{\mathcal{G}}_{(W,P)}$. For $\la , \nu \in {^PW}$, we let $\Path (\la\to \nu)$ denote the set of all paths in ${\rm Path}_{(W,P)}$ beginning at
  $\la$ and ending at $\nu$. When $\la = \varnothing$ 
 we set $\Path ( \nu):=\Path (\varnothing\to \nu)$. 
Let $\underline{w}$ be an expression in the generators $S_W$. We set   $\Path (\la \to \nu,\w)$ to be the set of paths $\SSTT\in {\rm Path}(\la \to \nu)$ with $\omega(\SSTT) = \underline{w}$. When $\la = \varnothing$ we set ${\rm Path}(\nu , \underline{w}) := {\rm Path}(\varnothing \to \nu, \underline{w})$. 
 
 Throughout the paper, we fix one reduced expression $\underline{\mu}$ for each $\mu \in {^PW}$.  The set of paths ${\rm Path}(\la, \underline{\mu})$ for $\la, \mu \in {^PW}$ will play a crucial role. Examples of such paths are  given in \cref{pppspspspsps}.
We will see in particular that,  for Hermitian symmetric pairs, the set $
\Path (\la,\underline{\mu})$ consists of either 0 or 1 elements.

\begin{figure}[ht!]
 \[ \scalefont{0.7}
\begin{tikzpicture} [scale=0.5]

\draw[magenta, line width=3] (2,-5)--(2,-1.5)  ;

\draw[darkgreen, line width=3] (2,0.5)--(2,-1.5)  ;

\draw[cyan, line width=3] (2,0.5)--(2,3) coordinate (hi);

\draw[orange, line width=3] (hi)--(5,5.5) coordinate (hi2);
\draw[gray, line width=3] (hi)--(-1,5.5) coordinate (hi3);
\draw[orange, line width=3] (hi3)--(2,8) coordinate (hi4);
\draw[gray, line width=3] (hi2)--(2,8) coordinate (hi4);
  
\draw[cyan, line width=3] (hi4)--++(90:2) coordinate (hi5);
\draw[darkgreen, line width=3] (hi5)--++(90:2.5) coordinate (hi6);

\draw[magenta, line width=3] (hi6)--++(90:2.5) ;

\path (0,-5) --++(-90:0.1)--++(180:0.252)coordinate (origin); 
\draw[fill=white,rounded corners](origin) rectangle ++(4.55,1.2) node [midway] {$\varnothing$};

\path (0,-2.5) --++(-90:0.1)--++(180:0.252)coordinate (origin); 
 
 \draw[fill=white,rounded corners](origin) rectangle ++(4.55,1.2) node [midway] {$\color{magenta}s_4$};

\path (0,0) --++(-90:0.1)--++(180:0.252)coordinate (origin); 

 \draw[fill=white,rounded corners](origin) rectangle ++(4.55,1.2) node [midway] {$
 \color{magenta}s_4
   \color{green!80!black}s_{3}
 $};

\path (0,2.5) --++(-90:0.1)--++(180:0.252)coordinate (origin); 

 \draw[fill=white,rounded corners](origin) rectangle ++(4.55,1.2) node [midway] {$
 \color{magenta}s_4
   \color{green!80!black}s_{3}
   \color{black} 
     \color{cyan}s_2
 $};

\path (-3,5) --++(-90:0.1) coordinate (origin); 
\draw[fill=white,rounded corners](origin) rectangle ++(4,1.2);

 \draw[fill=white,rounded corners](origin) rectangle ++(4.55,1.2) node [midway] {$
 \color{magenta}s_4
   \color{green!80!black}s_{3}
   \color{black} 
     \color{cyan}s_2
          \color{gray}s_1
 $};

\path (3,5) --++(-90:0.1) coordinate (origin);

 \draw[fill=white,rounded corners](origin) rectangle ++(4,1.2) node [midway] {$
 \color{magenta}s_4
   \color{green!80!black}s_{3}
   \color{black} 
     \color{cyan}s_2
          \color{orange}s_{1\!'\!'}
 $};

\path (0,7.5) --++(-90:0.1)--++(180:0.252)coordinate (origin);

 \draw[fill=white,rounded corners](origin) rectangle ++(4.55,1.2) node [midway] {$
 \color{magenta}s_4
   \color{green!80!black}s_{3}
   \color{black} 
     \color{cyan}s_2
      \color{gray}s_1       \color{orange}s_{1\!'\!'}
 $};

\path (0,10) --++(-90:0.1)--++(180:0.252)coordinate (origin);

 \draw[fill=white,rounded corners](origin) rectangle ++(4.55,1.2) node [midway] {$
 \color{magenta}s_4
   \color{green!80!black}s_{3}
   \color{black} 
     \color{cyan}s_2
      \color{gray}s_1       \color{orange}s_{1\!'\!'}     \color{cyan}s_2
 $};

\path (0,12.5) --++(-90:0.1)--++(180:0.252)coordinate (origin);

 \draw[fill=white,rounded corners](origin) rectangle ++(4.55,1.2) node [midway] {$
 \color{magenta}s_4
   \color{green!80!black}s_{3}
   \color{black} 
     \color{cyan}s_2
      \color{gray}s_1       \color{orange}s_{1\!'\!'}     \color{cyan}s_2   \color{green!80!black}s_{3}
 $};

\path (0,15) --++(-90:0.1)--++(180:0.35)coordinate (origin);

 \draw[fill=white,rounded corners](origin) rectangle ++(4+0.35*2,1.2) node [midway] {$
 \color{magenta}s_4
   \color{green!80!black}s_{3}
   \color{black} 
     \color{cyan}s_2
      \color{gray}s_1       \color{orange}s_{1\!'\!'}  \!   \color{cyan}s_2   \color{green!80!black}s_{3} \color{magenta}s_4
 $};

\path ( -2,0) coordinate (himid);

\draw[very thick] (himid)--++(180:2) coordinate (hi3); 
 
\draw[very thick ] (himid) --++(60:2) coordinate (hi2); 

\draw[very thick ] ( himid) --++(-60:2) coordinate (hi);

\draw[very thick] (hi3)--++(180:2) coordinate (hi4) ; 

\draw[very thick, fill=magenta] (hi4)  coordinate circle (7pt);

\draw[very thick, fill=orange] (hi2)  coordinate circle (7pt); 

\draw[very thick, fill=gray] (hi)  coordinate circle (7pt); 
 \draw[very thick, fill=darkgreen] ( hi3) coordinate circle (7pt); 

\draw[very thick, fill=cyan] (himid) coordinate circle (7pt);


\path (hi3)--++(180:0.5) coordinate (hi3TLX); 
\path (hi3)--++(120:0.5) coordinate (hi3TL);
\path (himid)--++(-120:0.5) coordinate (himidBL);
\path (himid)--++(0:0.5) coordinate (himidBL2);
\path (hi)--++(-120:0.5) coordinate (X);
\path (hi)--++(-60:0.5) coordinate (Y);
\path (hi)--++(0:0.5) coordinate (Z);
\path (hi2)--++(0:0.5) coordinate (hi2L);
\path (hi2)--++(60:0.5) coordinate (hi2LX);
\path (hi2)--++(120:0.5) coordinate (hi2T);
\path (himid)--++(120:0.5) coordinate (himidT);
\path (hi3)--++(-120:0.5) coordinate (hi3BL);
\path (hi3)--++(-120:0.5)--++(0:0.3) coordinate (hi3BL2);

\draw[thick, rounded corners,smooth] (hi3BL2)--(himidBL)
--(X)--(Y)--(Z)--(himidBL2)
--(hi2L)--(hi2LX)--(hi2T)--(himidT)--(hi3TL)--(hi3TLX)-- (hi3BL)--++(0:1); 
 \end{tikzpicture} \qquad
 \begin{tikzpicture} [scale=0.51]

%
%
%
%
%
%
%

\draw[magenta, line width=3] (2,-5)--(2,-1.5)  ;

\draw[darkgreen, line width=3] (2,0.5)--(2,-1.5)  ;

\draw[cyan, line width=3] (2,0.5)--(2,3) coordinate (hi);

 \draw[orange, line width=3] (hi)--(2,5.5) coordinate (hi2);
  
\draw[cyan, line width=3] (hi2)--++(90:2) coordinate (hi5);
\draw[darkgreen, line width=3] (hi5)--++(90:2.5) coordinate (hi6);

\draw[magenta, line width=3] (hi6)--++(90:2.5) ;

\path (0,-5) --++(-90:0.1)--++(180:0.252)coordinate (origin); 
\draw[fill=white,rounded corners](origin) rectangle ++(4.55,1.2) node [midway] {$\varnothing$};

  \path (0,-2.5) --++(-90:0.1)--++(180:0.252)coordinate (origin); 
 \draw[fill=white,rounded corners](origin) rectangle ++(4.55,1.2) node [midway] {$\color{magenta}s_4$};

%
%
%
%

\path (0,0) --++(-90:0.1)--++(180:0.252)coordinate (origin); 

 \draw[fill=white,rounded corners](origin) rectangle ++(4.55,1.2) node [midway] {$
 \color{magenta}s_4
   \color{green!80!black}s_{3}
 $};

\path (0,2.5) --++(-90:0.1)--++(180:0.252)coordinate (origin); 

 \draw[fill=white,rounded corners](origin) rectangle ++(4.55,1.2) node [midway] {$
 \color{magenta}s_4
   \color{green!80!black}s_{3}
   \color{black} 
     \color{cyan}s_2
 $};

\path (0,5) --++(-90:0.1)--++(180:0.252)coordinate (origin); 
  \draw[fill=white,rounded corners](origin) rectangle ++(4.55,1.2) node [midway] {$
 \color{magenta}s_4
   \color{green!80!black}s_{3}
   \color{black} 
    \color{cyan}s_2    \color{orange}s_{1'}
 $};

\path (0,7.5) --++(-90:0.1)--++(180:0.252)coordinate (origin);

   \draw[fill=white,rounded corners](origin) rectangle ++(4.55,1.2) node [midway] {$
  \color{magenta}s_4
   \color{green!80!black}s_{3}
   \color{black} 
    \color{cyan}s_2    \color{orange}s_{1'}
        \color{cyan}s_2   
 $};

\path (0,10) --++(-90:0.1)--++(180:0.252)coordinate (origin);

   \draw[fill=white,rounded corners](origin) rectangle ++(4.55,1.2) node [midway] {$
     \color{magenta}s_4
   \color{green!80!black}s_{3}
   \color{black} 
    \color{cyan}s_2    \color{orange}s_{1'} \color{cyan}s_2    \color{green!80!black}s_{3} 
 $};

\path (0,12.5) --++(-90:0.1)--++(180:0.252)coordinate (origin);

   \draw[fill=white,rounded corners](origin) rectangle ++(4.55,1.2) node [midway] {$
     \color{magenta}s_4
   \color{green!80!black}s_{3}
   \color{black} 
    \color{cyan}s_2    \color{orange}s_{1'} \color{cyan}s_2    \color{green!80!black}s_{3} 
     \color{magenta}s_4 $};

\path (0,15) --++(-90:0.1)--++(180:0.252)coordinate (origin); 
\fill[white,rounded corners](origin) rectangle ++(4.2,1.2);

\path ( -2.75,0) coordinate (himid);

\draw[very thick] (himid)--++(180:2) coordinate (hi3); 
 
\path (himid) --++(0:2) coordinate (hi2); 

\draw[very thick ] (hi2)--++(90:0.1)   --++(180:2) 
--++(-90:0.2)   --++(0: 2)  ; 

\draw[very thick] (hi3)--++(180:2) coordinate (hi6);

\draw[very thick, fill=magenta] (hi6)  coordinate circle (7pt);

\draw[very thick, fill=orange] (hi2)  coordinate circle (7pt); 

 \draw[very thick, fill=darkgreen] ( hi3) coordinate circle (7pt); 

\draw[very thick, fill=cyan] (himid) coordinate circle (7pt);


\path (hi3)--++(180:0.5) coordinate (hi3TLX); 
\path (hi3)--++(120:0.5) coordinate (hi3TL);
\path (hi2)--++(-60:0.5) coordinate (himidBR);
\path (hi2)--++(-60:0.5) coordinate (hi2L);
\path (hi2)--++(0:0.5) coordinate (hi2LX);
\path (hi2)--++(60:0.5) coordinate (hi2T);
\path (hi2)--++(120:0.5) coordinate (himidT);
\path (hi3)--++(-120:0.5) coordinate (hi3BL);
\path (hi3)--++(-120:0.5)--++(0:0.3) coordinate (hi3BL2);

\draw[thick, rounded corners,smooth] (hi3BL2) --(hi2L)--(hi2LX)--(hi2T)--(himidT)--(hi3TL)--(hi3TLX)-- (hi3BL)--++(0:1);  \end{tikzpicture}
\]
\caption{ The graph $\mathcal{G}_{(W,P)}$ for types  $(W,P)=(D_5, D_4)$ 
	and $(B_4 , B_3)$. The general case {$(D_{n+1},D_n)$ and $(B_n,B_{n-1})$ }is no more difficult (see  \cite[Section 1]{compan}) -- merely extend the top and bottom {vertical chains} of the graph.}  
\label{howTworks}
\end{figure}

\begin{figure} [H]
\[\qquad\quad
\begin{tikzpicture} [yscale=0.425,xscale=-0.425]
 
\path[magenta, line width=3] (2,0.5)--(2,3) coordinate (hi);

\path[cyan, line width=3] (hi)--(6,5.5) coordinate (hi2);
\path[darkgreen, line width=3] (hi)--(-2,5.5) coordinate (hi3);
\path[cyan, line width=3] (hi3)--(2,8) coordinate (hi4);
\path[darkgreen, line width=3] (hi2)--(2,8) coordinate (hi4);
\path[orange, line width=3] (hi3)--(-6,8) coordinate (hi5);
\path[cyan, line width=3] (hi5)--(-2,10.5) coordinate (hi6);
\path[orange, line width=3] (hi4)--(-2,10.5) coordinate (hi6);
\path[magenta, line width=3] (hi4)--(6,10.5) coordinate (hi7);

\path[magenta, line width=3] (hi6)--(2,13) coordinate (hi8);
\path[orange, line width=3] (hi7)--(2,13) coordinate (hi8);
\path[darkgreen, line width=3] (hi8)--(2,15.5) ;

\path (0,0) coordinate (origin); 



\draw[magenta, line width=3] (2,0.5)--(2,3) coordinate (hi);

\path (0,0) coordinate (origin); 



\path (0,2.5) coordinate (origin); 



\draw[darkgreen, line width=3] (hi)--(-2,5.5) coordinate (hi3);
 
\path (0,2.5) coordinate (origin); 



\path (-4,5) coordinate (origin); 



\draw[cyan, line width=3] (hi)--(6,5.5) coordinate (hi2);

\path (0,2.5) coordinate (origin); 



\path (4,5) coordinate (origin); 



\draw[cyan, line width=3] (hi3)--(2,8) coordinate (hi4);
\draw[darkgreen, line width=3] (hi2)--(2,8) coordinate (hi4);

\path (-4,5) coordinate (origin); 



\path (4,5) coordinate (origin); 



\path (0,7.5) coordinate (origin); 



\draw[orange, line width=3] (hi3)--(-6,8) coordinate (hi5);

\path (-4,5) coordinate (origin); 



\path (4,5) coordinate (origin); 



\path (-8,7.5) coordinate (origin); 



\draw[cyan, line width=3] (hi5)--(-2,10.5) coordinate (hi6);
\draw[orange, line width=3] (hi4)--(-2,10.5) coordinate (hi6);

\path (-8,7.5) coordinate (origin); 



\path (0,7.5) coordinate (origin); 



\path (-4,10) coordinate (origin); 



 \draw[magenta, line width=3] (hi4)--(6,10.5) coordinate (hi7);

\path (-8,7.5) coordinate (origin); 



\path (0,7.5) coordinate (origin); 



\path (4,10) coordinate (origin); 



\draw[magenta, line width=3] (hi6)--(2,13) coordinate (hi8);
\draw[orange, line width=3] (hi7)--(2,13) coordinate (hi8);

\path (-4,10) coordinate (origin); 



\path (4,10) coordinate (origin); 



\path (0,12.5) coordinate (origin); 



\draw[darkgreen, line width=3] (hi8)--(2,15) ; 

\path (0,12.5) coordinate (origin); 



\path (0,15) coordinate (origin); 




\fill(hi)circle (8pt);

     \draw[darkgreen, line width=3] (hi8)--++(90:2.5)  ;  
  \path (hi8)--++(90:2.5) coordinate (ghfj)
;  \fill  (ghfj)
circle (8pt);

      \draw[line width=3, wei ,zigzag ] (2,0.5)--++(90:2.5)--(-2,5.5)
 --(-6,8) --(-2,10.5)--(2,13)--++(90:2.5)
 coordinate(hi);

\draw[fill=black] (hi)circle  (22pt); 

\foreach \i in {2,3,...,8}
{\fill(hi\i)circle (8pt);	}

\draw  (hi)
node { \color{white}$  \alpha$};

\fill(2,0.5)circle (8pt);
  \fill(2,3)circle (8pt);

 \path (2,15.5) circle (22pt);
\end{tikzpicture}\qquad
\quad\begin{tikzpicture} [yscale=0.425,xscale=-0.425]
 
\path[magenta, line width=3] (2,0.5)--(2,3) coordinate (hi);

\path[cyan, line width=3] (hi)--(6,5.5) coordinate (hi2);
\path[darkgreen, line width=3] (hi)--(-2,5.5) coordinate (hi3);
\path[cyan, line width=3] (hi3)--(2,8) coordinate (hi4);
\path[darkgreen, line width=3] (hi2)--(2,8) coordinate (hi4);
\path[orange, line width=3] (hi3)--(-6,8) coordinate (hi5);
\path[cyan, line width=3] (hi5)--(-2,10.5) coordinate (hi6);
\path[orange, line width=3] (hi4)--(-2,10.5) coordinate (hi6);
\path[magenta, line width=3] (hi4)--(6,10.5) coordinate (hi7);

\path[magenta, line width=3] (hi6)--(2,13) coordinate (hi8);
\path[orange, line width=3] (hi7)--(2,13) coordinate (hi8);
\path[darkgreen, line width=3] (hi8)--(2,15.5) ;

\path (0,0) coordinate (origin); 



\draw[magenta, line width=3] (2,0.5)--(2,3) coordinate (hi);

\path (0,0) coordinate (origin); 



\path (0,2.5) coordinate (origin); 



\draw[darkgreen, line width=3] (hi)--(-2,5.5) coordinate (hi3);
 
\path (0,2.5) coordinate (origin); 



\path (-4,5) coordinate (origin); 



\draw[cyan, line width=3] (hi)--(6,5.5) coordinate (hi2);

\path (0,2.5) coordinate (origin); 



\path (4,5) coordinate (origin); 



\draw[cyan, line width=3] (hi3)--(2,8) coordinate (hi4);
\draw[darkgreen, line width=3] (hi2)--(2,8) coordinate (hi4);

\path (-4,5) coordinate (origin); 



\path (4,5) coordinate (origin); 



\path (0,7.5) coordinate (origin); 



\draw[orange, line width=3] (hi3)--(-6,8) coordinate (hi5);

\path (-4,5) coordinate (origin); 



\path (4,5) coordinate (origin); 



\path (-8,7.5) coordinate (origin); 



\draw[cyan, line width=3] (hi5)--(-2,10.5) coordinate (hi6);
\draw[orange, line width=3] (hi4)--(-2,10.5) coordinate (hi6);

\path (-8,7.5) coordinate (origin); 



\path (0,7.5) coordinate (origin); 



\path (-4,10) coordinate (origin); 



 \draw[magenta, line width=3] (hi4)--(6,10.5) coordinate (hi7);

\path (-8,7.5) coordinate (origin); 



\path (0,7.5) coordinate (origin); 



\path (4,10) coordinate (origin); 



\draw[magenta, line width=3] (hi6)--(2,13) coordinate (hi8);
\draw[orange, line width=3] (hi7)--(2,13) coordinate (hi8);

\path (-4,10) coordinate (origin); 



\path (4,10) coordinate (origin); 



\path (0,12.5) coordinate (origin); 



\draw[darkgreen, line width=3] (hi8)--(2,15) ; 

\path (0,12.5) coordinate (origin); 



\path (0,15) coordinate (origin); 



  
\draw[line width=3  ,zigzag,wei]  (2,0.5)--++(90:2.5) coordinate(hi);
    
\draw[line width=3, wei ,zigzag ] (2,0.5)--++(90:2.5)--(-2,5.5)
coordinate(hi);
    
\draw[line width=3  ,wei ] (hi) to
 [out=180,in=-120] 
(-4,6.75) 
 to [out=60,in=90] (-2,5.5) 
;
   
 \draw[line width=3 ,wei ,zigzag ] (-2,5.5) --(2,8);

  \draw[line width=3 ,wei ]   (2,8)  
to [out=0,in=-60] 
(4,9.25) 
  to [out=120,in=90] (2,8) 
; 
\fill(hi)circle (8pt);
   
  \draw[line width=3 ,wei,zigzag ]   (2,8)--(6,5.5)  
  coordinate (hi);

     \draw[darkgreen, line width=3] (hi8)--++(90:2.5)  ;\foreach \i in {2,3,...,8}
{\fill(hi\i)circle (8pt);	}
 
  \path (hi8)--++(90:2.5) coordinate (ghfj)
;  \fill  (ghfj)
circle (8pt);

\draw[fill=black] (hi)circle  (22pt);

\draw  (hi)
node { \color{white}$  \beta$};

\fill(2,0.5)circle (8pt);
  \fill(2,3)circle (8pt);

 \path (2,15.5) circle (22pt);
\end{tikzpicture}
\]  
\caption{On the left we depict the unique path in ${\rm Path}(\alpha, \underline{\alpha})$ corresponding with a choice of reduced word $\underline{\alpha}$ and on the right we depict 
the unique element of ${\rm Path}(\beta,\underline{\alpha})$ for $\alpha=\color{magenta}s_2\color{darkgreen}s_3\color{orange}s_4  \color{cyan}s_1 \color{magenta}s_2\color{darkgreen}s_3$ and 
$\beta=\color{magenta}s_2   \color{cyan}s_1  $.  
These are paths  on $\widehat{\mathcal{G}}_{ {(A_4, A_1\times A_2)} } $   (also known as ``Bruhat strolls") but we depict only the edges in $\mathcal{G}_{ {(A_4,A_1\times A_2)} }$ (for readability).
  }
\label{pppspspspsps}
\end{figure}

%
%
%

    The following path theoretic  definition of Kazhdan--Lusztig polynomials was for a long time talked about implicitly in the literature, see for example \cite{MR1065215} (in particular Proposition 3.5 and Section 4, and also Section 5 for the parabolic setting).  It is explicitly proven to be equivalent to the classical definition of these polynomials in \cite[Proposition 3.3]{Soergel}.

   \begin{defn}\label{troll}
Let $\la,\mu \in {{^P}W}$.  
We set 
$b_{\la,\underline{\la}}(q)=1=n_{\la,\la}(q)$.  
For $\la\neq \mu$, we recursively define the polynomials
\[
b_{\la,\underline{\mu}}(q)\in \ZZ [q+q^{-1}] \qquad
n_{\la,\mu}(q)\in q\ZZ [q]
\]
  by induction on the Bruhat order $\leq$ as follows
\begin{equation}\label{kl-pol-d}
b_{\la,\underline{\mu}}(q) +n_{\lambda,\mu}(q) =
\!\!\! \sum_{\SSTS \in {\rm Path}(\la,\underline{\mu})}\!\!\!\!\!\! q^{\deg(\SSTS)} -
 \!\!\sum_{
\begin{subarray}c
\la <  \nu <  \mu\end{subarray}
}   \!\!\!   n_{\lambda,\nu}(q)
b_{\nu,\underline{\mu}}(q).  
\end{equation}
The polynomials $n_{\la,\mu}(q)$ are called the {\sf anti-spherical Kazhdan--Lusztig polynomials} associated to $\la,\mu \in {^PW}$.
   \end{defn}

We can reformulate the above in terms of matrix multiplication.  
We define the {\sf matrix of light leaves polynomials}
\[
\Delta^{(W,P)}:=( \Delta_{\la,\underline{\mu}}(q))_{\la,\mu\in {{^P}W}}
\qquad \qquad
\Delta_{\la,\underline{\mu}}(q)=\sum_{\SSTS \in {\rm Path}(\la,\underline{\mu})}\!\!\!\!\!\! q^{\deg(\SSTS)}
\]
to be the (square) lower uni-triangular matrix whose entries record the degrees of paths in ${\rm Path}(\la ,\underline{\mu})$.   
This matrix can be factorised {\em uniquely} as a product 
$\Delta^{(W,P)}= N^{(W,P)} \times B^{(W,P)}$ of  
lower uni-triangular matrices 
\[N^{(W,P)}:=(n_{\la,\nu}(q))_{\la,\mu\in {{^P}W}}
\qquad 
B^{(W,P)}:=(b_{\nu,\underline{\mu}}(q))_{\nu,\mu\in {{^P}W}}
\]
such that $n_{\la,\nu}(q)\in q\ZZ[q]$ 
for $\la \neq \nu$ 
and 
$b_{\nu,\underline{\mu}}(q)\in  \ZZ[q+q^{-1}]$.

\begin{eg}
The matrix $\Delta$ in type $(A_3,A_1\times A_1)$ is depicted below. \[
 \begin{array}{r|cccccc}
\Delta  
&  \color{magenta}s_2\color{cyan}s_1  \color{green!80!black}s_3 \color{magenta}s_2
    &  \color{magenta}s_2\color{cyan}s_1  \color{green!80!black}s_3 
&  \color{magenta}s_2 \color{cyan}s_1 
    &\color{magenta}s_2 \color{green!80!black}s_3   
&  \phantom{s_2}\color{magenta}s_2 \phantom{s_2}  
& \phantom{s_2} \varnothing \phantom{s_2}      
         \\ \hline
 \color{magenta}s_2\color{cyan}s_1  \color{green!80!black}s_3 \color{magenta}s_2 
   & 1 & \cdot  & \cdot  & \cdot  & \cdot  & \cdot  \\
 	 \color{magenta}s_2\color{cyan}s_1  \color{green!80!black}s_3 			 & q & 1 & \cdot  & \cdot  & \cdot  & \cdot  \\
 \color{magenta}s_2\color{cyan}s_1  & \cdot  & q & 1 & \cdot  & \cdot  & \cdot  \\
 \color{magenta}s_2  \color{green!80!black}s_3    & \cdot  & q & \cdot  & 1 & \cdot  & \cdot  \\
  \color{magenta}s_2  & q & q^2 & q & q & 1 & \cdot  \\
\varnothing   & q^2 & \cdot  & \cdot  & \cdot  & q & 1
\end{array}\]
  The factorisation of this matrix is trivial, with $N=\Delta$ and $B={\rm Id}_{6\times 6}$ the identity matrix.
 \end{eg}

\begin{eg}\label{bad6}
For $(C_3,A_2)$ the factorisation $\Delta= N\times B$ is given below.  The rows of the matrix can be taken to be ordered with respect to 
any total refinement of the Bruhat order (there are two such total orders), see \cref{BruhatexampleD} for the corresponding graph $\mathcal{G}_{(W,P)}$.  
\[\hspace{-0.2cm}\scalefont{0.9}
\left(\begin{array}{cccccccc}
1  & \cdot & \cdot & \cdot & \cdot & \cdot & \cdot & \cdot \\
q  &1 & \cdot & \cdot & \cdot & \cdot & \cdot & \cdot \\
\color{gray}1  & q & 1& \cdot & \cdot & \cdot & \cdot & \cdot \\
\color{gray}q  & \cdot & q & 1 & \cdot & \cdot & \cdot & \cdot \\
\color{gray}q  & \cdot & q & \cdot & 1& \cdot & \cdot & \cdot \\
\color{gray}q^2  & q & q^2 & q& q & 1 & \cdot & \cdot \\
q  & q^2 & \cdot & \cdot & \color{gray}1 & q & 1 & \cdot \\
q^2  & \cdot & \cdot & \cdot &\color{gray}q & \cdot & q & 1 \\
 \end{array}\right)
 =
 \left(\begin{array}{cccccccc}
1  & \cdot & \cdot & \cdot & \cdot & \cdot & \cdot & \cdot \\
q  &1 & \cdot & \cdot & \cdot & \cdot & \cdot & \cdot \\
\cdot  & q & 1& \cdot & \cdot & \cdot & \cdot & \cdot \\
\cdot  & \cdot & q & 1 & \cdot & \cdot & \cdot & \cdot \\
\cdot  & \cdot & q & \cdot & 1& \cdot & \cdot & \cdot \\
\cdot  & q & q^2 & q& q & 1 & \cdot & \cdot \\
q  & q^2 & \cdot & \cdot & \cdot & q & 1 & \cdot \\
q^2  & \cdot & \cdot & \cdot & \cdot & \cdot & q & 1 \\
 \end{array}\right)
\left(\begin{array}{cccccccc}
 1  & \cdot & \cdot & \cdot & \cdot & \cdot & \cdot & \cdot \\
  \cdot  & 1 & \cdot & \cdot & \cdot & \cdot & \cdot & \cdot \\
1  & \cdot & 1 & \cdot & \cdot & \cdot & \cdot & \cdot \\
 \cdot  & \cdot & \cdot & 1& \cdot & \cdot & \cdot & \cdot \\
  \cdot  & \cdot & \cdot & \cdot &1 & \cdot & \cdot & \cdot \\
   \cdot  & \cdot & \cdot & \cdot & \cdot & 1& \cdot & \cdot \\
    \cdot  & \cdot & \cdot & \cdot & 1 & \cdot &  1& \cdot \\
     \cdot  & \cdot & \cdot & \cdot & \cdot & \cdot & \cdot & 1 \\
 \end{array}\right)
  \]

\end{eg}

\begin{eg}
The  matrices $\Delta$ for exceptional types of Hermitian symmetric pairs are given in \cref{exceptional}.  Again, we have that $B$ is the identity matrix in these cases.  
\end{eg}

The paths  $\SSTS \in \Path(\la,\underline{\mu})$ enumerate a   ``light leaf basis"  of the Hecke category.  
We refer to \cite{compan} for an algorithmic construction of these basis elements in the language of this paper, 
see also \cite{antiLW}.     
 We are now able to  restate Libedinsky--Williamson's goal (from the introduction) more precisely using the language of paths.  
They ask if it  is possible to produce (via a closed combinatorial algorithm) 
 a set  $\NPath (\la,\underline{\nu}) \subseteq \Path (\la,\underline{\nu})$  and a canonical basis for a space 
\[
\bigoplus_{\sts  \in \NPath (\la,\underline{\nu})}
\mathbb{R}\sts 
\xrightarrow{ \sim }
\mathbb N_{\la,\nu}
\]
so that, upon taking graded dimensions, we get 
\[
n_{\la,\nu} = \sum _{\sts \in \NPath (\la,\underline{\nu})} q^{\deg(\sts)}.
\]
In this paper, we answer this question for $(W,P)$ a Hermitian symmetric pair (see \cref{coxeterlabelD2} for a list of such pairs).
In fact, we go further and produce closed combinatorial descriptions of canonical bases for spaces 
 \[
\bigoplus_{\sts  \in    \NPath (\la,\underline{\nu})}
\mathbb{R}\sts 
\xrightarrow{ \sim }
\mathbb N_{\la,\nu}
\qquad\qquad
\bigoplus_{\sts  \in    \BPath (\la,\underline{\nu})}
\mathbb{R}\sts 
\xrightarrow{ \sim }
\mathbb B_{\la,\underline{\nu}}
\]
so that, upon taking graded dimensions, we get 
\[
n_{\la,\nu} = \sum _{\sts \in   \NPath (\la,\underline{\nu})} q^{\deg(\sts)}
\qquad
\qquad
b_{\nu,\underline{\mu}} = \sum _{\sts \in   \BPath (\nu,\underline{\mu})} q^{\deg(\sts)} 
\]
for subsets $\NPath (\la,\underline{\nu}) \subseteq \Path (\la,\underline{\nu})$ and 
$ \BPath (\nu,\underline{\mu}) \subseteq  \Path (\nu,\underline{\mu})$. 

We note further that, for Hermitian symmetric pairs, the subsets  $\NPath (\la,\underline{\nu})$ and $\BPath (\nu,\underline{\mu})$ are independent of the choice of reduced expressions for $\mu$ and $\nu$. This follows from the fact that the elements of ${^PW}$ are fully commutative (see Section 3.1 below).

 \section{The oriented Temperley--Lieb algebras}

We will assume from now on that $(W,P)$ is a {\sf Hermitian symmetric pair}, that is it is one of the following  infinite families  $(A_n  ,  A_{k-1} \times A_{n-k}  )$  with $1 \leq k \leq n$,
    $(D_n , A_{n-1}  )$,  $(D_n , D_{n-1}  )$,  $(B_n , B_{n-1}  )$, 
     $(C_n , A_{n-1}  )$ for  $n\geq 2$ or  is one of the exceptional cases $(E_6,D_5)$,   $(E_7,E_6)$. 
     The corresponding Coxeter graphs of these pairs are recorded in \cref{coxeterlabelD2}.

\subsection{The oriented Temperley--Lieb algebras and strong full commutativity } For $\underline{w}, \underline{w}'$  two expressions in the generators $s_i\in S_W$, we say that $\underline{w'}$ is a subword of $\underline{w}$ if there are expressions $\underline{u}$ and $\underline{v}$ such that $\underline{w} = \underline{u} \, \underline{w'} \,\underline{v}$.
One of the crucial property of Hermitian symmetric pairs for this paper is that the elements of $^PW$ are {\sf fully commutative} (as defined by Stembridge \cite[Introduction]{MR1406459}).  We recall that an element $w\in W$ is called fully commutative if and only if any two reduced expression  of $w$ are related by applying only the commutation relations in $W$.  Equivalently, no reduced expressions of $w$ contains $s_is_js_i$ as a subword for $m_{i,j}=3$ or $s_is_js_is_j$ as a subword for $m_{i,j}=4$. 
In fact, 
the elements of ${^PW}$ satisfy the following slightly stronger property.

\begin{defn}\label{strongly} We say that an element $w\in W$ is {\sf strongly fully commutative} if no reduced expression of $w$ contains $s_is_js_i$ as a subword for any  $s_i,s_j\in S_W$ with either $m_{i,j}=3$ or $m_{i,j}=4$ 
when $\alpha_i$ is a short root. We denote by $W_{sfc}$ the set of all strongly fully commutative elements of $W$.
\end{defn}

\begin{lem}\label{fullcomm} Let $(W,P)$ be a Hermitian symmetric pair. Then every element of ${^PW}$ is strongly fully commutative. 
\end{lem} 
   
 \begin{proof}
 This is well-known and can be seen for example from the explicit description of the elements of $^PW$ in terms of tilings given in \cite[Appendix: Diagrams of Hermitian types]{MR3363009}.
 \end{proof} 
 
 \begin{defn}\label{maindef}
Let  $(W,P)$ be a Hermitian symmetric pair.  
The  {\sf oriented Temperley--Lieb algebra of type  $(W,P)$}, ${\rm TL}_{(W,P)}(q)$,  is defined to be the unital associative  $\mathbb{Z}[q,q^{-1}]$-algebra generated by elements
\[\{{\sf 1}_\la \mid \lambda \in {^PW}\} \cup \{  E_i \mid s_i\in S_W\}
\]
subject to the following relations.  The idempotent relations,
 \begin{equation}
 \label{idempotentrel}
\begin{split} {\sf 1}= \textstyle\sum_{\la\in {^PW}} {\sf 1}_\la, 
 \quad  
  {\sf 1}_\la E_i {\sf 1}_\la= 0   \text{ if }  	\la s_i \notin {^PW},
 %
\\
{\sf 1}_\la  1_\mu =\delta_{\la,\mu}1_\la ,  \quad
 {\sf 1}_\la E_i {\sf 1}_\mu= 0   \text{ if }  	\mu \not \in \{\la,\la s_i \} 
 \end{split}
\end{equation}
for all $\lambda, \mu \in {^PW}$.
For all $s_i\in S_W$, any $\lambda\in {^PW}$ with $\la s_i  \in {^PW}$ and $\mu, \nu \in \{\lambda, \lambda s_i\}$ we have
\begin{equation}\label{Esquare}
{\sf 1}_\mu E_{ i}  {\sf 1}_\la   E_{ i}  {\sf 1}_\nu  =  
 q^{  \ell(\la s_i)-\ell(\la)} {\sf 1}_\mu E_i  {\sf 1}_\nu.
 \end{equation}
 If $m_{i,j}=2$ or $3$, then
\begin{equation}\label{commutation}
 E_i E_j=E_jE_i, \qquad\quad 
 E_iE_jE_i = E_i,
\end{equation}respectively.  
If $m_{i,j}=4$ and $\alpha_i$ is a short root then   we have that 
\begin{equation}\label{mequals4}
{\sf 1}_\mu E_i {\sf 1}_\lambda E_j {\sf 1}_\lambda E_i {\sf 1}_\nu =  {\sf 1}_\mu E_i {\sf 1}_\nu,
\end{equation}
for any $\lambda\in {^PW}$ with $\lambda s_i, \lambda s_j\in {^PW}$ and $\mu, \nu \in \{\lambda, \lambda s_i\}$.

 \end{defn}
 
 \begin{rmk}\label{altgen} It follows from the relations (\ref{idempotentrel}) that ${\rm TL}_{(W,P)}(q)$ is generated by the elements  ${\sf 1}_\la$ for $\la \in {^PW}$ and ${\sf 1}_\la E_i {\sf 1}_\mu$ for all $\la \in {^PW}$ with $\la s_i\in {^PW}$ and $\mu \in \{ \la , \la s_i\}$.
 \end{rmk} 
  
 \begin{rmk}\label{rmkcommutation}
 Note that for $m_{i,j}=2$ we have
 ${\sf 1}_\lambda E_i {\sf 1}_\mu E_j {\sf 1}_\nu \neq 0$ if and only if $\lambda, \lambda s_i, \lambda s_j\in {^PW}$ and either $\mu = \lambda$ and $\nu\in \{\lambda, \lambda s_j\}$, or $\mu = \lambda s_i$ and $\nu \in \{ \lambda s_i, \lambda s_i s_j\}$. So we have that the first relation in (\ref{commutation}) is equivalent to 
 \[{\sf 1}_\lambda E_i {\sf 1}_\lambda E_j {\sf 1}_\nu = {\sf 1}_\lambda E_j {\sf 1}_\nu E_i {\sf 1}_\nu,
\quad \text{ and } \quad  {\sf 1}_\lambda E_i {\sf 1}_{\lambda s_i} E_j {\sf 1}_\nu = {\sf 1}_\lambda E_j {\sf 1}_{\nu s_i} E_i {\sf 1}_\nu\]
 for $\nu \in \{ \lambda, \lambda s_j\}$ and  $\nu \in \{ \lambda s_i, \lambda s_is_j\}$ respectively, and every $\la\in {^PW}$.
 \end{rmk}   
 
 \begin{rmk}\label{rmkmequals3}
  Note that for $m_{i,j}=3$, using \cref{fullcomm}, we have
 $ E_i {\sf 1}_\mu E_j {\sf 1}_\nu E_i \neq 0$ implies $\nu = \mu$ and $\mu s_i, \mu s_j\in {^PW}$. Now the second relation in   (\ref{commutation}) 
   is equivalent to
 \[{\sf 1}_\lambda E_i  E_j  E_i {\sf 1}_\eta = {\sf 1}_\lambda E_i {\sf 1}_\eta\]
 for all $\la, \eta \in {^PW}$ with $\la s_i, \eta s_i\in {^PW}$. Note that for each such $\la$, using \cref{fullcomm}, we have either $\la s_j\in {^PW}$ or $\la s_i s_j\in {^PW}$ but not both. Thus  the second relation in   (\ref{commutation}) 
 is also  equivalent to 
\[{\sf 1}_\lambda E_i {\sf 1}_\la E_j {\sf 1}_\la E_i {\sf 1}_\eta = {\sf 1}_\lambda E_i {\sf 1}_\eta\]
 for all $\lambda\in {^PW}$ with $\lambda s_i, \la s_j\in {^PW}$ and $\eta = \{\lambda, \la s_i\}$, and 
 \[{\sf 1}_\lambda E_i {\sf 1}_{\la s_i} E_j {\sf 1}_{\la s_i} E_i {\sf 1}_\eta = {\sf 1}_\lambda E_i {\sf 1}_\eta\]
 for all $\lambda\in {^PW}$ with $\lambda s_i, \la s_i s_j\in {^PW}$ and $\eta = \{\lambda , \la s_i\}$.
 \end{rmk}
 
 \begin{rmk}\label{rmkmequals4}
 Note that for $m_{i,j}=4$ with $\alpha_i$ being a short root, we have 
 \[{\sf 1}_\lambda E_i  E_j  E_i {\sf 1}_\mu  = {\sf 1}_\lambda E_i {\sf 1}_\lambda E_j {\sf 1}_\lambda E_i {\sf 1}_\mu + {\sf 1}_\lambda E_i {\sf 1}_{\lambda s_i} E_j {\sf 1}_{\lambda s_i} E_i {\sf 1}_\mu =  2 ({\sf 1}_\lambda E_i {\sf 1}_\mu)\]
 for $\lambda, \lambda s_i, \lambda s_j , \lambda s_i s_j\in {^PW}$ and $\mu \in \{\lambda, \lambda s_i\}$.
 This implies that $E_iE_jE_i = 2E_i$ and so $E_iE_jE_iE_j = 2E_iE_j$ and $E_jE_iE_jE_i = 2E_jE_i$. 
 \end{rmk}

\subsection{Path basis and the anti-spherical module} 
 
For any path
\[\SSTT : \lambda_1 \xrightarrow{i_1} \lambda_2 \xrightarrow{i_2} \lambda_3 \xrightarrow{i_3} \ldots \xrightarrow{i_{k-1}} \lambda_k
 \] on $\widehat{\mathcal{G}}_{(W,P)}$, we write 
 \[E_\SSTT := 
  {\sf 1}_{\lambda_1}
 E_{i_1}
 {\sf 1}_{\lambda_2}  
  E_{i_2}  {\sf 1}_{\lambda_3} \ldots    {\sf 1}_{\lambda_{k-1}}  
  E_{i_{k-1}}    {\sf 1}_{\lambda_{i_k}}
 \] to be the corresponding element in ${\rm TL}_{(W,P)}(q)$. 
 
 Recall that we denote by  $W_{sfc}$ the set of all strongly fully commutative elements of $W$.  We now fix one reduced expression $\underline{w}$ for each $w\in W_{sfc}$.

\begin{thm}\label{spanningset}
The algebra ${\rm TL}_{(W,P)}(q)$ has a $\mathbb{Z}[q,q^{-1}]$-basis given by the set 
\[\{ E_{\SSTT} \, : \, \SSTT\in{\rm Path}_{(W,P)} \,\, \text{with} \,\,  \omega(\SSTT)=\underline{w} \,\, \text{for some  $w\in W_{sfc}$}\}.\]
\end{thm}

\begin{proof}
It is clear from relations (\ref{idempotentrel})  that 
\[ 
  {\sf 1}_{\lambda_1}
 E_{i_1}
 {\sf 1}_{\lambda_2}  
  E_{i_2}   {\sf 1}_{\lambda_3}   \ldots    {\sf 1}_{\lambda_{k-1}}  
  E_{i_{k-1}}    {\sf 1}_{\lambda_{k}}
  \neq 0
 \]
  implies that \[\SSTT : \lambda_1 \xrightarrow{i_1} \lambda_2 \xrightarrow{i_2} \lambda_3 \xrightarrow{i_3} \ldots \xrightarrow{i_{k-1}} \lambda_k
 \] is a path on $\widehat{\mathcal{G}}_{(W,P)}$.  Thus we have that ${\rm TL}_{(W,P)}(q)$ is spanned by elements of the form $E_\SSTT$ where $\SSTT\in {\rm Path}_{(W,P)}$. Now it follows from relations (\ref{Esquare})--(\ref{mequals4}) that any such $E_\SSTT = q^x E_\SSTS$ for some $x\in \mathbb{Z}$ and some path $\SSTS$ such that $\omega(\SSTS)=\underline{w}$ for some $w\in W_{sfc}$. Set ${\rm Path}_{(W,P)}^{sfc}$ to be the set of all $\SSTT \in {\rm Path}_{(W,P)}$ such that $\omega(\SSTT) = \underline{w}$ for some $w\in W_{sfc}$.  It remains to show that the set $\{ E_{\SSTT} \, : \, \omega(\SSTT) \in {\rm Path}_{(W,P)}^{sfc} \}$ is linearly independent. We do this by constructing a formal ${\rm TL}_{(W,P)}(q)$-module $M$ with $\mathbb{Z}[q,q^{-1}]$-basis labelled by $[\SSTT]$ for $\SSTT\in{\rm {Path}}_{(W,P)}^{sfc}$.  To define the action on this module we will need to \lq reduce' any path on  ${\widehat{\mathcal{G}}_{(W,P)}}$ using the following local operations.
\begin{itemize}[leftmargin=*]
\item If $m_{ij}=2$ then for $\nu \in \{\la, \la s_j\}$ we set
\[[\la \xrightarrow{i} \la \xrightarrow{j} \nu ] \qquad \implies  \qquad  [ \la \xrightarrow{j} \nu \xrightarrow{i} \la]\]
 and for $\nu \in \{\la s_i, \la s_i s_j\}$ we set
\[[\la \xrightarrow{i} \la s_i \xrightarrow{j} \nu] \qquad \implies \qquad [\la \xrightarrow{j} \nu s_i \xrightarrow{i} \nu].\]
\item If $m_{ij}=3$  or $m_{ij} = 4$ and $\alpha_i$ is a short root then for $\eta \in \{\la  , \la s_i\}$ we set
\[ [ \la \xrightarrow{i} \la \xrightarrow{j} \la \xrightarrow{i} \eta] \qquad \implies \qquad [\la \xrightarrow{i} \eta]\]
and 
\[[ \la \xrightarrow{i} \la s_i  \xrightarrow{j} \la s_i  \xrightarrow{i} \eta] \qquad \implies \qquad [\la \xrightarrow{i} \eta].\]

\item  For for $\mu, \nu \in \{\la , \la s_i\}$, we set
\[[\mu \xrightarrow{i} \la \xrightarrow{i} \nu ] \qquad \implies \qquad q^{\ell (\la s_i) - \ell (\la)} [\mu \xrightarrow{i} \nu].\]
\end{itemize}
(Note that these follow exactly the relations given in \cref{maindef}, see also \cref{rmkcommutation}, \cref{rmkmequals3} and \cref{rmkmequals4}.)
It is clear that starting with any path $\SSTT\in {\rm Path}_{(W,P)}$, applying these operations repeatedly, we obtain a uniquely defined power $q^{x(\SSTT)}$ and a unique path ${\rm rex}(\SSTT)\in{\rm {Path}}_{{(W,P)}}^{sfc}$.  Now, for any $\SSTT \in {\rm {Path}}_{{(W,P)}}^{sfc}$ with 
\[\SSTT : \lambda_1 \xrightarrow{i_1} \lambda_2 \xrightarrow{i_2} \lambda_3 \xrightarrow{i_3} \ldots \xrightarrow{i_{k-1}} \lambda_k =\mu
 \] we set 
 \[[\SSTT] {\sf 1}_\nu = \delta_{\mu \nu} [\SSTT] \quad \mbox{and} \quad [\SSTT] {\sf 1}_\mu E_i {\sf 1}_\nu = q^{x(\SSTT')} [{\rm rex}(\SSTT')]\]
 for $\nu \in \{\mu , \mu s_i\}$ and $\mu, \mu s_i\in {^PW}$, where
 \[\SSTT'  =  \lambda_1 \xrightarrow{i_1} \lambda_2 \xrightarrow{i_2} \lambda_3 \xrightarrow{i_3} \ldots \xrightarrow{i_{k-1}} \lambda_k = \mu \xrightarrow{i} \nu .
 \]
 As the relations used to reduce the path correspond precisely to the relations defining the oriented Temperley--Lieb algebra, it is clear that this turns $M$ into a ${\rm TL}_{(W,P)}(q)$-module. 
 
We are now ready to prove that the set $\{ E_{\SSTT} \, : \, \SSTT\in{\rm {Path}}_{{(W,P)}}^{sfc}\}$ is linearly independent. Assume that 
\[\sum_{\SSTT} a_\SSTT E_\SSTT = 0 \qquad \mbox{for some $a_\SSTT \in \mathbb{Z}[q,q^{-1}]$}.\]
We need to show that $a_\SSTT = 0$ for all $\SSTT$. Fix $\la \in {^PW}$ and consider the trivial path $[\la]\in M$. Then we have 
\[[\la] \left( \sum_{\SSTT} a_\SSTT E_\SSTT \right) =
 \sum_{
 \begin{subarray}{c}
 {\nu\in {^PW}}\\
 \SSTT \in {\Path(\la\to\nu)}
 \end{subarray}
 } a_\SSTT [\SSTT] = 0.\]
As $\{[\SSTT] \, : \, \SSTT\in 	 {\Path(\la\to\nu)}, \nu \in {^PW}			\}$ is linearly independent in $M$, we deduce that $a_\SSTT = 0$ for all $\SSTT\in {\Path(\la\to\nu)}$, all $\nu \in {^PW}$. But this holds for all $\la \in {^PW}$ so we are done.
\end{proof}

\begin{defn} \label{maindef2}
We define the anti-spherical ${\rm TL}_{(W,P)}(q)$-module to be the right module
 $ {\sf 1}_\varnothing {\rm TL}_{(W,P)}(q)  $.  
\end{defn}

\begin{cor}\label{spanningset2}
The anti-spherical module ${\sf 1}_\varnothing {\rm TL}_{(W,P)}(q)  $  has a $\mathbb{Z}[q,q^{-1}]$-basis given by 
\[\{E_\SSTT \mid \SSTT \in \Path(\la, \underline{\mu}) \, : \, \la,\mu   \in {^PW}\}.\]
 \end{cor}

\begin{proof}
By \cref{spanningset} we have that the anti-spherical module has a basis given by all $E_\SSTT$ where $\SSTT$ is a path on $\widehat{\mathcal{G}}_{(W,P)}$ starting at $\varnothing$ with $\omega(\SSTT)=\underline{w}$ some strongly fully commutative element $w\in W$. Note that $\underline{w}$ must start with the unique $s\notin S_P$, as $\SSTT$ starts at $\varnothing$. Moreover, as $w$ is fully commutative, any other reduced expression $\underline{w'}$ for $w$ is obtained from $\underline{w}$ by applying only the commutation relations and so $\underline{w'}$ is also the weight of a path on $\widehat{\mathcal{G}}_{(W,P)}$ starting at $\varnothing$. In particular, $\underline{w'}$ also starts with the unique $s\notin S_P$. This implies that $w = \mu \in {^PW}$.
\end{proof}

\subsection{Grading} 

We can view ${\rm TL}_{(W,P)}(q)$ as a $\mathbb{Z}$-algebra in the usual way, by considering $q$ and $q^{-1}$ as additional central generators. The next proposition shows that, as such, ${\rm TL}_{(W,P)}(q)$ is a $\mathbb{Z}$-graded algebra. 

\begin{prop}\label{abstractdegree} 
Set ${\rm deg}({\sf 1}_\la )= 0$ for all $\lambda \in {^PW}$, 
\[{\rm deg}({\sf 1}_\la E_i{\sf 1}_{\la s_i}) = 0, \qquad {\rm deg}({\sf 1}_\la E_i{\sf 1}_{\la})  = \left\{ \begin{array}{ll} 1 & \mbox{if $\la s_i >\la$}\\ -1 & \mbox{if $\la s_i < \la$} \end{array}\right. \]
for all $\la\in {^PW}$, $s_i\in S_W$ with $\la s_i \in {^PW}$,
${\rm deg}(q) = 1$ and ${\rm deg}(q^{-1}) = -1$. \\
This defines a $\mathbb{Z}$-grading on the $\mathbb{Z}$-algebra ${\rm TL}_{(W,P)}(q)$. In particular, we have $\deg (E_\SSTT) = \deg (\SSTT)$ for all $\SSTT\in {\rm Path}_{(W,P)}$.
\end{prop}

\begin{proof}
We need to check that the defining relations (\ref{idempotentrel})--(\ref{mequals4}) are (equivalent to) homogeneous relations.
For relations (\ref{idempotentrel}), there is nothing to prove. We now consider  relation  (\ref{Esquare}). We need to show that for any $\mu,\nu\in \{\la, \la s_i\}$ we have
\[{\rm deg}({\sf 1}_\mu E_i {\sf 1}_\la E_i {\sf 1}_\nu) = (\ell(\la s_i) - \ell(\la)) + {\rm deg}({\sf 1}_\mu E_i {\sf 1}_\nu).\]
If $\mu = \nu = \la$ we have 
\begin{align*}
{\rm deg}({\sf 1}_\la E_i {\sf 1}_\la E_i {\sf 1}_\la) 
&= 2(\ell(\la s_i) - \ell(\la)), \quad  \text{ and }
\quad  {\rm deg}({\sf 1}_\la E_i {\sf 1}_\la )  = \ell(\la s_i) - \ell(\la)
 \end{align*}
as required. If $\mu = \nu = \la s_i$ we have 
\begin{align*}
{\rm deg}({\sf 1}_{\la s_i} E_i {\sf 1}_\la E_i {\sf 1}_{\la s_i}) &= 0
, \quad  \text{ and }
\quad  
 {\rm deg}({\sf 1}_{\la s_i} E_i {\sf 1}_{\la s_i})  = \ell(\la) - \ell(\la s_i)\end{align*}
as required. Finally if $\mu \neq \nu$ then we have
\begin{align*}
{\rm deg}({\sf 1}_\mu E_i {\sf 1}_\la E_i {\sf 1}_\nu) = \ell(\la s_i) - \ell(\la)
, \quad  \text{ and }
\quad  
 {\rm deg}({\sf 1}_\mu E_i {\sf 1}_\nu ) = 0
 \end{align*}
as required.
Next consider the leftmost relation in (\ref{commutation}). By \cref{rmkcommutation}, we need to show that the relations
\begin{align*}
{\sf 1}_\la E_i {\sf 1}_\la E_j {\sf 1}_\nu &=  {\sf 1}_\la E_j {\sf 1}_\nu E_i {\sf 1}_\nu \quad \mbox{for $\nu \in \{ \la , \la s_j\}$}
, \quad  \text{ and } \\
 {\sf 1}_\la E_i {\sf 1}_{\la s_i} E_j {\sf 1}_\nu &=  {\sf 1}_\la E_j {\sf 1}_{\nu s_i} E_i {\sf 1}_\nu \quad  \mbox{for $\nu \in \{ \la s_i , \la s_i s_j\}$}
 \end{align*}
are homogeneous. The former is trivial for $\nu = \la$ and follows from the fact that $\la s_j s_i > \la s_j$ if and only if $\la s_i > \la$ for $\nu = \la s_j$. 
The latter is trivial for $\nu = \la s_i s_j$ and follows from the fact that $\lambda s_is_j>\la s_i$ if and only if $\la s_j>\la$ for $\nu = \la s_i$.

We now consider the rightmost  relation in  (\ref{commutation}).  By \cref{rmkmequals3}, this relation  is equivalent to 
\[{\sf 1}_\lambda E_i {\sf 1}_\la E_j {\sf 1}_\la E_i {\sf 1}_\eta = {\sf 1}_\lambda E_i {\sf 1}_\eta\]
 for all $\lambda\in {^PW}$ with $\lambda s_i, \la s_j\in {^PW}$ and $\eta = \{\lambda, \la s_i\}$, and 
 \[{\sf 1}_\lambda E_i {\sf 1}_{\la s_i} E_j {\sf 1}_{\la s_i} E_i {\sf 1}_\eta = {\sf 1}_\lambda E_i {\sf 1}_\eta\]
 for all $\lambda\in {^PW}$ with $\lambda s_i, \la s_i s_j\in {^PW}$ and $\eta = \{\lambda , \la s_i\}$.
The former is homogeneous as $\lambda s_i >\lambda$ if and only if $\lambda s_j < \lambda$ and so 
\[{\rm deg}( {\sf 1}_\lambda E_i {\sf 1}_\la E_j {\sf 1}_\la) =0.\]
To see that the latter is also homogeneous, observe that $\lambda s_is_j > \lambda s_i$ if and only if $\lambda s_i > \lambda$ and so 
\[{\rm deg} ({\sf 1}_{\la s_i} E_j {\sf 1}_{\la s_i} ) = {\rm deg} ({\sf 1}_\lambda E_i {\sf 1}_\lambda) \quad \mbox{and} \quad {\rm deg}( {\sf 1}_{\la s_i} E_j {\sf 1}_{\la s_i} E_i {\sf 1}_{\lambda s_i}) = 0.\]
Finally, consider relation (\ref{mequals4}). For $\mu = \lambda$ and $\nu\in\{ \lambda, \lambda s_i\}$ we have to show that
\[{\sf 1}_\lambda E_i {\sf 1}_\la E_j {\sf 1}_\la E_i {\sf 1}_\nu = {\sf 1}_\lambda E_i {\sf 1}_\nu\]
is homogeneous, and for $\mu = \lambda s_i$  and $\nu\in\{ \lambda, \lambda s_i\}$ we need to show that 
\[{\sf 1}_{\lambda s_i} E_i {\sf 1}_\la E_j {\sf 1}_\la E_i {\sf 1}_\nu = {\sf 1}_{\lambda s_i} E_i {\sf 1}_\nu\]
is homogeneous. These look exactly the same as the relations for $m_{i,j}=3$ above (replacing $\lambda s_i$ with $\lambda$ for the second equation) and the same arguments apply here.
\end{proof}

We immediately obtain the following result.

\begin{cor}\label{abstractDelta} \begin{enumerate}
\item The anti-spherical module ${\sf 1}_\varnothing {\rm TL}_{(W,P)}(q)$ is a graded ${\rm TL}_{(W,P)}(q)$-module with homogeneous basis $\{ E_\SSTT \,: \, \SSTT \in {\rm Path}(\la, \underline{\mu}), \la, \mu \in {^PW}\}$ satisfying 
\[\deg E_\SSTT = \deg \SSTT.\]
\item The light leaves matrix can be computed follows. For any $\la , \mu \in {^PW}$ we have
\[\Delta_{\la , \mu } = \left\{ \begin{array}{ll} \sum_{\SSTT\in {\rm Path}(\la , \underline{\mu})} q^{\deg (E_{\SSTT})} & \mbox{ if ${\rm Path}(\la , \underline{\mu})\neq \emptyset$}\\ 0 & \mbox{otherwise.} \end{array}\right.\]
\end{enumerate}
\end{cor}

Thus the anti-spherical module for the oriented Temperley--Lieb algebra gives us a model to study the light leaves matrix and its factorisation for all Hermitian symmetric pairs. This will be done in details in each type $(W,P)$  in the next few sections but first we take a short detour to relate our oriented Temperley--Lieb algebras to the generalised Temperley--Lieb algebras associated to $W$.

\subsection{Relationship with Fan--Graham's Temperley--Lieb algebras}  
We now relate our $(W,P)$-Temperley--Lieb algebra  to the generalised Temperley--Lieb algebra associated to $W$ introduced by Fan (in the simply-laced type)  and Graham (in the non-simply laced type).
 \begin{defn}\label{genTL}
 The generalised Temperley--Lieb algebra ${\rm TL}_W(q)$  is defined as the $\mathbb{Z}[q, q^{-1}]$-algebra generated by 
 \[\{U_i \mid   s_i\in S_W\}\]
 subject to the following relations:
 For all $s_i\in S_W$ we have
 \begin{equation}\label{TL1}
 U_i^2 = (q+q^{-1})U_i. 
 \end{equation}
Furthermore, we have that 
 \begin{align}\label{TL234}
 U_iU_j = U_jU_i, 
 \quad \quad 
 U_iU_jU_i =U_i,
  \quad \quad 
  U_iU_jU_iU_j = 2U_iU_j 
  \end{align}
  for $m_{i,j} = 2, 3 $ or $4$, respectively.  
   \end{defn}
 
 \begin{prop}
 There is a $\mathbb{Z}[q,q^{-1}]$-algebra homomorphism from ${\rm TL}_W(q)$ to ${\rm TL}_{(W,P)}(q)$ defined by $U_i\mapsto E_i$ for all $s_i\in S_W$.
 \end{prop} 

\begin{proof}
We need to check that the $E_i$'s satisfy the relations (\ref{TL1}) and (\ref{TL234}). 
The leftmost two relations of (\ref{TL234}) are given by (\ref{commutation}). 
The rightmost  relation in (\ref{TL234}) follows from \cref{rmkmequals4}.
 It remains to show (\ref{TL1}). We have 
\[E_i = \sum_\la {\sf 1}_\la E_i ({\sf 1}_\la + {\sf 1}_{\la s_i})\]
where the sum is taken over all $\lambda \in {^PW}$ such that $\la s_i\in {^PW}$. Now we have 
\begin{align*}
E_i^2  &= \textstyle \sum_\la {\sf 1}_\la E_i ({\sf 1}_\la + {\sf 1}_{\la s_i})E_i ({\sf 1}_\la + {\sf 1}_{\la s_i})   
 =  (q+q^{-1})  \textstyle\sum_\la {\sf 1}_\la E_i ({\sf 1}_\la + {\sf 1}_{\la s_i})   
 =  (q+q^{-1})E_i
\end{align*}
as required.  Here the first equality follows   {by (\ref{idempotentrel})} and the second {by (\ref{Esquare})}, the third is trivial.
\end{proof}  
  
\begin{rmk} Note that this homomorphism is not injective in general. To see this, take for example $W$ of type $A_3$ and $P$ of type $A_2$, then $^PW = \{ 1, s_3, s_3s_2, s_3s_2s_1\}$. Then we claim that $E_1E_3=0$. To see this, note that $E_1 {\sf 1}_\la \neq 0$ implies that $\lambda = s_3s_2$ or $s_3s_2s_1$ but ${\sf 1}_\la E_3\neq 0$ implies $\lambda = 1$ or $s_3$. However, $U_1U_3\neq 0$ in ${\rm TL}_W(q)$.
\end{rmk}

\section{Light leaves matrix factorisation for the trivial and exceptional types} \label{exceptional}

\cref{abstractDelta} provides a way of studying the light leaves matrix using the oriented Temperley--Lieb algebra and its anti-spherical module for all Hermitian symmetric pairs $(W,P)$. In Sections 5 to 8, we will construct a diagrammatic version of  the oriented Temperley--Lieb algebras in types $(A_n, A_k\times A_{n-k-1})$, $(D_n, A_{n-1})$ and $(C_n, A_{n-1})$ which will provide closed combinatorial formulas for the light leaves matrix and its factorisation. This could also be done in types $(B_n, B_{n-1})$ and $(D_n, D_{n-1})$ but the extra effort is unwarranted as the  light leaves matrices and their factorisation can be easily described without them.   This will be done in this section, together with the exceptional types $(E_6, D_5)$ and $(E_7, E_6)$, which are best tackled with a computer (although  $(E_6,E_5)$ is  manageable by hand as well).  This provides a proof of Theorem A in these trivial and exceptional types.

 \subsection{Exceptional types}
 We first consider the exceptional Hermitian symmetric pairs.  
 One can calculate the $\Delta$ matrix for type $(E_6,D_5)$ easily by hand.  For 
 type  $(E_7,E_6)$ this is a much larger calculation, but can be readily done  using the Coxeter 3 package in SAGE which wraps 
 Folko Ducloux's  original work in $C^{++}$.  
The  matrices $\Delta$ are recorded in  \cref{tableExc,tableExc2}. In both cases, note that all off-diagonal entries belong to $q\NN_0[q]$ and so $\Delta=  N$ and $B={\rm Id}$. This gives a proof of Theorem A in these cases by setting every basis element $E_{\SSTT}\in {\sf 1}_{\varnothing}{\rm TL}_{(W,P)}(q)$ to be standard.

 \begin{figure}[ht!]
$$ \scriptsize  
\begin{array}{cccccccccccccccccccccccccccccccccccc}
&0& \cdot& \cdot& \cdot& \cdot& \cdot& \cdot& \cdot& \cdot& \cdot& \cdot& \cdot& \cdot& \cdot& \cdot& \cdot& \cdot& \cdot& \cdot& \cdot& \cdot& \cdot& \cdot& \cdot& \cdot& \cdot& \cdot\\
& 1& 0& \cdot& \cdot& \cdot& \cdot& \cdot& \cdot& \cdot& \cdot& \cdot& \cdot& \cdot& \cdot& \cdot& \cdot& \cdot& \cdot& \cdot& \cdot& \cdot& \cdot& \cdot& \cdot& \cdot& \cdot& \cdot\\& 
\cdot& 1& 0& \cdot& \cdot& \cdot& \cdot& \cdot& \cdot& \cdot& \cdot& \cdot& \cdot& \cdot& \cdot& \cdot& \cdot& \cdot& \cdot& \cdot& \cdot& \cdot& \cdot& \cdot& \cdot& \cdot& \cdot\\& 
\cdot& \cdot& 1& 0& \cdot& \cdot& \cdot& \cdot& \cdot& \cdot& \cdot& \cdot& \cdot& \cdot& \cdot& \cdot& \cdot& \cdot& \cdot& \cdot& \cdot& \cdot& \cdot& \cdot& \cdot& \cdot& \cdot\\& 
\cdot& \cdot& \cdot& 1& 0& \cdot& \cdot& \cdot& \cdot& \cdot& \cdot& \cdot& \cdot& \cdot& \cdot& \cdot& \cdot& \cdot& \cdot& \cdot& \cdot& \cdot& \cdot& \cdot& \cdot& \cdot& \cdot\\& 
\cdot& \cdot& \cdot& 1& \cdot& 0& \cdot& \cdot& \cdot& \cdot& \cdot& \cdot& \cdot& \cdot& \cdot& \cdot& \cdot& \cdot& \cdot& \cdot& \cdot& \cdot& \cdot& \cdot& \cdot& \cdot& \cdot\\& 
\cdot& \cdot& \cdot& 2& 1& 1& 0& \cdot& \cdot& \cdot& \cdot& \cdot& \cdot& \cdot& \cdot& \cdot& \cdot& \cdot& \cdot& \cdot& \cdot& \cdot& \cdot& \cdot& \cdot& \cdot& \cdot\\& 
\cdot& \cdot& \cdot& \cdot& \cdot& 1& \cdot& 0& \cdot& \cdot& \cdot& \cdot& \cdot& \cdot& \cdot& \cdot& \cdot& \cdot& \cdot& \cdot& \cdot& \cdot& \cdot& \cdot& \cdot& \cdot& \cdot\\& 
\cdot& \cdot& \cdot& \cdot& \cdot& 2& 1& 1& 0& \cdot& \cdot& \cdot& \cdot& \cdot& \cdot& \cdot& \cdot& \cdot& \cdot& \cdot& \cdot& \cdot& \cdot& \cdot& \cdot& \cdot& \cdot\\& 
\cdot& \cdot& 2& 1& \cdot& \cdot& 1& \cdot& \cdot& 0& \cdot& \cdot& \cdot& \cdot& \cdot& \cdot& \cdot& \cdot& \cdot& \cdot& \cdot& \cdot& \cdot& \cdot& \cdot& \cdot& \cdot\\& 
\cdot& \cdot& \cdot& \cdot& \cdot& \cdot& 2& \cdot& 1& 1& 0& \cdot& \cdot& \cdot& \cdot& \cdot& \cdot& \cdot& \cdot& \cdot& \cdot& \cdot& \cdot& \cdot& \cdot& \cdot& \cdot\\& 
\cdot& \cdot& \cdot& \cdot& 2& \cdot& 1& \cdot& \cdot& \cdot& 1& 0& \cdot& \cdot& \cdot& \cdot& \cdot& \cdot& \cdot& \cdot& \cdot& \cdot& \cdot& \cdot& \cdot& \cdot& \cdot\\& 
\cdot& 2& 1& \cdot& \cdot& \cdot& \cdot& \cdot& \cdot& 1& \cdot& \cdot& 0& \cdot& \cdot& \cdot& \cdot& \cdot& \cdot& \cdot& \cdot& \cdot& \cdot& \cdot& \cdot& \cdot& \cdot\\& 
\cdot& \cdot& \cdot& \cdot& \cdot& \cdot& \cdot& \cdot& \cdot& 2& 1& \cdot& 1& 0& \cdot& \cdot& \cdot& \cdot& \cdot& \cdot& \cdot& \cdot& \cdot& \cdot& \cdot& \cdot& \cdot\\& 
\cdot& \cdot& \cdot& \cdot& \cdot& \cdot& \cdot& \cdot& \cdot& \cdot& 2& 1& \cdot& 1& 0& \cdot& \cdot& \cdot& \cdot& \cdot& \cdot& \cdot& \cdot& \cdot& \cdot& \cdot& \cdot\\& 
\cdot& \cdot& \cdot& \cdot& \cdot& \cdot& \cdot& \cdot& 2& \cdot& 1& \cdot& \cdot& \cdot& 1& 0& \cdot& \cdot& \cdot& \cdot& \cdot& \cdot& \cdot& \cdot& \cdot& \cdot& \cdot\\& 
\cdot& \cdot& \cdot& \cdot& \cdot& \cdot& \cdot& 2& 1& \cdot& \cdot& \cdot& \cdot& \cdot& \cdot& 1& 0& \cdot& \cdot& \cdot& \cdot& \cdot& \cdot& \cdot& \cdot& \cdot& \cdot\\& 
2& 1& \cdot& \cdot& \cdot& \cdot& \cdot& \cdot& \cdot& \cdot& \cdot& \cdot& 1& \cdot& \cdot& \cdot& \cdot& 0& \cdot& \cdot& \cdot& \cdot& \cdot& \cdot& \cdot& \cdot& \cdot\\& 
\cdot& \cdot& \cdot& \cdot& \cdot& \cdot& \cdot& \cdot& \cdot& \cdot& \cdot& \cdot& 2& 1& \cdot& \cdot& \cdot& 1& 0& \cdot& \cdot& \cdot& \cdot& \cdot& \cdot& \cdot& \cdot\\& 
\cdot& \cdot& \cdot& \cdot& \cdot& \cdot& \cdot& \cdot& \cdot& \cdot& \cdot& \cdot& \cdot& 2& 1& \cdot& \cdot& \cdot& 1& 0& \cdot& \cdot& \cdot& \cdot& \cdot& \cdot& \cdot\\& 
\cdot& \cdot& \cdot& \cdot& \cdot& \cdot& \cdot& \cdot& \cdot& \cdot& \cdot& \cdot& \cdot& \cdot& 2& 1& \cdot& \cdot& \cdot& 1& 0& \cdot& \cdot& \cdot& \cdot& \cdot& \cdot\\& 
\cdot& \cdot& \cdot& \cdot& \cdot& \cdot& \cdot& \cdot& \cdot& \cdot& \cdot& \cdot& \cdot& \cdot& \cdot& 2& 1& \cdot& \cdot& \cdot& 1& 0& \cdot& \cdot& \cdot& \cdot& \cdot\\& 
\cdot& \cdot& \cdot& \cdot& \cdot& \cdot& \cdot& \cdot& \cdot& \cdot& \cdot& 2& \cdot& \cdot& 1& \cdot& \cdot& \cdot& \cdot& \cdot& 1& \cdot& 0& \cdot& \cdot& \cdot& \cdot\\& 
\cdot& \cdot& \cdot& \cdot& \cdot& \cdot& \cdot& \cdot& \cdot& \cdot& \cdot& \cdot& \cdot& \cdot& \cdot& \cdot& \cdot& \cdot& \cdot& \cdot& 2& 1& 1& 0& \cdot& \cdot& \cdot\\& 
\cdot& \cdot& \cdot& \cdot& \cdot& \cdot& \cdot& \cdot& \cdot& \cdot& \cdot& \cdot& \cdot& \cdot& \cdot& \cdot& \cdot& \cdot& \cdot& 2& 1& \cdot& \cdot& 1& 0& \cdot& \cdot\\& 
\cdot& \cdot& \cdot& \cdot& \cdot& \cdot& \cdot& \cdot& \cdot& \cdot& \cdot& \cdot& \cdot& \cdot& \cdot& \cdot& \cdot& \cdot& 2& 1& \cdot& \cdot& \cdot& \cdot& 1& 0& \cdot\\& 
\cdot& \cdot& \cdot& \cdot& \cdot& \cdot& \cdot& \cdot& \cdot& \cdot& \cdot& \cdot& \cdot& \cdot& \cdot& \cdot& \cdot& 2& 1& \cdot& \cdot& \cdot& \cdot& \cdot& \cdot& 1& 0\\\\
\end{array}			
			$$
\caption{The light leaves matrix $\Delta_{\la , \mu}$ of type $(E_6, D_5)$.  
For purposes of space, we record $q^i$ simply as $i$, and we record each zero polynomials as a dot.  For example, the matrix is uni-triangular with diagonal entries $q^0=1$.  
The rows  and columns  are ordered by a total refinement of the Bruhat order in which we prefer to add the reflection  with largest possible subscript.   
More specifically, the order is as follows 
  $\varnothing\color{black}   , \;      
\color{magenta}s_1    \color{black}   , \;           		 
\color{magenta}s_1  \color{cyan}s_2    \color{black}   , \;           
\color{magenta}s_1  \color{cyan}s_2  \color{green!80!black}s_3   \color{black}   , \;            
\color{magenta}s_1  \color{cyan}s_2  \color{green!80!black}s_3 \color{violet}s_6    \color{black}   , \;           
\color{magenta}s_1  \color{cyan}s_2  \color{green!80!black}s_3 \color{orange}s_4    \color{black}   , \;           
\color{magenta}s_1  \color{cyan}s_2  \color{green!80!black}s_3 \color{violet}s_6  \color{orange}s_4    \color{black}   , \;           
\color{magenta}s_1  \color{cyan}s_2  \color{green!80!black}s_3 \color{orange}s_4  \color{gray}s_5   \color{black}   , \;  $
$         
\color{magenta}s_1  \color{cyan}s_2  \color{green!80!black}s_3 \color{violet}s_6  \color{orange}s_4  \color{gray}s_5   \color{black}   , \;           
\color{magenta}s_1  \color{cyan}s_2  \color{green!80!black}s_3 \color{violet}s_6  \color{orange}s_4  \color{green!80!black}s_3   \color{black}   , \;           
\color{magenta}s_1  \color{cyan}s_2  \color{green!80!black}s_3 \color{violet}s_6  \color{orange}s_4  \color{gray}s_5 \color{green!80!black}s_3   \color{black}   , \;           
\color{magenta}s_1  \color{cyan}s_2  \color{green!80!black}s_3 \color{violet}s_6  \color{orange}s_4  \color{gray}s_5 \color{green!80!black}s_3 \color{orange}s_4    \color{black}   , \;           
\color{magenta}s_1  \color{cyan}s_2  \color{green!80!black}s_3 \color{violet}s_6  \color{orange}s_4  \color{green!80!black}s_3 \color{cyan}s_2    \color{black}   , \;           
\color{magenta}s_1  \color{cyan}s_2  \color{green!80!black}s_3 \color{violet}s_6  \color{orange}s_4  \color{gray}s_5 \color{green!80!black}s_3 \color{cyan}s_2    \color{black}   , \;  $
$        
\color{magenta}s_1  \color{cyan}s_2  \color{green!80!black}s_3 \color{violet}s_6  \color{orange}s_4  \color{gray}s_5 \color{green!80!black}s_3 \color{orange}s_4  \color{cyan}s_2    \color{black}   , \;           
\color{magenta}s_1  \color{cyan}s_2  \color{green!80!black}s_3 \color{violet}s_6  \color{orange}s_4  \color{gray}s_5 \color{green!80!black}s_3 \color{orange}s_4  \color{cyan}s_2  \color{green!80!black}s_3   \color{black}   , \;           
\color{magenta}s_1  \color{cyan}s_2  \color{green!80!black}s_3 \color{violet}s_6  \color{orange}s_4  \color{gray}s_5 \color{green!80!black}s_3 \color{orange}s_4  \color{cyan}s_2  \color{green!80!black}s_3 \color{violet}s_6    \color{black}   , \;           
\color{magenta}s_1  \color{cyan}s_2  \color{green!80!black}s_3 \color{violet}s_6  \color{orange}s_4  \color{green!80!black}s_3 \color{cyan}s_2  \color{magenta}s_1    \color{black}   , \;           $
$
\color{magenta}s_1  \color{cyan}s_2  \color{green!80!black}s_3 \color{violet}s_6  \color{orange}s_4  \color{gray}s_5 \color{green!80!black}s_3 \color{cyan}s_2  \color{magenta}s_1    \color{black}   , \;           
\color{magenta}s_1  \color{cyan}s_2  \color{green!80!black}s_3 \color{violet}s_6  \color{orange}s_4  \color{gray}s_5 \color{green!80!black}s_3 \color{orange}s_4  \color{cyan}s_2  \color{magenta}s_1    \color{black}   , \;           
\color{magenta}s_1  \color{cyan}s_2  \color{green!80!black}s_3 \color{violet}s_6  \color{orange}s_4  \color{gray}s_5 \color{green!80!black}s_3 \color{orange}s_4  \color{cyan}s_2  \color{green!80!black}s_3 \color{magenta}s_1    \color{black}   , \;           
\color{magenta}s_1  \color{cyan}s_2  \color{green!80!black}s_3 \color{violet}s_6  \color{orange}s_4  \color{gray}s_5 \color{green!80!black}s_3 \color{orange}s_4  \color{cyan}s_2  \color{green!80!black}s_3 \color{violet}s_6  \color{magenta}s_1    \color{black}   , \;     
      $
$
\color{magenta}s_1  \color{cyan}s_2  \color{green!80!black}s_3 \color{violet}s_6  \color{orange}s_4  \color{gray}s_5 \color{green!80!black}s_3 \color{orange}s_4  \color{cyan}s_2  \color{green!80!black}s_3 \color{magenta}s_1  \color{cyan}s_2    \color{black}   , \;           
\color{magenta}s_1  \color{cyan}s_2  \color{green!80!black}s_3 \color{violet}s_6  \color{orange}s_4  \color{gray}s_5 \color{green!80!black}s_3 \color{orange}s_4  \color{cyan}s_2  \color{green!80!black}s_3 \color{violet}s_6  \color{magenta}s_1  \color{cyan}s_2    \color{black}   , \;           
\color{magenta}s_1  \color{cyan}s_2  \color{green!80!black}s_3 \color{violet}s_6  \color{orange}s_4  \color{gray}s_5 \color{green!80!black}s_3 \color{orange}s_4  \color{cyan}s_2  \color{green!80!black}s_3 \color{violet}s_6  \color{magenta}s_1  \color{cyan}s_2  \color{green!80!black}s_3   \color{black}   , \;           $
$
\color{magenta}s_1  \color{cyan}s_2  \color{green!80!black}s_3 \color{violet}s_6  \color{orange}s_4  \color{gray}s_5 \color{green!80!black}s_3 \color{orange}s_4  \color{cyan}s_2  \color{green!80!black}s_3 \color{violet}s_6  \color{magenta}s_1  \color{cyan}s_2  \color{green!80!black}s_3 \color{orange}s_4    \color{black}   , \;           
\color{magenta}s_1  \color{cyan}s_2  \color{green!80!black}s_3 \color{violet}s_6  \color{orange}s_4  \color{gray}s_5 \color{green!80!black}s_3 \color{orange}s_4  \color{cyan}s_2  \color{green!80!black}s_3 \color{violet}s_6  \color{magenta}s_1  \color{cyan}s_2  \color{green!80!black}s_3 \color{orange}s_4  \color{gray}s_5   \color{black}  .         
$ 
    }
\label{tableExc}
 \end{figure}

 \begin{figure}[ht!]
 $$\scalefont{0.6}
\begin{array}{ccccccccccccccccccccccccccccccccccccccccccccccccccccccccccccccc}
         0	    \!\!\!&\!\!\!      \cdot    \!\!\!&\!\!\!      \cdot    \!\!\!&\!\!\!      \cdot    \!\!\!&\!\!\!      \cdot    \!\!\!&\!\!\!      \cdot    \!\!\!&\!\!\!      \cdot    \!\!\!&\!\!\!      \cdot    \!\!\!&\!\!\!      \cdot    \!\!\!&\!\!\!      \cdot    \!\!\!&\!\!\!      \cdot    \!\!\!&\!\!\!      \cdot    \!\!\!&\!\!\!      \cdot    \!\!\!&\!\!\!      \cdot    \!\!\!&\!\!\!      \cdot    \!\!\!&\!\!\!      \cdot    \!\!\!&\!\!\!      \cdot    \!\!\!&\!\!\!      \cdot    \!\!\!&\!\!\!      \cdot    \!\!\!&\!\!\!      \cdot    \!\!\!&\!\!\!      \cdot    \!\!\!&\!\!\!      \cdot    \!\!\!&\!\!\!      \cdot    \!\!\!&\!\!\!      \cdot    \!\!\!&\!\!\!      \cdot    \!\!\!&\!\!\!      \cdot    \!\!\!&\!\!\!      \cdot    \!\!\!&\!\!\!      \cdot    \!\!\!&\!\!\!      \cdot    \!\!\!&\!\!\!      \cdot    \!\!\!&\!\!\!      \cdot    \!\!\!&\!\!\!      \cdot    \!\!\!&\!\!\!      \cdot    \!\!\!&\!\!\!      \cdot    \!\!\!&\!\!\!      \cdot    \!\!\!&\!\!\!      \cdot    \!\!\!&\!\!\!      \cdot    \!\!\!&\!\!\!      \cdot    \!\!\!&\!\!\!      \cdot    \!\!\!&\!\!\!      \cdot    \!\!\!&\!\!\!      \cdot    \!\!\!&\!\!\!      \cdot    \!\!\!&\!\!\!      \cdot    \!\!\!&\!\!\!      \cdot    \!\!\!&\!\!\!      \cdot    \!\!\!&\!\!\!      \cdot    \!\!\!&\!\!\!      \cdot    \!\!\!&\!\!\!      \cdot    \!\!\!&\!\!\!      \cdot    \!\!\!&\!\!\!      \cdot    \!\!\!&\!\!\!      \cdot    \!\!\!&\!\!\!      \cdot    \!\!\!&\!\!\!      \cdot    \!\!\!&\!\!\!      \cdot    \!\!\!&\!\!\!      \cdot                \\
\end{array}
$$
\caption{The light leaves matrix $\Delta$ of type $(E_7, E_6)$.  
For purposes of space, we record $q^i$ simply as $i$, and we record each zero polynomials as a dot.  For example. the matrix is uni-triangular with diagonal entries $q^0=1$.  
The rows and columns are ordered by a total refinement of the Bruhat order in which add the reflection with largest possible subscript.  
 }
 \label{tableExc2}
\end{figure}

\subsection{Type   $(D_n,D_{n-1})$.}
  In this  case, there is precisely one element  in ${^PW}$ of each length $0\leq \ell \leq 2(n-1)$, $\ell \neq n-1$ and precisely two elements  of length $n-1$. Ordering the rows and columns of the matrix $\Delta$ by decreasing length we have
\[ 
 \Delta=\scalefont{0.9}
\left[ \begin{array}{ccccc|cc|cccccccc}
 1& \cdot    & \cdots & \cdot  & \cdot  			& \cdot  & \cdot  		& \cdot  & \cdot  & \cdots  & \cdot  & \cdot 		\\
 q & 1    & \cdots  & \cdot & \cdot  	& \cdot 	 & \cdot  				& \cdot  & \cdot &    &    & \cdot 		\\
\vdots  &   &\ddots &    & \vdots  			& \vdots  & \vdots  		& \vdots  &   & \ddots  & 	 &   \vdots 		\\
   \cdot  & \cdot   & \cdots  & 1 & \cdot  		& \cdot  & \cdot  		& \cdot  &  &    & \cdot   & \cdot 		\\
    \cdot  & \cdot   & \cdots  & q &  1 			& \cdot  & \cdot  		& \cdot  & \cdot  & \cdots  & \cdot  & \cdot 		\\ 
    \hline 
    \cdot  & \cdot   & \cdots  & \cdot  &  q 			&   1 		& \cdot & \cdot  & \cdot  & \cdots  & \cdot  & \cdot 		\\ 
        \cdot  & \cdot   & \cdots  & \cdot  &  q 			& \cdot  & 1 		& \cdot  & \cdot  & \cdots  & \cdot  & \cdot 		\\ 
        \hline 
        \cdot  & \cdot   & \cdots  & q &  q^2 			& q&q & 1 		  & \cdot  & \cdots  & \cdot  & \cdot 		\\ 
 \cdot	 &\cdot	&\cdots	&q^2		&\cdot	& \cdot			&\cdot& 	 q & 1    & \cdots  & \cdot & \cdot  	  	  \\
 \vdots 		&		&\iddots			&	&\vdots			&\vdots	&\vdots&\vdots		   &   &\ddots &    & \vdots  			 \\
 q		&q^2		&\cdots				&\cdot	&\cdot		&\cdot			&\cdot		&   \cdot  & \cdot   & \cdots  & 1 & \cdot  		 \\
 q^2		&\cdot		&	\cdots		&\cdot	&\cdot		&\cdot			&\cdot		  &  \cdot  & \cdot   & \cdots  & q &  1 			  
                 \end{array}\right]
\]
that is, the top left and bottom right $(n-1)\times (n-1)$-matrices have non-zero entries on the diagonal and sub-diagonal only;
the bottom left $(n-1)\times (n-1)$-matrix has non-zero entries on the anti-diagonal and sup-anti-diagonal only.  
In this case the matrix factorisation is trivial, with $B={\rm Id}$ and $N=\Delta$.  This gives a proof of Theorem A in type $(D_n, D_{n-1})$ by setting every basis elements in the anti-spherical module for ${\rm TL}_{(D_n, D_{n-1})}(q)$ to be standard.

\subsection{Type $(B_n,B_{n-1})$} 
In this type there is precisely one element $\la\in {^PW}$ of each length. Ordering the rows and columns of the matrix $\Delta$ by  decreasing length it is easy to see that 
\[
 \Delta=\scalefont{0.9}
\left[ \begin{array}{ccccc|ccccccccc}
 1& \cdot    & \cdots & \cdot  & \cdot  			& \cdot  & \cdot  		& \cdot  & \cdot    & \cdot 		\\
 q & 1    & \cdots  & \cdot & \cdot  	& \cdot 	 & \cdot  				& \cdot  & \cdot    & \cdot 		\\
\vdots  &   &\ddots &    & \vdots  			 	& \vdots  &   & \ddots  & 	 &   \vdots 		\\
   \cdot  & \cdot   & \cdots  & 1 & \cdot  		 		& \cdot  &  &    & \cdot   & \cdot 		\\
    \cdot  & \cdot   & \cdots  & q &  1 			  		& \cdot  & \cdot  & \cdots  & \cdot  & \cdot 		\\ 
        \hline 
        \cdot  & \cdot   & \cdots  & 1 &  q 			 & 1 		  & \cdot  & \cdots  & \cdot  & \cdot 		\\ 
 \cdot	 &\cdot	&\cdots	&q 		&\cdot	   	&	 q & 1    & \cdots  & \cdot & \cdot  	  	  \\
 \vdots 		&		&\iddots			&	&\vdots			 &\vdots		   &   &\ddots &    & \vdots  			 \\
1		&q		&\cdots				&\cdot	&\cdot		 		&   \cdot  & \cdot   & \cdots  & 1 & \cdot  		 \\
 q 		&\cdot		&	\cdots		&\cdot	&\cdot		 	  &  \cdot  & \cdot   & \cdots  & q &  1 			  
                 \end{array}\right]
\]
that is, the top left and bottom right $n\times n$-matrices have non-zero entries on the diagonal and sub-diagonal only;
the bottom left $n\times n$-matrix has non-zero entries on the anti-diagonal and sub-anti-diagonal only.  
We then immediately  deduce that the matrix factorisation is as follows 
\[\scalefont{0.9}
\left[ \begin{array}{ccccc|ccccccccc}
 1& \cdot    & \cdots & \cdot  & \cdot  			& \cdot  & \cdot  		& \cdot  & \cdot    & \cdot 		\\
 q & 1    & \cdots  & \cdot & \cdot  	& \cdot 	 & \cdot  				& \cdot  & \cdot    & \cdot 		\\
\vdots  &   &\ddots &    & \vdots  			 	& \vdots  &   & \ddots  & 	 &   \vdots 		\\
   \cdot  & \cdot   & \cdots  & 1 & \cdot  		 		& \cdot  &  &    & \cdot   & \cdot 		\\
    \cdot  & \cdot   & \cdots  & q &  1 			  		& \cdot  & \cdot  & \cdots  & \cdot  & \cdot 		\\ 
        \hline 
        \cdot  & \cdot   & \cdots  & \cdot &  q 			 & 1 		  & \cdot  & \cdots  & \cdot  & \cdot 		\\ 
 \cdot	 &\cdot	&\cdots	&\cdot 		&\cdot	   	&	 q & 1    & \cdots  & \cdot & \cdot  	  	  \\
 \vdots 		&		&\iddots			&	&\vdots			 &\vdots		   &   &\ddots &    & \vdots  			 \\
\cdot		&\cdot		&\cdots				&\cdot	&\cdot		 		&   \cdot  & \cdot   & \cdots  & 1 & \cdot  		 \\
 \cdot 		&\cdot	&	\cdots		&\cdot	&\cdot		 	  &  \cdot  & \cdot   & \cdots  & q &  1 			  
                 \end{array}\right]
\left[ \begin{array}{ccccc|ccccccccc}
 1& \cdot    & \cdots & \cdot  & \cdot  			& \cdot  & \cdot  		& \cdot  & \cdot    & \cdot 		\\
 \cdot & 1    & \cdots  & \cdot & \cdot  	& \cdot 	 & \cdot  				& \cdot  & \cdot    & \cdot 		\\
\vdots  &   &\ddots &    & \vdots  			 	& \vdots  &   & \ddots  & 	 &   \vdots 		\\
   \cdot  & \cdot   & \cdots  & 1 & \cdot  		 		& \cdot  &  &    & \cdot   & \cdot 		\\
    \cdot  & \cdot   & \cdots  & \cdot &  1 			  		& \cdot  & \cdot  & \cdots  & \cdot  & \cdot 		\\ 
        \hline 
        \cdot  & \cdot   & \cdots  &1 &  \cdot 			 & 1 		  & \cdot  & \cdots  & \cdot  & \cdot 		\\ 
 \cdot	 &\cdot	&\cdots	&\cdot		&\cdot 	   	&	 \cdot & 1    & \cdots  & \cdot & \cdot  	  	  \\
 \vdots 		&		&\iddots			&	&\vdots			 &\vdots		   &   &\ddots &    & \vdots  			 \\
1		&\cdot		&\cdots				&\cdot	&\cdot		 		&   \cdot  & \cdot   & \cdots  & 1 & \cdot  		 \\
 \cdot 		&\cdot		&	\cdots		&\cdot	&\cdot		 	  &  \cdot  & \cdot   & \cdots  &\cdot &  1 			  
                 \end{array}\right].
\]
where we note that the matrix on the left is the same as the matrix of Kazhdan--Lusztig polynomials of type $(A_{2n-1}, A_{2n-2})$, this will be explained in our companion paper \cite{compan}.
This gives a proof of Theorem A in type $(B_n, B_{n-1})$ by setting every basis element $E_{\SSTT}$ corresponding to non-zero entries in the matrix $N$ (on the left) to be standard.

From now, until the end of Section 8, we focus solely on the remaining cases, namely $(W,P) = (A_{n}, A_{k-1}\times A_{n-k}), (C_n, A_{n-1})$ and $(D_n, A_{n-1})$.

\section{Bruhat graphs in classical type}\label{section4}

We  now  introduce an  elementary way of visualising  the graphs $\mathcal{G}_{(W,P)}$ for   classical    Hermitian symmetric pairs  $(W,P)=(A_n, A_{k-1}\times A_{n-k}), (C_n, A_{n-1})$ and $(D_n, A_{n-1})$.  
We will use these in the next section to define a diagrammatic visualisation of the oriented Temperley--Lieb algebra in these types.

We start by recalling the description of the Coxeter groups of type $A_n, C_n$ and $D_n$ as groups of (signed) permutations. The Coxeter group of type $A_n$ is the symmetric group consisting of all permutations of $\{1, 2, \ldots , n+1\}$. The simple reflections $s_i$, for $1\leq i\leq n$ are given by 
\begin{equation*}
s_i = (i, i+1).
\end{equation*}
The Coxeter group of type $C_n$  is the signed permutation group, namely the group of all permutations $w$ of $\{\pm 2, \pm 3, \ldots , \pm (n+1)\}$ such that $w(-i)=-w(i)$ for all $2\leq i\leq n+1$. The simple reflections $s_i$ for $i=1', 2, \ldots , n$ are given by 
\begin{equation*}
s_{1'} = (2, -2) \quad \mbox{ and} \quad s_i = (i,i+1)(-i, -(i+1)) \,\,\mbox{for $2 \leq i \leq n$}.
\end{equation*}
The Coxeter group of type $D_n$ is the even signed permutation group, that is the subgroup of the group of all signed permutations on $\{\pm 1 , \pm 2, \ldots , \pm n\}$ consisting of all elements flipping an even number of signs. The simple reflections $s_i$ for $i=1'', 1 , 2, \ldots, n-1$ are given by
\[s_{1''} = (1, -2)(-1,2) \quad \mbox{and} \quad s_i = (i,i+1)(-i, -(i+1)) \,\, \mbox{for $1\leq i\leq n-1$}.\]

\begin{rmk}
The choice of labelling in  type $C_n$ might seem slightly unnatural at this point but  (apart from giving a uniform definition of the generators $s_i$ for $2\leq i \leq n-1$ in all types) it is required for compatibility with the
diagrammatic  Temperley--Lieb algebra of type $C$ given in \cref{results1}. 
\end{rmk}

Given this description of the Coxeter groups as (signed) permutation groups, we have the following natural visualisation of the cosets ${^PW}$. 
We will represent elements of ${^PW}$ as horizontal lines with $n(+1)$ points in positions $1, 2 , \ldots , n (,n+1)$ labelled with the symbols $\{\up, \down, \circ\}$. The generators of $W$ will act on these as follows: 
For $i\geq 1$, $s_i$ swaps the labels in positions $i$ and $i+1$, 
the generator $s_{1'}$ flips (through the horizontal axis) the label in second position, and 
$s_{1''}$ swaps and flips (through the horizontal line) the labels in first and second positions.
Now, we start by representing the identity coset, which we denote by $\varnothing$ (for the empty word in the generators), as follows:
\begin{itemize}
\item If $(W,P)=(A_n, A_{k-1}\times A_{n-k})$, draw $\varnothing$ as the horizontal line containing $n+1$ points with the first $k$ points labelled by $\up$ and the last $n-k+1$ points labelled by $\down$.
\item If $(W,P)=(C_n, A_{n-1})$ draw $\varnothing$ as the horizontal line with $n+1$ points with the first point labelled by $\circ$ and the last $n$ points labelled by $\down$. 
\item If $(W,P)=(D_n , A_{n-1})$ draw $\varnothing$ as the horizontal line with $n$ points all labelled by $\down$.
\end{itemize}
The elements of ${^PW}$ are then obtained as the elements of the orbit of $\varnothing$ under the action of $W$.

 \begin{figure}[ht!]
 $$
  \begin{tikzpicture} [scale=0.7]

\path (0,-0.7)--++(135:0.5)--++(-135:0.5) coordinate (minus1);

\path (minus1)--++(135:0.5)--++(-135:0.5) coordinate (minus2);

\path (minus2)--++(135:0.5)--++(-135:0.5) coordinate (minus3);

\path (minus3)--++(135:0.5)--++(-135:0.5) coordinate (minus4);

\path (0,-0.7)--++(45:0.5)--++(-45:0.5) coordinate (plus1);

\path (plus1)--++(45:0.5)--++(-45:0.5) coordinate (plus2);

\path (plus2)--++(45:0.5)--++(-45:0.5) coordinate (plus3);

\draw[very thick](plus3)--(minus4);

\path(minus1)--++(45:0.4) coordinate (minus1NE);
\path(minus1)--++(-45:0.4) coordinate (minus1SE);

\path(minus4)--++(135:0.4) coordinate (minus1NW);
\path(minus4)--++(-135:0.4) coordinate (minus1SW);

\path(minus2)--++(135:0.4) coordinate (start);

\draw[rounded corners, thick] (start)--(minus1NE)--(minus1SE)--(minus1SW)--(minus1NW)--(start);

\path(plus3)--++(45:0.4) coordinate (minus1NE);
\path(plus3)--++(-45:0.4) coordinate (minus1SE);

\path(plus1)--++(135:0.4) coordinate (minus1NW);
\path(plus1)--++(-135:0.4) coordinate (minus1SW);

\path(plus2)--++(135:0.4) coordinate (start);

\draw[rounded corners, thick] (start)--(minus1NE)--(minus1SE)--(minus1SW)--(minus1NW)--(start);

\draw[very thick,fill=magenta](0,-0.7) circle (4pt);

\draw[very thick,  fill=darkgreen](minus1) circle (4pt);

\draw[very thick,  fill=orange](minus2) circle (4pt);

\draw[very thick,  fill=lime!80!black](minus3) circle (4pt);

\draw[very thick,  fill=violet](minus4) circle (4pt);

\draw[very thick,  fill=gray!80](plus1) circle (4pt);

\draw[very thick,  fill=cyan](plus2) circle (4pt);

\draw[very thick,  fill=pink](plus3) circle (4pt);

\path (0,0) coordinate (origin2);

\begin{scope}

     \foreach \i in {0,1,2,3,4,5,6,7,8,9,10,11,12}
{
\path (origin2)--++(45:0.5*\i) coordinate (c\i); 
\path (origin2)--++(135:0.5*\i)  coordinate (d\i); 
  }

\path(origin2)  ++(135:2.5)   ++(-135:2.5) coordinate(corner1);
\path(origin2)  ++(45:2)   ++(135:7) coordinate(corner2);
\path(origin2)  ++(45:2)   ++(-45:2) coordinate(corner4);
\path(origin2)  ++(135:2.5)   ++(45:6.5) coordinate(corner3);

\draw[thick] (origin2)--(corner1)--(corner2)--(corner3)--(corner4)--(origin2);

\clip(corner1)--(corner2)--++(90:0.3)--++(0:6.5)--(corner3)--(corner4)
--++(90:-0.3)--++(180:6.5) --(corner1);

\path[name path=pathd1] (d1)--++(90:7);   
 \path[name path=top] (corner2)--(corner3);   
 \path [name intersections={of = pathd1 and top}];
   \coordinate (A)  at (intersection-1);
     \path(A)--++(-90:0.1) node { $\up$ };

\path[name path=pathd3] (d3)--++(90:7);   
 \path[name path=top] (corner2)--(corner3);   
 \path [name intersections={of = pathd3 and top}];
   \coordinate (A)  at (intersection-1);
  \path(A)--++(90:-0.1) node { $\up$ };

\path[name path=pathd5] (d5)--++(90:7);   
 \path[name path=top] (corner2)--(corner3);   
 \path [name intersections={of = pathd5 and top}];
   \coordinate (A)  at (intersection-1);
  \path(A)--++(90:0.1) node { $\down$ };

\path[name path=pathd7] (d7)--++(90:7);   
 \path[name path=top] (corner2)--(corner3);   
 \path [name intersections={of = pathd7 and top}];
   \coordinate (A)  at (intersection-1); 
     \path(A)--++(90:-0.1) node { $\up$ };

\path[name path=pathd9] (d9)--++(90:7);   
 \path[name path=top] (corner2)--(corner3);   
 \path [name intersections={of = pathd9 and top}];
   \coordinate (A)  at (intersection-1);
    \path(A)--++(-90:-0.1) node { $\down$ };

\path[name path=pathc1] (c1)--++(90:7);   
 \path[name path=top] (corner2)--(corner3);   
 \path [name intersections={of = pathc1 and top}];
   \coordinate (A)  at (intersection-1);
  \path(A)--++(-90:-0.1) node { $\down$ };

\path[name path=pathc3] (c3)--++(90:7);   
 \path[name path=top] (corner2)--(corner3);   
 \path [name intersections={of = pathc3 and top}];
   \coordinate (A)  at (intersection-1);
     \path(A)--++(-90:-0.1) node { $\down$ };

\path[name path=pathc5] (c5)--++(90:7);   
 \path[name path=top] (corner2)--(corner3);   
 \path [name intersections={of = pathc5 and top}];
   \coordinate (A)  at (intersection-1);
    \path(A)--++(90:-0.1) node { $\up$ };

\path[name path=pathc7] (c7)--++(90:7);   
 \path[name path=top] (corner2)--(corner3);   
 \path [name intersections={of = pathc7 and top}];
   \coordinate (A)  at (intersection-1);
   \path(A)--++(-90:0.1) node { $\up$ };

 \path[name path=pathd1] (d1)--++(-90:7);   
 \path[name path=bottom] (corner1)--(corner4);   
 \path [name intersections={of = pathd1 and bottom}];
   \coordinate (A)  at (intersection-1);
   \path (A)--++(90:-0.1) node { $\up$  };
   
 \path[name path=pathd3] (d3)--++(-90:7);   
 \path[name path=bottom] (corner1)--(corner4);   
 \path [name intersections={of = pathd3 and bottom}];
   \coordinate (A)  at (intersection-1);
   \path (A)--++(90:-0.1) node { $\up$  };

  \path[name path=pathd5] (d5)--++(-90:7);   
 \path[name path=bottom] (corner1)--(corner4);   
 \path [name intersections={of = pathd5 and bottom}];
   \coordinate (A)  at (intersection-1);
   \path (A)--++(90:-0.1) node { $\up$  };

 \path[name path=pathd7] (d7)--++(-90:7);   
 \path[name path=bottom] (corner1)--(corner4);   
 \path [name intersections={of = pathd7 and bottom}];
   \coordinate (A)  at (intersection-1);
   \path (A)--++(90:-0.1) node { $\up$  };
   
 \path[name path=pathd9] (d9)--++(-90:7);   
 \path[name path=bottom] (corner1)--(corner4);   
 \path [name intersections={of = pathd9 and bottom}];
   \coordinate (A)  at (intersection-1);
   \path (A)--++(90:-0.1) node { $\up$  };

 \path[name path=pathc1] (c1)--++(-90:7);   
 \path[name path=bottom] (corner1)--(corner4);   
 \path [name intersections={of = pathc1 and bottom}];
   \coordinate (A)  at (intersection-1);
   \path (A)--++(90:0.1) node { $\down$  };

    \path[name path=pathc3] (c3)--++(-90:7);   
 \path[name path=bottom] (corner1)--(corner4);   
 \path [name intersections={of = pathc3 and bottom}];
   \coordinate (A)  at (intersection-1);
   \path (A)--++(90:0.1) node { $\down$  };

 \path[name path=pathc5] (c5)--++(-90:7);   
 \path[name path=bottom] (corner1)--(corner4);   
 \path [name intersections={of = pathc5 and bottom}];
   \coordinate (A)  at (intersection-1);
   \path (A)--++(90:0.1) node { $\down$  };

 \path[name path=pathc7] (c7)--++(-90:7);   
 \path[name path=bottom] (corner1)--(corner4);   
 \path [name intersections={of = pathc7 and bottom}];
   \coordinate (A)  at (intersection-1);
   \path (A)--++(90:0.1) node { $\down$  };

\clip(corner1)--(corner2)--(corner3)--(corner4)--(corner1);

  \foreach \i in {1,3,5,7,9,11}
{
 \draw[thick](c\i)--++(90:7);   
 \draw[thick](c\i)--++(-90:7);
\draw[thick](d\i)--++(90:7);
\draw[thick](d\i)--++(-90:7);
   }

\end{scope}

\begin{scope}
\clip(0,0) --++(45:4)--++(135:1)
--++(135:1)--++(-135:1)--++(-135:1)--++(135:2)--++(-135:1)--++(135:1)
--++(-135:1)--++(-45:5);
  
\path (0,0) coordinate (origin2);

 \foreach \i\j in {0,1,2,3,4,5,6,7,8,9,10,11,12}
{
\path (origin2)--++(45:0.5*\i) coordinate (a\i); 
\path (origin2)--++(135:0.5*\i)  coordinate (b\j); 


   }
     
  \fill[white]
(0,0) --++(45:4)--++(135:5)--++(-135:4)--++(-45:5);

\draw[very thick,magenta](c1) --++(135:1) coordinate (cC1);
\draw[very thick,darkgreen](cC1) --++(135:1) coordinate (cC1);
\draw[very thick,orange](cC1) --++(135:1) coordinate (cC1);   
\draw[very thick,lime!80!black](cC1) --++(135:1) coordinate (cC1);
\draw[very thick,violet](cC1) --++(135:1) coordinate (cC1);

\draw[very thick,gray](c3) --++(135:1) coordinate (cC1);
\draw[very thick,magenta](cC1) --++(135:1) coordinate (cC1);
\draw[very thick,darkgreen](cC1) --++(135:1) coordinate (cC1);
\draw[very thick,orange](cC1) --++(135:1) coordinate (cC1);   
\draw[very thick,lime!80!black](cC1) --++(135:1) coordinate (cC1);

  \draw[very thick,cyan](c5) --++(135:1) coordinate (cC1);
 \draw[very thick,gray](cC1) --++(135:1) coordinate (cC1);
\draw[very thick,magenta](cC1) --++(135:1) coordinate (cC1);
\draw[very thick,darkgreen](cC1) --++(135:1) coordinate (cC1);
\draw[very thick,orange](cC1) --++(135:1) coordinate (cC1);   
\draw[very thick,lime!80!black](cC1) --++(135:1) coordinate (cC1);

  \draw[very thick,pink](c7) --++(135:1) coordinate (cC1);
  \draw[very thick,cyan](cC1) --++(135:1) coordinate (cC1);
 \draw[very thick,gray](cC1) --++(135:1) coordinate (cC1);
\draw[very thick,magenta](cC1) --++(135:1) coordinate (cC1);
\draw[very thick,darkgreen](cC1) --++(135:1) coordinate (cC1);
\draw[very thick,orange](cC1) --++(135:1) coordinate (cC1);   
\draw[very thick,lime!80!black](cC1) --++(135:1) coordinate (cC1);

\draw[very thick,magenta](d1) --++(45:1) coordinate (x1); 
\draw[very thick,gray](x1) --++(45:1) coordinate (x1); 
 \draw[very thick,cyan](x1) --++(45:1) coordinate (x1); 
 \draw[very thick,pink](x1) --++(45:1) coordinate (x1);

\draw[very thick,darkgreen](d3) --++(45:1) coordinate (x1);
\draw[very thick,magenta](x1) --++(45:1) coordinate (x1); 
\draw[very thick,gray](x1) --++(45:1) coordinate (x1); 
 \draw[very thick,cyan](x1) --++(45:1) coordinate (x1); 
 \draw[very thick,pink](x1) --++(45:1) coordinate (x1);

\draw[very thick,orange](d5) --++(45:1) coordinate (x1);
\draw[very thick,darkgreen](x1) --++(45:1) coordinate (x1);
\draw[very thick,magenta](x1) --++(45:1) coordinate (x1); 
\draw[very thick,gray](x1) --++(45:1) coordinate (x1); 
 \draw[very thick,cyan](x1) --++(45:1) coordinate (x1); 
 \draw[very thick,pink](x1) --++(45:1) coordinate (x1);

\draw[very thick,lime!80!black](d7) --++(45:1) coordinate (x1);
\draw[very thick,orange](x1) --++(45:1) coordinate (x1);

\path(x1) --++(45:1) coordinate (x1);
\path(x1) --++(45:1) coordinate (x1); 
\path(x1) --++(45:1) coordinate (x1); 
\path(x1) --++(45:1) coordinate (x1); 
\path(x1) --++(45:1) coordinate (x1); 

\draw[very thick,violet](d9) --++(45:1) coordinate (x1);

   \foreach \i in {0,1,2,3,4,5,6,7,8,9,10,11,12}
{
\path (origin2)--++(45:1*\i) coordinate (c\i); 
\path (origin2)--++(135:1*\i)  coordinate (d\i);  
\fill[gray!5,opacity=0.01,rounded corners] (c0)--++(135:5)--++(45:4)
 --++(-45:5)--++(-135:4)--++(135:1);
 }

\end{scope}

\path(origin2)  ++(135:2.5)   ++(-135:2.5) coordinate(corner1);
\path(origin2)  ++(45:2)   ++(135:7) coordinate(corner2);
\path(origin2)  ++(45:2)   ++(-45:2) coordinate(corner4);
\path(origin2)  ++(135:2.5)   ++(45:6.5) coordinate(corner3);

\draw[thick] (origin2)--(corner1)--(corner2)--(corner3)--(corner4)--(origin2);
\clip(origin2)--(corner1)--(corner2)--(corner3)--(corner4)--(origin2);

\end{tikzpicture}
$$
\caption{
We depict the 
 identity coset $\varnothing$ along the bottom of the diagram, 
  the coset $\la$ along the top of the diagram,    
  and its corresponding reduced expression
  $\color{magenta}s_5
  \color{darkgreen}s_4
    \color{orange}s_3  
      \color{lime!80!black}s_2
          \color{violet}s_1
                    \color{gray}s_6
          \color{magenta}s_5
  \color{darkgreen}s_4
    \color{orange}s_3  
     \color{cyan}s_7
                    \color{gray}s_6
                    \color{pink}s_8
                    \color{cyan}s_7
$. }
\label{typeAtiling-long}
\end{figure}

Moreover, we can construct the graph $\mathcal{G}_{(W,P)}$ starting from $\varnothing$ as follows.

\begin{itemize}[leftmargin=*]
\item Start  by drawing   $\varnothing$ at the bottom.
\item If applying a  simple reflection does   result in a new coset, then record this in the next level up in the diagram.  We record the colour of the reflection as an edge relating the two points in the graph.  

\end{itemize}
We repeat the above until the process terminates.  
This is best illustrated via the  examples in \cref{Bruhatexample,BruhatexampleD}.  

\begin{figure}[ht!]
 $$\begin{tikzpicture} [yscale=0.51,xscale=-0.51]

\draw[magenta, line width=3] (2,0.5)--(2,3) coordinate (hi);

\draw[cyan, line width=3] (hi)--(6,5.5) coordinate (hi2);
\draw[darkgreen, line width=3] (hi)--(-2,5.5) coordinate (hi3);
\draw[cyan, line width=3] (hi3)--(2,8) coordinate (hi4);
\draw[darkgreen, line width=3] (hi2)--(2,8) coordinate (hi4);
\draw[orange, line width=3] (hi3)--(-6,8) coordinate (hi5);
\draw[cyan, line width=3] (hi5)--(-2,10.5) coordinate (hi6);
\draw[orange, line width=3] (hi4)--(-2,10.5) coordinate (hi6);
\draw[magenta, line width=3] (hi4)--(6,10.5) coordinate (hi7);

\draw[magenta, line width=3] (hi6)--(2,13) coordinate (hi8);
\draw[orange, line width=3] (hi7)--(2,13) coordinate (hi8);
\draw[darkgreen, line width=3] (hi8)--(2,15) ;

\path (0,0) coordinate (origin); 
\draw[fill=white,rounded corners](origin) rectangle ++(4,1);
\path (origin)--++(0.5,0.5) coordinate (origin2);  
\draw(origin2)--++(0:3); 
\foreach \i in {1,2,3,4,5}
{
\path (origin2)--++(0:0.5*\i) coordinate (a\i); 
 
  }
\path(a1) --++(90:0.175) node  {  $  \down   $} ;
\path(a2) --++(90:0.175) node  {  $  \down   $} ;
\path(a3) --++(90:0.175) node  {  $  \down   $} ;

\path(a4) --++(-90:0.175) node  {  $  \up   $} ;
\path(a5) --++(-90:0.175) node  {  $  \up   $} ;

\path (0,2.5) coordinate (origin); 
\draw[fill=white,rounded corners](origin) rectangle ++(4,1);
\path (origin)--++(0.5,0.5) coordinate (origin2);  
\draw(origin2)--++(0:3); 
\foreach \i in {1,2,3,4,5}
{
\path (origin2)--++(0:0.5*\i) coordinate (a\i); 
 
  }
\path(a1) --++(90:0.175) node  {  $  \down   $} ;
\path(a2) --++(90:0.175) node  {  $  \down   $} ;
\path(a3) --++(-90:0.175) node  {  $  \up   $} ;

\path(a4) --++(90:0.175) node  {  $  \down   $} ;
\path(a5) --++(-90:0.175) node  {  $  \up   $} ;

\path (-4,5) coordinate (origin); 
\draw[fill=white,rounded corners](origin) rectangle ++(4,1);
\path (origin)--++(0.5,0.5) coordinate (origin2);  
\draw(origin2)--++(0:3); 
\foreach \i in {1,2,3,4,5}
{
\path (origin2)--++(0:0.5*\i) coordinate (a\i); 
 
  }
\path(a1) --++(90:0.175) node  {  $  \down   $} ;
\path(a2) --++(-90:0.175) node  {  $  \up   $} ;
\path(a3) --++(90:0.175) node  {  $  \down   $} ;

\path(a4) --++(90:0.175) node  {  $  \down   $} ;
\path(a5) --++(-90:0.175) node  {  $  \up   $} ;

\path (4,5) coordinate (origin); 
\draw[fill=white,rounded corners](origin) rectangle ++(4,1);
\path (origin)--++(0.5,0.5) coordinate (origin2);  
\draw(origin2)--++(0:3); 
\foreach \i in {1,2,3,4,5}
{
\path (origin2)--++(0:0.5*\i) coordinate (a\i); 
 
  }
\path(a1) --++(90:0.175) node  {  $  \down   $} ;
\path(a2) --++(90:0.175) node  {  $  \down   $} ;
\path(a3) --++(-90:0.175) node  {  $  \up   $} ;

\path(a4) --++(-90:0.175) node  {  $  \up   $} ;
\path(a5) --++(90:0.175) node  {  $  \down  $} ;

\path (0,7.5) coordinate (origin); 
\draw[fill=white,rounded corners](origin) rectangle ++(4,1);
\path (origin)--++(0.5,0.5) coordinate (origin2);  
\draw(origin2)--++(0:3); 
\foreach \i in {1,2,3,4,5}
{
\path (origin2)--++(0:0.5*\i) coordinate (a\i); 
 
  }
\path(a1) --++(90:0.175) node  {  $  \down   $} ;
\path(a2) --++(-90:0.175) node  {  $  \up  $} ;
\path(a3) --++(90:0.175) node  {  $  \down    $} ;

\path(a4) --++(-90:0.175) node  {  $  \up   $} ;
\path(a5) --++(90:0.175) node  {  $  \down  $} ;

\path (-8,7.5) coordinate (origin); 
\draw[fill=white,rounded corners](origin) rectangle ++(4,1);
\path (origin)--++(0.5,0.5) coordinate (origin2);  
\draw(origin2)--++(0:3); 
\foreach \i in {1,2,3,4,5}
{
\path (origin2)--++(0:0.5*\i) coordinate (a\i); 
 
  }
\path(a1) --++(-90:0.175) node  {  $  \up   $} ;
\path(a2) --++(90:0.175) node  {  $  \down  $} ;
\path(a3) --++(90:0.175) node  {  $  \down    $} ;

\path(a4) --++(90:0.175) node  {  $   \down   $} ;
\path(a5) --++(-90:0.175) node  {  $  \up   $} ;

\path (-4,10) coordinate (origin); 
\draw[fill=white,rounded corners](origin) rectangle ++(4,1);
\path (origin)--++(0.5,0.5) coordinate (origin2);  
\draw(origin2)--++(0:3); 
\foreach \i in {1,2,3,4,5}
{
\path (origin2)--++(0:0.5*\i) coordinate (a\i); 
 
  }
\path(a1) --++(-90:0.175) node  {  $  \up   $} ;
\path(a2) --++(90:0.175) node  {  $  \down  $} ;
\path(a3) --++(90:0.175) node  {  $  \down   $} ;

\path(a4) --++(-90:0.175) node  {  $  \up  $} ;
\path(a5) --++(90:0.175) node  {  $  \down  $} ;

\path (4,10) coordinate (origin); 
\draw[fill=white,rounded corners](origin) rectangle ++(4,1);
\path (origin)--++(0.5,0.5) coordinate (origin2);  
\draw(origin2)--++(0:3); 
\foreach \i in {1,2,3,4,5}
{
\path (origin2)--++(0:0.5*\i) coordinate (a\i); 
 
  }
\path(a1) --++(90:0.175) node  {  $  \down   $} ;
\path(a2) --++(-90:0.175) node  {  $  \up  $} ;
\path(a3) --++(-90:0.175) node  {  $  \up   $} ;

\path(a4) --++(90:0.175) node  {  $  \down   $} ;
\path(a5) --++(90:0.175) node  {  $  \down  $} ;

\path (0,12.5) coordinate (origin); 
\draw[fill=white,rounded corners](origin) rectangle ++(4,1);
\path (origin)--++(0.5,0.5) coordinate (origin2);  
\draw(origin2)--++(0:3); 
\foreach \i in {1,2,3,4,5}
{
\path (origin2)--++(0:0.5*\i) coordinate (a\i); 
 
  }
\path(a1) --++(-90:0.175) node  {  $  \up   $} ;
\path(a2) --++(90:0.175) node  {  $  \down $} ;
\path(a3) --++(-90:0.175) node  {  $  \up   $} ;

\path(a4) --++(90:0.175) node  {  $  \down   $} ;
\path(a5) --++(90:0.175) node  {  $  \down  $} ;

\path (0,15) coordinate (origin); 
\draw[fill=white,rounded corners](origin) rectangle ++(4,1);
\path (origin)--++(0.5,0.5) coordinate (origin2);  
\draw(origin2)--++(0:3); 
\foreach \i in {1,2,3,4,5}
{
\path (origin2)--++(0:0.5*\i) coordinate (a\i); 
 
  }
\path(a1) --++(-90:0.175) node  {  $  \up   $} ;
\path(a2) --++(-90:0.175) node  {  $  \up   $} ;
\path(a3) --++(90:0.175) node  {  $  \down   $} ;

\path(a4) --++(90:0.175) node  {  $  \down   $} ;
\path(a5) --++(90:0.175) node  {  $  \down  $} ;

\draw[very thick]( -10,3)--(-4,3); 
\draw[very thick, fill=orange] ( -10,3) coordinate circle (7pt); 

\draw[very thick, fill=darkgreen] ( -8,3) coordinate circle (7pt);  

\draw[very thick, fill=magenta] ( -6,3) coordinate circle (7pt); 
 
\draw[very thick, fill=cyan] ( -4,3) coordinate circle (7pt); 

\draw[very thick, rounded corners] (-10.8,2.5) rectangle (-7.2,3.5); 

\draw[very thick, rounded corners] (-4.8,2.5) rectangle (-3.2,3.5); 
\end{tikzpicture}
$$

\caption{The graph  $\mathcal{G}_{(W,P)}$ for  $(W,P)=(A_4,A_1 \times A_2  )$. } 
\label{Bruhatexample}
\end{figure}

\begin{figure}[ht!]
 $$ \begin{tikzpicture} [scale=0.51]

\draw[magenta, line width=3] (2,0.5)--(2,-1.5)  ;

\draw[cyan, line width=3] (2,0.5)--(2,3) coordinate (hi);

\draw[orange, line width=3] (hi)--(5,5.5) coordinate (hi2);
\draw[darkgreen, line width=3] (hi)--(-1,5.5) coordinate (hi3);
\draw[orange, line width=3] (hi3)--(2,8) coordinate (hi4);
\draw[darkgreen, line width=3] (hi2)--(2,8) coordinate (hi4);
  
\draw[cyan, line width=3] (hi4)--++(90:2) coordinate (hi5);
\draw[magenta, line width=3] (hi5)--++(90:2.5) coordinate (hi6);

\path (0,-2.5) coordinate (origin); 
\draw[fill=white,rounded corners](origin) rectangle ++(4,1);
\path (origin)--++(0.75,0.5) coordinate (origin2);  
\draw(origin2)--++(0:2.5); 
\foreach \i in {1,2,3,4}
{
\path (origin2)--++(0:0.5*\i) coordinate (a\i); 
 
  }
\path(a1) --++(90:0.175) node  {  $  \down   $} ;
\path(a2) --++(90:0.175) node  {  $  \down   $} ;
\path(a3) --++(90:0.175) node  {  $  \down   $} ;

\path(a4) --++(90:0.175) node  {  $  \down    $} ;

\path (0,0) coordinate (origin); 
\draw[fill=white,rounded corners](origin) rectangle ++(4,1);
\path (origin)--++(0.75,0.5) coordinate (origin2);  
\draw(origin2)--++(0:2.5); 
\foreach \i in {1,2,3,4}
{
\path (origin2)--++(0:0.5*\i) coordinate (a\i); 
 
  }
\path(a1) --++(-90:0.175) node  {  $  \up  $} ;
\path(a2) --++(-90:0.175) node  {  $  \up $} ;
\path(a3) --++(90:0.175) node  {  $  \down   $} ;

\path(a4) --++(90:0.175) node  {  $  \down    $} ;

\path (0,2.5) coordinate (origin); 
\draw[fill=white,rounded corners](origin) rectangle ++(4,1);
\path (origin)--++(0.75,0.5) coordinate (origin2);  
\draw(origin2)--++(0:2.5); 
\foreach \i in {1,2,3,4}
{
\path (origin2)--++(0:0.5*\i) coordinate (a\i); 
 
  }
\path(a1) --++(-90:0.175) node  {  $  \up  $} ;
\path(a2) --++(90:0.175) node  {  $  \down $} ;
\path(a3) --++(-90:0.175) node  {  $  \up   $} ;

\path(a4) --++(90:0.175) node  {  $  \down    $} ;

\path (-3,5) coordinate (origin); 
\draw[fill=white,rounded corners](origin) rectangle ++(4,1);
\path (origin)--++(0.75,0.5) coordinate (origin2);  
\draw(origin2)--++(0:2.5); 
\foreach \i in {1,2,3,4}
{
\path (origin2)--++(0:0.5*\i) coordinate (a\i); 
 
  }
\path(a1) --++(-90:0.175) node  {  $  \up  $} ;
\path(a4) --++(-90:0.175) node  {  $  \up $} ;
\path(a3) --++(90:0.175) node  {  $  \down   $} ;

\path(a2) --++(90:0.175) node  {  $  \down    $} ;

\path (3,5) coordinate (origin); 
\draw[fill=white,rounded corners](origin) rectangle ++(4,1);
\path (origin)--++(0.75,0.5) coordinate (origin2);  
\draw(origin2)--++(0:2.5); 
\foreach \i in {1,2,3,4}
{
\path (origin2)--++(0:0.5*\i) coordinate (a\i); 
 
  }
\path(a1) --++(90:0.175) node  {  $  \down  $} ;
\path(a4) --++(90:0.175) node  {  $   \down $} ;
\path(a3) --++(-90:0.175) node  {  $  \up   $} ;

\path(a2) --++(-90:0.175) node  {  $  \up    $} ;

\path (0,7.5) coordinate (origin); 
\draw[fill=white,rounded corners](origin) rectangle ++(4,1);
\path (origin)--++(0.75,0.5) coordinate (origin2);  
\draw(origin2)--++(0:2.5); 
\foreach \i in {1,2,3,4}
{
\path (origin2)--++(0:0.5*\i) coordinate (a\i); 
 
  }
\path(a1) --++(90:0.175) node  {  $  \down  $} ;
\path(a2) --++(-90:0.175) node  {  $  \up $} ;
\path(a3) --++(90:0.175) node  {  $  \down   $} ;

\path(a4) --++(-90:0.175) node  {  $  \up  $} ;

\path (0,10) coordinate (origin); 
\draw[fill=white,rounded corners](origin) rectangle ++(4,1);
\path (origin)--++(0.75,0.5) coordinate (origin2);  
\draw(origin2)--++(0:2.5); 
\foreach \i in {1,2,3,4}
{
\path (origin2)--++(0:0.5*\i) coordinate (a\i); 
 
  }
\path(a1) --++(90:0.175) node  {  $  \down  $} ;
\path(a2) --++(90:0.175) node  {  $  \down $} ;
\path(a3) --++(-90:0.175) node  {  $  \up   $} ;

\path(a4) --++(-90:0.175) node  {  $  \up  $} ;

\path (0,12.5) coordinate (origin); 
\draw[fill=white,rounded corners](origin) rectangle ++(4,1);
\path (origin)--++(0.75,0.5) coordinate (origin2);  
\draw(origin2)--++(0:2.5); 
\foreach \i in {1,2,3,4}
{
\path (origin2)--++(0:0.5*\i) coordinate (a\i); 
 
  }
\path(a1) --++(-90:0.175) node  {  $  \up  $} ;
\path(a2) --++(-90:0.175) node  {  $  \up $} ;
\path(a3) --++(-90:0.175) node  {  $  \up   $} ;

\path(a4) --++(-90:0.175) node  {  $   \up $} ;

\path ( -2,0) coordinate (himid);

\draw[very thick] (himid)--++(180:2) coordinate (hi3); 
 
\draw[very thick ] (himid) --++(60:2) coordinate (hi2); 

\draw[very thick ] ( himid) --++(-60:2) coordinate (hi); 

\draw[very thick, fill=orange] (hi2)  coordinate circle (7pt); 

\draw[very thick, fill=magenta] (hi)  coordinate circle (7pt); 
 \draw[very thick, fill=darkgreen] ( hi3) coordinate circle (7pt); 

\draw[very thick, fill=cyan] (himid) coordinate circle (7pt);


\path (hi3)--++(180:0.5) coordinate (hi3TLX); 
\path (hi3)--++(120:0.5) coordinate (hi3TL);
\path (himid)--++(-60:0.5) coordinate (himidBR);
\path (hi2)--++(0:0.5) coordinate (hi2L);
\path (hi2)--++(60:0.5) coordinate (hi2LX);
\path (hi2)--++(120:0.5) coordinate (hi2T);
\path (himid)--++(120:0.5) coordinate (himidT);
\path (hi3)--++(-120:0.5) coordinate (hi3BL);
\path (hi3)--++(-120:0.5)--++(0:0.3) coordinate (hi3BL2);

\draw[thick, rounded corners,smooth] (hi3BL2)--(himidBR)--(hi2L)--(hi2LX)--(hi2T)--(himidT)--(hi3TL)--(hi3TLX)-- (hi3BL)--++(0:1); 
 \end{tikzpicture} \qquad 
 \begin{tikzpicture} [scale=0.51]

\draw[magenta, line width=3] (2,0.5)--(2,-1.5)  ;

\draw[cyan, line width=3] (2,0.5)--(2,3) coordinate (hi);

\draw[magenta, line width=3] (hi)--(5,5.5) coordinate (hi2);
\draw[darkgreen, line width=3] (hi)--(-1,5.5) coordinate (hi3);
\draw[magenta, line width=3] (hi3)--(2,8) coordinate (hi4);
\draw[darkgreen, line width=3] (hi2)--(2,8) coordinate (hi4);
  
\draw[cyan, line width=3] (hi4)--++(90:2) coordinate (hi5);
\draw[magenta, line width=3] (hi5)--++(90:2.5) coordinate (hi6);

\path (0,-2.5) coordinate (origin); 
\draw[fill=white,rounded corners](origin) rectangle ++(4,1);
\path (origin)--++(0.5,0.5) coordinate (origin2);  
\draw(origin2)--++(0:3); 
\foreach \i in {0,1,2,3}
{
\path (origin2)--++(0:\i*2/3) --++(0:0.5)coordinate (a\i); 
 
  }
  \draw(a0)  circle (2pt); 

\path(a1) --++(90:0.175) node  {  $  \down   $} ;
\path(a2) --++(90:0.175) node  {  $  \down   $} ;
\path(a3) --++(90:0.175) node  {  $  \down   $} ;

\path (0,0) coordinate (origin); 
\draw[fill=white,rounded corners](origin) rectangle ++(4,1);
\path (origin)--++(0.5,0.5) coordinate (origin2);  
\draw(origin2)--++(0:3); 
\foreach \i in {0,1,2,3}
{
\path (origin2)--++(0:\i*2/3) --++(0:0.5)coordinate (a\i); 
 
  }
  \draw(a0)  circle (2pt);
\path(a1) --++(-90:0.175) node  {  $  \up   $} ;
\path(a2) --++(90:0.175) node  {  $  \down   $} ;
\path(a3) --++(90:0.175) node  {  $  \down   $} ;

\path (0,2.5) coordinate (origin); 
\draw[fill=white,rounded corners](origin) rectangle ++(4,1);
\path (origin)--++(0.5,0.5) coordinate (origin2);  
\draw(origin2)--++(0:3); 
\foreach \i in {0,1,2,3}
{
\path (origin2)--++(0:\i*2/3) --++(0:0.5)coordinate (a\i); 
 
  }
\path(a1) --++(90:0.175) node  {  $  \down   $} ;
\path(a2) --++(-90:0.175) node  {  $  \up   $} ;
\path(a3) --++(90:0.175) node  {  $  \down   $} ;

  \draw(a0)  circle (2pt);

\path (-3,5) coordinate (origin); 
\draw[fill=white,rounded corners](origin) rectangle ++(4,1);
\path (origin)--++(0.5,0.5) coordinate (origin2);  
\draw(origin2)--++(0:3); 
\foreach \i in {0,1,2,3}
{
\path (origin2)--++(0:\i*2/3) --++(0:0.5)coordinate (a\i); 
 
  }
\path(a1) --++(90:0.175) node  {  $  \down   $} ;
\path(a2) --++(90:0.175) node  {  $  \down   $} ;
\path(a3) --++(-90:0.175) node  {  $  \up   $} ;

%
  \draw(a0)  circle (2pt);

\path (3,5) coordinate (origin); 
\draw[fill=white,rounded corners](origin) rectangle ++(4,1);
\path (origin)--++(0.5,0.5) coordinate (origin2);  
\draw(origin2)--++(0:3); 
\foreach \i in {0,1,2,3}
{
\path (origin2)--++(0:\i*2/3) --++(0:0.5)coordinate (a\i); 
 
  }
\path(a1) --++(-90:0.175) node  {  $  \up   $} ;
\path(a2) --++(-90:0.175) node  {  $  \up  $} ;
\path(a3) --++(90:0.175) node  {  $  \down   $} ;

  \draw(a0)  circle (2pt);

\path (0,7.5) coordinate (origin); 
\draw[fill=white,rounded corners](origin) rectangle ++(4,1);
\path (origin)--++(0.5,0.5) coordinate (origin2);  
\draw(origin2)--++(0:3); 
\foreach \i in {0,1,2,3}
{
\path (origin2)--++(0:\i*2/3) --++(0:0.5)coordinate (a\i); 
 
  }
\path(a1) --++(-90:0.175) node  {  $  \up   $} ;
\path(a2) --++(90:0.175) node  {  $  \down   $} ;
\path(a3) --++(-90:0.175) node  {  $  \up   $} ;

  \draw(a0)  circle (2pt);

\path (0,10) coordinate (origin); 
\draw[fill=white,rounded corners](origin) rectangle ++(4,1);
\path (origin)--++(0.5,0.5) coordinate (origin2);  
\draw(origin2)--++(0:3); 
\foreach \i in {0,1,2,3}
{
\path (origin2)--++(0:\i*2/3) --++(0:0.5)coordinate (a\i); 
 
  }
\path(a1) --++(90:0.175) node  {  $  \down   $} ;
\path(a2) --++(-90:0.175) node  {  $  \up  $} ;
\path(a3) --++(-90:0.175) node  {  $  \up   $} ;

  \draw(a0)  circle (2pt);

\path (0,12.5) coordinate (origin); 
\draw[fill=white,rounded corners](origin) rectangle ++(4,1);
\path (origin)--++(0.5,0.5) coordinate (origin2);  
\draw(origin2)--++(0:3); 
\foreach \i in {0,1,2,3}
{
\path (origin2)--++(0:\i*2/3) --++(0:0.5)coordinate (a\i); 
 
  }
\path(a1) --++(-90:0.175) node  {  $  \up   $} ;
\path(a2) --++(-90:0.175) node  {  $  \up  $} ;
\path(a3) --++(-90:0.175) node  {  $  \up   $} ;

  \draw(a0)  circle (2pt);

\path ( -2.75,0) coordinate (himid);

\draw[very thick] (himid)--++(180:2) coordinate (hi3); 
 
\path (himid) --++(0:2) coordinate (hi2); 

\draw[very thick ] (hi2)--++(90:0.1)   --++(180:2) 
--++(-90:0.2)   --++(0: 2)  ;

\draw[very thick, fill=magenta] (hi2)  coordinate circle (7pt); 

  \draw[very thick, fill=darkgreen] ( hi3) coordinate circle (7pt); 

\draw[very thick, fill=cyan] (himid) coordinate circle (7pt);

\path (hi3)--++(180:0.5) coordinate (hi3TLX); 
\path (hi3)--++(120:0.5) coordinate (hi3TL);
\path (himid)--++(-60:0.5) coordinate (himidBR);
\path (himid)--++(-60:0.5) coordinate (hi2L);
\path (himid)--++(0:0.5) coordinate (hi2LX);
\path (himid)--++(60:0.5) coordinate (hi2T);
\path (himid)--++(120:0.5) coordinate (himidT);
\path (hi3)--++(-120:0.5) coordinate (hi3BL);
\path (hi3)--++(-120:0.5)--++(0:0.3) coordinate (hi3BL2);

\draw[thick, rounded corners,smooth] (hi3BL2) --(hi2L)--(hi2LX)--(hi2T)--(himidT)--(hi3TL)--(hi3TLX)-- (hi3BL)--++(0:1); 
 \end{tikzpicture}
$$

\caption{ The  graphs $\mathcal{G}_{(W,P)}$  for   $(W,P)=(D_4,A_3)$
and   $(C_3,A_2)$.  
}
\label{BruhatexampleD}
\end{figure}

 The {\sf Bruhat order} on ${{^P}W}$  can also easily be visualised in this setting, namely  $\la < \mu$ if either $\la$ contains strictly fewer $\up$ arrows than $\mu$ or
if $\la$ has the same number of $\up$ arrows
  as $\mu$ and $\mu$ is obtained from $\lambda$ by moving $\up$ arrows to the right.

\begin{rmk}\label{weightcosets} It is worth taking a moment to consider which  diagrams can appear elements of ${^PW}$. It is clear that there are $ {n+1}\choose {k}$ cosets for $(A_n,A_{k-1}\times A_{n-k})$ and that these  are given by all possible diagrams with $k$ $\up$-arrows and $n-k+1$ $\down$-arrows.  
There are $ 2^n$ cosets for $(C_n,A_{n-1})$ and these 
are given by freely choosing the $\up$ versus $\down$ decorations on the vertices
$\{ 2,\dots,n+1\}$ (the   decoration on vertex $1$ is always a $\circ$).  
There are $ 2^{n-1}$ cosets for $(D_n,A_{n-1})$ and these 
are given by freely choosing the $\up$ versus $\down$ decorations on the vertices
$\{ 1,2,\dots,n\}$ subject to the condition that  the total number of $\up$-arrows is even. 
\end{rmk}

For the remainder of the paper we will freely identify cosets with their diagrams  without further mention.  

\begin{rmk}\label{isomorphic copies}
In type $(D_n, A_{n-1})$, the description we gave of the cosets (and graph $\mathcal{G}_{(W,P)}$) uses the parabolic subgroup of type $A_{n-1}$ as the subgroup generated by $\{s_1, s_2, \ldots , s_{n-1}\}$. We could have used instead the parabolic subgroup of type $A_{n-1}$ generated by $\{s_{1''}, s_2, \ldots , s_{n-1}\}$. 
This gives an alternative construction of $\mathcal{G}_{(W,P)}$ in type $(D_n, A_{n-1})$, starting by defining the coset $\varnothing$ as the horizontal line with $n$ points where the first one is labelled by $\up$ and all others are labelled by $\down$.  So the cosets now must have an odd number of $\up$-arrows.
\end{rmk}

 \section{Diagrammatic oriented Temperley--Lieb algebras in classical type}\label{results1}

We now introduce a visualisation of oriented Temperley--Lieb algebras for $(W,P)$ of classical type, namely $(A_n, A_k\times A_{n-k-1})$, $(D_n, A_{n-1})$ and $(C_n, A_{n-1})$.

\subsection{Green's diagrammatic Temperley--Lieb algebras}
 We first recall Green's diagrammatic realisation of the generalised Temperley--Lieb algebras of type $W$ (see \cite{MR1618912}).

\begin{defn}
	\begin{enumerate}
	\item 
	An {\sf $n$-tangle} is a rectangular
	frame with $n$ vertices on the northern and southern boundaries which are paired-off by $n$ non-crossing strands and a finite number of non-crossing closed loops. Strands and loops can be decorated by a finite number of beads if they are left exposed (i.e. can be deformed to touch the left boundary of the frame). 
	We refer to a strand connecting a northern and southern vertex as a {\sf propagating strand}.  We refer to any strand connecting two northern vertices (or two southern vertices) to each other as an {\sf arc}.  Two $n$-tangles are equal if there exists an isotopy of the plane fixing the boundaries of the frame carrying one $n$-tangle to the other. We denote the set of all such $n$-tangles by $\mathbb{DT}_n$. 

\item We call an $n$-tangle $d\in \mathbb{DT}_n$ {\sf undecorated} if it has no beads on its strands or loops.

	\item For $R$ any commutative ring, we have that $R\mathbb{DT}_n$ has the structure of an $R$-algebra where the multiplication is given by the vertical concatenation of $n$-tangles. Specifically, for $d, d'\in  \mathbb{DT}_n$, we define  the product $d d'$ simply by placing $d'$ above $d$. 
	\end{enumerate}
	
\end{defn}

\begin{figure}[ht!]
	$$  
	\begin{tikzpicture} [scale=0.75]
		
		\clip(4,1.82) rectangle ++ (6,-2.7);
		
		\path (4,1) coordinate (origin); 
		\path (origin)--++(0.5,0.5) coordinate (origin2);  
		\draw(origin2)--++(0:5); 
		\foreach \i in {1,2,3,4,5,...,9}
		{
			\path (origin2)--++(0:0.5*\i) coordinate (a\i); 
			\path (origin2)--++(0:0.5*\i)--++(-90:0.08) coordinate (c\i); 
			\draw (a\i)circle (1pt);  }
		
		\foreach \i in {1,2,3,4,5,...,19}
		{
			\path (origin2)--++(0:0.25*\i) --++(-90:0.5) coordinate (b\i); 
			\path (origin2)--++(0:0.25*\i) --++(-90:0.7) coordinate (d\i); 
		}
		%

		\draw[  very thick](c3) to [out=-90,in=0] (b5) to [out=180,in=-90] (c2); 
		\draw[  very thick](c4) to [out=-90,in=0] (d5) to [out=180,in=-90] (c1);

		\draw[  very thick](c9) to [out=-90,in=0] (b17) to [out=180,in=-90] (c8);

		\draw[  very thick](c5) to [out=-90,in=15] (6,0.55);  
		
		\draw[  very thick](c6) to [out=-90,in=15] (7,0.55);

		\draw[  very thick](c7) to [out=-90,in=165] (8.5,0.55);

		\path (4,-1) coordinate (origin); 
		\path (origin)--++(0.5,0.5) coordinate (origin2);  
		\draw(origin2)--++(0:5); 
		\foreach \i in {1,2,3,4,5,...,9}
		{
			\path (origin2)--++(0:0.5*\i) coordinate (a\i); 
			\path (origin2)--++(0:0.5*\i)--++(90:0.08) coordinate (c\i); 
			\draw (a\i)circle (1pt);  }
		
		\foreach \i in {1,2,3,4,5,...,19}
		{
			\path (origin2)--++(0:0.25*\i) --++(90:0.5) coordinate (b\i); 
			\path (origin2)--++(0:0.25*\i) --++(90:0.65) coordinate (d\i); 
		} 
		
		\draw[  very thick](c6) to [out=90,in=0] (b11) to [out=180,in=90] (c5); 
		\draw[  very thick](c7) to [out=90,in=0] (d11) to [out=180,in=90] (c4); 
		
		\path(d11)--++(90:0.2) coordinate (e11); 
		\draw[  very thick](c8) to [out=90,in=0] (e11) to [out=180,in=90] (c3);

		\draw[very   thick](c1) to [out=90,in=-170] (6,0.55); 
		\draw[very   thick](c2) to [out=90,in=-170] (7,0.55);   
		
		\draw[  very thick](c9) to [out=90,in=-15] (8.5,0.55);

	\end{tikzpicture} 
	\begin{tikzpicture} [scale=0.75]
		
		\clip(4,1.82) rectangle ++ (6,-2.7);
		
		\path (4,1) coordinate (origin); 
		\path (origin)--++(0.5,0.5) coordinate (origin2);  
		\draw(origin2)--++(0:5); 
		\foreach \i in {1,2,3,4,5,...,9}
		{
			\path (origin2)--++(0:0.5*\i) coordinate (a\i); 
			\path (origin2)--++(0:0.5*\i)--++(-90:0.08) coordinate (c\i); 
			\draw (a\i)circle (1pt);  }
		
		\foreach \i in {1,2,3,4,5,...,19}
		{
			\path (origin2)--++(0:0.25*\i) --++(-90:0.5) coordinate (b\i); 
			\path (origin2)--++(0:0.25*\i) --++(-90:0.7) coordinate (d\i); 
		}
	 		
		\draw[  very thick](c5) to [out=-90,in=0] (b9) to [out=180,in=-90] (c4); 		
		\draw[  very thick](c3) to [out=-90,in=0] (b5) to [out=180,in=-90] (c2); 
		\draw[  very thick](c6) to [out=-90,in=0] (d7) to [out=180,in=-90] (c1);

 	\draw[  very thick](c8) to [out=-90,in=0] (b15) to [out=180,in=-90] (c7);

		\draw[  very thick](c9) to [out=-90,in=0] (7,0.6);  
 		
		\path (4,-1) coordinate (origin); 
		\path (origin)--++(0.5,0.5) coordinate (origin2);  
		\draw(origin2)--++(0:5); 
		\foreach \i in {1,2,3,4,5,...,9}
		{
			\path (origin2)--++(0:0.5*\i) coordinate (a\i); 
			\path (origin2)--++(0:0.5*\i)--++(90:0.08) coordinate (c\i); 
			\draw (a\i)circle (1pt);  }
		
		\foreach \i in {1,2,3,4,5,...,19}
		{
			\path (origin2)--++(0:0.25*\i) --++(90:0.5) coordinate (b\i); 
			\path (origin2)--++(0:0.25*\i) --++(90:0.65) coordinate (d\i); 
		} 
	
			\draw[very   thick](c1) 
 to [out=90,in=-180] (7,0.6);

%
%
%
%
 
 		\path(b11)--++(-90:0.1) coordinate (b11);

 		\path(c11)--++(-90:0.1) coordinate (c11); 
 		\path(d11)--++(-90:0.1) coordinate (d11);  
 
 		\draw[  very thick](c6) to [out=90,in=0] (b11) to [out=180,in=90] (c5); 
		\draw[  very thick](c7) to [out=90,in=0] (d11) to [out=180,in=90] (c4); 
		
		\path(d11)--++(90:0.2) coordinate (e11); 
		\draw[  very thick](c8) to [out=90,in=0] (e11) to [out=180,in=90] (c3); 
		
	 		\path(e11)--++(90:0.2) coordinate (e11);	
 
		\draw[  very thick](c9) to [out=90,in=0] (e11) to [out=180,in=90] (c2);

		

	\end{tikzpicture}  
	\begin{tikzpicture} [scale=0.75]
		
		\clip(4,1.82) rectangle ++ (6,-2.7);
		
		\path (4,1) coordinate (origin); 
		\path (origin)--++(0.5,0.5) coordinate (origin2);  
		\draw(origin2)--++(0:5); 
		\foreach \i in {1,2,3,4,5,...,9}
		{
			\path (origin2)--++(0:0.5*\i) coordinate (a\i); 
			\path (origin2)--++(0:0.5*\i)--++(-90:0.08) coordinate (c\i); 
			\draw (a\i)circle (1pt);  }

		
		\foreach \i in {1,2,3,4,5,...,19}
		{
			\path (origin2)--++(0:0.25*\i) --++(-90:0.5) coordinate (b\i); 
			\path (origin2)--++(0:0.25*\i) --++(-90:0.7) coordinate (d\i); 
		}
	 		
		\draw[  very thick](c5) to [out=-90,in=0] (b9) to [out=180,in=-90] (c4); 		
		\draw[  very thick](c3) to [out=-90,in=0] (b5) to [out=180,in=-90] (c2); 
		\draw[  very thick](c6) to [out=-90,in=0] (d7) to [out=180,in=-90] (c1);

 	\draw[  very thick](c8) to [out=-90,in=0] (b15) to [out=180,in=-90] (c7);

%
%
%
%
%
		\draw[  very thick](c9) to [out=-90,in=15] (7,0.55);  

		\path (4,-1) coordinate (origin); 
		\path (origin)--++(0.5,0.5) coordinate (origin2);  
		\draw(origin2)--++(0:5); 
		\foreach \i in {1,2,3,4,5,...,9}
		{
			\path (origin2)--++(0:0.5*\i) coordinate (a\i); 
			\path (origin2)--++(0:0.5*\i)--++(90:0.08) coordinate (c\i); 
			\draw (a\i)circle (1pt);  }
		
		\foreach \i in {1,2,3,4,5,...,19}
		{
			\path (origin2)--++(0:0.25*\i) --++(90:0.5) coordinate (b\i); 
			\path (origin2)--++(0:0.25*\i) --++(90:0.65) coordinate (d\i); 
		} 
		
		\draw[  very thick](c8) to [out=90,in=0] (b15) to [out=180,in=90] (c7); 
		\draw[  very thick](c9) to [out=90,in=0] (d15) to [out=180,in=90] (c6); 
		
		
				\draw[  very thick](c2) to [out=90,in=0] (b3) to [out=180,in=90] (c1); 
		
				\draw[  very thick](c5) to [out=90,in=0] (b9) to [out=180,in=90] (c4); 

		\draw[very   thick](c3) to [out=90,in=-170] (7,0.55); 
%
%
		
%
	\end{tikzpicture}  
	$$
	 	\caption{  Examples of undecorated tangles.  
 	}
	\label{typeAtiling2lkjs}
\end{figure}

\begin{figure}[ht!]
$$	\begin{tikzpicture} [scale=0.75]
		
 	 	\clip(3,1.82) rectangle ++ (6.5,-2.7);

		\path (4,1) coordinate (origin); 
		\path (origin)--++(0.5,0.5) coordinate (origin2);  
		\draw(origin2)--++(0:4.5) (origin2)--++(180:1); 
		\foreach \i in {-1,0,1,2,3,4,5,...,8}
		{
			\path (origin2)--++(0:0.5*\i) coordinate (a\i); 
			\path (origin2)--++(0:0.5*\i)--++(-90:0.08) coordinate (c\i); 
			\draw (a\i)circle (1pt);  }
		
		\foreach \i in {-1,0,1,2,3,4,5,...,19}
		{
			\path (origin2)--++(0:0.25*\i) --++(-90:0.5) coordinate (b\i); 
			\path (origin2)--++(0:0.25*\i) --++(-90:0.7) coordinate (d\i); 
		} 
		
		\draw[  very thick](c3) to [out=-90,in=0] (b5) to [out=180,in=-90] (c2); 
		\draw[  very thick](c4) to [out=-90,in=0] (d5) to [out=180,in=-90] (c1); 
		
		\fill (d5)  circle (3pt);

		\draw[  very thick](c8) to [out=-90,in=0] (b15) to [out=180,in=-90] (c7);

	 	\draw[  very thick](c0) to [out=-90,in=0] (b-1) to [out=180,in=-90] (c-1); 
	
	\fill (b-1)  circle (3pt);    	
		
		\draw[  very thick](c5) to [out=-90,in=145] (7.5,0.55);  
		
		\draw[  very thick](c6) to [out=-90,in=145] (8,0.55);


		\path (4,-1) coordinate (origin); 
		\path (origin)--++(0.5,0.5) coordinate (origin2);  
		\draw(origin2)--++(0:4.5); 
				\draw(origin2)--++(180:1); 
		\foreach \i in {-1,0,1,2,3,4,5,...,8}
		{
			\path (origin2)--++(0:0.5*\i) coordinate (a\i); 
			\path (origin2)--++(0:0.5*\i)--++(90:0.08) coordinate (c\i); 
			\draw (a\i)circle (1pt);  }
		
		\foreach \i in {-1,0,1,2,3,4,5,...,19}
		{
			\path (origin2)--++(0:0.25*\i) --++(90:0.5) coordinate (b\i); 
			\path (origin2)--++(0:0.25*\i) --++(90:0.65) coordinate (d\i); 
		}

 	\draw[  very thick](c0) to [out=90,in=0] (b-1) to [out=180,in=90] (c-1); 
 	\draw[  very thick](c2) to [out=90,in=0] (b3) to [out=180,in=90] (c1); 
 	\draw[  very thick](c4) to [out=90,in=0] (b7) to [out=180,in=90] (c3); 	
 	\draw[  very thick](c6) to [out=90,in=0] (b11) to [out=180,in=90] (c5); 	
 \draw[  very thick](c8) to [out=90,in=-35] (8,0.55) ;
		 \draw[  very thick](c7) to [out=90,in=-35] (7.5,0.55) coordinate (here);
\fill (here)  circle (3pt);     

\fill (b-1)  circle (3pt);    	
\fill (b3)  circle (3pt);    		
\fill (b7)  circle (3pt);    		\fill (b11)  circle (3pt);    				
	\end{tikzpicture}
		\begin{tikzpicture} [scale=0.75]
		
 	 	\clip(3,1.82) rectangle ++ (6.5,-2.7);

		\path (4,1) coordinate (origin); 
		\path (origin)--++(0.5,0.5) coordinate (origin2);  
		\draw(origin2)--++(0:4.5) (origin2)--++(180:1); 
		\foreach \i in {-1,0,1,2,3,4,5,...,8}
		{
			\path (origin2)--++(0:0.5*\i) coordinate (a\i); 
			\path (origin2)--++(0:0.5*\i)--++(-90:0.08) coordinate (c\i); 
			\draw (a\i)circle (1pt);  }
		
		\foreach \i in {-1,0,1,2,3,4,5,...,19}
		{
			\path (origin2)--++(0:0.25*\i) --++(-90:0.5) coordinate (b\i); 
			\path (origin2)--++(0:0.25*\i) --++(-90:0.7) coordinate (d\i); 
		} 
		
		\draw[  very thick](c3) to [out=-90,in=0] (b5) to [out=180,in=-90] (c2); 
		\draw[  very thick](c4) to [out=-90,in=0] (d5) to [out=180,in=-90] (c1); 
		
		\fill (d5)  circle (3pt);

		\draw[  very thick](c8) to [out=-90,in=0] (b15) to [out=180,in=-90] (c7);

	 	\draw[  very thick](c0) to [out=-90,in=0] (b-1) to [out=180,in=-90] (c-1); 
	
		
		\draw[  very thick](c5) to [out=-90,in=145] (7.5,0.55);  
		
		\draw[  very thick](c6) to [out=-90,in=145] (8,0.55);


		\path (4,-1) coordinate (origin); 
		\path (origin)--++(0.5,0.5) coordinate (origin2);  
		\draw(origin2)--++(0:4.5); 
				\draw(origin2)--++(180:1); 
		\foreach \i in {-1,0,1,2,3,4,5,...,8}
		{
			\path (origin2)--++(0:0.5*\i) coordinate (a\i); 
			\path (origin2)--++(0:0.5*\i)--++(90:0.08) coordinate (c\i); 
			\draw (a\i)circle (1pt);  }
		
		\foreach \i in {-1,0,1,2,3,4,5,...,19}
		{
			\path (origin2)--++(0:0.25*\i) --++(90:0.5) coordinate (b\i); 
			\path (origin2)--++(0:0.25*\i) --++(90:0.65) coordinate (d\i); 
		}

 	\draw[  very thick](c0) to [out=90,in=0] (b-1) to [out=180,in=90] (c-1); 
 	\draw[  very thick](c2) to [out=90,in=0] (b3) to [out=180,in=90] (c1); 
 	\draw[  very thick](c4) to [out=90,in=0] (b7) to [out=180,in=90] (c3); 	
 	\draw[  very thick](c6) to [out=90,in=0] (b11) to [out=180,in=90] (c5); 	
 \draw[  very thick](c8) to [out=90,in=-35] (8,0.55) ;
		 \draw[  very thick](c7) to [out=90,in=-35] (7.5,0.55) coordinate (here);
\fill (here)  circle (3pt);     

\fill (b-1)  circle (3pt);    	
\fill (b3)  circle (3pt);    		
\fill (b7)  circle (3pt);    		\fill (b11)  circle (3pt);    				
	\end{tikzpicture}
	 \begin{tikzpicture}[scale=0.75] 
		\clip(4,1.82) rectangle ++ (6,-2.7);

		\path (4,1) coordinate (origin); 
		\path (origin)--++(0.5,0.5) coordinate (origin2);  
		\draw(origin2)--++(0:5); 
		\foreach \i in {1,2,3,4,5,...,9}
		{
			\path (origin2)--++(0:0.5*\i) coordinate (a\i); 
			\path (origin2)--++(0:0.5*\i)--++(-90:0.08) coordinate (c\i); 
			\draw (a\i)circle (1pt);  }
		
		\foreach \i in {1,2,3,4,5,...,19}
		{
			\path (origin2)--++(0:0.25*\i) --++(-90:0.5) coordinate (b\i); 
			\path (origin2)--++(0:0.25*\i) --++(-90:0.7) coordinate (d\i); 
		}

		\draw[  very thick](c3) to [out=-90,in=0] (b5) to [out=180,in=-90] (c2); 
		\draw[  very thick](c4) to [out=-90,in=0] (d5) to [out=180,in=-90] (c1);

		\draw[  very thick](c9) to [out=-90,in=0] (b17) to [out=180,in=-90] (c8);

		\draw[  very thick](c5) to [out=-90,in=15] (6,0.55);  
		
		\draw[  very thick](c6) to [out=-90,in=15] (7,0.55);

		\draw[  very thick](c7) to [out=-90,in=165] (8.5,0.55);

			\draw[  very thick](c3) to [out=-90,in=0] (b5) to [out=180,in=-90] (c2); 
		\draw[  very thick](c4) to [out=-90,in=0] (d5) to [out=180,in=-90] (c1); 
		
		\fill (d5)  circle (3pt);

		\path (4,-1) coordinate (origin); 
		\path (origin)--++(0.5,0.5) coordinate (origin2);  
		\draw(origin2)--++(0:5); 
		\foreach \i in {1,2,3,4,5,...,9}
		{
			\path (origin2)--++(0:0.5*\i) coordinate (a\i); 
			\path (origin2)--++(0:0.5*\i)--++(90:0.08) coordinate (c\i); 
			\draw (a\i)circle (1pt);  }
		
		\foreach \i in {1,2,3,4,5,...,19}
		{
			\path (origin2)--++(0:0.25*\i) --++(90:0.5) coordinate (b\i); 
			\path (origin2)--++(0:0.25*\i) --++(90:0.65) coordinate (d\i); 
		}
%
		
		\draw[  very thick](c6) to [out=90,in=0] (b11) to [out=180,in=90] (c5); 
		\draw[  very thick](c7) to [out=90,in=0] (d11) to [out=180,in=90] (c4); 
		
		\path(d11)--++(90:0.2) coordinate (e11); 
		\draw[  very thick](c8) to [out=90,in=0] (e11) to [out=180,in=90] (c3); 
		

		\draw[very   thick](c1) to [out=90,in=-170] (6,0.55); 
		\draw[very   thick](c2) to [out=90,in=-170] (7,0.55);   
		
		\draw[  very thick](c9) to [out=90,in=-15] (8.5,0.55);

		\draw[very   thick](c1) to [out=90,in=-170] (6,0.55)
		[midway] coordinate (X);
		\fill (X)  circle (3pt);     
		
		\draw[very   thick](c2) to [out=90,in=-170] (7,0.55)
		;
		
		\draw[  very thick](c9) to [out=90,in=-15] (8.5,0.55) ;

	\end{tikzpicture}$$
	 	\caption{  Examples of tangles.  In the first and third diagrams we decorate {\em every} left-exposed strand. 
 	}
	\label{typeAtiling2lkjs2}
\end{figure}


%
%
%

We define ${\sf e}_i$ for $i\geq 1$ to be the undecorated $n$-tangle with  
a single pair of arcs connecting the $i$th and $(i+1)$th northern (respectively  southern) vertices,  
and with 
$(n-2)$ vertical strands.  
We set ${\sf e}_{1'}={\sf e}_{1''}$ to be the $n$-tangle
which has a single pair of arcs connecting the $1$st and $2$nd  northern (respectively  southern) vertices, both of which carry a single bead, and 
with 
$(n-2)$ vertical undecorated strands.  
We use the distinct subscripts  $1'$ and $1''$ to remind the reader  of  the corresponding Coxeter labels in \cref{coxeterlabelD2} for what follows.
Examples of these elements are depicted in \cref{gens}.

\begin{figure}[H]
	\begin{tikzpicture}  [scale=0.85]
		
		\clip(4,1.82) rectangle ++ (3.5,-2.7);
		
		\path (4,1) coordinate (origin); 
		\path (origin)--++(0.5,0.5) coordinate (origin2);  
		\draw(origin2)--++(0:3); 
		\foreach \i in {1,2,3,4,5}
		{
			\path (origin2)--++(0:0.5*\i) coordinate (a\i); 
			\path (origin2)--++(0:0.5*\i)--++(-90:0.08) coordinate (c\i); 
			\draw (a\i)circle (1pt); 
		}
		
		\foreach \i in {1,2,3,4,5}
		{
			\path (origin2)--++(0:0.25*\i) --++(-90:0.5) coordinate (b\i); 
			\path (origin2)--++(0:0.25*\i) --++(-90:0.7) coordinate (d\i); 
		}

		\draw[  very thick](c2) to [out=-90,in=0] (b3) to [out=180,in=-90] (c1);

		\draw[fill=black](b3) circle (3pt);

		\draw[very   thick](c3)  --  (6,0.55);

		\draw[very   thick](c4)  --  (6.5,0.55);
		
		\draw[very   thick](c5)  --  (7,0.55);

		\path (4,-1) coordinate (origin); 
		\path (origin)--++(0.5,0.5) coordinate (origin2);  
		\draw(origin2)--++(0:3); 
		\foreach \i in {1,2,3,4,5}
		{
			\path (origin2)--++(0:0.5*\i) coordinate (a\i); 
			\path (origin2)--++(0:0.5*\i)--++(90:0.08) coordinate (c\i); 
			\draw (a\i)circle (1pt);  }
		
		\foreach \i in {1,2,3,4,5,...,19}
		{
			\path (origin2)--++(0:0.25*\i) --++(90:0.5) coordinate (b\i); 
			\path (origin2)--++(0:0.25*\i) --++(90:0.65) coordinate (d\i); 
		}

		\draw[  very thick](c2) to [out=90,in=0] (b3) to [out=180,in=90] (c1);

		\draw[very   thick](c3)  --  (6,0.55);

		\draw[fill=black](b3) circle (3pt);
		
		\draw[very   thick](c4)  --  (6.5,0.55);
		
		\draw[very   thick](c5)  --  (7,0.55);

	\end{tikzpicture}  
\begin{tikzpicture}  [scale=0.85]
		
		\clip(4,1.82) rectangle ++ (3.5,-2.7);
		
		\path (4,1) coordinate (origin); 
		\path (origin)--++(0.5,0.5) coordinate (origin2);  
		\draw(origin2)--++(0:3); 
		\foreach \i in {1,2,3,4,5}
		{
			\path (origin2)--++(0:0.5*\i) coordinate (a\i); 
			\path (origin2)--++(0:0.5*\i)--++(-90:0.08) coordinate (c\i); 
			\draw (a\i)circle (1pt); 
		}
		
		\foreach \i in {1,2,3,4,5}
		{
			\path (origin2)--++(0:0.25*\i) --++(-90:0.5) coordinate (b\i); 
			\path (origin2)--++(0:0.25*\i) --++(-90:0.7) coordinate (d\i); 
		}

		\draw[  very thick](c2) to [out=-90,in=0] (b3) to [out=180,in=-90] (c1);

		\draw[very   thick](c3)  --  (6,0.55);

		\draw[very   thick](c4)  --  (6.5,0.55);
		
		\draw[very   thick](c5)  --  (7,0.55);

		\path (4,-1) coordinate (origin); 
		\path (origin)--++(0.5,0.5) coordinate (origin2);  
		\draw(origin2)--++(0:3); 
		\foreach \i in {1,2,3,4,5}
		{
			\path (origin2)--++(0:0.5*\i) coordinate (a\i); 
			\path (origin2)--++(0:0.5*\i)--++(90:0.08) coordinate (c\i); 
			\draw (a\i)circle (1pt);  }
		
		\foreach \i in {1,2,3,4,5,...,19}
		{
			\path (origin2)--++(0:0.25*\i) --++(90:0.5) coordinate (b\i); 
			\path (origin2)--++(0:0.25*\i) --++(90:0.65) coordinate (d\i); 
		}

		\draw[  very thick](c2) to [out=90,in=0] (b3) to [out=180,in=90] (c1);

		\draw[very   thick](c3)  --  (6,0.55);


		\draw[very   thick](c4)  --  (6.5,0.55);
		
		\draw[very   thick](c5)  --  (7,0.55);

	\end{tikzpicture}  
	\begin{tikzpicture}  [scale=0.85]
		
		\clip(4,1.82) rectangle ++ (3.5,-2.7);
		
		\path (4,1) coordinate (origin); 
		\path (origin)--++(0.5,0.5) coordinate (origin2);  
		\draw(origin2)--++(0:3); 
		\foreach \i in {1,2,3,4,5}
		{
			\path (origin2)--++(0:0.5*\i) coordinate (a\i); 
			\path (origin2)--++(0:0.5*\i)--++(-90:0.08) coordinate (c\i); 
			\draw (a\i)circle (1pt); 
		}
		
		\foreach \i in {1,2,3,4,5}
		{
			\path (origin2)--++(0:0.25*\i) --++(-90:0.5) coordinate (b\i); 
			\path (origin2)--++(0:0.25*\i) --++(-90:0.7) coordinate (d\i); 
		}

		\draw[  very thick](c3) to [out=-90,in=0] (b5) to [out=180,in=-90] (c2);

		\draw[very   thick](c1)  --  (5,0.55);

		\draw[very   thick](c4)  --  (6.5,0.55);
		
		\draw[very   thick](c5)  --  (7,0.55);

		\path (4,-1) coordinate (origin); 
		\path (origin)--++(0.5,0.5) coordinate (origin2);  
		\draw(origin2)--++(0:3); 
		\foreach \i in {1,2,3,4,5}
		{
			\path (origin2)--++(0:0.5*\i) coordinate (a\i); 
			\path (origin2)--++(0:0.5*\i)--++(90:0.08) coordinate (c\i); 
			\draw (a\i)circle (1pt);  }
		
		\foreach \i in {1,2,3,4,5,...,19}
		{
			\path (origin2)--++(0:0.25*\i) --++(90:0.5) coordinate (b\i); 
			\path (origin2)--++(0:0.25*\i) --++(90:0.65) coordinate (d\i); 
		}

		\draw[  very thick](c3) to [out=90,in=0] (b5) to [out=180,in=90] (c2);

		\draw[very   thick](c1)  --  (5,0.55);


		\draw[very   thick](c4)  --  (6.5,0.55);
		
		\draw[very   thick](c5)  --  (7,0.55);

	\end{tikzpicture}  
	\quad	\begin{tikzpicture}  [xscale=-0.85,yscale=0.85]
		
		\clip(4,1.82) rectangle ++ (3.5,-2.7);
		
		\path (4,1) coordinate (origin); 
		\path (origin)--++(0.5,0.5) coordinate (origin2);  
		\draw(origin2)--++(0:3); 
		\foreach \i in {1,2,3,4,5}
		{
			\path (origin2)--++(0:0.5*\i) coordinate (a\i); 
			\path (origin2)--++(0:0.5*\i)--++(-90:0.08) coordinate (c\i); 
			\draw (a\i)circle (1pt); 
		}
		
		\foreach \i in {1,2,3,4,5}
		{
			\path (origin2)--++(0:0.25*\i) --++(-90:0.5) coordinate (b\i); 
			\path (origin2)--++(0:0.25*\i) --++(-90:0.7) coordinate (d\i); 
		}

		\draw[  very thick](c3) to [out=-90,in=0] (b5) to [out=180,in=-90] (c2);

		\draw[very   thick](c1)  --  (5,0.55);

		\draw[very   thick](c4)  --  (6.5,0.55);
		
		\draw[very   thick](c5)  --  (7,0.55);

		\path (4,-1) coordinate (origin); 
		\path (origin)--++(0.5,0.5) coordinate (origin2);  
		\draw(origin2)--++(0:3); 
		\foreach \i in {1,2,3,4,5}
		{
			\path (origin2)--++(0:0.5*\i) coordinate (a\i); 
			\path (origin2)--++(0:0.5*\i)--++(90:0.08) coordinate (c\i); 
			\draw (a\i)circle (1pt);  }
		
		\foreach \i in {1,2,3,4,5,...,19}
		{
			\path (origin2)--++(0:0.25*\i) --++(90:0.5) coordinate (b\i); 
			\path (origin2)--++(0:0.25*\i) --++(90:0.65) coordinate (d\i); 
		}

		\draw[  very thick](c3) to [out=90,in=0] (b5) to [out=180,in=90] (c2);

		\draw[very   thick](c1)  --  (5,0.55);


		\draw[very   thick](c4)  --  (6.5,0.55);
		
		\draw[very   thick](c5)  --  (7,0.55);

	\end{tikzpicture}  
	\caption{
 The $n$-tangle   $\diag_{1'} =\diag_{1''}  $, $\diag_1$,  $\diag_2$, and $\diag_3$    
  for $n = 5$.
	}
	\label{gens}
\end{figure} 

We now recall the diagrammatic visualisation of the generalised Temperley--Lieb algebras ${\rm TL}_W(q)$ given in \cref{genTL}. From now on, we take $R=\mathbb{Z}[q,q^{-1}]$.

\begin{thm}[Kauffman \cite{Kau87} ]\label{thmGreenA} The Temperley--Lieb algebra ${\rm TL}_{A_n}(q)$ of type $A_n$,  is isomorphic to the subquotient of the algebra $R\mathbb{DT}_{n+1}$ generated by ${\sf e}_1, \ldots , {\sf e}_n$  subject to the relation
\begin{align*} 
\begin{minipage}{1.55cm}
\begin{tikzpicture}[yscale=1,xscale=-1]
\draw[densely dotted,rounded corners] (-22pt,-22pt) rectangle (22pt,22pt);
\clip (0,0) circle (22pt);
\draw[very thick] 
 circle (8pt);
 \end{tikzpicture}
\end{minipage}=q +q ^{-1}
 \end{align*}It has a basis given by the set of all undecorated $(n+1)$-tangles with no loops, which we denote by $\mathbb{DT}[A_n]$.
\end{thm}

\begin{thm}[Green \cite{MR1618912}]\label{thmGreenB} The generalised Temperley--Lieb algebra  ${\rm TL}_{C_n}(q)$ of type $C_n$  is isomorphic to the subquotient of the algebra $R\mathbb{DT}_{n+1}$ generated by ${\sf e}_{1'}, {\sf e}_2, \ldots , {\sf e}_n$ subject to the relations
 \[
\begin{minipage}{1.55cm}
\begin{tikzpicture}[yscale=1,xscale=-1]
\draw[densely dotted,rounded corners] (-22pt,-22pt) rectangle (22pt,22pt);
\draw[very thick] (0,-22pt)--++(90:44pt);
\fill[very thick] (0,-8pt)
 circle (4pt);
\fill[very thick] (0,8pt)
 circle (4pt);

 \end{tikzpicture}
\end{minipage}=
\begin{minipage}{1.55cm}
\begin{tikzpicture}[yscale=1,xscale=-1]\draw[densely dotted,rounded corners] (-22pt,-22pt) rectangle (22pt,22pt);
\fill[very thick] (0,0)
 circle (4pt);
\draw[very thick] (0,-22pt)--++(90:44pt);
 \end{tikzpicture}
\end{minipage}
\qquad\quad\quad  
\begin{minipage}{1.55cm}
\begin{tikzpicture}[yscale=1,xscale=-1]\draw[densely dotted,rounded corners] (-22pt,-22pt) rectangle (22pt,22pt);
\clip (0,0) circle (22pt);
\draw[very thick] 
 circle (8pt);
  \fill[very thick] (0pt,8pt)  circle (4pt) ;
 \end{tikzpicture}
\end{minipage}=
\begin{minipage}{1.85cm}
\begin{tikzpicture}[yscale=1.2,xscale=-1.2]
\draw[densely dotted,rounded corners] (-22pt,-22pt) rectangle (22pt,22pt);
\clip (0,0) circle (22pt);
\draw[very thick] 
 circle (8pt);
 \end{tikzpicture}
\end{minipage}=q +q ^{-1}
\]
 It has a basis given by the set of all $(n+1)$-tangles with no loops and at most one decoration on each strand and satisfying one of the following (mutually exclusive) conditions.
\begin{enumerate}[leftmargin=*]
\item The leftmost northern vertex is connected to the leftmost southern vertex by an undecorated strand.
\item The leftmost northern vertex is connected to the leftmost southern vertex  by a decorated strand and there is at least one northern and one southern arc.
\item The strands emerging from the leftmost northern and southern vertices are distinct and both decorated.
\end{enumerate}  
We denote the set of all such $(n+1)$-tangles by $\mathbb{DT}[C_n]$.
\end{thm}

\begin{thm}[Green \cite{MR1618912}]\label{thmGreenD} The generalised Temperley--Lieb algebra ${\rm TL}_{D_n}(q)$ of type $D_n$  is isomorphic to  the subquotient of the algebra $R\mathbb{DT}_{n}$ generated by ${\sf e}_{1''}, {\sf e}_1, {\sf e}_2, \ldots {\sf e}_{n-1}$ subject to relations
$$
\begin{minipage}{1.55cm}
\begin{tikzpicture}[yscale=1,xscale=-1]
\draw[densely dotted,rounded corners] (-22pt,-22pt) rectangle (22pt,22pt);
 \draw[very thick] (0,-22pt)--++(90:44pt);
\fill[very thick] (0,-8pt)
 circle (4pt);
\fill[very thick] (0,8pt)
 circle (4pt);
  \end{tikzpicture}
\end{minipage}=
\begin{minipage}{1.55cm}
\begin{tikzpicture}[yscale=1,xscale=-1]
\draw[densely dotted,rounded corners] (-22pt,-22pt) rectangle (22pt,22pt);
\draw[very thick] (0,-22pt)--++(90:44pt);
  \end{tikzpicture}
\end{minipage}
\qquad\qquad 
 \begin{minipage}{1.55cm}
\begin{tikzpicture}[yscale=1,xscale=-1]
\draw[densely dotted,rounded corners] (-22pt,-22pt) rectangle (22pt,22pt);
\clip (0,0) circle (22pt);
\draw[very thick] 
 circle (8pt);
   \end{tikzpicture}
\end{minipage}=q +q ^{-1}
\qquad\qquad
 \begin{minipage}{1.55cm}
\begin{tikzpicture}[yscale=1,xscale=1]
\draw[densely dotted,rounded corners] (-22pt,-22pt) rectangle (22pt,22pt);
\clip (0,0) circle (22pt);
\draw[very thick] (-5pt,0)
 circle (8pt);
  \fill[very thick] (-5pt,8pt)  circle (4pt) ;
  \draw[very thick] (0,-22pt) to [out=90, in=-90] (16pt,0) to [out=90, in=-90] (0,22pt);
   \fill[very thick] (15pt,0pt)  circle (4pt) ;
  \end{tikzpicture}
\end{minipage}=
 \begin{minipage}{1.55cm}
\begin{tikzpicture}[yscale=1,xscale=1]
\draw[densely dotted,rounded corners] (-22pt,-22pt) rectangle (22pt,22pt);
\clip (0,0) circle (22pt);
\draw[very thick] (-5pt,0)
 circle (8pt);
  \fill[very thick] (-5pt,8pt)  circle (4pt) ;
  \draw[very thick] (0,-22pt) to [out=90, in=-90] (16pt,0) to [out=90, in=-90] (0,22pt);

  \end{tikzpicture}\end{minipage}
$$It has a basis given by the set  of all $n$-tangles with at most one decoration on each strand or loop, and which satisfy one of the following (mutually exclusive) conditions.
\begin{enumerate}[leftmargin=*]
\item It contains one decorated loop and no other loops or decorations, and there is at least one northern and one southern arc.
\item It contains no loops and the number of decorations is even.
\end{enumerate}
We denote the set of all such $n$-tangles by $\mathbb{DT}[D_n]$.
\end{thm}

\subsection{Diagrammatic oriented Temperley--Lieb algebras}

Recall from Section 5 that  we can represent each coset representative $\lambda \in {^PW}$ by its coset diagram, which we also denote by $\lambda$. For $W\in\{A_n, C_n, D_n\}$, we let $\mathbb{DT}[W]$ be the set of tangles defined in \cref{thmGreenA,thmGreenB,thmGreenD}. Now starting with any tangle in $\mathbb{DT}[W]$ we will form so-called {\sf oriented tangles} by considering elements of the form $\lambda d \mu$ for $\lambda, \mu \in {^PW}$ by placing $\lambda$, respectively $\mu$ on the southern, respectively  northern, boundary of $d$. Each strand of $d$ in $\lambda d \mu$ now connects two symbols from the set $\{ \up, \down, \circ \}$ from the boundaries. 
A strand connecting two symbols from the set $\{\up, \down\}$	is said to be  {\sf oriented} if one
		 arrow points into the strand and the other arrow point out of the strand.	
We say that a strand connecting two symbols from the set $\{\up, \down\}$ is {\sf flip-oriented} if either both arrows point into the strand or both arrows point out of the strand. 
A strand connecting $\circ$ with one symbol from the set $\{\up , \down\}$ is said to be both oriented and flip-oriented.

\begin{defn} 
For $d\in \mathbb{DT}[W]$ and $\lambda , \mu\in {^PW}$ we say that $\lambda d \mu$ is an {\sf oriented tangle of type $(W,P)$} if the following conditions holds.
\begin{itemize}
\item Every undecorated strand is oriented.
\end{itemize}
Moreover, if $(W,P) = (D_n, A_{n-1})$ then
\begin{itemize}
\item every decorated strand is flip-oriented, and 
\item there are no loops.
\end{itemize}
If $(W,P)=(C_n, A_{n-1})$ then 
\begin{itemize}
\item there are no decorations on the strand connecting $\circ$ to $\circ$.
\end{itemize}
We denote by $\mathbb{ODT}[W,P]$ the set of all oriented tangles of type $(W,P)$ and refer to these as oriented Temperley--Lieb diagrams of type $(W,P)$.
\end{defn}

\noindent Examples of oriented tangles are given in \cref{typeAtiling2lkjsdfghslkdfjghsldk,typeAtiling2lkjsdfghslkdfjghsldkXXX}.

\begin{figure} [H]
	$$   \begin{tikzpicture} [scale=0.85]
		
		\clip(4,1.82) rectangle ++ (6,-2.7);
		
		\path (4,1) coordinate (origin); 
		\path (origin)--++(0.5,0.5) coordinate (origin2);  
		\draw(origin2)--++(0:5); 
		\foreach \i in {1,2,3,4,5,...,9}
		{
			\path (origin2)--++(0:0.5*\i) coordinate (a\i); 
			\path (origin2)--++(0:0.5*\i)--++(-90:0.08) coordinate (c\i); 
			\draw (a\i)circle (1pt);  }
		
		\foreach \i in {1,2,3,4,5,...,19}
		{
			\path (origin2)--++(0:0.25*\i) --++(-90:0.5) coordinate (b\i); 
			\path (origin2)--++(0:0.25*\i) --++(-90:0.7) coordinate (d\i); 
		}
		\path(a1) --++(90:0.175) node  {  $  \down   $} ;
		\path(a2) --++(90:0.175) node  {  $  \down   $} ;
		\path(a3) --++(-90:0.175) node  {  $  \up   $} ;
		\path(a4) --++(-90:0.175) node  {  $  \up   $} ;
		\path(a5) --++(-90:0.175) node  {  $  \up  $} ;
		\path(a6) --++(90:-0.175) node  {  $  \up  $} ;
		\path(a7) --++(90:0.175) node  {  $  \down  $} ;
		\path(a8) --++(90:0.175) node  {  $  \down  $} ; 
		\path(a9) --++(-90:0.175) node  {  $  \up  $} ;

		\draw[  very thick](c3) to [out=-90,in=0] (b5) to [out=180,in=-90] (c2); 
		\draw[  very thick](c4) to [out=-90,in=0] (d5) to [out=180,in=-90] (c1);

		\draw[  very thick](c9) to [out=-90,in=0] (b17) to [out=180,in=-90] (c8);

		\draw[  very thick](c5) to [out=-90,in=15] (6,0.55);  
		
		\draw[  very thick](c6) to [out=-90,in=15] (7,0.55);

		\draw[  very thick](c7) to [out=-90,in=165] (8.5,0.55);

		\path (4,-1) coordinate (origin); 
		\path (origin)--++(0.5,0.5) coordinate (origin2);  
		\draw(origin2)--++(0:5); 
		\foreach \i in {1,2,3,4,5,...,9}
		{
			\path (origin2)--++(0:0.5*\i) coordinate (a\i); 
			\path (origin2)--++(0:0.5*\i)--++(90:0.08) coordinate (c\i); 
			\draw (a\i)circle (1pt);  }
		
		\foreach \i in {1,2,3,4,5,...,19}
		{
			\path (origin2)--++(0:0.25*\i) --++(90:0.5) coordinate (b\i); 
			\path (origin2)--++(0:0.25*\i) --++(90:0.65) coordinate (d\i); 
		}
		\path(a1) --++(90:-0.175) node  {  $  \up  $} ;
		\path(a2) --++(90:-0.175) node  {  $  \up   $} ;
		\path(a3) --++(90:-0.175) node  {  $  \up   $} ;
		\path(a4) --++(90:-0.175) node  {  $ \up   $} ;
		\path(a5) --++(90:-0.175) node  {  $  \up  $} ;
		
		\path(a6) --++(90:0.175) node  {  $ \down     $} ;
		\path(a7) --++(90:0.175) node  {  $  \down     $} ;
		\path(a8) --++(90:0.175) node  {  $ \down     $} ;
		\path(a9) --++(90:0.175) node  {  $  \down     $} ;
		
		\draw[  very thick](c6) to [out=90,in=0] (b11) to [out=180,in=90] (c5); 
		\draw[  very thick](c7) to [out=90,in=0] (d11) to [out=180,in=90] (c4); 
		
		\path(d11)--++(90:0.2) coordinate (e11); 
		\draw[  very thick](c8) to [out=90,in=0] (e11) to [out=180,in=90] (c3); 
		

		\draw[very   thick](c1) to [out=90,in=-170] (6,0.55); 
		\draw[very   thick](c2) to [out=90,in=-170] (7,0.55);   
		
		\draw[  very thick](c9) to [out=90,in=-15] (8.5,0.55);

	\end{tikzpicture} 
     \begin{tikzpicture} [scale=0.85]
		\clip(4,1.82) rectangle ++ (6,-2.7);

		\path (4,1) coordinate (origin); 
		\path (origin)--++(0.5,0.5) coordinate (origin2);  
		\draw(origin2)--++(0:5); 
		\foreach \i in {1,2,3,4,5,...,9}
		{
			\path (origin2)--++(0:0.5*\i) coordinate (a\i); 
			\path (origin2)--++(0:0.5*\i)--++(-90:0.08) coordinate (c\i); 
			\draw (a\i)circle (1pt);  }
		
		\foreach \i in {1,2,3,4,5,...,19}
		{
			\path (origin2)--++(0:0.25*\i) --++(-90:0.5) coordinate (b\i); 
			\path (origin2)--++(0:0.25*\i) --++(-90:0.7) coordinate (d\i); 
		}
		\path(a1) --++(90:0.175) node  {  $  \down   $} ;
		\path(a3) --++(90:0.175) node  {  $  \down   $} ;
		\path(a2) --++(-90:0.175) node  {  $  \up   $} ;
		\path(a4) --++(-90:0.175) node  {  $  \up   $} ;
		\path(a5) --++(-90:0.175) node  {  $  \up  $} ;
		\path(a6) --++(90:-0.175) node  {  $  \up  $} ;
		\path(a7) --++(90:0.175) node  {  $  \down  $} ;
		\path(a8) --++(90:0.175) node  {  $  \down  $} ; 
		\path(a9) --++(-90:0.175) node  {  $  \up  $} ;

		\draw[  very thick](c3) to [out=-90,in=0] (b5) to [out=180,in=-90] (c2); 
		\draw[  very thick](c4) to [out=-90,in=0] (d5) to [out=180,in=-90] (c1);

		\draw[  very thick](c9) to [out=-90,in=0] (b17) to [out=180,in=-90] (c8);

		\draw[  very thick](c5) to [out=-90,in=15] (6,0.55);  
		
		\draw[  very thick](c6) to [out=-90,in=15] (7,0.55);

		\draw[  very thick](c7) to [out=-90,in=165] (8.5,0.55);

		\path (4,-1) coordinate (origin); 
		\path (origin)--++(0.5,0.5) coordinate (origin2);  
		\draw(origin2)--++(0:5); 
		\foreach \i in {1,2,3,4,5,...,9}
		{
			\path (origin2)--++(0:0.5*\i) coordinate (a\i); 
			\path (origin2)--++(0:0.5*\i)--++(90:0.08) coordinate (c\i); 
			\draw (a\i)circle (1pt);  }
		
		\foreach \i in {1,2,3,4,5,...,19}
		{
			\path (origin2)--++(0:0.25*\i) --++(90:0.5) coordinate (b\i); 
			\path (origin2)--++(0:0.25*\i) --++(90:0.65) coordinate (d\i); 
		}
		\path(a1) --++(90:-0.175) node  {  $  \up  $} ;
		\path(a2) --++(90:-0.175) node  {  $  \up   $} ;
		\path(a3) --++(90:-0.175) node  {  $  \up   $} ;
		\path(a4) --++(90:-0.175) node  {  $ \up   $} ;
		\path(a5) --++(90:-0.175) node  {  $  \up  $} ;
		
		\path(a6) --++(90:0.175) node  {  $ \down     $} ;
		\path(a7) --++(90:0.175) node  {  $  \down     $} ;
		\path(a8) --++(90:0.175) node  {  $ \down     $} ;
		\path(a9) --++(90:0.175) node  {  $  \down     $} ;
		
		\draw[  very thick](c6) to [out=90,in=0] (b11) to [out=180,in=90] (c5); 
		\draw[  very thick](c7) to [out=90,in=0] (d11) to [out=180,in=90] (c4); 
		
		\path(d11)--++(90:0.2) coordinate (e11); 
		\draw[  very thick](c8) to [out=90,in=0] (e11) to [out=180,in=90] (c3); 
		

		\draw[very   thick](c1) to [out=90,in=-170] (6,0.55); 
		\draw[very   thick](c2) to [out=90,in=-170] (7,0.55);   
		
		\draw[  very thick](c9) to [out=90,in=-15] (8.5,0.55);

	\end{tikzpicture} 
 \begin{tikzpicture} [scale=0.85]
		
		\clip(4,1.82) rectangle ++ (6,-2.7);
		
		\path (4,1) coordinate (origin); 
		\path (origin)--++(0.5,0.5) coordinate (origin2);  
		\draw(origin2)--++(0:5); 
		\foreach \i in {1,2,3,4,5,...,9}
		{
			\path (origin2)--++(0:0.5*\i) coordinate (a\i); 
			\path (origin2)--++(0:0.5*\i)--++(-90:0.08) coordinate (c\i); 
			\draw (a\i)circle (1pt);  }

		\path(a2) --++(90:0.175) node  {  $  \down   $} ;
		\path(a1) --++(-90:0.175) node  {  $  \up    $} ;
		\path(a3) --++(-90:0.175) node  {  $  \up    $} ;
		\path(a6) --++(-90:-0.175) node  {  $  \down    $} ;
		\path(a5) --++(-90:-0.175) node  {  $  \down  $} ;
		\path(a4) --++(90:-0.175) node  {  $  \up  $} ;
		\path(a7) --++(90:0.175) node  {  $  \down  $} ;
 		\path(a9) --++(-90:0.175) node  {  $  \up  $} ;
		\path(a8) --++(-90:0.175) node  {  $  \up  $} ;
		
		\foreach \i in {1,2,3,4,5,...,19}
		{
			\path (origin2)--++(0:0.25*\i) --++(-90:0.5) coordinate (b\i); 
			\path (origin2)--++(0:0.25*\i) --++(-90:0.7) coordinate (d\i); 
		}
	 		
		\draw[  very thick](c5) to [out=-90,in=0] (b9) to [out=180,in=-90] (c4); 		
		\draw[  very thick](c3) to [out=-90,in=0] (b5) to [out=180,in=-90] (c2); 
		\draw[  very thick](c6) to [out=-90,in=0] (d7) to [out=180,in=-90] (c1);

 	\draw[  very thick](c8) to [out=-90,in=0] (b15) to [out=180,in=-90] (c7);

%
%
%
%
%
		\draw[  very thick](c9) to [out=-90,in=15] (7,0.55);  

		\path (4,-1) coordinate (origin); 
		\path (origin)--++(0.5,0.5) coordinate (origin2);  
		\draw(origin2)--++(0:5); 
		\foreach \i in {1,2,3,4,5,...,9}
		{
			\path (origin2)--++(0:0.5*\i) coordinate (a\i); 
			\path (origin2)--++(0:0.5*\i)--++(90:0.08) coordinate (c\i); 
			\draw (a\i)circle (1pt);  }
		
		\foreach \i in {1,2,3,4,5,...,19}
		{
			\path (origin2)--++(0:0.25*\i) --++(90:0.5) coordinate (b\i); 
			\path (origin2)--++(0:0.25*\i) --++(90:0.65) coordinate (d\i); 
		} 
		
		\draw[  very thick](c8) to [out=90,in=0] (b15) to [out=180,in=90] (c7); 
		\draw[  very thick](c9) to [out=90,in=0] (d15) to [out=180,in=90] (c6); 
		
		
				\draw[  very thick](c2) to [out=90,in=0] (b3) to [out=180,in=90] (c1); 
		
				\draw[  very thick](c5) to [out=90,in=0] (b9) to [out=180,in=90] (c4); 

		\draw[very   thick](c3) to [out=90,in=-170] (7,0.55); 
%
%
		
%
	\path(a1) --++(90:0.175) node  {  $  \down   $} ;
		\path(a2) --++(-90:0.175) node  {  $  \up    $} ;
		\path(a3) --++(-90:0.175) node  {  $  \up    $} ;
		\path(a6) --++(-90:-0.175) node  {  $  \down    $} ;
		\path(a5) --++(-90:-0.175) node  {  $  \down  $} ;
		\path(a4) --++(90:-0.175) node  {  $  \up  $} ;
		\path(a7) --++(90:0.175) node  {  $  \down  $} ;
 		\path(a9) --++(-90:0.175) node  {  $  \up  $} ;
		\path(a8) --++(-90:0.175) node  {  $  \up  $} ;					
	\end{tikzpicture} 	$$ 	\caption{ Oriented tangles of type $(A_{8},A_4\times A_3)$ obtained from the diagrams of  \cref{typeAtiling2lkjs}. 
	 }
	\label{typeAtiling2lkjsdfghslkdfjghsldk}
\end{figure}

\begin{rmk}\label{rmkdiagTL}
Note that, by definition, we have that $\la \diag_i \mu\in \mathbb{ODT}[W,P]$ if and only if $\la \xrightarrow{i} \mu$ is an edge in the graph $\widehat{\mathcal{G}}_{(W,P)}$.
\end{rmk}

The following proposition follows directly from \cref{weightcosets}.

\begin{prop}
Let $d\in \mathbb{DT}[W]$. Then there exists $\lambda, \mu \in {^PW}$ such that $\lambda d \mu\in \mathbb{ODT}[W,P]$ if and only if
\begin{itemize}[leftmargin=*]
\item  $(W,P) = (A_n, A_k\times A_{n-k-1})$ and $d$ has at most $\min \{k+1, n-k\}$ northern/southern arcs.
\item  $(W,P) = (C_n, A_{n-1})$ and  $d$ contains no decoration on a strand connecting $\circ$ to $\circ$ (i.e. $d$ satisfies condition (1) or (3) from \cref{thmGreenB}).
\item $(W,P) = (D_n, A_{n-1})$ and $d$ contains no loops (i.e. $d$ satisfies condition (2) from \cref{thmGreenD}).
\end{itemize}
We denote the set of all tangles described above by $\mathbb{DT}[W,P]$.
\end{prop}

\begin{defn}\label{diagrelations}
The {\sf diagrammatic oriented Temperley--Lieb algebra} of type   $(W,P)$, denoted by ${\rm TL}^{{\up}{\down}}_{(W,P)}(q)$,  is the $\mathbb{Z}[q,q^{-1}]$-algebra with basis $\mathbb{ODT}[W,P]$ and multiplication defined as follows.
For $\lambda d \mu, \lambda' d' \mu'\in  \mathbb{ODT}[W,P]$ we have
\begin{equation}\label{matchy}
(\lambda d \mu)(\lambda' d' \mu')=0 \quad \mbox{if $\mu \neq \lambda'$}
\end{equation}
and vertical concatenation, denoted by $\la d \mu d' \mu'$, if $\mu = \lambda'$, subject to the following relations.

\begin{enumerate}[leftmargin=*]
\item {\bf Closed loop relation. } 
   Remove any closed  loop and replace it by $q$ if its rightmost vertex is labelled by $\down$, and by $q^{-1}$ if its rightmost vertex is labelled by $\up$.

  $$
 \begin{minipage}{1.68cm}
 \begin{tikzpicture}
\draw[very thick] (0,0pt)
 circle (14pt);

       \draw[very thick,white] (0,0)--++(65:22pt);
 \draw[very thick,white] (0,0)--++(75:22pt);
  \draw[very thick,white] (0,0)--++(85:22pt);
    \draw[very thick,white] (0,0)--++(95:22pt);
      \draw[very thick,white] (0,0)--++(105:22pt);
      \draw[very thick,white] (0,0)--++(115:22pt);

       \draw[very thick,white] (-0,0)--++(-65:22pt);
 \draw[very thick,white] (-0,0)--++(-75:22pt);
  \draw[very thick,white] (-0,0)--++(-85:22pt);
    \draw[very thick,white] (-0,0)--++(-95:22pt);
      \draw[very thick,white] (-0,0)--++(-105:22pt);
      \draw[very thick,white] (-0,0)--++(-115:22pt);

 \draw[densely dotted,rounded corners] (-22pt,-22pt) rectangle (22pt,22pt);

  \draw  (0,0)--++(0:22pt) (0,0)--++(0:-22pt);;

 \draw [fill=white] (14pt,0) circle (5pt) node  {$x$};

 \draw[very thick] (-14pt,0)  --++(-55:0.24);
  \draw[very thick] (-14pt,0)  --++(-125:0.24);
 
   \end{tikzpicture}
\end{minipage}=\;
\begin{minipage}{1.68cm}
 \begin{tikzpicture}[yscale=-1,xscale=1]
\draw[very thick] (0,0pt)
 circle (14pt);

       \draw[very thick,white] (0,0)--++(65:22pt);
 \draw[very thick,white] (0,0)--++(75:22pt);
  \draw[very thick,white] (0,0)--++(85:22pt);
    \draw[very thick,white] (0,0)--++(95:22pt);
      \draw[very thick,white] (0,0)--++(105:22pt);
      \draw[very thick,white] (0,0)--++(115:22pt);

       \draw[very thick,white] (-0,0)--++(-65:22pt);
 \draw[very thick,white] (-0,0)--++(-75:22pt);
  \draw[very thick,white] (-0,0)--++(-85:22pt);
    \draw[very thick,white] (-0,0)--++(-95:22pt);
      \draw[very thick,white] (-0,0)--++(-105:22pt);
      \draw[very thick,white] (-0,0)--++(-115:22pt);

 \draw[densely dotted,rounded corners] (-22pt,-22pt) rectangle (22pt,22pt);

  \draw  (0,0)--++(0:22pt) (0,0)--++(0:-22pt);;

 \draw [fill=white] (14pt,0) circle (5pt) node  {$x$};

 \draw[very thick] (-14pt,0)  --++(-55:0.24);
  \draw[very thick] (-14pt,0)  --++(-125:0.24);
 
   \end{tikzpicture}
\end{minipage}=\; 
\begin{minipage}{1.68cm}
 \begin{tikzpicture}
\draw[very thick] (0,0pt)
 circle (14pt);

       \draw[very thick,white] (0,0)--++(65:22pt);
 \draw[very thick,white] (0,0)--++(75:22pt);
  \draw[very thick,white] (0,0)--++(85:22pt);
    \draw[very thick,white] (0,0)--++(95:22pt);
      \draw[very thick,white] (0,0)--++(105:22pt);
      \draw[very thick,white] (0,0)--++(115:22pt);

       \draw[very thick,white] (-0,0)--++(-65:22pt);
 \draw[very thick,white] (-0,0)--++(-75:22pt);
  \draw[very thick,white] (-0,0)--++(-85:22pt);
    \draw[very thick,white] (-0,0)--++(-95:22pt);
      \draw[very thick,white] (-0,0)--++(-105:22pt);
      \draw[very thick,white] (-0,0)--++(-115:22pt);

 \draw[densely dotted,rounded corners] (-22pt,-22pt) rectangle (22pt,22pt);

  \draw  (0,0)--++(0:22pt) (0,0)--++(0:-22pt);;

%
%

 \draw [fill=white] (14pt,0) circle (5pt) node  {$x$};

 \draw[fill=white]  (-14pt,0)   circle (1.7pt);
  
   \end{tikzpicture}
\end{minipage}=
\begin{cases}
q &\text{if } x=\down \\
q^{-1} &\text{if } x=\up \\
\end{cases}
$$Once all closed loops have been removed, delete all remaining symbols coming from $\mu = \lambda'$ and apply the following relations.

\medskip

\item {\bf Double bead relation. } We have that 
$$\begin{minipage}{1.55cm}
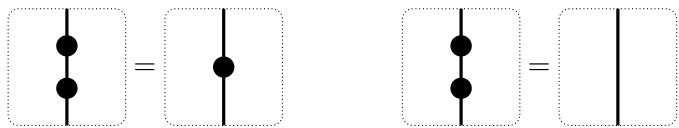
\begin{tikzpicture}[yscale=1,xscale=-1]
\draw[densely dotted,rounded corners] (-22pt,-22pt) rectangle (22pt,22pt);
\draw[very thick] (0,-22pt)--++(90:44pt);
\fill[very thick] (0,-8pt)
 circle (4pt);
\fill[very thick] (0,8pt)
 circle (4pt);

 \end{tikzpicture}
\end{minipage}=
\begin{minipage}{1.55cm}
\begin{tikzpicture}[yscale=1,xscale=-1]\draw[densely dotted,rounded corners] (-22pt,-22pt) rectangle (22pt,22pt);
\fill[very thick] (0,0)
 circle (4pt);
\draw[very thick] (0,-22pt)--++(90:44pt);
 \end{tikzpicture}
\end{minipage}
\qquad\quad\quad\begin{minipage}{1.55cm}
\begin{tikzpicture}[yscale=1,xscale=-1]
\draw[densely dotted,rounded corners] (-22pt,-22pt) rectangle (22pt,22pt);
 \draw[very thick] (0,-22pt)--++(90:44pt);
\fill[very thick] (0,-8pt)
 circle (4pt);
\fill[very thick] (0,8pt)
 circle (4pt);
  \end{tikzpicture}
\end{minipage}=
\begin{minipage}{1.55cm}
\begin{tikzpicture}[yscale=1,xscale=-1]
\draw[densely dotted,rounded corners] (-22pt,-22pt) rectangle (22pt,22pt);
\draw[very thick] (0,-22pt)--++(90:44pt);
  \end{tikzpicture}
\end{minipage}
$$in types $C_n$ and $D_n$ respectively.   

\medskip

\item {\bf Non-simply laced bead relation. } In type $C_n$, we have that 
$$\begin{minipage}{1.55cm}
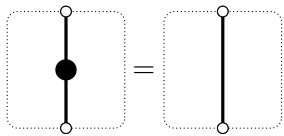
\begin{tikzpicture}[yscale=1,xscale=-1]
\draw[fill=white] (0,-22pt) circle (2pt);
\draw[fill=white] (0,22pt) circle (2pt);
\draw[densely dotted,rounded corners] (-22pt,-22pt) rectangle (22pt,22pt);
 \draw[very thick] (0,-22pt)--++(90:44pt);
\fill[very thick] (0,0pt)
 circle (4pt);
\draw[fill=white] (0,-22pt) circle (2pt);
\draw[fill=white] (0,22pt) circle (2pt);

  \end{tikzpicture}
\end{minipage}=
\begin{minipage}{1.55cm}
\begin{tikzpicture}[yscale=1,xscale=-1]
\draw[fill=white] (0,-22pt) circle (2pt);
\draw[fill=white] (0,22pt) circle (2pt);
\draw[densely dotted,rounded corners] (-22pt,-22pt) rectangle (22pt,22pt);
\draw[very thick] (0,-22pt)--++(90:44pt);
\draw[fill=white] (0,-22pt) circle (2pt);
\draw[fill=white] (0,22pt) circle (2pt);

  \end{tikzpicture}
\end{minipage}
$$

\end{enumerate}

\end{defn}

\begin{rmk} 
Note that the multiplication is associative. Indeed, consider the product of three oriented tangles forming a closed loop as pictured below. Then we have that the red zigzag-strand is not left-exposed and so must be oriented. This implies that the symbols at $x$ and $y$ are either both $\up$ or both $\down$ and hence we will get the same result whichever way we apply the two multiplications.

$$
  
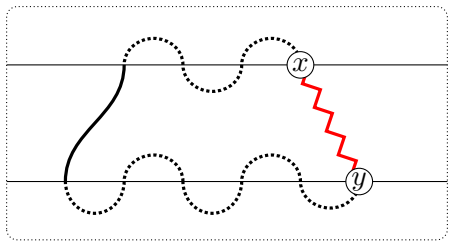
\begin{tikzpicture} 
  \draw[densely dotted,rounded corners] (-22*1.5pt,-66pt) rectangle (22*6pt,22pt);

 \draw[very thick,densely dotted] (22*3.5pt,0) to [out=90,in=0]
 (22*3pt,10pt)
  to [out=180,in=90]
   (22*2.5pt,0pt);

    \draw[very thick,densely dotted] (22*2.5pt,0) to [out=-90,in=0]
 (22*2pt,-10pt)
  to [out=180,in=-90]
   (22*1.5pt,0pt);
   
    \draw[very thick,densely dotted] (22*1.5pt,0) to [out=90,in=0]
 (22*1pt,10pt)
  to [out=180,in=90]
   (22*0.5pt,0pt) coordinate (X);

  \draw[very thick,red,zigzag] (22*4.5pt,-44pt)  to  (22*3.5pt,0);

%
%
  
    \draw[very thick,densely dotted] (22*4.5pt,-44pt) to [out=-90,in=0]
 (22*4pt,-54pt)
  to [out=180,in=-90]
   (22*3.5pt,-44pt);
    \draw[very thick,densely dotted] (22*3.5pt,-44pt) to [out=90,in=0]
 (22*3pt,-34pt)
  to [out=180,in=90]
   (22*2.5pt,-44pt);


     \draw[very thick,densely dotted] (22*2.5pt,-44pt) to [out=-90,in=0]
 (22*2pt,-56pt)
  to [out=180,in=-90]
   (22*1.5pt,-44pt) ;

    \draw[very thick,densely dotted] (22*1.5pt,-44pt) to [out=90,in=0]
 (22*1pt,-34pt)
  to [out=180,in=90]
   (22*0.5pt,-44pt);
   
        \draw[very thick,densely dotted] (22*0.5pt,-44pt) to [out=-90,in=0]
 (22*0pt,-56pt)
  to [out=180,in=-90]
   (22*-0.5pt,-44pt) ;

   \draw[ very thick]    (22*-0.5pt,-44pt)  to [out=90,in=-90] (X);

 \draw  (0,0)--++(0:22pt*6) (0,0)--++(0:-22*1.5pt);;
 \draw  (0,-44pt)--++(0:22pt*6) (0,-44pt)--++(0:-22*1.5pt);;

 \draw [fill=white] (22*3.5pt,0) circle (5pt) node  {$x$};

 \draw [fill=white] (22*4.5pt,-44pt) circle (5pt) node  {$y$};

   \end{tikzpicture}
 $$Therefore ${\rm TL}^{{\up}{\down}}_{(W,P)}(q)$ is a unital associative algebra with identity $\sum_{\lambda \in {^PW}} \lambda {\sf 1}\lambda$ where ${\sf 1}$ is the tangle in $\mathbb{DT}(W)$ containing only undecorated propagating strands  (that is, the identity element in ${\rm TL}_{W}(q)$).
\end{rmk}

\begin{figure} [H]
	$$\hspace{-0.1cm} \begin{tikzpicture} [scale=0.85]
		
 	 	\clip(3,1.82) rectangle ++ (6.5,-2.7);

		\path (4,1) coordinate (origin); 
		\path (origin)--++(0.5,0.5) coordinate (origin2);  
		\draw(origin2)--++(0:4.5) (origin2)--++(180:1); 
		\foreach \i in {-1,0,1,2,3,4,5,...,8}
		{
			\path (origin2)--++(0:0.5*\i) coordinate (a\i); 
			\path (origin2)--++(0:0.5*\i)--++(-90:0.08) coordinate (c\i); 
  }
		
		\foreach \i in {-1,0,1,2,3,4,5,...,19}
		{
			\path (origin2)--++(0:0.25*\i) --++(-90:0.5) coordinate (b\i); 
			\path (origin2)--++(0:0.25*\i) --++(-90:0.7) coordinate (d\i); 
		} 
		
		\draw[  very thick](c3) to [out=-90,in=0] (b5) to [out=180,in=-90] (c2); 
		\draw[  very thick](c4) to [out=-90,in=0] (d5) to [out=180,in=-90] (c1); 
		
		\fill (d5)  circle (3pt);

 		\path(a0) --++(90:0.175) node  {  $  \down  $} ;
		
 		\path(a1) --++(90:-0.175) node  {  $  \up  $} ;
		\path(a2) --++(90:-0.175) node  {  $  \up   $} ;
		\path(a3) --++(90:0.175) node  {  $  \down   $} ;
		\path(a4) --++(90:0.175) node  {  $ \down   $} ;
		\path(a5) --++(90:-0.175) node  {  $  \up  $} ;
		
		\path(a6) --++(90:0.175) node  {  $ \down     $} ;
		\path(a7) --++(90:0.175) node  {  $  \down     $} ;
		\path(a8) --++(90:-0.175) node  {  $ \up     $} ;
\draw		(a-1)  circle (2pt);

		\draw[  very thick](c8) to [out=-90,in=0] (b15) to [out=180,in=-90] (c7);

	 	\draw[  very thick](c0) to [out=-90,in=0] (b-1) to [out=180,in=-90] (c-1); 
	
	\fill (b-1)  circle (3pt);    	
		
		\draw[  very thick](c5) to [out=-90,in=145] (7.5,0.55);  
		
		\draw[  very thick](c6) to [out=-90,in=145] (8,0.55);


		\path (4,-1) coordinate (origin); 
		\path (origin)--++(0.5,0.5) coordinate (origin2);  
		\draw(origin2)--++(0:4.5); 
				\draw(origin2)--++(180:1); 
		\foreach \i in {-1,0,1,2,3,4,5,...,8}
		{
			\path (origin2)--++(0:0.5*\i) coordinate (a\i); 
			\path (origin2)--++(0:0.5*\i)--++(90:0.08) coordinate (c\i); 
	 }
		
		\foreach \i in {-1,0,1,2,3,4,5,...,19}
		{
			\path (origin2)--++(0:0.25*\i) --++(90:0.5) coordinate (b\i); 
			\path (origin2)--++(0:0.25*\i) --++(90:0.65) coordinate (d\i); 
		}

 	\draw[  very thick](c0) to [out=90,in=0] (b-1) to [out=180,in=90] (c-1); 
 	\draw[  very thick](c2) to [out=90,in=0] (b3) to [out=180,in=90] (c1); 
 	\draw[  very thick](c4) to [out=90,in=0] (b7) to [out=180,in=90] (c3); 	
 	\draw[  very thick](c6) to [out=90,in=0] (b11) to [out=180,in=90] (c5); 	
 \draw[  very thick](c8) to [out=90,in=-35] (8,0.55) ;
		 \draw[  very thick](c7) to [out=90,in=-35] (7.5,0.55) coordinate (here);
\fill (here)  circle (3pt);     

\fill (b-1)  circle (3pt);    	
\fill (b3)  circle (3pt);    		
\fill (b7)  circle (3pt);    		\fill (b11)  circle (3pt);

\draw		(a-1)  circle (2pt);
				\path(a0) --++(90:0.175) node  {  $  \down   $} ;

		\path(a1) --++(90:0.175) node  {  $  \down   $} ;
				\path(a2) --++(90:0.175) node  {  $  \down   $} ;
						\path(a3) --++(90:0.175) node  {  $  \down   $} ;
		\path(a4) --++(90:0.175) node  {  $  \down   $} ;
				\path(a5) --++(90:0.175) node  {  $  \down   $} ;
						\path(a6) --++(90:0.175) node  {  $ \down     $} ;
		\path(a7) --++(90:0.175) node  {  $  \down     $} ;
		\path(a8) --++(90:0.175) node  {  $ \down     $} ;
 			
	\end{tikzpicture}   \begin{tikzpicture} [scale=0.85]
		
 	 	\clip(3,1.82) rectangle ++ (6.5,-2.7);

		\path (4,1) coordinate (origin); 
		\path (origin)--++(0.5,0.5) coordinate (origin2);  
		\draw(origin2)--++(0:4.5) (origin2)--++(180:1); 
		\foreach \i in {-1,0,1,2,3,4,5,...,8}
		{
			\path (origin2)--++(0:0.5*\i) coordinate (a\i); 
			\path (origin2)--++(0:0.5*\i)--++(-90:0.08) coordinate (c\i); 
  }
		
		\foreach \i in {-1,0,1,2,3,4,5,...,19}
		{
			\path (origin2)--++(0:0.25*\i) --++(-90:0.5) coordinate (b\i); 
			\path (origin2)--++(0:0.25*\i) --++(-90:0.7) coordinate (d\i); 
		} 
		
		\draw[  very thick](c3) to [out=-90,in=0] (b5) to [out=180,in=-90] (c2); 
		\draw[  very thick](c4) to [out=-90,in=0] (d5) to [out=180,in=-90] (c1); 
		
		\fill (d5)  circle (3pt);

 		\path(a-1) --++(-90:0.175) node  {  $  \up $} ;
 		\path(a0) --++(90:0.175) node  {  $  \down  $} ;
		
 		\path(a1) --++(90:0.175) node  {  $  \down  $} ;
		\path(a2) --++(90:-0.175) node  {  $  \up   $} ;
		\path(a3) --++(90:0.175) node  {  $  \down   $} ;
		\path(a4) --++(90:0.175) node  {  $ \down   $} ;
		\path(a5) --++(90:-0.175) node  {  $  \up  $} ;
		
		\path(a6) --++(90:0.175) node  {  $ \down     $} ;
		\path(a7) --++(90:0.175) node  {  $  \down     $} ;
		\path(a8) --++(90:-0.175) node  {  $ \up     $} ;

		\draw[  very thick](c8) to [out=-90,in=0] (b15) to [out=180,in=-90] (c7);

	 	\draw[  very thick](c0) to [out=-90,in=0] (b-1) to [out=180,in=-90] (c-1);

		\draw[  very thick](c5) to [out=-90,in=145] (7.5,0.55);  
		
		\draw[  very thick](c6) to [out=-90,in=145] (8,0.55);


		\path (4,-1) coordinate (origin); 
		\path (origin)--++(0.5,0.5) coordinate (origin2);  
		\draw(origin2)--++(0:4.5); 
				\draw(origin2)--++(180:1); 
		\foreach \i in {-1,0,1,2,3,4,5,...,8}
		{
			\path (origin2)--++(0:0.5*\i) coordinate (a\i); 
			\path (origin2)--++(0:0.5*\i)--++(90:0.08) coordinate (c\i); 
  }
		
		\foreach \i in {-1,0,1,2,3,4,5,...,19}
		{
			\path (origin2)--++(0:0.25*\i) --++(90:0.5) coordinate (b\i); 
			\path (origin2)--++(0:0.25*\i) --++(90:0.65) coordinate (d\i); 
		}

 	\draw[  very thick](c0) to [out=90,in=0] (b-1) to [out=180,in=90] (c-1); 
 	\draw[  very thick](c2) to [out=90,in=0] (b3) to [out=180,in=90] (c1); 
 	\draw[  very thick](c4) to [out=90,in=0] (b7) to [out=180,in=90] (c3); 	
 	\draw[  very thick](c6) to [out=90,in=0] (b11) to [out=180,in=90] (c5); 	
 \draw[  very thick](c8) to [out=90,in=-35] (8,0.55) ;
		 \draw[  very thick](c7) to [out=90,in=-35] (7.5,0.55) coordinate (here);
\fill (here)  circle (3pt);     

\fill (b-1)  circle (3pt);    	
\fill (b3)  circle (3pt);    		
\fill (b7)  circle (3pt);    		\fill (b11)  circle (3pt);

 	\path(a-1) --++(90:0.175) node  {  $  \down   $} ;
				\path(a0) --++(90:0.175) node  {  $  \down   $} ;

		\path(a1) --++(90:0.175) node  {  $  \down   $} ;
				\path(a2) --++(90:0.175) node  {  $  \down   $} ;
						\path(a3) --++(90:0.175) node  {  $  \down   $} ;
		\path(a4) --++(90:0.175) node  {  $  \down   $} ;
				\path(a5) --++(90:0.175) node  {  $  \down   $} ;
						\path(a6) --++(90:0.175) node  {  $ \down     $} ;
		\path(a7) --++(90:0.175) node  {  $  \down     $} ;
		\path(a8) --++(90:0.175) node  {  $ \down     $} ;
 			
	\end{tikzpicture}   
	    \begin{tikzpicture}[scale=0.85] 
		\clip(4,1.82) rectangle ++ (6,-2.7);

		\path (4,1) coordinate (origin); 
		\path (origin)--++(0.5,0.5) coordinate (origin2);  
		\draw(origin2)--++(0:5); 
		\foreach \i in {1,2,3,4,5,...,9}
		{
			\path (origin2)--++(0:0.5*\i) coordinate (a\i); 
			\path (origin2)--++(0:0.5*\i)--++(-90:0.08) coordinate (c\i); 
   }
		
		\foreach \i in {1,2,3,4,5,...,19}
		{
			\path (origin2)--++(0:0.25*\i) --++(-90:0.5) coordinate (b\i); 
			\path (origin2)--++(0:0.25*\i) --++(-90:0.7) coordinate (d\i); 
		}
		
\draw		(a1)  circle (2pt);
		\path(a3) --++(90:0.175) node  {  $  \down   $} ;
		\path(a2) --++(-90:0.175) node  {  $  \up    $} ;
		\path(a4) --++(-90:-0.175) node  {  $  \down    $} ;
		\path(a5) --++(-90:-0.175) node  {  $  \down  $} ;
		\path(a6) --++(90:-0.175) node  {  $  \up  $} ;
		\path(a7) --++(90:0.175) node  {  $  \down  $} ;
		\path(a8) --++(90:0.175) node  {  $  \down  $} ; 
		\path(a9) --++(-90:0.175) node  {  $  \up  $} ;

		\draw[  very thick](c3) to [out=-90,in=0] (b5) to [out=180,in=-90] (c2); 
		\draw[  very thick](c4) to [out=-90,in=0] (d5) to [out=180,in=-90] (c1);

		\draw[  very thick](c9) to [out=-90,in=0] (b17) to [out=180,in=-90] (c8);

		\draw[  very thick](c5) to [out=-90,in=15] (6,0.55);  
		
		\draw[  very thick](c6) to [out=-90,in=15] (7,0.55);

		\draw[  very thick](c7) to [out=-90,in=165] (8.5,0.55);

			\draw[  very thick](c3) to [out=-90,in=0] (b5) to [out=180,in=-90] (c2); 
		\draw[  very thick](c4) to [out=-90,in=0] (d5) to [out=180,in=-90] (c1); 
		
		\fill (d5)  circle (3pt);

		\path (4,-1) coordinate (origin); 
		\path (origin)--++(0.5,0.5) coordinate (origin2);  
		\draw(origin2)--++(0:5); 
		\foreach \i in {1,2,3,4,5,...,9}
		{
			\path (origin2)--++(0:0.5*\i) coordinate (a\i); 
			\path (origin2)--++(0:0.5*\i)--++(90:0.08) coordinate (c\i); 
 }
		
		\foreach \i in {1,2,3,4,5,...,19}
		{
			\path (origin2)--++(0:0.25*\i) --++(90:0.5) coordinate (b\i); 
			\path (origin2)--++(0:0.25*\i) --++(90:0.65) coordinate (d\i); 
		}
\draw		(a1)  circle (2pt);
		\path(a2) --++(90:-0.175) node  {  $  \up   $} ;
		\path(a3) --++(90:-0.175) node  {  $  \up   $} ;
		\path(a4) --++(90:-0.175) node  {  $ \up   $} ;
		\path(a5) --++(90:-0.175) node  {  $  \up  $} ;
		
		\path(a6) --++(90:0.175) node  {  $ \down     $} ;
		\path(a7) --++(90:0.175) node  {  $  \down     $} ;
		\path(a8) --++(90:0.175) node  {  $ \down     $} ;
		\path(a9) --++(90:0.175) node  {  $  \down     $} ;
		
		\draw[  very thick](c6) to [out=90,in=0] (b11) to [out=180,in=90] (c5); 
		\draw[  very thick](c7) to [out=90,in=0] (d11) to [out=180,in=90] (c4); 
		
		\path(d11)--++(90:0.2) coordinate (e11); 
		\draw[  very thick](c8) to [out=90,in=0] (e11) to [out=180,in=90] (c3); 
		

		\draw[very   thick](c1) to [out=90,in=-170] (6,0.55); 
		\draw[very   thick](c2) to [out=90,in=-170] (7,0.55);   
		
		\draw[  very thick](c9) to [out=90,in=-15] (8.5,0.55);

		\draw[very   thick](c1) to [out=90,in=-170] (6,0.55)
		[midway] coordinate (X);
		\fill (X)  circle (3pt);     
		
		\draw[very   thick](c2) to [out=90,in=-170] (7,0.55)
		;
		
		\draw[  very thick](c9) to [out=90,in=-15] (8.5,0.55) ;

	\end{tikzpicture}
	$$ 	\caption{ Oriented  tangles of type $(C_{9},A_8)$, $(D_{10},A_9)$, and $(C_8,A_7)$ respectively  obtained by orienting  the diagrams from  \cref{typeAtiling2lkjs2}.  }
	\label{typeAtiling2lkjsdfghslkdfjghsldkXXX}
\end{figure}

\begin{defn}\label{parity}
We define the {\sf parity specialisation map}
$${^{A_{n-1}}C_n}\to{ ^{A_{n}}D_{n+1}}
\qquad
:\la \mapsto \overline{\la}$$by replacing the $\circ$ in $\la$  by either an $\up$ or $\down$ arrow in such a way that the resulting 
total number of $\up$ arrows is even.  
\end{defn}

We note that the parity specialisation  map is a bijection 
between our diagrammatic coset representatives, by definition.  

In \cref{results2}, we will see that whilst 
the matrices   of light leaves polynomials of types 
$(D_{n+1},A_n)$  and $(C_n,A_{n-1})$ are genuinely 
distinct, the underlying Kazhdan--Lusztig polynomials are the same (see also \cite{MR957071,MR3518556}).  The key to understanding this phenomenon will be the  
following:

\begin{prop}\label{CtoD}
There is a surjective
algebra homomorphism 
\[\varphi:{\rm TL}_{(C_n,A_{n-1})}^{\up\down}(q)\to
{\rm TL}_{(D_{n+1},A_{n})}^{\up\down}(q)\]
by setting $\varphi(\la d \mu)$ to be the oriented tangle obtained from 
$\overline\la d \overline\mu$ by removing the beads from all  decorated oriented strands.  
\end{prop}

\begin{proof}
Clearly the map is surjective.
We now verify that $\varphi$ preserves the relations. 
The closed loops relation is trivially preserved as it only involves the rightmost decoration on the loop, which is unchanged by the parity specialisation map.  
\Cref{matchy} is also trivially preserved.  
We now verify the bead relations of \cref{diagrelations}.  
We slightly abuse notation by setting $d$ 
 to be the  tangle consisting of a single   decorated  propagating strand and by setting
$u$ to be the  tangle consisting of a single  undecorated  propagating strand. 
We   check that 
\[
\varphi(\la d \mu d \nu)= \varphi\left(\begin{minipage}{1.55cm}
\begin{tikzpicture}
\draw[densely dotted,rounded corners] (-22pt,33pt) rectangle (22pt,22pt) node [midway] {$\la$};

\draw[densely dotted,rounded corners] (-22pt,0pt) rectangle (22pt,22pt);
 \draw[very thick] (0,-0pt)--++(90:22pt);
\fill[very thick] (0,11pt)
 circle (4pt);

\draw[densely dotted,rounded corners] (-22pt,-11pt) rectangle (22pt,0pt) node [midway] {$\mu$};

\draw[densely dotted,rounded corners] (-22pt,-11pt) rectangle (22pt,-33pt);
 \draw[very thick] (0,-33pt)--++(90:22pt);
\fill[very thick] (0,-22pt)
 circle (4pt);

 \draw[densely dotted,rounded corners] (-22pt,-33pt) rectangle (22pt,-44pt) node [midway] {$\nu$};


 \end{tikzpicture} 
\end{minipage}\right)=
\varphi\left(\begin{minipage}{1.55cm}
\begin{tikzpicture}
\draw[densely dotted,rounded corners] (-22pt,33pt) rectangle (22pt,22pt) node [midway] {$\la$};

\draw[densely dotted,rounded corners] (-22pt,-33pt) rectangle (22pt,22pt);
 \draw[very thick] (0,-0pt)--++(90:22pt);
 

 \draw[very thick] (0,-33pt)--++(90:33pt);
\fill[very thick] (0,-5.5pt)
 circle (4pt);

 \draw[densely dotted,rounded corners] (-22pt,-33pt) rectangle (22pt,-44pt) node [midway] {$\mu$};

 \end{tikzpicture} 
\end{minipage}\right)
\varphi\left(\begin{minipage}{1.55cm}
\begin{tikzpicture}
\draw[densely dotted,rounded corners] (-22pt,33pt) rectangle (22pt,22pt) node [midway] {$\mu$};

\draw[densely dotted,rounded corners] (-22pt,-33pt) rectangle (22pt,22pt);
 \draw[very thick] (0,-0pt)--++(90:22pt);
 

 \draw[very thick] (0,-33pt)--++(90:33pt);
\fill[very thick] (0,-5.5pt)
 circle (4pt);

 \draw[densely dotted,rounded corners] (-22pt,-33pt) rectangle (22pt,-44pt) node [midway] {$\nu$};

 \end{tikzpicture} 
\end{minipage}\right)
=
\varphi(\la d \mu) \varphi(\mu d \nu) \]
by breaking  this up into three cases depending on 
  $\la,\mu,\nu \in \{\circ,\up,\down\}$.  
 Note that $\overline{\la},\overline{\mu},\overline{\nu}$ depend on the entirety of the coset diagram, not just the label of the strand that  we are considering (as they are calculated by the overall parity).  The three cases are as follows:

 \begin{itemize}[leftmargin=*]
\item[$(i)$] If $\overline \la=\overline \mu=\overline \nu  $ then $\overline \la d \overline \mu$, $\overline \mu d\overline \nu$, and 
 $\overline \la d\overline \nu$
   are    oriented strands and so 
\[
\varphi(\la d \mu d \nu)=
\varphi(\la d    \nu) 
= \overline{\la} u  \overline{\nu}
= (\overline{\la} u  \overline{\mu})
(  \overline{\mu} u  \overline{\nu})=
\varphi(\la d \mu) \varphi(\mu d \nu) .\]
 \item[$(ii)$] If $\overline \la=\overline \mu\neq \overline \nu  $ then $\overline \la d \overline \mu$ is oriented and
 $\overline \la d\overline \nu$,  $\overline \mu d\overline \nu$  are unoriented   strands  and so 
\[
\varphi(\la d \mu d \nu)=
\varphi(\la d    \nu) 
= \overline{\la} d  \overline{\nu}
= (\overline{\la} u  \overline{\mu})
(  \overline{\mu} d  \overline{\nu})=
\varphi(\la d \mu) \varphi(\mu d \nu) .\]

\item[$(iii)$]
 If $\overline \la\neq \overline \mu\neq \overline \nu  $ then 
   $\overline \la d\overline \nu$ is oriented   
and    $\overline \la d \overline \mu$, 
 $\overline \mu d\overline \nu$ 
 are unoriented
      strands  and so 
\begin{align*}
\varphi(\la d \mu d \nu)=
\varphi(\la d    \nu) 
= \overline{\la} u  \overline{\nu}
= (\overline{\la} d  \overline{\mu})
(  \overline{\mu} d  \overline{\nu})=
\varphi(\la d \mu) \varphi(\mu d \nu) .  
\end{align*}
\end{itemize} 
The result follows.
\end{proof}

We now state the main result of this section, which establishes that the 
abstract and diagrammatically defined  algebras are, in fact, isomorphic. Recall, from \cref{rmkdiagTL} that for any edge $\la \xrightarrow{i} \mu$ in the graph $\widehat{\mathcal{G}}_{(W,P)}$ we have an element $\la \diag_i \mu\in {\rm TL}^{\up \down}_{(W,P)}(q)$. More generally, for any path $T =  \lambda_1 \xrightarrow{i_1} \lambda_2 \xrightarrow{i_2}  \ldots \la_{i_{k-1}}\xrightarrow{i_{k-1}} \lambda_k$ on the graph $\widehat{\mathcal{G}}_{(W,P)}$ we can form the product 
\[\diag_{\SSTT} := \la_1 \diag_{i_1} \la_2 \diag_{i_2} \ldots \la_{i_{k-1}}\diag_{i_{k-1}} \la_k \in {\rm TL}^{\up \down}_{(W,P)}(q).\]

\begin{thm}\label{orientediso} Let $(W,P)=(A_n, A_k\times A_{n-k-1}), (C_n, A_{n-1})$ or $(D_n, A_{n-1})$. 
There is an isomorphism  of $\mathbb{Z}[q,q^{-1}]$-algebras 
\[\vartheta: {\rm TL}_{(W,P)}(q)  \longrightarrow {\rm TL}^{{\up}{\down}}_{(W,P)}(q)\]
defined by $\vartheta({\sf 1}_\la) = \lambda {\sf 1}\lambda$, and $\vartheta ({\sf 1}_\la E_i {\sf 1}_\mu) = \lambda {\sf e}_i \mu$. In particular, for any $\SSTT\in {\rm Path}_{(W,P)}$ we have $\vartheta(E_\SSTT) = \diag_\SSTT$. 
\end{thm}

\begin{rmk}
Note that we have $\vartheta (E_i) =\vartheta(\sum_\la ({\sf 1}_\la E_i {\sf 1}_\la + \la E_i ({\sf 1}_{\la s_i}))) = \sum_\la (\la \diag_i \la +\la \diag_i (\la s_i))$
where the sum is over all $\la \in {^PW}$ with $\la s_i \in {^PW}$.
\end{rmk}

We will prove this theorem in the rest of this section.

\begin{prop}
The map $\vartheta$ is a $\mathbb{Z}[q,q^{-1}]$-algebra homomorphism.
\end{prop}

\begin{proof}
We need to check that the relations (\ref{idempotentrel})--(\ref{mequals4}) are preserved under $\vartheta$. The two leftmost relations in (\ref{idempotentrel}) are clear by definition of the multiplication in ${\rm TL}^{{\up}{\down}}_{(W,P)}(q)$.  The two rightmost ones follow from the description of the graph $\mathcal{G}_{(W,P)}$ given in \cref{section4}. 
Relation (\ref{Esquare}) is satisfied by the closed loop relation noting that $\down \down < \up \up$ and $\up \down < \down \up$. The leftmost relation in  (\ref{commutation}) holds as the corresponding tangles are isotopic for $\{i,j\}\neq \{1, 1''\}$. Note that when $\{i,j\} = \{1, 1''\}$ in type $(D_n, A_{n-1})$, we have that  $E_1{\sf 1}_\la E_{1''} = E_{1''}{\sf 1}_\la E_1= 0$  for all $\la\in {^PW}$, so $E_1E_{1''} = E_{1''}E_1 = 0$ and there is nothing to check. The rightmost relation in (\ref{commutation}) is also satisfied using the fact that  the corresponding tangles are isotopic and the double bead relation when $\{i,j\} = \{1'', 2\}$ in type $(D_n, A_{n-1})$. Relation (\ref{mequals4}) only applies when $i=2$ and $j=1'$ in type $(C_n, A_{n-1})$. Now, it is easy to see that it holds in the diagrammatic algebra using isotopy, the double bead relations in type $C_n$ and the non-simply laced bead relation.  
\end{proof}

To show that $\vartheta$ is an isomorphism, we will need to show that every oriented Temperley--Lieb diagram can be written as a product of generators. We start with the non-oriented diagrams.

\subsection{Closed and iterative constructions of Temperley--Lieb diagrams}\label{productofgenerators}

For $W=A_{n-1}, C_{n-1}$ or $D_n$, we consider a tiling of a vertical strip of the plane with square tiles labelled by the simple reflections $s\in S_{W}$ as illustrated in Figure \ref{DiagramsforTLtiling}.

 \begin{figure}[ht!]
$$\hspace{-0.4cm} 
\begin{tikzpicture} [scale=0.67]

 \path(0,0)--++(-90:1.5) coordinate(origin3);
\path(origin3)--++(-90:1)  ++(135:0.5)   ++(-135:0.5)--++(180:0.2)  coordinate(corner1);
	\path(origin3)  ++(45:6 )   ++(135:7)--++( 90:1.7)--++(180:0.2)  coordinate(corner2);
	\path(origin3)--++(-90:1)  ++(45:4.5)   ++(-45:4.5) --++( 0:0.2)   coordinate(corner4);
	\path(origin3)   ++(135:2)   ++(45:11)--++( 90:1.7)--++( 0:0.2) coordinate(corner3)    ;

 	\clip (corner1)--(corner2)--(corner3)--(corner4);

\path(0,-1)--++(45:2)--++(-45:2)--++(90:-0.75) coordinate (origin);

\path (origin)--++(135:0.5)--++(-135:0.5) coordinate (minus1);

\path (minus1)--++(135:0.5)--++(-135:0.5) coordinate (minus2);

\path (minus2)--++(135:0.5)--++(-135:0.5) coordinate (minus3);

\path (minus3)--++(135:0.5)--++(-135:0.5)--++(+90:0.4) coordinate (minus4);

\path (minus3)--++(135:0.5)--++(-135:0.5)--++(-90:0.4) coordinate (minus4a);

\path (origin)--++(45:0.5)--++(-45:0.5) coordinate (plus1);

\path (plus1)--++(45:0.5)--++(-45:0.5)  coordinate (plus3);

 \path (plus3)--++(45:0.5)--++(-45:0.5)   coordinate (plus4);

\path (plus4)--++(45:0.5)--++(-45:0.5)   coordinate (plus2);

\path(plus2)--++(45:0.4) coordinate (minus1NE);
\path(plus2)--++(-45:0.4) coordinate (minus1SE);

\path(minus3)--++(135:0.2)-++(45:0.2)  coordinate (minus1NW);
\path(minus3)--++(-135:0.2)--++(-45:0.2) coordinate (minus1SW);

\path(minus2)--++(135:0.4) coordinate (start);

\path(minus4a)--++(180:0.4) coordinate (minus4aNE);
\path(minus4a)--++(-90:0.4) coordinate (minus4aSE);

\path (start)--(minus1NE)--(minus1SE)--(minus1SW)--(minus4aSE)--(minus4aNE)--(minus1NW)--(start);

\path(0,-1)--++(45:2)--++(-45:2)--++(90:-0.5) coordinate (origin);

\path (origin)--++(135:0.5)--++(-135:0.5) coordinate (minus1);

\path (minus1)--++(135:0.5)--++(-135:0.5) coordinate (minus2);

\path (minus2)--++(135:0.5)--++(-135:0.5) coordinate (minus3);

\path (minus3)--++(135:0.5)--++(-135:0.5) coordinate (minus4);

\path (origin)--++(45:0.5)--++(-45:0.5) coordinate (plus1);

\path (plus1)	--++(45:0.5)--++(-45:0.5)		 coordinate (plus2);

\path (plus2) --++(45:0.5)--++(-45:0.5) coordinate (plus3);

\path (plus3) --++(45:0.5)--++(-45:0.5) coordinate (plus4);

\draw[very thick](plus4)--(minus3);

 \draw[very thick](minus3) --++(180:0.7);
%

\path(plus4)--++(45:0.4) coordinate (minus1NE);
\path(plus4)--++(-45:0.4) coordinate (minus1SE);

\path(minus3)--++(135:0.4) coordinate (minus1NW);
\path(minus3)--++(-135:0.4) coordinate (minus1SW);

\path(minus2)--++(135:0.4) coordinate (start);

\draw[very thick,  fill=pink](plus2) circle (4pt);
\draw[very thick,  fill=violet](plus3) circle (4pt);
\draw[very thick,  fill=brown](plus4) circle (4pt);

\draw[very thick,fill=lime!80!black](origin) circle (4pt);

\draw[very thick,  fill=darkgreen](minus1) circle (4pt);

\draw[very thick,  fill=orange](minus2) circle (4pt);

\draw[very thick,  fill=gray](minus3) circle (4pt);

\draw[very thick,  fill=magenta](minus4) circle (4pt);

\draw[very thick,  fill=cyan](plus1) circle (4pt);

\path(minus2)--++(-135:1) coordinate (grayguy);
\draw[very thick,  fill=white, white](grayguy) circle (4pt);
\path(grayguy)--++(-135:0.3);

\path (0,0)--++(-90:0.5) coordinate (origin2);

\begin{scope}
	
  	\clip(corner1)--(corner2)--(corner3)--(corner4);

\end{scope}

\draw[very thick,densely dotted ](origin2)--++(-45:0.4);\draw[very thick,densely dotted ](origin2)--++(-135:0.4);
	\draw[very thick](origin2)--++(45:1)--++(135:1)--++(45:1)--++(135:1)--++(45:1)--++(135:1)--++(45:1)--++(135:1)--++(45:1)--++(135:1)--++(45:1)--++(135:1)
 coordinate (origin3) 
--++(-135:1)--++(-45:1)--++(-135:1)--++(-45:1)--++(-135:1)--++(-45:1)--++(-135:1)--++(-45:1)--++(-135:1)--++(-45:1)--++(-135:1)--++(-45:1)	;

\draw[very thick,densely dotted ](origin3)--++(45:0.4);\draw[very thick,densely dotted ](origin3)--++(135:0.4);

	\path (origin2)--++(45:0.5)--++(135:0.5)
	node {$\color{magenta}1$}
	--++(45:1)--++(135:1)
	node {$\color{magenta}1$}
	--++(45:1)--++(135:1)
	node {$\color{magenta}1$}
	--++(45:1)--++(135:1)
	node {$\color{magenta}1$}
	--++(45:1)--++(135:1)
	node {$\color{magenta}1$}
	--++(45:1)--++(135:1)
	node {$\color{magenta}1$};

	\path (origin2)--++(45:1.5)--++(135:0.5)
	node {$\color{gray}2$}
	--++(45:1)--++(135:1)
	node {$\color{gray}2$}
	--++(45:1)--++(135:1)
	node {$\color{gray}2$}
	--++(45:1)--++(135:1)
	node {$\color{gray}2$}
	--++(45:1)--++(135:1)
	node {$\color{gray}2$}
 ;

\path(origin2)--++(45:1)--++(-45:1) coordinate (origin2);
\draw[very thick,densely dotted ](origin2)--++(-45:0.4);\draw[very thick,densely dotted ](origin2)--++(-135:0.4);
	\draw[very thick](origin2)--++(45:1)--++(135:1)--++(45:1)--++(135:1)--++(45:1)--++(135:1)--++(45:1)--++(135:1)--++(45:1)--++(135:1)--++(45:1)--++(135:1)
 coordinate (origin3) 
--++(-135:1)--++(-45:1)--++(-135:1)--++(-45:1)--++(-135:1)--++(-45:1)--++(-135:1)--++(-45:1)--++(-135:1)--++(-45:1)--++(-135:1)--++(-45:1)	;

\draw[very thick,densely dotted ](origin3)--++(45:0.4);\draw[very thick,densely dotted ](origin3)--++(135:0.4);

	\path (origin2)--++(45:0.5)--++(135:0.5)
	node {$\color{orange}3$}
	--++(45:1)--++(135:1)
	node {$\color{orange}3$}
	--++(45:1)--++(135:1)
	node {$\color{orange}3$}
	--++(45:1)--++(135:1)
	node {$\color{orange}3$}
	--++(45:1)--++(135:1)
	node {$\color{orange}3$}
	--++(45:1)--++(135:1)
	node {$\color{orange}3$};

	\path (origin2)--++(45:1.5)--++(135:0.5)
	node {$\color{darkgreen}4$}
	--++(45:1)--++(135:1)
	node {$\color{darkgreen}4$}
	--++(45:1)--++(135:1)
	node {$\color{darkgreen}4$}
	--++(45:1)--++(135:1)
	node {$\color{darkgreen}4$}
	--++(45:1)--++(135:1)
	node {$\color{darkgreen}4$}
 ;

\path(origin2)--++(45:1)--++(-45:1) coordinate (origin2);
\draw[very thick,densely dotted ](origin2)--++(-45:0.4);\draw[very thick,densely dotted ](origin2)--++(-135:0.4);
	\draw[very thick](origin2)--++(45:1)--++(135:1)--++(45:1)--++(135:1)--++(45:1)--++(135:1)--++(45:1)--++(135:1)--++(45:1)--++(135:1)--++(45:1)--++(135:1)
 coordinate (origin3) 
--++(-135:1)--++(-45:1)--++(-135:1)--++(-45:1)--++(-135:1)--++(-45:1)--++(-135:1)--++(-45:1)--++(-135:1)--++(-45:1)--++(-135:1)--++(-45:1)	;
\draw[very thick,densely dotted ](origin3)--++(45:0.4);\draw[very thick,densely dotted ](origin3)--++(135:0.4);

	\path (origin2)--++(45:0.5)--++(135:0.5)
	node {$\color{green}5$}
	--++(45:1)--++(135:1)
	node {$\color{green}5$}
	--++(45:1)--++(135:1)
	node {$\color{green}5$}
	--++(45:1)--++(135:1)
	node {$\color{green}5$}
	--++(45:1)--++(135:1)
	node {$\color{green}5$}
	--++(45:1)--++(135:1)
	node {$\color{green}5$};

	\path (origin2)--++(45:1.5)--++(135:0.5)
	node {$\color{cyan}6$}
	--++(45:1)--++(135:1)
	node {$\color{cyan}6$}
	--++(45:1)--++(135:1)
	node {$\color{cyan}6$}
	--++(45:1)--++(135:1)
	node {$\color{cyan}6$}
	--++(45:1)--++(135:1)
	node {$\color{cyan}6$}
 ;

\path(origin2)--++(45:1)--++(-45:1) coordinate (origin2);
\draw[very thick,densely dotted ](origin2)--++(-45:0.4);\draw[very thick,densely dotted ](origin2)--++(-135:0.4);
	\draw[very thick](origin2)--++(45:1)--++(135:1)--++(45:1)--++(135:1)--++(45:1)--++(135:1)--++(45:1)--++(135:1)--++(45:1)--++(135:1)--++(45:1)--++(135:1)
 coordinate (origin3) 
--++(-135:1)--++(-45:1)--++(-135:1)--++(-45:1)--++(-135:1)--++(-45:1)--++(-135:1)--++(-45:1)--++(-135:1)--++(-45:1)--++(-135:1)--++(-45:1)	;
\draw[very thick,densely dotted ](origin3)--++(45:0.4);\draw[very thick,densely dotted ](origin3)--++(135:0.4);

	\path (origin2)--++(45:0.5)--++(135:0.5)
	node {$\color{pink}7$}
	--++(45:1)--++(135:1)
	node {$\color{pink}7$}
	--++(45:1)--++(135:1)
	node {$\color{pink}7$}
	--++(45:1)--++(135:1)
	node {$\color{pink}7$}
	--++(45:1)--++(135:1)
	node {$\color{pink}7$}
	--++(45:1)--++(135:1)
	node {$\color{pink}7$};

	\path (origin2)--++(45:1.5)--++(135:0.5)
	node {$\color{violet}8$}
	--++(45:1)--++(135:1)
	node {$\color{violet}8$}
	--++(45:1)--++(135:1)
	node {$\color{violet}8$}
	--++(45:1)--++(135:1)
	node {$\color{violet}8$}
	--++(45:1)--++(135:1)
	node {$\color{violet}8$}
 ;

\path(origin2)--++(45:1)--++(-45:1) coordinate (origin2);
\draw[very thick,densely dotted ](origin2)--++(-45:0.4);\draw[very thick,densely dotted ](origin2)--++(-135:0.4);
	\draw[very thick](origin2)--++(45:1)--++(135:1)--++(45:1)--++(135:1)--++(45:1)--++(135:1)--++(45:1)--++(135:1)--++(45:1)--++(135:1)--++(45:1)--++(135:1)
 coordinate (origin3) 
--++(-135:1)--++(-45:1)--++(-135:1)--++(-45:1)--++(-135:1)--++(-45:1)--++(-135:1)--++(-45:1)--++(-135:1)--++(-45:1)--++(-135:1)--++(-45:1)	;
\draw[very thick,densely dotted ](origin3)--++(45:0.4);\draw[very thick,densely dotted ](origin3)--++(135:0.4);

	\path (origin2)--++(45:0.5)--++(135:0.5)
	node {$\color{brown}9$}
	--++(45:1)--++(135:1)
	node {$\color{brown}9$}
	--++(45:1)--++(135:1)
	node {$\color{brown}9$}
	--++(45:1)--++(135:1)
	node {$\color{brown}9$}
	--++(45:1)--++(135:1)
	node {$\color{brown}9$}
	--++(45:1)--++(135:1)
	node {$\color{brown}9$};

\end{tikzpicture}
  \begin{tikzpicture} [scale=0.67]

\path(0,0) coordinate(origin3);

 \path(0,0)--++(-90:1.5) coordinate(origin3);
\path(origin3) --++(-90:1) ++(135:0.5)   ++(-135:0.5)--++(180:0.2)  coordinate(corner1);
	\path(origin3)  ++(45:6 )   ++(135:7)--++( 90:1.7)--++(180:0.2)  coordinate(corner2);
	\path(origin3) --++(-90:1) ++(45:4.5)   ++(-45:4.5) --++( 0:0.2)   coordinate(corner4);
	\path(origin3)  ++(135:2)   ++(45:11)--++( 90:1.7)--++( 0:0.2) coordinate(corner3)    ;

 	\clip (corner1)--(corner2)--(corner3)--(corner4);
 
 \path (0,0) coordinate (origin2); 
 
\path(0,0) coordinate(origin3);
\path(origin3)  ++(135:0.5)   ++(-135:0.5) coordinate(corner1);

\path(origin3)  ++(45:5)   ++(135:7) coordinate(corner2);
\path(origin3)  ++(45:3.5)   ++(-45:3.5) coordinate(corner4);
\path(origin3)  ++(135:2.5)   ++(45:9.5) coordinate(corner3);

\path(0,-1)--++(45:2)--++(-45:2)--++(90:-0.75) coordinate (origin);

\path (origin)--++(135:0.5)--++(-135:0.5) coordinate (minus1);

\path (minus1)--++(135:0.5)--++(-135:0.5) coordinate (minus2);

\path (minus2)--++(135:0.5)--++(-135:0.5) coordinate (minus3);

\path (minus3)--++(135:0.5)--++(-135:0.5)--++(+90:0.4) coordinate (minus4);

\path (minus3)--++(135:0.5)--++(-135:0.5)--++(-90:0.4) coordinate (minus4a);

\path (origin)--++(45:0.5)--++(-45:0.5) coordinate (plus1);

\path (plus1)--++(45:0.5)--++(-45:0.5)  coordinate (plus3);

 \path (plus3)--++(45:0.5)--++(-45:0.5)   coordinate (plus4);

\path (plus4)--++(45:0.5)--++(-45:0.5)   coordinate (plus2);

\draw[very thick](plus2)--(minus3);

\draw[very thick](minus3)-- (minus4);
\draw[very thick](minus3)--(minus4a);

\path(plus2)--++(45:0.4) coordinate (minus1NE);
\path(plus2)--++(-45:0.4) coordinate (minus1SE);

\path(minus3)--++(135:0.2)-++(45:0.2)  coordinate (minus1NW);
\path(minus3)--++(-135:0.2)--++(-45:0.2) coordinate (minus1SW);

\path(minus2)--++(135:0.4) coordinate (start);

\path(minus4a)--++(180:0.4) coordinate (minus4aNE);
\path(minus4a)--++(-90:0.4) coordinate (minus4aSE);


\draw[very thick,  fill=pink](plus3) circle (4pt);
\draw[very thick,  fill=violet](plus4) circle (4pt);
\draw[very thick,  fill=brown](plus2) circle (4pt);

\draw[very thick,fill=lime](origin) circle (4pt);

\draw[very thick,  fill=darkgreen](minus1) circle (4pt);

\draw[very thick,  fill=orange](minus2) circle (4pt);

\draw[very thick,  fill=gray](minus3) circle (4pt);

\draw[very thick,  fill= black, ](minus4a) circle (4pt);

\draw[very thick,  fill=magenta](minus4) circle (4pt);

\draw[very thick,  fill=cyan](plus1) circle (4pt);

\path(minus2)--++(-135:1) coordinate (grayguy);
 \draw[very thick,  fill=white, white](grayguy) circle (4pt);
\path(grayguy)--++(-135:0.3);

\path (0,0)--++(-90:0.5) coordinate (origin2);

\begin{scope}
	
  	\clip(corner1)--(corner2)--(corner3)--(corner4);

\end{scope}

\draw[very thick,densely dotted ](origin2)--++(-45:0.4);\draw[very thick,densely dotted ](origin2)--++(-135:0.4);
	\draw[very thick](origin2)--++(45:1)--++(135:1)--++(45:1)--++(135:1)--++(45:1)--++(135:1)--++(45:1)--++(135:1)--++(45:1)--++(135:1)--++(45:1)--++(135:1)
 coordinate (origin3) 
--++(-135:1)--++(-45:1)--++(-135:1)--++(-45:1)--++(-135:1)--++(-45:1)--++(-135:1)--++(-45:1)--++(-135:1)--++(-45:1)--++(-135:1)--++(-45:1)	;

\draw[very thick,densely dotted ](origin3)--++(45:0.4);\draw[very thick,densely dotted ](origin3)--++(135:0.4);

	\path (origin2)--++(45:0.5)--++(135:0.5)
	node {$\color{magenta}1''$}
	--++(45:1)--++(135:1)
	node {$ 1$ }
	--++(45:1)--++(135:1)
	node {$\color{magenta}1''$}
	--++(45:1)--++(135:1)
	node {$ 1$}
	--++(45:1)--++(135:1)
	node {$\color{magenta}1''$}
	--++(45:1)--++(135:1)
	node {$ 1$};

	\path (origin2)--++(45:1.5)--++(135:0.5)
	node {$\color{gray}2$}
	--++(45:1)--++(135:1)
	node {$\color{gray}2$}
	--++(45:1)--++(135:1)
	node {$\color{gray}2$}
	--++(45:1)--++(135:1)
	node {$\color{gray}2$}
	--++(45:1)--++(135:1)
	node {$\color{gray}2$}
 ;

\path(origin2)--++(45:1)--++(-45:1) coordinate (origin2);
\draw[very thick,densely dotted ](origin2)--++(-45:0.4);\draw[very thick,densely dotted ](origin2)--++(-135:0.4);
	\draw[very thick](origin2)--++(45:1)--++(135:1)--++(45:1)--++(135:1)--++(45:1)--++(135:1)--++(45:1)--++(135:1)--++(45:1)--++(135:1)--++(45:1)--++(135:1)
 coordinate (origin3) 
--++(-135:1)--++(-45:1)--++(-135:1)--++(-45:1)--++(-135:1)--++(-45:1)--++(-135:1)--++(-45:1)--++(-135:1)--++(-45:1)--++(-135:1)--++(-45:1)	;

\draw[very thick,densely dotted ](origin3)--++(45:0.4);\draw[very thick,densely dotted ](origin3)--++(135:0.4);

	\path (origin2)--++(45:0.5)--++(135:0.5)
	node {$\color{orange}3$}
	--++(45:1)--++(135:1)
	node {$\color{orange}3$}
	--++(45:1)--++(135:1)
	node {$\color{orange}3$}
	--++(45:1)--++(135:1)
	node {$\color{orange}3$}
	--++(45:1)--++(135:1)
	node {$\color{orange}3$}
	--++(45:1)--++(135:1)
	node {$\color{orange}3$};

	\path (origin2)--++(45:1.5)--++(135:0.5)
	node {$\color{darkgreen}4$}
	--++(45:1)--++(135:1)
	node {$\color{darkgreen}4$}
	--++(45:1)--++(135:1)
	node {$\color{darkgreen}4$}
	--++(45:1)--++(135:1)
	node {$\color{darkgreen}4$}
	--++(45:1)--++(135:1)
	node {$\color{darkgreen}4$}
 ;

\path(origin2)--++(45:1)--++(-45:1) coordinate (origin2);
\draw[very thick,densely dotted ](origin2)--++(-45:0.4);\draw[very thick,densely dotted ](origin2)--++(-135:0.4);
	\draw[very thick](origin2)--++(45:1)--++(135:1)--++(45:1)--++(135:1)--++(45:1)--++(135:1)--++(45:1)--++(135:1)--++(45:1)--++(135:1)--++(45:1)--++(135:1)
 coordinate (origin3) 
--++(-135:1)--++(-45:1)--++(-135:1)--++(-45:1)--++(-135:1)--++(-45:1)--++(-135:1)--++(-45:1)--++(-135:1)--++(-45:1)--++(-135:1)--++(-45:1)	;
\draw[very thick,densely dotted ](origin3)--++(45:0.4);\draw[very thick,densely dotted ](origin3)--++(135:0.4);

	\path (origin2)--++(45:0.5)--++(135:0.5)
	node {$\color{green}5$}
	--++(45:1)--++(135:1)
	node {$\color{green}5$}
	--++(45:1)--++(135:1)
	node {$\color{green}5$}
	--++(45:1)--++(135:1)
	node {$\color{green}5$}
	--++(45:1)--++(135:1)
	node {$\color{green}5$}
	--++(45:1)--++(135:1)
	node {$\color{green}5$};

	\path (origin2)--++(45:1.5)--++(135:0.5)
	node {$\color{cyan}6$}
	--++(45:1)--++(135:1)
	node {$\color{cyan}6$}
	--++(45:1)--++(135:1)
	node {$\color{cyan}6$}
	--++(45:1)--++(135:1)
	node {$\color{cyan}6$}
	--++(45:1)--++(135:1)
	node {$\color{cyan}6$}
 ;

\path(origin2)--++(45:1)--++(-45:1) coordinate (origin2);
\draw[very thick,densely dotted ](origin2)--++(-45:0.4);\draw[very thick,densely dotted ](origin2)--++(-135:0.4);
	\draw[very thick](origin2)--++(45:1)--++(135:1)--++(45:1)--++(135:1)--++(45:1)--++(135:1)--++(45:1)--++(135:1)--++(45:1)--++(135:1)--++(45:1)--++(135:1)
 coordinate (origin3) 
--++(-135:1)--++(-45:1)--++(-135:1)--++(-45:1)--++(-135:1)--++(-45:1)--++(-135:1)--++(-45:1)--++(-135:1)--++(-45:1)--++(-135:1)--++(-45:1)	;
\draw[very thick,densely dotted ](origin3)--++(45:0.4);\draw[very thick,densely dotted ](origin3)--++(135:0.4);

	\path (origin2)--++(45:0.5)--++(135:0.5)
	node {$\color{pink}7$}
	--++(45:1)--++(135:1)
	node {$\color{pink}7$}
	--++(45:1)--++(135:1)
	node {$\color{pink}7$}
	--++(45:1)--++(135:1)
	node {$\color{pink}7$}
	--++(45:1)--++(135:1)
	node {$\color{pink}7$}
	--++(45:1)--++(135:1)
	node {$\color{pink}7$};

	\path (origin2)--++(45:1.5)--++(135:0.5)
	node {$\color{violet}8$}
	--++(45:1)--++(135:1)
	node {$\color{violet}8$}
	--++(45:1)--++(135:1)
	node {$\color{violet}8$}
	--++(45:1)--++(135:1)
	node {$\color{violet}8$}
	--++(45:1)--++(135:1)
	node {$\color{violet}8$}
 ;

\path(origin2)--++(45:1)--++(-45:1) coordinate (origin2);
\draw[very thick,densely dotted ](origin2)--++(-45:0.4);\draw[very thick,densely dotted ](origin2)--++(-135:0.4);
	\draw[very thick](origin2)--++(45:1)--++(135:1)--++(45:1)--++(135:1)--++(45:1)--++(135:1)--++(45:1)--++(135:1)--++(45:1)--++(135:1)--++(45:1)--++(135:1)
 coordinate (origin3) 
--++(-135:1)--++(-45:1)--++(-135:1)--++(-45:1)--++(-135:1)--++(-45:1)--++(-135:1)--++(-45:1)--++(-135:1)--++(-45:1)--++(-135:1)--++(-45:1)	;
\draw[very thick,densely dotted ](origin3)--++(45:0.4);\draw[very thick,densely dotted ](origin3)--++(135:0.4);

	\path (origin2)--++(45:0.5)--++(135:0.5)
	node {$\color{brown}9$}
	--++(45:1)--++(135:1)
	node {$\color{brown}9$}
	--++(45:1)--++(135:1)
	node {$\color{brown}9$}
	--++(45:1)--++(135:1)
	node {$\color{brown}9$}
	--++(45:1)--++(135:1)
	node {$\color{brown}9$}
	--++(45:1)--++(135:1)
	node {$\color{brown}9$};

\end{tikzpicture}  
\begin{tikzpicture} [scale=0.67]

 \path(0,0)--++(-90:1.5) coordinate(origin3);
\path(origin3)--++(-90:1)  ++(135:0.5)   ++(-135:0.5)--++(180:0.2)  coordinate(corner1);
	\path(origin3)  ++(45:6 )   ++(135:7)--++( 90:1.7)--++(180:0.2)  coordinate(corner2);
	\path(origin3)--++(-90:1)  ++(45:4.5)   ++(-45:4.5) --++( 0:0.2)   coordinate(corner4);
	\path(origin3)   ++(135:2)   ++(45:11)--++( 90:1.7)--++( 0:0.2) coordinate(corner3)    ;

 	\clip (corner1)--(corner2)--(corner3)--(corner4);

\path(0,-1)--++(45:2)--++(-45:2)--++(90:-0.75) coordinate (origin);

\path (origin)--++(135:0.5)--++(-135:0.5) coordinate (minus1);

\path (minus1)--++(135:0.5)--++(-135:0.5) coordinate (minus2);

\path (minus2)--++(135:0.5)--++(-135:0.5) coordinate (minus3);

\path (minus3)--++(135:0.5)--++(-135:0.5)--++(+90:0.4) coordinate (minus4);

\path (minus3)--++(135:0.5)--++(-135:0.5)--++(-90:0.4) coordinate (minus4a);

\path (origin)--++(45:0.5)--++(-45:0.5) coordinate (plus1);

\path (plus1)--++(45:0.5)--++(-45:0.5)  coordinate (plus3);

 \path (plus3)--++(45:0.5)--++(-45:0.5)   coordinate (plus4);

\path (plus4)--++(45:0.5)--++(-45:0.5)   coordinate (plus2);

\path(plus2)--++(45:0.4) coordinate (minus1NE);
\path(plus2)--++(-45:0.4) coordinate (minus1SE);

\path(minus3)--++(135:0.2)-++(45:0.2)  coordinate (minus1NW);
\path(minus3)--++(-135:0.2)--++(-45:0.2) coordinate (minus1SW);

\path(minus2)--++(135:0.4) coordinate (start);

\path(minus4a)--++(180:0.4) coordinate (minus4aNE);
\path(minus4a)--++(-90:0.4) coordinate (minus4aSE);

\path (start)--(minus1NE)--(minus1SE)--(minus1SW)--(minus4aSE)--(minus4aNE)--(minus1NW)--(start);

\path(0,-1)--++(45:2)--++(-45:2)--++(90:-0.5) coordinate (origin);

\path (origin)--++(135:0.5)--++(-135:0.5) coordinate (minus1);

\path (minus1)--++(135:0.5)--++(-135:0.5) coordinate (minus2);

\path (minus2)--++(135:0.5)--++(-135:0.5) coordinate (minus3);

\path (minus3)--++(135:0.5)--++(-135:0.5) coordinate (minus4);

\path (origin)--++(45:0.5)--++(-45:0.5) coordinate (plus1);

\path (plus1)	--++(45:0.5)--++(-45:0.5)		 coordinate (plus2);

\path (plus2) --++(45:0.5)--++(-45:0.5) coordinate (plus3);

\path (plus3) --++(45:0.5)--++(-45:0.5) coordinate (plus4);

\draw[very thick](plus4)--(minus3);

\draw[very thick](minus3)--++(90:0.1)--++(180:0.7);
\draw[very thick](minus3)--++(-90:0.1)--++(180:0.7);

\path(plus4)--++(45:0.4) coordinate (minus1NE);
\path(plus4)--++(-45:0.4) coordinate (minus1SE);

\path(minus3)--++(135:0.4) coordinate (minus1NW);
\path(minus3)--++(-135:0.4) coordinate (minus1SW);

\path(minus2)--++(135:0.4) coordinate (start);

\draw[very thick,  fill=pink](plus2) circle (4pt);
\draw[very thick,  fill=violet](plus3) circle (4pt);
\draw[very thick,  fill=brown](plus4) circle (4pt);

\draw[very thick,fill=lime!80!black](origin) circle (4pt);

\draw[very thick,  fill=darkgreen](minus1) circle (4pt);

\draw[very thick,  fill=orange](minus2) circle (4pt);

\draw[very thick,  fill=gray](minus3) circle (4pt);

\draw[very thick,  fill=magenta](minus4) circle (4pt);

\draw[very thick,  fill=cyan](plus1) circle (4pt);

\path(minus2)--++(-135:1) coordinate (grayguy);
\draw[very thick,  fill=white, white](grayguy) circle (4pt);
\path(grayguy)--++(-135:0.3);

\path (0,0)--++(-90:0.5) coordinate (origin2);

\begin{scope}
	
  	\clip(corner1)--(corner2)--(corner3)--(corner4);

\end{scope}

\draw[very thick,densely dotted ](origin2)--++(-45:0.4);\draw[very thick,densely dotted ](origin2)--++(-135:0.4);
	\draw[very thick](origin2)--++(45:1)--++(135:1)--++(45:1)--++(135:1)--++(45:1)--++(135:1)--++(45:1)--++(135:1)--++(45:1)--++(135:1)--++(45:1)--++(135:1)
 coordinate (origin3) 
--++(-135:1)--++(-45:1)--++(-135:1)--++(-45:1)--++(-135:1)--++(-45:1)--++(-135:1)--++(-45:1)--++(-135:1)--++(-45:1)--++(-135:1)--++(-45:1)	;

\draw[very thick,densely dotted ](origin3)--++(45:0.4);\draw[very thick,densely dotted ](origin3)--++(135:0.4);

	\path (origin2)--++(45:0.5)--++(135:0.5)
	node {$\color{magenta}1'$}
	--++(45:1)--++(135:1)
	node {$\color{magenta}1'$}
	--++(45:1)--++(135:1)
	node {$\color{magenta}1'$}
	--++(45:1)--++(135:1)
	node {$\color{magenta}1'$}
	--++(45:1)--++(135:1)
	node {$\color{magenta}1'$}
	--++(45:1)--++(135:1)
	node {$\color{magenta}1'$};

	\path (origin2)--++(45:1.5)--++(135:0.5)
	node {$\color{gray}2$}
	--++(45:1)--++(135:1)
	node {$\color{gray}2$}
	--++(45:1)--++(135:1)
	node {$\color{gray}2$}
	--++(45:1)--++(135:1)
	node {$\color{gray}2$}
	--++(45:1)--++(135:1)
	node {$\color{gray}2$}
 ;

\path(origin2)--++(45:1)--++(-45:1) coordinate (origin2);
\draw[very thick,densely dotted ](origin2)--++(-45:0.4);\draw[very thick,densely dotted ](origin2)--++(-135:0.4);
	\draw[very thick](origin2)--++(45:1)--++(135:1)--++(45:1)--++(135:1)--++(45:1)--++(135:1)--++(45:1)--++(135:1)--++(45:1)--++(135:1)--++(45:1)--++(135:1)
 coordinate (origin3) 
--++(-135:1)--++(-45:1)--++(-135:1)--++(-45:1)--++(-135:1)--++(-45:1)--++(-135:1)--++(-45:1)--++(-135:1)--++(-45:1)--++(-135:1)--++(-45:1)	;

\draw[very thick,densely dotted ](origin3)--++(45:0.4);\draw[very thick,densely dotted ](origin3)--++(135:0.4);

	\path (origin2)--++(45:0.5)--++(135:0.5)
	node {$\color{orange}3$}
	--++(45:1)--++(135:1)
	node {$\color{orange}3$}
	--++(45:1)--++(135:1)
	node {$\color{orange}3$}
	--++(45:1)--++(135:1)
	node {$\color{orange}3$}
	--++(45:1)--++(135:1)
	node {$\color{orange}3$}
	--++(45:1)--++(135:1)
	node {$\color{orange}3$};

	\path (origin2)--++(45:1.5)--++(135:0.5)
	node {$\color{darkgreen}4$}
	--++(45:1)--++(135:1)
	node {$\color{darkgreen}4$}
	--++(45:1)--++(135:1)
	node {$\color{darkgreen}4$}
	--++(45:1)--++(135:1)
	node {$\color{darkgreen}4$}
	--++(45:1)--++(135:1)
	node {$\color{darkgreen}4$}
 ;

\path(origin2)--++(45:1)--++(-45:1) coordinate (origin2);
\draw[very thick,densely dotted ](origin2)--++(-45:0.4);\draw[very thick,densely dotted ](origin2)--++(-135:0.4);
	\draw[very thick](origin2)--++(45:1)--++(135:1)--++(45:1)--++(135:1)--++(45:1)--++(135:1)--++(45:1)--++(135:1)--++(45:1)--++(135:1)--++(45:1)--++(135:1)
 coordinate (origin3) 
--++(-135:1)--++(-45:1)--++(-135:1)--++(-45:1)--++(-135:1)--++(-45:1)--++(-135:1)--++(-45:1)--++(-135:1)--++(-45:1)--++(-135:1)--++(-45:1)	;
\draw[very thick,densely dotted ](origin3)--++(45:0.4);\draw[very thick,densely dotted ](origin3)--++(135:0.4);

	\path (origin2)--++(45:0.5)--++(135:0.5)
	node {$\color{green}5$}
	--++(45:1)--++(135:1)
	node {$\color{green}5$}
	--++(45:1)--++(135:1)
	node {$\color{green}5$}
	--++(45:1)--++(135:1)
	node {$\color{green}5$}
	--++(45:1)--++(135:1)
	node {$\color{green}5$}
	--++(45:1)--++(135:1)
	node {$\color{green}5$};

	\path (origin2)--++(45:1.5)--++(135:0.5)
	node {$\color{cyan}6$}
	--++(45:1)--++(135:1)
	node {$\color{cyan}6$}
	--++(45:1)--++(135:1)
	node {$\color{cyan}6$}
	--++(45:1)--++(135:1)
	node {$\color{cyan}6$}
	--++(45:1)--++(135:1)
	node {$\color{cyan}6$}
 ;

\path(origin2)--++(45:1)--++(-45:1) coordinate (origin2);
\draw[very thick,densely dotted ](origin2)--++(-45:0.4);\draw[very thick,densely dotted ](origin2)--++(-135:0.4);
	\draw[very thick](origin2)--++(45:1)--++(135:1)--++(45:1)--++(135:1)--++(45:1)--++(135:1)--++(45:1)--++(135:1)--++(45:1)--++(135:1)--++(45:1)--++(135:1)
 coordinate (origin3) 
--++(-135:1)--++(-45:1)--++(-135:1)--++(-45:1)--++(-135:1)--++(-45:1)--++(-135:1)--++(-45:1)--++(-135:1)--++(-45:1)--++(-135:1)--++(-45:1)	;
\draw[very thick,densely dotted ](origin3)--++(45:0.4);\draw[very thick,densely dotted ](origin3)--++(135:0.4);

	\path (origin2)--++(45:0.5)--++(135:0.5)
	node {$\color{pink}7$}
	--++(45:1)--++(135:1)
	node {$\color{pink}7$}
	--++(45:1)--++(135:1)
	node {$\color{pink}7$}
	--++(45:1)--++(135:1)
	node {$\color{pink}7$}
	--++(45:1)--++(135:1)
	node {$\color{pink}7$}
	--++(45:1)--++(135:1)
	node {$\color{pink}7$};

	\path (origin2)--++(45:1.5)--++(135:0.5)
	node {$\color{violet}8$}
	--++(45:1)--++(135:1)
	node {$\color{violet}8$}
	--++(45:1)--++(135:1)
	node {$\color{violet}8$}
	--++(45:1)--++(135:1)
	node {$\color{violet}8$}
	--++(45:1)--++(135:1)
	node {$\color{violet}8$}
 ;

\path(origin2)--++(45:1)--++(-45:1) coordinate (origin2);
\draw[very thick,densely dotted ](origin2)--++(-45:0.4);\draw[very thick,densely dotted ](origin2)--++(-135:0.4);
	\draw[very thick](origin2)--++(45:1)--++(135:1)--++(45:1)--++(135:1)--++(45:1)--++(135:1)--++(45:1)--++(135:1)--++(45:1)--++(135:1)--++(45:1)--++(135:1)
 coordinate (origin3) 
--++(-135:1)--++(-45:1)--++(-135:1)--++(-45:1)--++(-135:1)--++(-45:1)--++(-135:1)--++(-45:1)--++(-135:1)--++(-45:1)--++(-135:1)--++(-45:1)	;
\draw[very thick,densely dotted ](origin3)--++(45:0.4);\draw[very thick,densely dotted ](origin3)--++(135:0.4);

	\path (origin2)--++(45:0.5)--++(135:0.5)
	node {$\color{brown}9$}
	--++(45:1)--++(135:1)
	node {$\color{brown}9$}
	--++(45:1)--++(135:1)
	node {$\color{brown}9$}
	--++(45:1)--++(135:1)
	node {$\color{brown}9$}
	--++(45:1)--++(135:1)
	node {$\color{brown}9$}
	--++(45:1)--++(135:1)
	node {$\color{brown}9$};

\end{tikzpicture}
$$
\caption{The Temperley--Lieb algebra tiling pictures for types $A_9$, $D_{10}$, and $C_9$ respectively. }
\label{DiagramsforTLtiling}
\end{figure}

Now, for $(W,P)=(A_{n-1}, A_{k-1}\times A_{n-k-1})$, $(C_{n-1}, A_{n-2})$ or $(D_n, A_{n-1})$ and $d\in \mathbb{DT}[W,P]$ we have that $d$ has $n$ vertices on the northern edge, labelled by $1, \ldots , n$ (from left to right) and $n$ vertices on the southern edge, labelled by $1', \ldots , n'$ (from left to right).   We write $i'\geq j'$ whenever $i\geq j$.  
Each such diagram $d$ will correspond to a finite region on our tiling with boundary given by the path 
\[\pi(d) = (\pi(1), \pi(2), \ldots , \pi(n), \pi(n') , \pi((n-1)'), \ldots, \pi(2'), \pi(1'))\]
defined as follows: Start at the leftmost corner of a tile labelled by $1$, or $1'$ if $W=C_{n-1}$, then for $1\leq k \leq n$ we have 
\[\pi(k) = \left\{ \begin{array}{ll} {\rm NE}  & \mbox{if $k$ is connected to a vertex $l>k$ or $l'\geq k'$  by an undecorated strand}\\
{\rm SE} & \mbox{otherwise}.
\end{array}\right.\]
\[\pi(k') = \left\{ \begin{array}{ll} {\rm NW}  & \mbox{if $k'$ is connected to a vertex $l>k$ or $l'>k'$ by an undecorated strand}\\
{\rm SW} & \mbox{otherwise}.
\end{array}\right.\]   
Examples are given in Figure \ref{resulting} and \ref{resulting2XXX}. We have drawn vertical lines through all tiles not included in $R(d)$. We see that the vertical line starting at vertex $i$, respectively  $i'$ meets the path $\pi$ at $\pi(i)$, respectively  $\pi(i')$ for each $1\leq i\leq n$.

 Given a subset $X \subseteq \{1,\dots, n , n' \dots, 1'\}$ we set 
\begin{align*}
N_\pi (X) &=  \{ k\in X \mid \pi (k)= {\rm NE}\} \sqcup\{ k' \in X \mid \pi (k')={\rm NW}\} \\
 S_\pi (X) &=  \{ k\in X \mid \pi (k)= {\rm SE}\} \sqcup\{ k'\in X  \mid \pi (k')={\rm SW}\}  .
 \end{align*}
 First note that the path $\pi(d)$ takes $n$ steps to the East and $n$ steps to the West. So the path starts and finishes on the left boundary of the tiling. Moreover, the number of steps to the North is precisely the number of undecorated strands in $d$. So if $d$ is an undecorated $n$-tangle, then the path starts and ends at the same point; this is because  
\[
  |N_ \pi\{1,\dots, n,n',\dots 1'\}|=n
\qquad
 | S_ \pi\{1,\dots, n,n',\dots 1'\}|=2n-n=n.
\]
But if $d$ has at least one bead, then the path will end strictly below where it started, this is because in this case 
\[
  |N_ \pi\{1,\dots, n,n',\dots 1'\} |< n
\qquad
  |S_ \pi\{1,\dots, n,n',\dots 1'\}|>2n-n=n.
\]
More generally, the second half of the path (going West) is always weakly to the South of the first half of the path (going East). 
To see this, observe that for $1\leq k\leq n$, the difference in height in the path after $k$ steps and after $2n-k$ steps is equal to
\[
  |S_ \pi\{k+1,\dots, n,n',\dots (k+1)'\} |- 
|  N_ \pi\{k+1,\dots, n,n',\dots (k+1)'\} |\geq 0
\]
where the inequality follows by definition as every $l$ (or $l'$) in  $N_ \pi\{k+1,\dots, n,n',\dots (k+1)'\}$ 
 is connected to a vertex in  $S_ \pi\{k+1,\dots, n,n',\dots (k+1)'\}$.
Thus $\pi(d)$ defines a region $R(d)$ in the tiling which contains a finite set of tiles $t_1, \ldots, t_r$ (which are labelled by simple reflections in $S_W$).

\begin{figure} [H]
$$
\begin{tikzpicture} [scale=0.7]
	
		\clip (-3.75,-1.2) rectangle (3.15,6.7);
	
	\path (0,-0.7)--++(135:0.5)--++(-135:0.5) coordinate (minus1);

	\path (minus1)--++(135:0.5)--++(-135:0.5) coordinate (minus2);

	\path (minus2)--++(135:0.5)--++(-135:0.5) coordinate (minus3);

	\path (minus3)--++(135:0.5)--++(-135:0.5) coordinate (minus4);

	\path (0,-0.7)--++(45:0.5)--++(-45:0.5) coordinate (plus1);

	\path (plus1)--++(45:0.5)--++(-45:0.5) coordinate (plus2);

	\path (plus2)--++(45:0.5)--++(-45:0.5) coordinate (plus3);

	\draw[very thick](plus3)--(minus4);

	\path(minus1)--++(45:0.4) coordinate (minus1NE);
	\path(minus1)--++(-45:0.4) coordinate (minus1SE);

	\path(minus4)--++(135:0.4) coordinate (minus1NW);
	\path(minus4)--++(-135:0.4) coordinate (minus1SW);
	
	\path(minus2)--++(135:0.4) coordinate (start);


	\path(plus3)--++(45:0.4) coordinate (minus1NE);
	\path(plus3)--++(-45:0.4) coordinate (minus1SE);

	\path(plus1)--++(135:0.4) coordinate (minus1NW);
	\path(plus1)--++(-135:0.4) coordinate (minus1SW);
	
	\path(plus2)--++(135:0.4) coordinate (start);


	\draw[very thick,fill=magenta](0,-0.7) circle (4pt);

	\draw[very thick,  fill=darkgreen](minus1) circle (4pt);

	\draw[very thick,  fill=orange](minus2) circle (4pt);

	\draw[very thick,  fill=lime](minus3) circle (4pt);

	\draw[very thick,  fill=violet](minus4) circle (4pt);

	\draw[very thick,  fill=gray!80](plus1) circle (4pt);

	\draw[very thick,  fill=cyan](plus2) circle (4pt);

	\draw[very thick,  fill=pink](plus3) circle (4pt);

	\path (0,0) coordinate (origin2); 
 
	\begin{scope}

		\foreach \i in {0,1,2,3,4,5,6,7,8,9,10,11,12}
		{
			\path (origin2)--++(45:0.5*\i) coordinate (c\i); 
			\path (origin2)--++(135:0.5*\i)  coordinate (d\i); 
		}

		\path(origin2)  ++(135:2.5)   ++(-135:2.5) coordinate(corner1);
		\path(origin2)  ++(45:2)   ++(135:7) coordinate(corner2);
		\path(origin2)  ++(45:2)   ++(-45:2) coordinate(corner4);
		\path(origin2)  ++(135:2.5)   ++(45:6.5) coordinate(corner3);
		
		\draw[thick] (origin2)--(corner1)--(corner2)--(corner3)--(corner4)--(origin2);
		
		\clip(corner1)--(corner2)--++(90:0.3)--++(0:6.5)--(corner3)--(corner4)
		--++(90:-0.3)--++(180:6.5) --(corner1);

		\path[name path=pathd1] (d1)--++(90:7);   
		\path[name path=top] (corner2)--(corner3);   
		\path [name intersections={of = pathd1 and top}];
		\coordinate (A)  at (intersection-1);
		\path(A)--++(-90:0.1) node { \color{white}$\up$ };

		\path[name path=pathd3] (d3)--++(90:7);   
		\path[name path=top] (corner2)--(corner3);   
		\path [name intersections={of = pathd3 and top}];
		\coordinate (A)  at (intersection-1);
		\path(A)--++(-90:0.1) node { \color{white}$\up$ };

		\path[name path=pathd5] (d5)--++(90:7);   
		\path[name path=top] (corner2)--(corner3);   
		\path [name intersections={of = pathd5 and top}];
		\coordinate (A)  at (intersection-1);
		\path(A)--++(-90:0.1) node { \color{white}$\up$ };

		\path[name path=pathd7] (d7)--++(90:7);   
		\path[name path=top] (corner2)--(corner3);   
		\path [name intersections={of = pathd7 and top}];
		\coordinate (A)  at (intersection-1);
		\path(A)--++(-90:-0.1) node { \color{white}$\down$ };

		\path[name path=pathd9] (d9)--++(90:7);   
		\path[name path=top] (corner2)--(corner3);   
		\path [name intersections={of = pathd9 and top}];
		\coordinate (A)  at (intersection-1);
		\path(A)--++(-90:-0.1) node { \color{white}$\down$ };

		\path[name path=pathc1] (c1)--++(90:7);   
		\path[name path=top] (corner2)--(corner3);   
		\path [name intersections={of = pathc1 and top}];
		\coordinate (A)  at (intersection-1);
		\path(A)--++(-90:0.1) node { \color{white}$\up$ };

		\path[name path=pathc3] (c3)--++(90:7);   
		\path[name path=top] (corner2)--(corner3);   
		\path [name intersections={of = pathc3 and top}];
		\coordinate (A)  at (intersection-1);
		\path(A)--++(-90:-0.1) node { \color{white}$\down$ };

		\path[name path=pathc5] (c5)--++(90:7);   
		\path[name path=top] (corner2)--(corner3);   
		\path [name intersections={of = pathc5 and top}];
		\coordinate (A)  at (intersection-1);
		\path(A)--++(-90:-0.1) node { \color{white}$\down$ };

		\path[name path=pathc7] (c7)--++(90:7);   
		\path[name path=top] (corner2)--(corner3);   
		\path [name intersections={of = pathc7 and top}];
		\coordinate (A)  at (intersection-1);
		\path(A)--++(-90:0.1) node { \color{white}$\up$ };

		\path[name path=pathd1] (d1)--++(-90:7);   
		\path[name path=bottom] (corner1)--(corner4);   
		\path [name intersections={of = pathd1 and bottom}];
		\coordinate (A)  at (intersection-1);
		\path (A)--++(90:-0.1) node { \color{white}$\up$  };
		
		\path[name path=pathd3] (d3)--++(-90:7);   
		\path[name path=bottom] (corner1)--(corner4);   
		\path [name intersections={of = pathd3 and bottom}];
		\coordinate (A)  at (intersection-1);
		\path (A)--++(90:-0.1) node { \color{white}$\up$  };

		\path[name path=pathd5] (d5)--++(-90:7);   
		\path[name path=bottom] (corner1)--(corner4);   
		\path [name intersections={of = pathd5 and bottom}];
		\coordinate (A)  at (intersection-1);
		\path (A)--++(90:-0.1) node { \color{white}$\up$  };

		\path[name path=pathd7] (d7)--++(-90:7);   
		\path[name path=bottom] (corner1)--(corner4);   
		\path [name intersections={of = pathd7 and bottom}];
		\coordinate (A)  at (intersection-1);
		\path (A)--++(90:-0.1) node { \color{white}$\up$  };
		
		\path[name path=pathd9] (d9)--++(-90:7);   
		\path[name path=bottom] (corner1)--(corner4);   
		\path [name intersections={of = pathd9 and bottom}];
		\coordinate (A)  at (intersection-1);
		\path (A)--++(90:-0.1) node { \color{white}$\up$  };

		\path[name path=pathc1] (c1)--++(-90:7);   
		\path[name path=bottom] (corner1)--(corner4);   
		\path [name intersections={of = pathc1 and bottom}];
		\coordinate (A)  at (intersection-1);
		\path (A)--++(90:0.1) node { \color{white}$\down$  };

		\path[name path=pathc3] (c3)--++(-90:7);   
		\path[name path=bottom] (corner1)--(corner4);   
		\path [name intersections={of = pathc3 and bottom}];
		\coordinate (A)  at (intersection-1);
		\path (A)--++(90:0.1) node { \color{white}$\down$  };

		\path[name path=pathc5] (c5)--++(-90:7);   
		\path[name path=bottom] (corner1)--(corner4);   
		\path [name intersections={of = pathc5 and bottom}];
		\coordinate (A)  at (intersection-1);
		\path (A)--++(90:0.1) node { \color{white}$\down$  };

		\path[name path=pathc7] (c7)--++(-90:7);   
		\path[name path=bottom] (corner1)--(corner4);   
		\path [name intersections={of = pathc7 and bottom}];
		\coordinate (A)  at (intersection-1);
		\path (A)--++(90:0.1) node { \color{white}$\down$  };
		
		\clip(corner1)--(corner2)--(corner3)--(corner4)--(corner1);

		\foreach \i in {1,3,5,7,9,11}
		{
			\draw[thick](c\i)--++(90:7);   
			\draw[thick](c\i)--++(-90:7);
			\draw[thick](d\i)--++(90:7);
			\draw[thick](d\i)--++(-90:7);
		}

	\end{scope}

	\begin{scope}

			\clip(corner1)--(corner2)--(corner3)--(corner4)--(corner1);
		
		\path (0,0) coordinate (origin2);

		\foreach \i\j in {0,1,2,3,4,5,6,7,8,9,10,11,12}
		{
			\path (origin2)--++(45:0.5*\i) coordinate (a\i); 
			\path (origin2)--++(135:0.5*\i)  coordinate (b\j);

		}

		\fill[white]
		(0,0) --++(45:4)--++(135:1)--++(-135:2)
		--++(135:3)--++(135:1)
		--++(-135:1)--++(-135:1)
		--++(-45:5);

		\draw(a1)  ++(135:0.5)   ++(-135:0.5) coordinate(next1);
		\draw[very thick,magenta](a1) to [out=90,in=90] (next1);

		\draw(a3)  ++(135:0.5)   ++(-135:0.5) coordinate(upnext1);
		\draw[very thick,gray](a3) to [out=90,in=90] (upnext1);

		\draw(upnext1)  ++(135:0.5)   ++(-135:0.5) coordinate(upnext2);
		\draw[very thick,magenta](upnext1) to [out=-90,in=-90] (upnext2);

		\draw(upnext2)  ++(135:0.5)   ++(-135:0.5) coordinate(upnext3);
		\draw[very thick,darkgreen](upnext2) to [out=90,in=90] (upnext3);


		\draw(a5)  ++(135:0.5)   ++(-135:0.5) coordinate(upnext1);
		\draw[very thick,cyan](a5) to [out=90,in=90] (upnext1);

		\draw(upnext1)  ++(135:0.5)   ++(-135:0.5) coordinate(upnext2);
		\draw[very thick,gray](upnext1) to [out=-90,in=-90] (upnext2);

		\draw(upnext2)  ++(135:0.5)   ++(-135:0.5) coordinate(upnext3);
		\draw[very thick,magenta](upnext2) to [out=90,in=90] (upnext3);

		\draw(upnext3)  ++(135:0.5)   ++(-135:0.5) coordinate(upnext4);
		\draw[very thick, darkgreen](upnext3) to [out=-90,in=-90] (upnext4);

		\draw(upnext4)  ++(135:0.5)   ++(-135:0.5) coordinate(upnext5);
		\draw[very thick,orange](upnext4) to [out=90,in=90] (upnext5);


		\draw(a7)  ++(135:0.5)   ++(-135:0.5) coordinate(upnext1);
		\draw[very thick,pink](a7) to [out=90,in=90] (upnext1);

		\draw(upnext1)  ++(135:0.5)   ++(-135:0.5) coordinate(upnext2);
		\draw[very thick,cyan](upnext1) to [out=-90,in=-90] (upnext2);

		\draw(upnext2)  ++(135:0.5)   ++(-135:0.5) coordinate(upnext3);
		\path(upnext2) to [out=90,in=90] (upnext3);

		\draw(upnext3)  ++(135:0.5)   ++(-135:0.5) coordinate(upnext4);
		\draw[very thick, magenta](upnext3) to [out=-90,in=-90] (upnext4);

		\draw(upnext4)  ++(135:0.5)   ++(-135:0.5) coordinate(upnext5);
		\draw[very thick,darkgreen](upnext4) to [out=90,in=90] (upnext5);

		\draw(upnext5)  ++(135:0.5)   ++(-135:0.5) coordinate(upnext6);
		\draw[very thick, orange](upnext5) to [out=-90,in=-90] (upnext6);

		\draw(upnext6)  ++(135:0.5)   ++(-135:0.5) coordinate(upnext7);
		\draw[very thick,lime](upnext6) to [out=90,in=90] (upnext7);


		\draw(a9)  ++(135:0.5)   ++(-135:0.5) coordinate(upnext1);
		\path(a9) to [out=90,in=90] (upnext1);

		\draw(upnext1)  ++(135:0.5)   ++(-135:0.5) coordinate(upnext2);
		\draw[very thick,pink](upnext1) to [out=-90,in=-90] (upnext2);

		\draw(upnext2)  ++(135:0.5)   ++(-135:0.5) coordinate(upnext3);
		\path(upnext2) to [out=90,in=90] (upnext3);

		\draw(upnext3)  ++(135:0.5)   ++(-135:0.5) coordinate(upnext4);
		\path(upnext3) to [out=-90,in=-90] (upnext4);

		\draw(upnext4)  ++(135:0.5)   ++(-135:0.5) coordinate(upnext5);
		\path(upnext4) to [out=90,in=90] (upnext5);

		\draw(upnext5)  ++(135:0.5)   ++(-135:0.5) coordinate(upnext6);
		\draw[very thick, darkgreen ](upnext5) to [out=-90,in=-90] (upnext6);

		\draw(upnext6)  ++(135:0.5)   ++(-135:0.5) coordinate(upnext7);
		\draw[very thick,orange ](upnext6) to [out=90,in=90] (upnext7);

		\draw(upnext7)  ++(135:0.5)   ++(-135:0.5) coordinate(upnext8);
		\draw[very thick,lime](upnext7) to [out=-90,in=-90]   (upnext8);

		\draw(upnext8)  ++(135:0.5)   ++(-135:0.5) coordinate(upnext9);
		\draw[very thick, violet ](upnext8) to [out=90,in=90] (upnext9);


		\path(a4)  --++(135:3.5)    coordinate(upnext6);
		
		\draw(upnext6)  ++(135:0.5)   ++(-135:0.5) coordinate(upnext7);
		\draw[very thick,orange ](upnext6) to[out=-90,in=-90]   (upnext7);

		\draw(upnext7)  ++(135:0.5)   ++(-135:0.5) coordinate(upnext8);
		\draw[very thick,lime ]
		(upnext7) to  [out=90,in=90]   (upnext8);

		\draw(upnext8)  ++(135:0.5)   ++(-135:0.5) coordinate(upnext9);
		\draw[very thick, violet ](upnext8) to[out=-90,in=-90] (upnext9);
		

		\path(a4)  --++(135:4.5)    coordinate(upnext6);
		
		\draw(upnext6)  ++(135:0.5)   ++(-135:0.5) coordinate(upnext7);
		\draw[very thick,lime ](upnext6) to[out=-90,in=-90]   (upnext7);

			\path (0,0) coordinate (origin2);

		\foreach \i in {0,1,2,3,4,5,6,7,8,9,10,11,12}
		{
			\path (origin2)--++(-45:2)--++(-135:2) --++(45:1*\i) coordinate (c\i); 
			\path (origin2)--++(-45:2)--++(-135:2) --++(135:1*\i)  coordinate (d\i); 
			\draw[thick,densely dotted] (c\i)--++(135:14);
			\draw[thick,densely dotted] (d\i)--++(45:14);
		}

	\end{scope}
	
	\path(origin2)  ++(135:2.5)   ++(-135:2.5) coordinate(corner1);
	\path(origin2)  ++(45:2)   ++(135:7) coordinate(corner2);
	\path(origin2)  ++(45:2)   ++(-45:2) coordinate(corner4);
	\path(origin2)  ++(135:2.5)   ++(45:6.5) coordinate(corner3);
	
	\draw[thick] (origin2)--(corner1)--(corner2)--(corner3)--(corner4)--(origin2);

%
%
	\draw[line width=1.6,->](origin2)--++(135:1) coordinate (X); 
	\draw[line width=1.6,->](X)--++(135:1) coordinate (X); 
	\draw[line width=1.6,->](X)--++(135:1) coordinate (X);
		\draw[line width=1.6,->](X)--++(135:1) coordinate (X);
			\draw[line width=1.6,->](X)--++(135:1) coordinate (X); 
 
 	\draw[fill=white,line width=1.5](X) circle (2.5pt) coordinate (Y);

\draw[line width=1.6,->](X)--++(45:1) coordinate (X);  

 	\draw[fill=white,line width=1.5](Y) circle (2.5pt);

\draw[line width=1.6,->](X)--++(45:1) coordinate (X);   
 
\draw[line width=1.6,->](X)--++(-45:1) coordinate (X);   
\draw[line width=1.6,->](X)--++(-45:1) coordinate (X);   \draw[line width=1.6,->](X)--++(-45:1) coordinate (X);   \draw[line width=1.6,->](X)--++(-45:1) coordinate (X);    
 
 \draw[line width=1.6,->](X)--++(45:1) coordinate (X);   \draw[line width=1.6,->](X)--++(45:1) coordinate (X);   
 \draw[line width=1.6,->](X)--++(-45:1) coordinate (X);

 \draw[line width=1.6,->](X)--++(-135:1) coordinate (X);    \draw[line width=1.6,->](X)--++(-135:1) coordinate (X);    \draw[line width=1.6,->](X)--++(-135:1) coordinate (X);    \draw[line width=1.6,->](X)--++(-135:1) coordinate (X);

 	\clip(origin2)--(corner1)--(corner2)--(corner3)--(corner4)--(origin2);

		\path(0,0)--++(135:1) coordinate (graystart);
	
	\fill[gray,opacity=0.2]  (graystart) --++(0:1.414)--++(45:1)--++(180:2*1.414);

		\path(graystart)--++(135:2) coordinate (graystart);
	
	\fill[gray,opacity=0.2]  (graystart) --++(0:1.414*3)--++(45:1)
	--++(180:1*1.414)
	--++(-135:1)--++(135:1)--++(180:2*1.414);

			\path(graystart)--++(135:2) coordinate (graystart);
\fill[gray,opacity=0.2]  (graystart) --++(0:1.414*2)--++(135:1)--++(180:1*1.414);

\end{tikzpicture}\qquad
 \begin{tikzpicture} [scale=0.7]

	  		\clip (-3.75,-1.2) rectangle (3.15,6.7);
	
	\path (0,-0.7)--++(135:0.5)--++(-135:0.5) coordinate (minus1);

	\path (minus1)--++(135:0.5)--++(-135:0.5) coordinate (minus2);

	\path (minus2)--++(135:0.5)--++(-135:0.5) coordinate (minus3);

	\path (minus3)--++(135:0.5)--++(-135:0.5) coordinate (minus4);

	\path (0,-0.7)--++(45:0.5)--++(-45:0.5) coordinate (plus1);

	\path (plus1)--++(45:0.5)--++(-45:0.5) coordinate (plus2);

	\path (plus2)--++(45:0.5)--++(-45:0.5) coordinate (plus3);

	\draw[very thick](plus3)--(minus4);

	\path(minus1)--++(45:0.4) coordinate (minus1NE);
	\path(minus1)--++(-45:0.4) coordinate (minus1SE);

	\path(minus4)--++(135:0.4) coordinate (minus1NW);
	\path(minus4)--++(-135:0.4) coordinate (minus1SW);
	
	\path(minus2)--++(135:0.4) coordinate (start);

	\path(plus3)--++(45:0.4) coordinate (minus1NE);
	\path(plus3)--++(-45:0.4) coordinate (minus1SE);

	\path(plus1)--++(135:0.4) coordinate (minus1NW);
	\path(plus1)--++(-135:0.4) coordinate (minus1SW);
	
	\path(plus2)--++(135:0.4) coordinate (start);


	\draw[very thick,fill=magenta](0,-0.7) circle (4pt);

	\draw[very thick,  fill=darkgreen](minus1) circle (4pt);

	\draw[very thick,  fill=orange](minus2) circle (4pt);

	\draw[very thick,  fill=lime!80!black](minus3) circle (4pt);

	\draw[very thick,  fill=violet](minus4) circle (4pt);

	\draw[very thick,  fill=gray!80](plus1) circle (4pt);

	\draw[very thick,  fill=cyan](plus2) circle (4pt);

	\draw[very thick,  fill=pink](plus3) circle (4pt);

	\path (0,0) coordinate (origin2); 
	

	\begin{scope}

		\foreach \i in {0,1,2,3,4,5,6,7,8,9,10,11,12}
		{
			\path (origin2)--++(45:0.5*\i) coordinate (c\i); 
			\path (origin2)--++(135:0.5*\i)  coordinate (d\i); 
		}

		\path(origin2)  ++(135:2.5)   ++(-135:2.5) coordinate(corner1);
		\path(origin2)  ++(45:2)   ++(135:7) coordinate(corner2);
		\path(origin2)  ++(45:2)   ++(-45:2) coordinate(corner4);
		\path(origin2)  ++(135:2.5)   ++(45:6.5) coordinate(corner3);
		
		\draw[thick] (origin2)--(corner1)--(corner2)--(corner3)--(corner4)--(origin2);
		
		\clip(corner1)--(corner2)--++(90:0.3)--++(0:6.5)--(corner3)--(corner4)
		--++(90:-0.3)--++(180:6.5) --(corner1);

		\path[name path=pathd1] (d1)--++(90:7);   
		\path[name path=top] (corner2)--(corner3);   
		\path [name intersections={of = pathd1 and top}];
		\coordinate (A)  at (intersection-1);
		\path(A)--++(-90:0.1) node { \color{white}$\up$ };

		\path[name path=pathd3] (d3)--++(90:7);   
		\path[name path=top] (corner2)--(corner3);   
		\path [name intersections={of = pathd3 and top}];
		\coordinate (A)  at (intersection-1);
		\path(A)--++(-90:0.1) node { \color{white}$\up$ };

		\path[name path=pathd5] (d5)--++(90:7);   
		\path[name path=top] (corner2)--(corner3);   
		\path [name intersections={of = pathd5 and top}];
		\coordinate (A)  at (intersection-1);
		\path(A)--++(-90:0.1) node { \color{white}$\up$ };

		\path[name path=pathd7] (d7)--++(90:7);   
		\path[name path=top] (corner2)--(corner3);   
		\path [name intersections={of = pathd7 and top}];
		\coordinate (A)  at (intersection-1);
		\path(A)--++(-90:-0.1) node { \color{white}$\down$ };

		\path[name path=pathd9] (d9)--++(90:7);   
		\path[name path=top] (corner2)--(corner3);   
		\path [name intersections={of = pathd9 and top}];
		\coordinate (A)  at (intersection-1);
		\path(A)--++(-90:-0.1) node { \color{white}$\down$ };

		\path[name path=pathc1] (c1)--++(90:7);   
		\path[name path=top] (corner2)--(corner3);   
		\path [name intersections={of = pathc1 and top}];
		\coordinate (A)  at (intersection-1);
		\path(A)--++(-90:0.1) node { \color{white}$\up$ };

		\path[name path=pathc3] (c3)--++(90:7);   
		\path[name path=top] (corner2)--(corner3);   
		\path [name intersections={of = pathc3 and top}];
		\coordinate (A)  at (intersection-1);
		\path(A)--++(-90:-0.1) node { \color{white}$\down$ };

		\path[name path=pathc5] (c5)--++(90:7);   
		\path[name path=top] (corner2)--(corner3);   
		\path [name intersections={of = pathc5 and top}];
		\coordinate (A)  at (intersection-1);
		\path(A)--++(-90:-0.1) node { \color{white}$\down$ };

		\path[name path=pathc7] (c7)--++(90:7);   
		\path[name path=top] (corner2)--(corner3);   
		\path [name intersections={of = pathc7 and top}];
		\coordinate (A)  at (intersection-1);
		\path(A)--++(-90:0.1) node { \color{white}$\up$ };

		\path[name path=pathd1] (d1)--++(-90:7);   
		\path[name path=bottom] (corner1)--(corner4);   
		\path [name intersections={of = pathd1 and bottom}];
		\coordinate (A)  at (intersection-1);
		\path (A)--++(90:-0.1) node { \color{white}$\up$  };
		
		\path[name path=pathd3] (d3)--++(-90:7);   
		\path[name path=bottom] (corner1)--(corner4);   
		\path [name intersections={of = pathd3 and bottom}];
		\coordinate (A)  at (intersection-1);
		\path (A)--++(90:-0.1) node { \color{white}$\up$  };

		\path[name path=pathd5] (d5)--++(-90:7);   
		\path[name path=bottom] (corner1)--(corner4);   
		\path [name intersections={of = pathd5 and bottom}];
		\coordinate (A)  at (intersection-1);
		\path (A)--++(90:-0.1) node { \color{white}$\up$  };

		\path[name path=pathd7] (d7)--++(-90:7);   
		\path[name path=bottom] (corner1)--(corner4);   
		\path [name intersections={of = pathd7 and bottom}];
		\coordinate (A)  at (intersection-1);
		\path (A)--++(90:-0.1) node { \color{white}$\up$  };
		
		\path[name path=pathd9] (d9)--++(-90:7);   
		\path[name path=bottom] (corner1)--(corner4);   
		\path [name intersections={of = pathd9 and bottom}];
		\coordinate (A)  at (intersection-1);
		\path (A)--++(90:-0.1) node { \color{white}$\up$  };

		\path[name path=pathc1] (c1)--++(-90:7);   
		\path[name path=bottom] (corner1)--(corner4);   
		\path [name intersections={of = pathc1 and bottom}];
		\coordinate (A)  at (intersection-1);
		\path (A)--++(90:0.1) node { \color{white}$\down$  };

		\path[name path=pathc3] (c3)--++(-90:7);   
		\path[name path=bottom] (corner1)--(corner4);   
		\path [name intersections={of = pathc3 and bottom}];
		\coordinate (A)  at (intersection-1);
		\path (A)--++(90:0.1) node { \color{white}$\down$  };

		\path[name path=pathc5] (c5)--++(-90:7);   
		\path[name path=bottom] (corner1)--(corner4);   
		\path [name intersections={of = pathc5 and bottom}];
		\coordinate (A)  at (intersection-1);
		\path (A)--++(90:0.1) node { \color{white}$\down$  };

		\path[name path=pathc7] (c7)--++(-90:7);   
		\path[name path=bottom] (corner1)--(corner4);   
		\path [name intersections={of = pathc7 and bottom}];
		\coordinate (A)  at (intersection-1);
		\path (A)--++(90:0.1) node { \color{white}$\down$  };
		
		\clip(corner1)--(corner2)--(corner3)--(corner4)--(corner1);

		\foreach \i in {1,3,5,7,9,11}
		{
			\draw[thick](c\i)--++(90:7);   
			\draw[thick](c\i)--++(-90:7);
			\draw[thick](d\i)--++(90:7);
			\draw[thick](d\i)--++(-90:7);
		}

	\end{scope}

	\begin{scope}
		\clip(corner1)--(corner2)--(corner3)--(corner4)--(corner1);
		
		\path (0,0) coordinate (origin2);

		\foreach \i\j in {0,1,2,3,4,5,6,7,8,9,10,11,12}
		{
			\path (origin2)--++(45:0.5*\i) coordinate (a\i); 
			\path (origin2)--++(135:0.5*\i)  coordinate (b\j); 


		}

		\path(0,0)--++(135:5) coordinate(pirir);
		
		\fill[white] 
		(pirir) --++(45:2) --++(-45:1)--++(45:1) --++(-45:2)--++(45:1) --++(-45:2)
		--++(-135:2)--++(135:2)
	 --++(-135:1)
	 --++(135:2)	 --++(-135:1)
	 --++(135:1);

		\fill[white] 
		(0,0) --++(45:2) --++(135:4)--++(-135:2) ;

		\draw(a1)  ++(135:0.5)   ++(-135:0.5) coordinate(next1);

		\draw(a3)  ++(135:0.5)   ++(-135:0.5) coordinate(upnext1);

		\draw(upnext1)  ++(135:0.5)   ++(-135:0.5) coordinate(upnext2);

		\draw(upnext2)  ++(135:0.5)   ++(-135:0.5) coordinate(upnext3);


		\draw(a5)  ++(135:0.5)   ++(-135:0.5) coordinate(upnext1);
		\draw[very thick,cyan](a5) to [out=90,in=90] (upnext1);

		\draw(upnext1)  ++(135:0.5)   ++(-135:0.5) coordinate(upnext2);

		\draw(upnext2)  ++(135:0.5)   ++(-135:0.5) coordinate(upnext3);

		\draw(upnext3)  ++(135:0.5)   ++(-135:0.5) coordinate(upnext4);

		\draw(upnext4)  ++(135:0.5)   ++(-135:0.5) coordinate(upnext5);


		\draw(a7)  ++(135:0.5)   ++(-135:0.5) coordinate(upnext1);
		\draw[very thick,pink](a7) to [out=90,in=90] (upnext1);

		\draw(upnext1)  ++(135:0.5)   ++(-135:0.5) coordinate(upnext2);
		\draw[very thick,cyan](upnext1) to [out=-90,in=-90] (upnext2);

		\draw(upnext2)  ++(135:0.5)   ++(-135:0.5) coordinate(upnext3);
		\draw[very thick,gray](upnext2) to [out=90,in=90] (upnext3);

		\path(upnext2) to [out=90,in=90] (upnext3);

		\draw(upnext3)  ++(135:0.5)   ++(-135:0.5) coordinate(upnext4);

		\draw(upnext4)  ++(135:0.5)   ++(-135:0.5) coordinate(upnext5);
		\draw[very thick,darkgreen](upnext4) to [out=90,in=90] (upnext5);

		\draw(upnext5)  ++(135:0.5)   ++(-135:0.5) coordinate(upnext6);

		\draw(upnext6)  ++(135:0.5)   ++(-135:0.5) coordinate(upnext7);


		\draw(a9)  ++(135:0.5)   ++(-135:0.5) coordinate(upnext1);
		\path(a9) to [out=90,in=90] (upnext1);

		\draw(upnext1)  ++(135:0.5)   ++(-135:0.5) coordinate(upnext2);
		\draw[very thick,pink](upnext1) to [out=-90,in=-90] (upnext2);

		\draw(upnext2)  ++(135:0.5)   ++(-135:0.5) coordinate(upnext3);
\draw[very thick,cyan](upnext2) to [out=90,in=90] (upnext3);		
		
		\path(upnext2) to [out=90,in=90] (upnext3);

		\draw(upnext3)  ++(135:0.5)   ++(-135:0.5) coordinate(upnext4);
		\draw[very thick,gray](upnext3) to [out=-90,in=-90] (upnext4);

		\path(upnext3) to [out=-90,in=-90] (upnext4);

		\draw(upnext4)  ++(135:0.5)   ++(-135:0.5) coordinate(upnext5);
	\draw[very thick,magenta](upnext4) to [out=90,in=90] (upnext5);

	\path(upnext4) to [out=90,in=90] (upnext5);

		\draw(upnext5)  ++(135:0.5)   ++(-135:0.5) coordinate(upnext6);
		\draw[very thick, darkgreen ](upnext5) to [out=-90,in=-90] (upnext6);

		\draw(upnext6)  ++(135:0.5)   ++(-135:0.5) coordinate(upnext7);
		\draw[very thick,orange ](upnext6) to [out=90,in=90] (upnext7);

		\draw(upnext7)  ++(135:0.5)   ++(-135:0.5) coordinate(upnext8);

		\draw(upnext8)  ++(135:0.5)   ++(-135:0.5) coordinate(upnext9);
		\draw[very thick, violet ](upnext8) to [out=90,in=90] (upnext9);


		\path(a4)  --++(135:3.5)    coordinate(upnext6);

		\path	(upnext6)--++(45:0.5)--++(135:0.5)  coordinate (upnext00);
		\path	(upnext6)--++(45:0.5)--++(135:0.5) --++(45:0.5)--++(-45:0.5)  coordinate (upnext001);

			\draw[very thick,darkgreen ](upnext00) to[out=-90,in=-90]   (upnext001);

		\path	(upnext6)--++(45:0.5)--++(-45:0.5)  coordinate (upnext67);
			\draw[very thick,darkgreen]
		(upnext6) to  [out=90,in=90]   (upnext67);

				\path	(upnext67)--++(45:0.5)--++(-45:0.5)  coordinate (upnext68);
		
			\draw[very thick,magenta ](upnext67) to[out=-90,in=-90]   (upnext68);

				\path	(upnext68)--++(45:0.5)--++(-45:0.5)  coordinate (upnext69);
								\path	(upnext69)--++(45:0.5)--++(-45:0.5)  coordinate (upnext70);
			\draw[very thick,cyan ](upnext69) to[out=-90,in=-90]   (upnext70);

		\draw(upnext6)  ++(135:0.5)   ++(-135:0.5) coordinate(upnext7);
		\draw[very thick,orange ](upnext6) to[out=-90,in=-90]   (upnext7);

		\draw(upnext7)  ++(135:0.5)   ++(-135:0.5) coordinate(upnext8);
		\draw[very thick,lime!80!black ]
		(upnext7) to  [out=90,in=90]   (upnext8);

		\draw(upnext8)  ++(135:0.5)   ++(-135:0.5) coordinate(upnext9);
		\draw[very thick, violet ](upnext8) to[out=-90,in=-90] (upnext9);
		

		\path(a4)  --++(135:4.5)    coordinate(upnext6);
		
		\draw(upnext6)  ++(135:0.5)   ++(-135:0.5) coordinate(upnext7);
		\draw[very thick,lime!80!black ](upnext6) to[out=-90,in=-90]   (upnext7);

		\foreach \i in {0,1,2,3,4,5,6,7,8,9,10,11,12}
		{
			\path (origin2)--++(-45:2)--++(-135:2) --++(45:1*\i) coordinate (c\i); 
			\path (origin2)--++(-45:2)--++(-135:2) --++(135:1*\i)  coordinate (d\i); 
			\draw[thick,densely dotted] (c\i)--++(135:14);
			\draw[thick,densely dotted] (d\i)--++(45:14);
		}

	\end{scope}
	
	\path(origin2)  ++(135:2.5)   ++(-135:2.5) coordinate(corner1);
	\path(origin2)  ++(45:2)   ++(135:7) coordinate(corner2);
	\path(origin2)  ++(45:2)   ++(-45:2) coordinate(corner4);
	\path(origin2)  ++(135:2.5)   ++(45:6.5) coordinate(corner3);

	\draw[thick] (origin2)--(corner1)--(corner2)--(corner3)--(corner4)--(origin2);

	\path(origin2)--++(135:5) coordinate (X);
		\path(X)  coordinate (XT);
	\draw[line width=1.5,->](X)--++(45:1) coordinate (X);  
	\draw[line width=1.5,->](X)--++(45:1) coordinate (X);  
	\draw[line width=1.5,->](X)--++(-45:1) coordinate (X);  	 
	 	\draw[line width=1.5,->](X)--++(45:1) coordinate (X);  
	\draw[line width=1.5,->](X)--++(-45:1) coordinate (X);  	 	\draw[line width=1.5,->](X)--++(-45:1) coordinate (X);  	 	\draw[line width=1.5,->](X)--++(45:1) coordinate (X);  	 	 
	\draw[line width=1.5,->](X)--++(-45:1) coordinate (X);  	 	\draw[line width=1.5,->](X)--++(-45:1) coordinate (X);  	 
	\draw[line width=1.5,->](X)--++(-135:1) coordinate (X);  	 
	\draw[line width=1.5,->](X)--++(-135:1) coordinate (X);  	 	\draw[line width=1.5,->](X)--++(-135:1) coordinate (X);  	 	\draw[line width=1.5,->](X)--++(-135:1) coordinate (X);

		\draw[line width=1.5,->](X)--++(135:1) coordinate (X);  	 
		\draw[line width=1.5,->](X)--++(135:1) coordinate (X);  	 		\draw[line width=1.5,->](X)--++(135:1) coordinate (X);  	 		\draw[line width=1.5,->](X)--++(135:1) coordinate (X);  	 		\draw[line width=1.5,->](X)--++(135:1) coordinate (X);

\path(0,0)--++(45:1)--++(135:1) 			coordinate (tile);	
\path(tile)--++(45:0.5) 	coordinate (SE);		
\path(tile)--++(135:0.5) 	coordinate (SW);			
\path(tile)--++(135:1)--++(45:0.5)  	coordinate (NW);	
\path(tile)--++(135:0.5)--++(45:1)  	coordinate (NE);		

	\draw[ very thick,magenta] 		(SE) to [out=90,in=90] (SW) (NE) to [out=-90,in=-90] (NW)				;

\path(0,0) 		coordinate (tile);	
\path(tile)--++(45:0.5) 	coordinate (SE);		
\path(tile)--++(135:0.5) 	coordinate (SW);			
\path(tile)--++(135:1)--++(45:0.5)  	coordinate (NW);	
\path(tile)--++(135:0.5)--++(45:1)  	coordinate (NE);		

	\draw[ very thick,magenta] 		(SE) to [out=90,in=90] (SW) (NE) to [out=-90,in=-90] (NW)				;

\path(0,0)--++(45:1) 		coordinate (tile);	
\path(tile)--++(45:0.5) 	coordinate (SE);		
\path(tile)--++(135:0.5) 	coordinate (SW);			
\path(tile)--++(135:1)--++(45:0.5)  	coordinate (NW);	
\path(tile)--++(135:0.5)--++(45:1)  	coordinate (NE);		

	\draw[ very thick,gray] 		(SE) to [out=90,in=90] (SW) (NE) to [out=-90,in=-90] (NW)				;

	\path(0,0) --++(135:1) 			coordinate (tile);	
\path(tile)--++(45:0.5) 	coordinate (SE);		
\path(tile)--++(135:0.5) 	coordinate (SW);			
\path(tile)--++(135:1)--++(45:0.5)  	coordinate (NW);	
\path(tile)--++(135:0.5)--++(45:1)  	coordinate (NE);		

	\draw[ very thick,darkgreen] 		(SE) to [out=90,in=90] (SW) (NE) to [out=-90,in=-90] (NW)				;

	\path(0,0) --++(135:2) 			coordinate (tile);	
\path(tile)--++(45:0.5) 	coordinate (SE);		
\path(tile)--++(135:0.5) 	coordinate (SW);			
\path(tile)--++(135:1)--++(45:0.5)  	coordinate (NW);	
\path(tile)--++(135:0.5)--++(45:1)  	coordinate (NE);		

	\draw[ very thick,orange] 		(SE) to [out=90,in=90] (SW) (NE) to [out=-90,in=-90] (NW)				;

	\path(0,0) --++(135:3) 			coordinate (tile);	
\path(tile)--++(45:0.5) 	coordinate (SE);		
\path(tile)--++(135:0.5) 	coordinate (SW);			
\path(tile)--++(135:1)--++(45:0.5)  	coordinate (NW);	
\path(tile)--++(135:0.5)--++(45:1)  	coordinate (NE);		

	\draw[ very thick,lime!80!black] 		(SE) to [out=90,in=90] (SW) (NE) to [out=-90,in=-90] (NW)				;

	 
	 	\draw[fill=white,line width=1.5](XT) circle (2.5pt);
		 	\draw[fill=white,line width=1.5](X) circle (2.5pt);
  
 	\clip(origin2)--(corner1)--(corner2)--(corner3)--(corner4)--(origin2);

		\path(0,0)--++(135:1) coordinate (graystart);
	
	\fill[gray,opacity=0.2]  (graystart) --++(0:1.414)--++(45:1)--++(180:2*1.414);

			\path(graystart)--++(135:2) coordinate (graystart);
\fill[gray,opacity=0.2]  (graystart) --++(0:1.414*3)--++(45:1)--++(180:4*1.414);

			\path(graystart)--++(135:2) coordinate (graystart);
	\fill[gray,opacity=0.2]  (graystart) --++(0:1.414*4)--++(135:1)
	--++(-135:1)--++(135:1)--++(180:2*1.414);

\end{tikzpicture}\qquad
\begin{tikzpicture} [scale=0.7]
	
	\clip (-3.75,-1.2) rectangle (3.15,6.7);
	
	\path (0,-0.7)--++(135:0.5)--++(-135:0.5) coordinate (minus1);

	\path (minus1)--++(135:0.5)--++(-135:0.5) coordinate (minus2);

	\path (minus2)--++(135:0.5)--++(-135:0.5) coordinate (minus3);

	\path (minus3)--++(135:0.5)--++(-135:0.5) coordinate (minus4);

	\path (0,-0.7)--++(45:0.5)--++(-45:0.5) coordinate (plus1);

	\path (plus1)--++(45:0.5)--++(-45:0.5) coordinate (plus2);

	\path (plus2)--++(45:0.5)--++(-45:0.5) coordinate (plus3);

	\draw[very thick](plus3)--(minus4);

	\path(minus1)--++(45:0.4) coordinate (minus1NE);
	\path(minus1)--++(-45:0.4) coordinate (minus1SE);

	\path(minus4)--++(135:0.4) coordinate (minus1NW);
	\path(minus4)--++(-135:0.4) coordinate (minus1SW);
	
	\path(minus2)--++(135:0.4) coordinate (start);


	\path(plus3)--++(45:0.4) coordinate (minus1NE);
	\path(plus3)--++(-45:0.4) coordinate (minus1SE);

	\path(plus1)--++(135:0.4) coordinate (minus1NW);
	\path(plus1)--++(-135:0.4) coordinate (minus1SW);
	
	\path(plus2)--++(135:0.4) coordinate (start);


	\draw[very thick,fill=magenta](0,-0.7) circle (4pt);

	\draw[very thick,  fill=darkgreen](minus1) circle (4pt);

	\draw[very thick,  fill=orange](minus2) circle (4pt);

	\draw[very thick,  fill=lime](minus3) circle (4pt);

	\draw[very thick,  fill=violet](minus4) circle (4pt);

	\draw[very thick,  fill=gray!80](plus1) circle (4pt);

	\draw[very thick,  fill=cyan](plus2) circle (4pt);

	\draw[very thick,  fill=pink](plus3) circle (4pt);

	\path (0,0) coordinate (origin2); 
	

	\begin{scope}

		\foreach \i in {0,1,2,3,4,5,6,7,8,9,10,11,12}
		{
			\path (origin2)--++(45:0.5*\i) coordinate (c\i); 
			\path (origin2)--++(135:0.5*\i)  coordinate (d\i); 
		}

		\path(origin2)  ++(135:2.5)   ++(-135:2.5) coordinate(corner1);
		\path(origin2)  ++(45:2)   ++(135:7) coordinate(corner2);
		\path(origin2)  ++(45:2)   ++(-45:2) coordinate(corner4);
		\path(origin2)  ++(135:2.5)   ++(45:6.5) coordinate(corner3);
		
		\draw[thick] (origin2)--(corner1)--(corner2)--(corner3)--(corner4)--(origin2);
		
		\clip(corner1)--(corner2)--++(90:0.3)--++(0:6.5)--(corner3)--(corner4)
		--++(90:-0.3)--++(180:6.5) --(corner1);

		\path[name path=pathd1] (d1)--++(90:7);   
		\path[name path=top] (corner2)--(corner3);   
		\path [name intersections={of = pathd1 and top}];
		\coordinate (A)  at (intersection-1);
		\path(A)--++(-90:0.1) node { \color{white}$\up$ };

		\path[name path=pathd3] (d3)--++(90:7);   
		\path[name path=top] (corner2)--(corner3);   
		\path [name intersections={of = pathd3 and top}];
		\coordinate (A)  at (intersection-1);
		\path(A)--++(-90:0.1) node { \color{white}$\up$ };

		\path[name path=pathd5] (d5)--++(90:7);   
		\path[name path=top] (corner2)--(corner3);   
		\path [name intersections={of = pathd5 and top}];
		\coordinate (A)  at (intersection-1);
		\path(A)--++(-90:0.1) node { \color{white}$\up$ };

		\path[name path=pathd7] (d7)--++(90:7);   
		\path[name path=top] (corner2)--(corner3);   
		\path [name intersections={of = pathd7 and top}];
		\coordinate (A)  at (intersection-1);
		\path(A)--++(-90:-0.1) node { \color{white}$\down$ };

		\path[name path=pathd9] (d9)--++(90:7);   
		\path[name path=top] (corner2)--(corner3);   
		\path [name intersections={of = pathd9 and top}];
		\coordinate (A)  at (intersection-1);
		\path(A)--++(-90:-0.1) node { \color{white}$\down$ };

		\path[name path=pathc1] (c1)--++(90:7);   
		\path[name path=top] (corner2)--(corner3);   
		\path [name intersections={of = pathc1 and top}];
		\coordinate (A)  at (intersection-1);
		\path(A)--++(-90:0.1) node { \color{white}$\up$ };

		\path[name path=pathc3] (c3)--++(90:7);   
		\path[name path=top] (corner2)--(corner3);   
		\path [name intersections={of = pathc3 and top}];
		\coordinate (A)  at (intersection-1);
		\path(A)--++(-90:-0.1) node { \color{white}$\down$ };

		\path[name path=pathc5] (c5)--++(90:7);   
		\path[name path=top] (corner2)--(corner3);   
		\path [name intersections={of = pathc5 and top}];
		\coordinate (A)  at (intersection-1);
		\path(A)--++(-90:-0.1) node { \color{white}$\down$ };

		\path[name path=pathc7] (c7)--++(90:7);   
		\path[name path=top] (corner2)--(corner3);   
		\path [name intersections={of = pathc7 and top}];
		\coordinate (A)  at (intersection-1);
		\path(A)--++(-90:0.1) node { \color{white}$\up$ };

		\path[name path=pathd1] (d1)--++(-90:7);   
		\path[name path=bottom] (corner1)--(corner4);   
		\path [name intersections={of = pathd1 and bottom}];
		\coordinate (A)  at (intersection-1);
		\path (A)--++(90:-0.1) node { \color{white}$\up$  };
		
		\path[name path=pathd3] (d3)--++(-90:7);   
		\path[name path=bottom] (corner1)--(corner4);   
		\path [name intersections={of = pathd3 and bottom}];
		\coordinate (A)  at (intersection-1);
		\path (A)--++(90:-0.1) node { \color{white}$\up$  };

		\path[name path=pathd5] (d5)--++(-90:7);   
		\path[name path=bottom] (corner1)--(corner4);   
		\path [name intersections={of = pathd5 and bottom}];
		\coordinate (A)  at (intersection-1);
		\path (A)--++(90:-0.1) node { \color{white}$\up$  };

		\path[name path=pathd7] (d7)--++(-90:7);   
		\path[name path=bottom] (corner1)--(corner4);   
		\path [name intersections={of = pathd7 and bottom}];
		\coordinate (A)  at (intersection-1);
		\path (A)--++(90:-0.1) node { \color{white}$\up$  };
		
		\path[name path=pathd9] (d9)--++(-90:7);   
		\path[name path=bottom] (corner1)--(corner4);   
		\path [name intersections={of = pathd9 and bottom}];
		\coordinate (A)  at (intersection-1);
		\path (A)--++(90:-0.1) node { \color{white}$\up$  };

		\path[name path=pathc1] (c1)--++(-90:7);   
		\path[name path=bottom] (corner1)--(corner4);   
		\path [name intersections={of = pathc1 and bottom}];
		\coordinate (A)  at (intersection-1);
		\path (A)--++(90:0.1) node { \color{white}$\down$  };

		\path[name path=pathc3] (c3)--++(-90:7);   
		\path[name path=bottom] (corner1)--(corner4);   
		\path [name intersections={of = pathc3 and bottom}];
		\coordinate (A)  at (intersection-1);
		\path (A)--++(90:0.1) node { \color{white}$\down$  };

		\path[name path=pathc5] (c5)--++(-90:7);   
		\path[name path=bottom] (corner1)--(corner4);   
		\path [name intersections={of = pathc5 and bottom}];
		\coordinate (A)  at (intersection-1);
		\path (A)--++(90:0.1) node { \color{white}$\down$  };

		\path[name path=pathc7] (c7)--++(-90:7);   
		\path[name path=bottom] (corner1)--(corner4);   
		\path [name intersections={of = pathc7 and bottom}];
		\coordinate (A)  at (intersection-1);
		\path (A)--++(90:0.1) node { \color{white}$\down$  };
		
		\clip(corner1)--(corner2)--(corner3)--(corner4)--(corner1);

		\foreach \i in {1,3,5,7,9,11}
		{
			\draw[thick](c\i)--++(90:7);   
			\draw[thick](c\i)--++(-90:7);
			\draw[thick](d\i)--++(90:7);
			\draw[thick](d\i)--++(-90:7);
		}

	\end{scope}

	\begin{scope}
 			\clip(corner1)--(corner2)--(corner3)--(corner4)--(corner1);
	
		\path (0,0) coordinate (origin2);

		\foreach \i\j in {0,1,2,3,4,5,6,7,8,9,10,11,12}
		{
			\path (origin2)--++(45:0.5*\i) coordinate (a\i); 
			\path (origin2)--++(135:0.5*\i)  coordinate (b\j); 


		}

		\path(0,0)--++(135:5) coordinate(pirir);
		
		\fill[white] 
		(pirir) --++(45:2) --++(-45:1)--++(45:1) --++(-45:2)--++(45:1) --++(-45:2)
		--++(-135:2)--++(135:2)
	 --++(-135:1)
	 --++(135:2)	 --++(-135:1)
	 --++(135:1);

		\draw(a1)  ++(135:0.5)   ++(-135:0.5) coordinate(next1);

		\draw(a3)  ++(135:0.5)   ++(-135:0.5) coordinate(upnext1);

		\draw(upnext1)  ++(135:0.5)   ++(-135:0.5) coordinate(upnext2);

		\draw(upnext2)  ++(135:0.5)   ++(-135:0.5) coordinate(upnext3);


		\draw(a5)  ++(135:0.5)   ++(-135:0.5) coordinate(upnext1);
		\draw[very thick,cyan](a5) to [out=90,in=90] (upnext1);

		\draw(upnext1)  ++(135:0.5)   ++(-135:0.5) coordinate(upnext2);

		\draw(upnext2)  ++(135:0.5)   ++(-135:0.5) coordinate(upnext3);

		\draw(upnext3)  ++(135:0.5)   ++(-135:0.5) coordinate(upnext4);

		\draw(upnext4)  ++(135:0.5)   ++(-135:0.5) coordinate(upnext5);


		\draw(a7)  ++(135:0.5)   ++(-135:0.5) coordinate(upnext1);
		\draw[very thick,pink](a7) to [out=90,in=90] (upnext1);

		\draw(upnext1)  ++(135:0.5)   ++(-135:0.5) coordinate(upnext2);
		\draw[very thick,cyan](upnext1) to [out=-90,in=-90] (upnext2);

		\draw(upnext2)  ++(135:0.5)   ++(-135:0.5) coordinate(upnext3);
		\draw[very thick,gray](upnext2) to [out=90,in=90] (upnext3);

		\path(upnext2) to [out=90,in=90] (upnext3);

		\draw(upnext3)  ++(135:0.5)   ++(-135:0.5) coordinate(upnext4);

		\draw(upnext4)  ++(135:0.5)   ++(-135:0.5) coordinate(upnext5);
		\draw[very thick,darkgreen](upnext4) to [out=90,in=90] (upnext5);

		\draw(upnext5)  ++(135:0.5)   ++(-135:0.5) coordinate(upnext6);

		\draw(upnext6)  ++(135:0.5)   ++(-135:0.5) coordinate(upnext7);


		\draw(a9)  ++(135:0.5)   ++(-135:0.5) coordinate(upnext1);
		\path(a9) to [out=90,in=90] (upnext1);

		\draw(upnext1)  ++(135:0.5)   ++(-135:0.5) coordinate(upnext2);
		\draw[very thick,pink](upnext1) to [out=-90,in=-90] (upnext2);

		\draw(upnext2)  ++(135:0.5)   ++(-135:0.5) coordinate(upnext3);
\draw[very thick,cyan](upnext2) to [out=90,in=90] (upnext3);		
		
		\path(upnext2) to [out=90,in=90] (upnext3);

		\draw(upnext3)  ++(135:0.5)   ++(-135:0.5) coordinate(upnext4);
		\draw[very thick,gray](upnext3) to [out=-90,in=-90] (upnext4);

		\path(upnext3) to [out=-90,in=-90] (upnext4);

		\draw(upnext4)  ++(135:0.5)   ++(-135:0.5) coordinate(upnext5);
	\draw[very thick,magenta](upnext4) to [out=90,in=90] (upnext5);

	\path(upnext4) to [out=90,in=90] (upnext5);

		\draw(upnext5)  ++(135:0.5)   ++(-135:0.5) coordinate(upnext6);
		\draw[very thick, darkgreen ](upnext5) to [out=-90,in=-90] (upnext6);

		\draw(upnext6)  ++(135:0.5)   ++(-135:0.5) coordinate(upnext7);
		\draw[very thick,orange ](upnext6) to [out=90,in=90] (upnext7);

		\draw(upnext7)  ++(135:0.5)   ++(-135:0.5) coordinate(upnext8);

		\draw(upnext8)  ++(135:0.5)   ++(-135:0.5) coordinate(upnext9);
		\draw[very thick, violet ](upnext8) to [out=90,in=90] (upnext9);


		\path(a4)  --++(135:3.5)    coordinate(upnext6);

		\path	(upnext6)--++(45:0.5)--++(135:0.5)  coordinate (upnext00);
		\path	(upnext6)--++(45:0.5)--++(135:0.5) --++(45:0.5)--++(-45:0.5)  coordinate (upnext001);

			\draw[very thick,darkgreen ](upnext00) to[out=-90,in=-90]   (upnext001);

		\path	(upnext6)--++(45:0.5)--++(-45:0.5)  coordinate (upnext67);
			\draw[very thick,darkgreen]
		(upnext6) to  [out=90,in=90]   (upnext67);

				\path	(upnext67)--++(45:0.5)--++(-45:0.5)  coordinate (upnext68);
		
			\draw[very thick,magenta ](upnext67) to[out=-90,in=-90]   (upnext68);

				\path	(upnext68)--++(45:0.5)--++(-45:0.5)  coordinate (upnext69);
								\path	(upnext69)--++(45:0.5)--++(-45:0.5)  coordinate (upnext70);
			\draw[very thick,cyan ](upnext69) to[out=-90,in=-90]   (upnext70);

		\draw(upnext6)  ++(135:0.5)   ++(-135:0.5) coordinate(upnext7);
		\draw[very thick,orange ](upnext6) to[out=-90,in=-90]   (upnext7);

		\draw(upnext7)  ++(135:0.5)   ++(-135:0.5) coordinate(upnext8);
		\draw[very thick,lime ]
		(upnext7) to  [out=90,in=90]   (upnext8);

		\draw(upnext8)  ++(135:0.5)   ++(-135:0.5) coordinate(upnext9);
		\draw[very thick, violet ](upnext8) to[out=-90,in=-90] (upnext9);
		

		\path(a4)  --++(135:4.5)    coordinate(upnext6);
		
		\draw(upnext6)  ++(135:0.5)   ++(-135:0.5) coordinate(upnext7);
		\draw[very thick,lime ](upnext6) to[out=-90,in=-90]   (upnext7);

		\foreach \i in {0,1,2,3,4,5,6,7,8,9,10,11,12}
		{
			\path (origin2)--++(-45:2)--++(-135:2) --++(45:1*\i) coordinate (c\i); 
			\path (origin2)--++(-45:2)--++(-135:2) --++(135:1*\i)  coordinate (d\i); 
			\draw[thick,densely dotted] (c\i)--++(135:14);
			\draw[thick,densely dotted] (d\i)--++(45:14);
		}

	\end{scope}
	
	\path(origin2)  ++(135:2.5)   ++(-135:2.5) coordinate(corner1);
	\path(origin2)  ++(45:2)   ++(135:7) coordinate(corner2);
	\path(origin2)  ++(45:2)   ++(-45:2) coordinate(corner4);
	\path(origin2)  ++(135:2.5)   ++(45:6.5) coordinate(corner3);

	\draw[thick] (origin2)--(corner1)--(corner2)--(corner3)--(corner4)--(origin2);

	\path(origin2)--++(135:5) coordinate (X);
	\draw[line width=1.5,->](X)--++(45:1) coordinate (X);  
	\draw[line width=1.5,->](X)--++(45:1) coordinate (X);  
	\draw[line width=1.5,->](X)--++(-45:1) coordinate (X);  	 
	 	\draw[line width=1.5,->](X)--++(45:1) coordinate (X);  
	\draw[line width=1.5,->](X)--++(-45:1) coordinate (X);  	 	\draw[line width=1.5,->](X)--++(-45:1) coordinate (X);  	 	\draw[line width=1.5,->](X)--++(45:1) coordinate (X);  	 	 
	\draw[line width=1.5,->](X)--++(-45:1) coordinate (X);  	 	\draw[line width=1.5,->](X)--++(-45:1) coordinate (X);

	\draw[line width=1.5,->](X)--++(-135:1) coordinate (X);   \draw[line width=1.5,->](X)--++(-135:1) coordinate (X);  	\draw[line width=1.5,->](X)--++( 135:1) coordinate (X);   \draw[line width=1.5,->](X)--++(135:1) coordinate (X);  
	\draw[line width=1.5,->](X)--++(-135:1) coordinate (X);	
		\draw[line width=1.5,->](X)--++( 135:1) coordinate (X);   \draw[line width=1.5,->](X)--++(135:1) coordinate (X);  	\draw[line width=1.5,->](X)--++(-135:1) coordinate (X);	
		\draw[line width=1.5,->](X)--++( 135:1) coordinate (X);
	 
	 	\draw[fill=white,line width=1.5](X) circle (2.5pt);
  
 	\clip(origin2)--(corner1)--(corner2)--(corner3)--(corner4)--(origin2);

		\path(0,0)--++(135:1) coordinate (graystart);

			\path(graystart)--++(135:2) coordinate (graystart);
\fill[gray,opacity=0.2]  
(graystart) --++(0:1.414*1)--++(45:1)
--++(-45:1)--++(0:1*1.414)
--++(45:1)--++(180:3*1.414)--++(-45:1);

			\path(graystart)--++(135:2) coordinate (graystart);
	\fill[gray,opacity=0.2]  (graystart) --++(0:1.414*4)--++(135:1)
	--++(-135:1)--++(135:1)--++(180:2*1.414);

\end{tikzpicture}
$$
\caption{The tilings of  the diagrams  from \cref{typeAtiling2lkjsdfghslkdfjghsldk}.  
In each case the path begins at the western-most point, which is denoted with a circle; the path then  follows the orientation depicted on the diagram.
An example of the associated $\diag_\w$  for the rightmost diagram  is 
$\diag_\w=
\color{cyan}
\diag_{7}
\color{darkgreen}
\diag_{4}\color{gray} \diag_{6} \color{pink}\diag_{8}
\color{violet}
\diag_{1}\color{orange} \diag_{3} \color{magenta}\diag_{5}\color{cyan}
\diag_{7}\color{green}
\diag_{2}\color{darkgreen}
\diag_{4}
$. 		}
\label{resulting}
\end{figure}

\begin{figure}[ht!]
$$\hspace{-0.4cm}

$$
\caption{The tiling  of the two leftmost diagrams from \cref{typeAtiling2lkjsdfghslkdfjghsldkXXX} and an additional one in type $(C_9,A_8)$.
The path begins at the northerly western-most point; the path then  follows the orientation depicted on the diagram until terminating at the southernly western-most point (both of which are  denoted with  circles). Some of the  strands in  type $C$ carry two dots which can be simplified to one dot. }
\label{resulting2XXX}
\end{figure}

We now explain how  the  region  $R(d)$  defines a reduced word of a 
strongly fully commutative element of $W$.   
Enumerate the tiles in $R(d)$, $t_{i_1}, \ldots , t_{i_r}$,  in such a way that for each $j$, the tiles in $R(d)$ to the ${\rm SW}$ and ${\rm SE}$ of $t_{i_j}$ appear  before $t_{i_j}$. Taking the corresponding ordered product  of simple reflections gives a word $\underline{w}=\underline{w}(d)$ in the elements of $S_W$. There is, of course, more than one way of enumerating the tiles in $R(d)$ in this fashion and  these give  (all the) different reduced expressions for the same element of $W$ (as they differ only by commutation relations).

Conversely, any expression $\underline{w} = s_{i_1}s_{i_2} \ldots s_{i_k}$ with  $s_{i_j}\in S_W$ defines a region $R(\underline{w})$ in the tiling by picking any horizontal line in the tiling and stacking the boxes in order, starting from $s_{i_1}$, then $s_{i_2}$, \ldots, and finally $s_{i_k}$. (This is an alternative description of the \lq heaps' introduced by Stembridge \cite{MR1406459}).

\begin{prop}\label{Rw}  
\begin{enumerate}[leftmargin=*]
\item 
Let $\underline{w}$ is a reduced expression for a strongly fully commutative element of $W$ such that $\underline{w} = \omega(\SSTT)$ for some $T\in {\rm Path}_{(W,P)}$. Then $R(\underline{w})$ does not have a boundary containing an inadmissible section of the form depicted in \cref{resulting2XXXY}. 
\item For any $d\in \mathbb{DT}[W,P]$, the region $R(d)$ does not have a boundary containing an inadmissible section of the form depicted in \cref{resulting2XXXY}. In particular, $\underline{w}(d)$ is a reduced expression for a strongly fully commutative element of $W$. 
\item For any $d\in \mathbb{DT}[W,P]$,  if $\underline{w}(d) = \underline{w} = s_{i_1} s_{i_2} \ldots s_{i_k}$ then we have
\[d = \diag_{\underline{w}} := \diag_{i_1} \diag_{i_2} \ldots \diag_{i_k}.\]
Moreover, if $\underline{w'}$ is another reduced expression for the same strongly fully commutative element of $W$ then $\diag_{\underline{w'}} = \diag_{\underline{w}}$.
\end{enumerate}

\begin{figure}[ht!]
$$
\begin{tikzpicture}[xscale=1]

\path(0,0) 
coordinate (newstart);

	\path(newstart)  coordinate (X);

	\fill[magenta!20] (X)
	--++(-45:1)--++(-135:1) 	--++(-45:1)--++(-135:1)
	--++(135:1)--++(45:1) 	--++(135:1)--++(45:1) ;
	
	\path(X) --++(-45:0.5)--++(-135:0.5)  coordinate (here);	
		\path(here) node {$i$};
	\path(X) --++(-45:1.5)--++(-135:1.5)  coordinate (here);	
		\path(here) node {$i$};

 	\draw[densely dotted, thick](X)--++(135:0.2) ; 
 	\draw[densely dotted, thick](X)--++(45:0.2) ;  
	\draw[densely dotted, thick](X)--++(-135:1) coordinate (Y1) --++(-45:1)  --++(-135:1)coordinate (Y2) --++(-45:1)  --++(45:1)    ;
	\draw[densely dotted, thick](X)
--++(-45:1);

	\draw[densely dotted, thick](Y1)--++(-135:0.2) ;
 
		\draw[densely dotted, thick](Y2)--++(-135:0.2) ;
	\draw[densely dotted, thick](Y1)--++(135:0.2) ;
		\draw[densely dotted, thick](Y2)--++(135:0.2) ;

	  	\path(X)--++(-135: 2)--++(-45: 2) coordinate (T);
 	\draw[densely dotted, thick](T)--++(-135:0.2) ; 
 	\draw[densely dotted, thick](T)--++(-45:0.2) ;

\path(X)--++(-45:1) coordinate (zig);
	 	 
	\draw[line width= 2.2] 
	(zig)   --++(-135:1) --++(-45:1) ;

	\draw[densely dotted, thick](X)--++(-135:0.2) ; 

  \end{tikzpicture}
\qquad\quad
\begin{tikzpicture}[xscale=-1]

\path(0,0) 
coordinate (newstart);

	\path(newstart)  coordinate (X);

	\fill[cyan!20] (X)
	--++(-45:1)--++(-135:1) 	--++(-45:1)--++(-135:1)
	--++(135:1)--++(45:1) 	--++(135:1)--++(45:1) ;
	
	\path(X) --++(-45:0.5)--++(-135:0.5)  coordinate (here);	
		\path(here) node {$j$};
	\path(X) --++(-45:1.5)--++(-135:1.5)  coordinate (here);	
		\path(here) node {$j$};

 	\draw[densely dotted, thick](X)--++(135:0.2) ; 
 	\draw[densely dotted, thick](X)--++(45:0.2) ;  
	\draw[densely dotted, thick](X)--++(-135:1) coordinate (Y1) --++(-45:1)  --++(-135:1)coordinate (Y2) --++(-45:1)  --++(45:1)    ;
	\draw[densely dotted, thick](X)
--++(-45:1);

	\draw[densely dotted, thick](Y1)--++(-135:0.2) ;
 
		\draw[densely dotted, thick](Y2)--++(-135:0.2) ;
	\draw[densely dotted, thick](Y1)--++(135:0.2) ;
		\draw[densely dotted, thick](Y2)--++(135:0.2) ;

	  	\path(X)--++(-135: 2)--++(-45: 2) coordinate (T);
 	\draw[densely dotted, thick](T)--++(-135:0.2) ; 
 	\draw[densely dotted, thick](T)--++(-45:0.2) ;

\path(X)--++(-45:1) coordinate (zig);
	 	 
	\draw[line width= 2.2] 
	(zig)   --++(-135:1) --++(-45:1) ;

	\draw[densely dotted, thick](X)--++(-135:0.2) ; 

  \end{tikzpicture}
  \qquad 
  \begin{tikzpicture}[xscale=1]

\path(0,0) 
coordinate (newstart);

	\path(newstart)  coordinate (X);

	\fill[orange!20] (X)
	--++(-45:1)--++(-135:1) 	 coordinate (change)
	--++(135:1)--++(45:1) 	 ;
		\fill[green!20] (change)
	--++(-45:1)--++(-135:1) 	 coordinate (change)
	--++(135:1)--++(45:1) 	 ;

	\path(X) --++(-45:0.5)--++(-135:0.5)  coordinate (here);	
		\path(here) node {$1$};
	\path(X) --++(-45:1.5)--++(-135:1.5)  coordinate (here);	
		\path(here) node {$1''$};

 	\draw[densely dotted, thick](X)--++(135:0.2) ; 
 	\draw[densely dotted, thick](X)--++(45:0.2) ;  
	\draw[densely dotted, thick](X)--++(-135:1) coordinate (Y1) --++(-45:1)  --++(-135:1)coordinate (Y2) --++(-45:1)  --++(45:1)    ;
	\draw[densely dotted, thick](X)
--++(-45:1);

%

	  	\path(X)--++(-135: 2)--++(-45: 2) coordinate (T);
 	\draw[densely dotted, thick](T)--++(-135:0.2) ; 
 	\draw[densely dotted, thick](T)--++(-45:0.2) ;

\path(X)--++(-45:1) coordinate (zig);
	 	 
	\draw[line width= 2.2] 
	(zig)   --++(-135:1) --++(-45:1) ;

	\draw[densely dotted, thick](X)--++(-135:0.2) ; 

  \end{tikzpicture}
\qquad  
  \begin{tikzpicture}[yscale=-1]

\path(0,0) 
coordinate (newstart);

	\path(newstart)  coordinate (X);

	\fill[orange!20] (X)
	--++(-45:1)--++(-135:1) 	 coordinate (change)
	--++(135:1)--++(45:1) 	 ;
		\fill[green!20] (change)
	--++(-45:1)--++(-135:1) 	 coordinate (change)
	--++(135:1)--++(45:1) 	 ;

	\path(X) --++(-45:0.5)--++(-135:0.5)  coordinate (here);	
		\path(here) node {$1$};
	\path(X) --++(-45:1.5)--++(-135:1.5)  coordinate (here);	
		\path(here) node {$1''$};

 	\draw[densely dotted, thick](X)--++(135:0.2) ; 
 	\draw[densely dotted, thick](X)--++(45:0.2) ;  
	\draw[densely dotted, thick](X)--++(-135:1) coordinate (Y1) --++(-45:1)  --++(-135:1)coordinate (Y2) --++(-45:1)  --++(45:1)    ;
	\draw[densely dotted, thick](X)
--++(-45:1);

%

	  	\path(X)--++(-135: 2)--++(-45: 2) coordinate (T);
 	\draw[densely dotted, thick](T)--++(-135:0.2) ; 
 	\draw[densely dotted, thick](T)--++(-45:0.2) ;

\path(X)--++(-45:1) coordinate (zig);
	 	 
	\draw[line width= 2.2] 
	(zig)   --++(-135:1) --++(-45:1) ;

	\draw[densely dotted, thick](X)--++(-135:0.2) ; 

  \end{tikzpicture}$$
\caption{ 
We emphasise the inadmissible section of the  border-region by drawing it thickly.  
In type $A_{n-1}$ we require that there is no such region for any $i,j \in \{1,2,\dots,n-1\}$.  
In type $D_{n}$ we require that there is no such region for any $i \in \{2,\dots,n-1\}$ or $j \in \{2,\dots,n-1\}$.    
In type $C_{n-1}$ we require that there is no such region for any $i \in \{1',2,\dots,n-1\}$ or $j \in \{2,\dots,n-1\}$.    
We note that the latter two regions (involving a $1$ and a $1''$ tile) occur only in type $D_{n-1}$.}
\label{resulting2XXXY}
\end{figure}

\end{prop}

\begin{proof}
\begin{enumerate}[leftmargin=*]
\item The conditions on $i,j$ follow from \cref{strongly,} and the  definition  of  a  reduced expression (see \cref{reduced-defn}).
 For the conditions involving $1$ and $1''$, we  note that any path on $\widehat{\mathcal{G}}_{(D_n, A_{n-1})}$ containing an $s_1$-step and an $s_{1''}$-step must contain an $s_2$-step in between. 
 \item Note that the path $\pi(d)$ takes all steps to the East first, followed by all steps to the West, this implies that  boundary of $R(d)$ does not contain any inadmissible sections and hence $\underline{w}(d)$ is a reduced expression for a fully commutative element of $W$. 
 \item Now recall that for each simple reflection $s_i\in S_W$, we have a generator $\diag_i\in {\rm TL}_{W}(q)$, and so any reduced expression $\underline{w} = s_{i_1}s_{i_2} \ldots s_{i_k}$  for a strongly fully commutative element of $W$ define a product of the generators, $\diag_{\underline{w}} = \diag_{i_1}  \diag_{i_2} \ldots \diag_{i_k}\in {\rm TL}_W(q)$. These are pictured in  \cref{resulting,resulting2XXX}  for $\underline{w} = \underline{w}(d)$. Again, it is clear that for a different choice of reduced word $\underline{w}'$ we have $\diag_{\underline{w}'} = \diag_{\underline{w}}$ (as these differ only by commutation relations). We claim that $\diag_{\underline{w}} = d$. To see this, we first partition the region $R(d)$ by splitting it along 
the horizontal lines through the vertices of the tiles. 
This gives a partition of $R(d)$ into horizontal strips.
 If a strip intersects the left boundary in more than one point, merge it with the strip above or below so that the new wider strip now contains precisely two vertices of tiles on the left boundary. 
 These strips are shown in   \cref{resulting,resulting2XXX} (by  alternating between grey and white shading). 
  It is clear that the strands in $\diag_{\underline{w}}$ are in one-to-one correspondence with the strips of $R(d)$. 
Each strip contains precisely two edges of the path $\pi(d)$ and joins the corresponding vertices in $d$; 
one thus  recovers the (decorated) $n$-tangle $d$. 
\end{enumerate}
\end{proof}

\begin{prop}\label{uniquepath}
Let $\underline{w}$ be a reduced expression for a strongly fully commutative element of $W$. Suppose that we have $\lambda, \mu\in {^PW}$ with $\lambda \diag_{\underline{w}}\mu \in \mathbb{ODT}[W,P]$. Then there is a unique path $\SSTT$ on $\widehat{\mathcal{G}}_{(W,P)}$ starting at $\lambda$ and ending at $\mu$ with $w(\SSTT) = \underline{w}$. Moreover, we have $\lambda \diag_{\underline{w}} \mu = \diag_\SSTT$.
\end{prop}

\begin{proof}
Having fixed the orientation at the top and bottom of the product of generators $\diag_{\underline{w}}$, there are only two situations in which we have a choice of orientation for strand segments in this product. The first one is when the diagram contains a closed loop. These are formed by a tile configuration of the form depicted on the left of \cref{resulting2XXXZ}.
 The second one is when $(W,P)=(C_{n-1},A_{n-2})$ and we have  a tile configuration of the form depicted on the right of \cref{resulting2XXXZ}.  
By  \cref{Rw}, neither of these can happen. Therefore the top and bottom orientations uniquely determine  the orientation of every strand segment in the diagram, proving the result.
\end{proof}

 \begin{figure}[ht!]
$$ \scalefont{0.9}
\begin{tikzpicture}[xscale=1]
 \clip (-4.1,-3) rectangle (1,0.4);
 
\path(0,0) 
coordinate (newstart);

	\path(newstart)  coordinate (X);

	\fill[magenta!20] (X)
	--++(-45:1)--++(-135:1) 	--++(-45:1)--++(-135:1)
	--++(135:1)--++(45:1) 	--++(135:1)--++(45:1) ;
	
	\path(X) --++(-45:0.5)--++(-135:0.5)  coordinate (here);	
		\path(here) --++(90:0.05) node {$i$};
	\path(X) --++(-45:1.5)--++(-135:1.5)  coordinate (here);	
		\path(here)--++(-90:0.05) node {$i$};

 	\draw[densely dotted, thick](X)--++(135:0.2) ; 
 	\draw[densely dotted, thick](X)--++(45:0.2) ;  
	\draw[densely dotted, thick](X)--++(-135:1) coordinate (Y1) --++(-45:1)  --++(-135:1)coordinate (Y2) --++(-45:1)  --++(45:1)    ;
	\draw[densely dotted, thick](X)
--++(-45:1);

	\draw[densely dotted, thick](Y1)--++(-135:0.2) ;
 
		\draw[densely dotted, thick](Y2)--++(-135:0.2) ;
	\draw[densely dotted, thick](Y1)--++(135:0.2) ;
		\draw[densely dotted, thick](Y2)--++(135:0.2) ;

	  	\path(X)--++(-135: 2)--++(-45: 2) coordinate (T);
 	\draw[densely dotted, thick](T)--++(-135:0.2) ; 
 	\draw[densely dotted, thick](T)--++(-45:0.2) ;

\path(X)--++(-45:1) coordinate (zig);
	 	 
	\draw[line width= 2.2]  
	(zig)   --++(-135:1) --++(-45:1) ;

	\path(newstart)--++(135:2)--++(-135:2)  coordinate (ZZ);

	\fill[cyan!20] (ZZ)
	--++(-45:1)--++(-135:1) 	--++(-45:1)--++(-135:1)
	--++(135:1)--++(45:1) 	--++(135:1)--++(45:1) ;
	
	\path(ZZ) --++(-45:0.5)--++(-135:0.5)  coordinate (here);	
		\path(here) --++(90:0.05) node {$j$};
	\path(ZZ) --++(-45:1.5)--++(-135:1.5)  coordinate (here);	
		\path(here)--++(-90:0.05) node {$j$};

 	\draw[densely dotted, thick](ZZ)--++(45:0.2) ; 
 	\draw[densely dotted, thick](ZZ)--++(135:0.2) ;  
	\draw[densely dotted, thick](ZZ)--++(-45:1) coordinate (Y1) --++(-135:1)  --++(-45:1)coordinate (Y2) --++(-135:1)  --++(135:1)    ;
	\draw[densely dotted, thick](ZZ)
--++(-135:1);

	\draw[densely dotted, thick](Y1)--++(-45:0.2) ;
 
		\draw[densely dotted, thick](Y2)--++(-45:0.2) ;
	\draw[densely dotted, thick](Y1)--++(45:0.2) ;
		\draw[densely dotted, thick](Y2)--++(45:0.2) ;

	  	\path(ZZ)--++(-45: 2)--++(-135: 2) coordinate (T);
 	\draw[densely dotted, thick](T)--++(-45:0.2) ; 
 	\draw[densely dotted, thick](T)--++(-135:0.2) ;

\path(ZZ)--++(-135:1) coordinate (zig);
	 	 
	\draw[line width= 2.2]  
	(zig)   --++(-45:1) --++(-135:1) ;

	\draw[densely dotted, thick](X)--++(-135:0.2) ;

\path(X)--++(-45:1)--++(-135:0.5) coordinate (here);	
\path(here)--++(-135:0.5)--++(135:0.5) coordinate (there);	

\draw[very thick] (here) to [out=90,in=90] (there);

\path(there)--++(-135:0.5)--++(135:0.5) coordinate (here2);	

\path(here2)--++(-135:0.5)--++(135:0.5) coordinate (here3);	

\path(here3)--++(-135:0.5)--++(135:0.5) coordinate (here4);	
\path(here4)--++(-135:0.5)--++(135:0.5) coordinate (here5);

\draw[densely dotted,very thick] (there) to [out=-90,in=-90] (here2) to [out=90,in=90] (here3)
to [out=-90,in=-90] (here4);
\draw	[very thick] 		(here4) to [out=90,in=90]  (here5);

\path(here)--++(-45:0.5)--++(-135:0.5) coordinate (Here);	
\path(Here)--++(-135:0.5)--++(135:0.5) coordinate (tHere);	

\draw[very thick] (Here) to [out=-90,in=-90] (tHere);

\path(tHere)--++(-135:0.5)--++(135:0.5) coordinate (Here2);	

\path(Here2)--++(-135:0.5)--++(135:0.5) coordinate (Here3);	

\path(Here3)--++(-135:0.5)--++(135:0.5) coordinate (Here4);	
\path(Here4)--++(-135:0.5)--++(135:0.5) coordinate (Here5);

\draw[densely dotted,very thick] (tHere) to [out=90,in=90] (Here2) to [out=-90,in=-90] (Here3)
to [out=90,in=90] (Here4);
\draw	[very thick] 		(Here4) to [out=-90,in=-90]  (Here5);

\draw	[very thick] (Here5)--(here5);
\draw	[very thick] (Here)--(here);

  \end{tikzpicture}
\qquad\quad\quad\quad
\begin{tikzpicture}[xscale=1]

  \clip (-4.1,-3) rectangle (1,0.4);
\path(0,0) 
coordinate (newstart);

	\path(newstart)  coordinate (X);

	\fill[magenta!20] (X)
	--++(-45:1)--++(-135:1) 	--++(-45:1)--++(-135:1)
	--++(135:1)--++(45:1) 	--++(135:1)--++(45:1) ;
	
	\path(X) --++(-45:0.5)--++(-135:0.5)  coordinate (here);	
		\path(here) --++(90:0.05) node {$i$};
	\path(X) --++(-45:1.5)--++(-135:1.5)  coordinate (here);	
		\path(here)--++(-90:0.05) node {$i$};

 	\draw[densely dotted, thick](X)--++(135:0.2) ; 
 	\draw[densely dotted, thick](X)--++(45:0.2) ;  
	\draw[densely dotted, thick](X)--++(-135:1) coordinate (Y1) --++(-45:1)  --++(-135:1)coordinate (Y2) --++(-45:1)  --++(45:1)    ;
	\draw[densely dotted, thick](X)
--++(-45:1);

	\draw[densely dotted, thick](Y1)--++(-135:0.2) ;
 
		\draw[densely dotted, thick](Y2)--++(-135:0.2) ;
	\draw[densely dotted, thick](Y1)--++(135:0.2) ;
		\draw[densely dotted, thick](Y2)--++(135:0.2) ;

	  	\path(X)--++(-135: 2)--++(-45: 2) coordinate (T);
 	\draw[densely dotted, thick](T)--++(-135:0.2) ; 
 	\draw[densely dotted, thick](T)--++(-45:0.2) ;

\path(X)--++(-45:1) coordinate (zig);
	 	 
	\draw[line width= 2.2]  
	(zig)   --++(-135:1) --++(-45:1) ;

	\path(newstart)--++(135:2)--++(-135:2)  coordinate (ZZ);

  	\draw[densely dotted, thick](ZZ)--++(45:0.2) ; 
	\draw[densely dotted, thick](ZZ)--++(-45:1) coordinate (Y1) --++(-135:1)  --++(-45:1)coordinate (Y2) --++(-135:1)  --++(135:1)    ;
	\draw[densely dotted, thick](ZZ)
--++(-135:1);

	\draw[densely dotted, thick](Y1)--++(-45:0.2) ;
 
		\draw[densely dotted, thick](Y2)--++(-45:0.2) ;
	\draw[densely dotted, thick](Y1)--++(45:0.2) ;
		\draw[densely dotted, thick](Y2)--++(45:0.2) ;

	  	\path(ZZ)--++(-45: 2)--++(-135: 2) coordinate (T);
 	\draw[densely dotted, thick](T)--++(-45:0.2) ; 

\path(ZZ)--++(-135:1) coordinate (zig);
	 	 
	\draw[thick, densely dotted]  
	(zig)   --++(-45:1) --++(-135:1) ;
	\draw[thick, densely dotted]  
	(zig)   --++(-135:1) --++(-45:1) ;

	\draw[densely dotted, thick](X)--++(-135:0.2) ;

\path(X)--++(-45:1)--++(-135:0.5) coordinate (here);	
\path(here)--++(-135:0.5)--++(135:0.5) coordinate (there);	

\draw[very thick] (here) to [out=90,in=90] (there);

\path(there)--++(-135:0.5)--++(135:0.5) coordinate (here2);	

\path(here2)--++(-135:0.5)--++(135:0.5) coordinate (here3);	

\path(here3)--++(-135:0.5)--++(135:0.5) coordinate (here4);	
\path(here4)--++(-135:0.5)--++(135:0.5) coordinate (here5);	
\path(here5)--++(-135:0.5)--++(135:0.5) coordinate (here6);

\draw[densely dotted,very thick] (there) to [out=-90,in=-90] (here2) to [out=90,in=90] (here3)
to [out=-90,in=-90] (here4);
\draw	[very thick] 		(here4) to [out=90,in=90]  (here5);

\draw	[very thick] 		(here5) to [out=-90,in=-90]  (here6)--++(90:0.9) coordinate (X);	
\draw	[very thick,densely dotted] 	(X)--++(90:0.3);

\path(here)--++(-45:0.5)--++(-135:0.5) coordinate (Here);	
\path(Here)--++(-135:0.5)--++(135:0.5) coordinate (tHere);	

\draw[very thick] (Here) to [out=-90,in=-90] (tHere);

\path(tHere)--++(-135:0.5)--++(135:0.5) coordinate (Here2);	

\path(Here2)--++(-135:0.5)--++(135:0.5) coordinate (Here3);	

\path(Here3)--++(-135:0.5)--++(135:0.5) coordinate (Here4);	
\path(Here4)--++(-135:0.5)--++(135:0.5) coordinate (Here5);	
\path(Here5)--++(-135:0.5)--++(135:0.5) coordinate (Here6);

\draw[densely dotted,very thick] (tHere) to [out=90,in=90] (Here2) to [out=-90,in=-90] (Here3)
to [out=90,in=90] (Here4);

\draw	[very thick] 		(Here4) to [out=-90,in=-90]  (Here5);	

\draw	[very thick] 		(Here5) to [out=90,in=90]  (Here6)
 --++(-90:0.9) coordinate (X);	
\draw	[very thick,densely dotted] 	(X)--++(-90:0.3);

\draw	[very thick] (Here)--(here);

\path(here6)--++(45:0.25)--++(-45:0.25)--++(-90:0.2) coordinate (dot);
\fill (dot) circle (2.5pt);

\path(Here6)--++(45:0.25)--++(-45:0.25)--++(90:0.2) coordinate (dot);
\fill (dot) circle (2.5pt);

  \end{tikzpicture}$$
\caption{ 
On the left is the tile configuration for a closed loop. On the right the tile configuration of a propagating strand connecting the first northern and southern vertices in type $(C_{n-1},A_{n-2})$. The orientation of the closed loop is not determined by the orientation at the northern and southern edges of the diagram (obviously). The orientation of strand segment  {\em between} the two decorations is also not determined by the orientation at the northern and southern edges of the diagram. The colouring of tiles should be compared with that of \cref{resulting2XXXY}.}
 \label{resulting2XXXZ}
\end{figure}

\begin{cor}
The map $\vartheta$ is a $\mathbb{Z}[q,q^{-1}]$-module isomorphism.
\end{cor}

\begin{proof}
Using \cref{Rw,uniquepath}, we have that every oriented Temperley--Lieb diagram $\lambda d \mu = \lambda \diag_{\underline{w}} \mu = \diag_\SSTT = \vartheta (E_\SSTT)$ and so the map $\vartheta$ is surjective. 
To show that it is injective, note that the basis elements  $E_\SSTT$ for ${\rm TL}_{(W,P)}(q)$ given in \cref{spanningset} are mapped to distinct basis elements in ${\rm TL}^{\up\down}_{(W,P)}(q)$.
\end{proof}

This completes the proof of \cref{orientediso}. From now on, we identify ${\rm TL}_{(W,P)}(q) = {\rm TL}^{\up\down}_{(W,P)}(q)$ and use the diagrammatic notation for its elements.

As noted earlier, if $\underline{w}$ and $\underline{w}'$ are two reduced expressions for the same strongly fully commutative element $w$, then we have $\diag_{\underline{w}} = \diag_{\underline{w}'}$. So we will denote this element by $\diag_w$.  In particular, for each $\mu\in {^PW}$ we have the corresponding element $\diag_\mu\in {\rm TL}_W(q)$.  (We will give a closed combinatorial description of this element in Section 8.)
Now, restricting our attention to the anti-spherical module  and using  \cref{spanningset2} gives the following.

\begin{cor}\label{ungradedbasis}
The anti-spherical module  $\varnothing {\rm TL}_{(W,P)}(q)$ has basis given by 
\[\{\diag_\SSTT = \varnothing \diag_{\mu} \lambda \mid \SSTT \in \Path(\la,\underline{\mu}),  \la,\mu   \in {^PW}\}\]
In particular we have  
\[|{\rm Path}(\lambda , \underline{\mu})|\leq 1\qquad \mbox{for all $\la, \mu \in {^PW}.$}\]
 \end{cor}

Thus, the diagrammatic oriented Temperley--Lieb algebra setting (and its anti-spherical module in particular) provides a diagrammatic model for studying the matrix $\Delta_{\lambda \mu}$.  
More precisely, we have  shown that 
\[
\Delta_{\la,\mu}=\left\{ \begin{array}{ll} q^{\deg(\varnothing \diag_\mu \lambda)} & \mbox{if $\varnothing \diag_\mu \lambda\in \mathbb{ODT}[W,P]$}\\
0 & \mbox{otherwise.}\end{array}\right.
\]
In the next section, we investigate the degree of oriented Temperley--Lieb diagrams.

\section{A closed combinatorial interpretation of the grading}\label{sectiongrading}

The isomorphism $\vartheta$ given in \cref{orientediso} gives a grading on the diagrammatic oriented Temperley--Lieb algebra. Explicitly, we have that the degree of any oriented Temperley--Lieb diagram $\la d\mu = \la \diag_w \mu = \diag_{\SSTT}$ is equal to $\deg(\SSTT)$. This is computed as the sum of the degree of each step in the path $\SSTT$. 
 We now provide a closed combinatorial description of this degree in terms of the diagram $\la d \mu$ itself. 
 \begin{thm}\label{picdegree}
 We define the  {\sf degree} of an oriented Temperley--Lieb diagram by 
 assigning a degree to each of its strands and then summing over all strands.
  We define the degree of a northern arc whose rightmost vertex is labelled by $\down$ to be  $+1$,
 and the degree of a southern arc whose rightmost vertex is labelled by $\up$ to be $-1$.  All other strands are defined to have degree $0$ (in particular, all propagating strands have degree $0$). 
 Diagrammatically, we record the degree of an oriented Temperley--Lieb diagram as follows
 	 \[
\sharp\left\{  \begin{minipage}{1.2cm}
 \begin{tikzpicture}[scale=0.7]
 \draw[thick](-1,-0.155) node {$\boldsymbol\up$};
  \draw[very thick](0,0.1) node {$\boldsymbol\down$};
  \draw[very thick]  (-1,-0) to [out=-90,in=180] (-0.5,-0.75) to [out=0,in=-90] (0,0);
 \end{tikzpicture}
 \end{minipage}
%
%
 , \; 
  \begin{minipage}{1.27cm}
 \begin{tikzpicture}[xscale=-0.7,yscale=0.7]
   \draw[very thick] 
   (-1,0)  to [in=180,out=-90]  (-0.5,-0.75)  ; 
     \draw[very thick,densely dotted] 
  (-0.5,-0.75)  to [out=0,in=-90] (0,0);
 \fill 	(-0.5,-0.75) 	 circle (4pt);
 
   \draw[very thick](-1,0.15) node {$\boldsymbol\down$};
 \end{tikzpicture}
 \end{minipage}
%
%
\right\}	
-\sharp
\left\{\; \begin{minipage}{1.2cm}
 \begin{tikzpicture}[xscale=0.7,yscale=-0.7]
 \draw[thick]( 0,0.155) node {$\boldsymbol\up$};
  \draw[very thick](-1,-0.1) node {$\boldsymbol\down$};
  \draw[very thick]  (-1,-0) to [out=-90,in=180] (-0.5,-0.75) to [out=0,in=-90] (0,0);
 \end{tikzpicture}  \end{minipage}
  , \;\;
%
  \begin{minipage}{1.27cm}
 \begin{tikzpicture}[xscale=0.7,yscale=0.7]
   \draw[very thick,densely dotted] 
   (-1,0)  to [in=180,out= 90]  (-0.5,0.75)   ;
    \draw[very thick  ]
 (-0.5,0.75)   to [out=0,in= 90] (0,0);
 \fill 	(-0.5, 0.75) 	 circle (4pt);
  \draw[thick]( 0,-0.155) node {$\boldsymbol\up$};

  \end{tikzpicture}
 \end{minipage}
%
%
%
 \; 
\right\} 	 \]
where the dots are to emphasise that the lefthand-side of the arc does not contribute to the degree.
 Then for any oriented Temperley--Lieb diagram $\la d \mu = \diag_{\SSTT}$ we have
 $\deg(\la d \mu) = \deg(\SSTT)$. 	 
	  
 	\end{thm}

\begin{proof}

We first check that  the degree of the generators $\lambda \diag_i \mu$ (for $\mu = \lambda$ or $\lambda s_i$) are correct. 
Note that we have $\lambda < \lambda s_i\in {^PW}$ precisely when $\lambda = \ldots \up \down \ldots$ and $\lambda s_i = \ldots \down \up \ldots$ for $i\in \{1, \ldots n-1\}$, or $\lambda = \down \down \ldots$ and $\lambda s_{1''} = \up \up \ldots$, or $\lambda = \circ \down \ldots $ and $\lambda s_{1'} = \circ \up \ldots$. So the degrees of the generators are given by
$$
 \la \diag_ i \la 
=
\begin{cases}
1+0=1		&\text{if } \la <\la s_i \\
0-1=-1		&\text{if } \la >\la s_i \\
\end{cases}
\qquad
\la \diag_ i (\la s_i)
=
\begin{cases}
0+0=0		&\text{if } \la <\la s_i \\
 1-1=0		&\text{if } \la >\la s_i \\
\end{cases}.$$These are illustrated in \cref{breakitup} for $i\in \{1, \ldots , n-1\}$.
Comparing these with \cref{abstractdegree} proves the result for the generators.

 \begin{figure}[ht!]
 $$ 
\begin{tikzpicture} [scale=1.6]

 \path(0,0) coordinate (origin2)	;

	\draw[densely dotted] (0,0) --++(135:0.5) coordinate(NE)--++(135:0.5)coordinate(Y)	--++(-135:0.5)coordinate(NW)--++(-135:0.5) 	
	--++(-45:0.5) coordinate (SW)	
	--++(-45:0.5)	
		--++(45:0.5) coordinate (SE)	
	--++(45:0.5)	;
	
\path(NW)--(NE) coordinate [midway] (mid);
\path(mid)--++(-90:0.3) coordinate   (mid);	

			\draw[very thick] (NW) to [out=-90,in=180] (mid)
		  to [out=0,in=-90] (NE);;;

 \draw[very thick] (NE)  --++(55:0.15);
  \draw[very thick] (NE)  --++(125:0.15);

 \draw[very thick] (NW)  --++(-55:0.15);
  \draw[very thick] (NW)  --++(-125:0.15);

\draw[very thick] (SE)  --++(55:0.15);
  \draw[very thick] (SE)  --++(125:0.15);

 \draw[very thick] (SW)  --++(-55:0.15);
  \draw[very thick] (SW)  --++(-125:0.15);

\path(mid)--++(90:0.3) coordinate   (below);	

 \draw(below) node { $\boldsymbol+$};

\path(SW)--(SE) coordinate [midway] (midlow);
\path(midlow)--++(90:0.3) coordinate   (midlow);

	\draw[very thick] (SW) to [out=90,in=180] (midlow)
		  to [out=0,in=90] (SE);;;

 \path(midlow)--++(90:-0.3) coordinate   (below);	

 \draw(below) node { $\boldsymbol \circ$};

\end{tikzpicture}
\qquad
\begin{tikzpicture} [scale=1.6]

 \path(0,0) coordinate (origin2)	;

	\draw[densely dotted] (0,0) --++(135:0.5) coordinate(NE)--++(135:0.5)coordinate(Y)	--++(-135:0.5)coordinate(NW)--++(-135:0.5) 	
	--++(-45:0.5) coordinate (SW)	
	--++(-45:0.5)	
		--++(45:0.5) coordinate (SE)	
	--++(45:0.5)	;
	
\path(NW)--(NE) coordinate [midway] (mid);
\path(mid)--++(-90:0.3) coordinate   (mid);	

			\draw[very thick] (NW) to [out=-90,in=180] (mid)
		  to [out=0,in=-90] (NE);;;

 \draw[very thick] (SW)  --++(55:0.15);
  \draw[very thick] (SW)  --++(125:0.15);

 \draw[very thick] (SE)  --++(-55:0.15);
  \draw[very thick] (SE)  --++(-125:0.15);

\draw[very thick] (NW)  --++(55:0.15);
  \draw[very thick] (NW)  --++(125:0.15);

 \draw[very thick] (NE)  --++(-55:0.15);
  \draw[very thick] (NE)  --++(-125:0.15);

\path(mid)--++(90:0.3) coordinate   (below);	

 \draw(below) node { $\boldsymbol\circ $};

\path(SW)--(SE) coordinate [midway] (midlow);
\path(midlow)--++(90:0.3) coordinate   (midlow);

	\draw[very thick] (SW) to [out=90,in=180] (midlow)
		  to [out=0,in=90] (SE);;;

 \path(midlow)--++(90:-0.3) coordinate   (below);	

 \draw(below) node { $\boldsymbol - $};

\end{tikzpicture}
\qquad
\begin{tikzpicture} [scale=1.6]

 \path(0,0) coordinate (origin2)	;

	\draw[densely dotted] (0,0) --++(135:0.5) coordinate(NE)--++(135:0.5)coordinate(Y)	--++(-135:0.5)coordinate(NW)--++(-135:0.5) 	
	--++(-45:0.5) coordinate (SW)	
	--++(-45:0.5)	
		--++(45:0.5) coordinate (SE)	
	--++(45:0.5)	;
	
\path(NW)--(NE) coordinate [midway] (mid);
\path(mid)--++(-90:0.3) coordinate   (mid);	

			\draw[very thick] (NW) to [out=-90,in=180] (mid)
		  to [out=0,in=-90] (NE);;;

 \draw[very thick] (NW)  --++(55:0.15);
  \draw[very thick] (NW)  --++(125:0.15);

 \draw[very thick] (NE)  --++(-55:0.15);
  \draw[very thick] (NE)  --++(-125:0.15);

\draw[very thick] (SE)  --++(55:0.15);
  \draw[very thick] (SE)  --++(125:0.15);

 \draw[very thick] (SW)  --++(-55:0.15);
  \draw[very thick] (SW)  --++(-125:0.15);

\path(mid)--++(90:0.3) coordinate   (below);	

 \draw(below) node { $\boldsymbol \circ $};

\path(SW)--(SE) coordinate [midway] (midlow);
\path(midlow)--++(90:0.3) coordinate   (midlow);

	\draw[very thick] (SW) to [out=90,in=180] (midlow)
		  to [out=0,in=90] (SE);;;

 \path(midlow)--++(90:-0.3) coordinate   (below);	

 \draw(below) node { $\boldsymbol \circ$};

\end{tikzpicture}
\qquad
\begin{tikzpicture} [scale=-1.6]

 \path(0,0) coordinate (origin2)	;

	\draw[densely dotted] (0,0) --++(135:0.5) coordinate(NE)--++(135:0.5)coordinate(Y)	--++(-135:0.5)coordinate(NW)--++(-135:0.5) 	
	--++(-45:0.5) coordinate (SW)	
	--++(-45:0.5)	
		--++(45:0.5) coordinate (SE)	
	--++(45:0.5)	;
	
\path(NW)--(NE) coordinate [midway] (mid);
\path(mid)--++(-90:0.3) coordinate   (mid);	

			\draw[very thick] (NW) to [out=-90,in=180] (mid)
		  to [out=0,in=-90] (NE);;;

 \draw[very thick] (NW)  --++(55:0.15);
  \draw[very thick] (NW)  --++(125:0.15);

 \draw[very thick] (NE)  --++(-55:0.15);
  \draw[very thick] (NE)  --++(-125:0.15);

\draw[very thick] (SE)  --++(55:0.15);
  \draw[very thick] (SE)  --++(125:0.15);

 \draw[very thick] (SW)  --++(-55:0.15);
  \draw[very thick] (SW)  --++(-125:0.15);

\path(mid)--++(90:0.3) coordinate   (below);	

 \draw(below) node { $\boldsymbol - $};

\path(SW)--(SE) coordinate [midway] (midlow);
\path(midlow)--++(90:0.3) coordinate   (midlow);

	\draw[very thick] (SW) to [out=90,in=180] (midlow)
		  to [out=0,in=90] (SE);;;

 \path(midlow)--++(90:-0.3) coordinate   (below);	

 \draw(below) node { $+ $};

\end{tikzpicture}
 $$
\caption{The four  generators $\lambda \diag_i \mu$ for $i\in \{1, \ldots n-1\}$ of degree 
$1=1+0$, $-1=0-1$, $0=0+0$, and $0=1-1$ respectively. We record the degrees of the strand segments within the tile (with $+=1$, $-=-1$ and $\circ=0$). }
\label{breakitup}
 \end{figure}

Now as explained in \cref{productofgenerators}, we can write any oriented Temperley--Lieb diagram  as a product of generators $\diag_\SSTT$ with $\omega(\SSTT)$ a reduced word for a strongly fully commutative element of $W$ using a region in the tiling for $W$. Moreover, we have seen how each strand of the diagram corresponds to a horizontal strip in the tiling. We can then obtain the degree of strand in the product by adding the degree of each small arc coming from the generators. We now run through the various types of strands and check that the result holds in each case.

 First observe that any  northern or southern arc  passes through an odd number of tiles and any propagating strand passes through an even number of tiles.  
 The undecorated northern arcs are illustrated in \cref{undecoratednorthernarcs}.
 
 \begin{figure}[ht!]
\[
 		\begin{tikzpicture} [yscale=0.65,xscale=-0.65]

 \clip(0,-1) rectangle (-4*1.41421356237,1.8);
 \path(0,0) coordinate (origin2)	;
	
	\fill[magenta!20] (0,0)	--++(135:1)--++ 	(180:3*1.41421356237)
		--++(-135:1)--(0,0);

	\path(0,0)--++(90:0.42)--++(135:0.5)--++(-135:0.5) node {\scalefont{0.7}$+ $}
	--++(135:1)--++(-135:1) node {\scalefont{0.7}$+ $}	--++(135:1)--++(-135:1) node {\scalefont{0.7}$+ $}	--++(135:1)--++(-135:1) node {\scalefont{0.7}$+ $};
	
	\draw[densely dotted] (0,0) --++(135:0.5) coordinate(NE)--++(135:0.5)coordinate(Y)	--++(-135:0.5)coordinate(NW)--++(-135:0.5) 	coordinate (X1)
	--++(-45:1) 		--++(45:1) 	;

			\draw[very thick,->](NW) to [out=-90,in=-90] (NE)--++(90:1.2);;;

	\draw[densely dotted]  (X1)--++(135:0.5) coordinate(NE)--++(135:0.5) --++(-135:0.5)
	coordinate(NW)--++(-135:0.5) 	
	coordinate (X2)
	--++(-45:1) 		--++(45:1) 	;

			\draw[very thick](NE) to [out=-90,in=-90] (NW);

				\draw[densely dotted]   (X2)--++(135:0.5) coordinate(NE)--++(135:0.5) --++(-135:0.5)
	coordinate(NW)--++(-135:0.5) 	
	coordinate (X3)
	--++(-45:1) 		--++(45:1) 	;

				\draw[very thick](NE) to [out=-90,in=-90] (NW);

		\draw[densely dotted]  (X3) --++(135:0.5) coordinate(NE)--++(135:0.5) --++(-135:0.5)
	coordinate(NW)--++(-135:0.5) 	
	coordinate (X4)
	--++(-45:1) 		--++(45:1) 	;

				\draw[very thick,-<](NE) to [out=-90,in=-90] (NW)--++(90:1.2);;

	\draw[densely dotted]  (Y)--++(135:1) 	--++(-135:1) 	coordinate (Y2)
 	;	
	
		\path(Y)--++(-90:0.4)--++(135:0.5)--++(-135:0.5) node {\scalefont{0.7}$- $}
	--++(135:1)--++(-135:1) node {\scalefont{0.7}$- $}	--++(135:1)--++(-135:1) node {\scalefont{0.7}$- $};

	\path  (Y) --++(-135:0.5) coordinate(NE)--++(-135:0.5) --++(135:0.5)
	coordinate(NW)--++(135:0.5) 	
	coordinate (Y2);

				\draw[very thick](NE) to [out=90,in=90] (NW);

	\draw[densely dotted]  (Y2)--++(135:1) 	--++(-135:1) 	coordinate (Y3)
;

	\path  (Y2) --++(-135:0.5) coordinate(NE)--++(-135:0.5) --++(135:0.5)
	coordinate(NW)--++(135:0.5) 	
	coordinate (Y3); 

 		\draw[densely dotted]  (Y3)--++(135:1) 	--++(-135:1) 	coordinate (Y4)
; 

					\draw[very thick,densely dotted](NE) to [out=90,in=90] (NW);

	\path  (Y3) --++(-135:0.5) coordinate(NE)--++(-135:0.5) --++(135:0.5)
	coordinate(NW)--++(135:0.5) 	
	coordinate (Y4); 
					\draw[very thick](NE) to [out=90,in=90] (NW);

		\foreach \i in {0,1,2,3,4,5,6,7,8,9,10,11,12}
		{
			\path (origin2)--++(45:0.5*\i) coordinate (c\i); 
			\path (origin2)--++(135:0.5*\i)  coordinate (d\i); 
		}

		\end{tikzpicture} 
		\qquad\quad
	 \begin{tikzpicture} [scale=0.65]
	

 \clip(0,-1) rectangle (-4*1.41421356237,1.8);
 \path(0,0) coordinate (origin2)	;
	
	\fill[magenta!20] (0,0)	--++(135:1)--++ 	(180:3*1.41421356237)
		--++(-135:1)--(0,0);

	\path(0,0)--++(90:0.42)--++(135:0.5)--++(-135:0.5) node {\scalefont{0.7}$\circ $}
	--++(135:1)--++(-135:1) node {\scalefont{0.7}$\circ $}	--++(135:1)--++(-135:1) node {\scalefont{0.7}$\circ $}	--++(135:1)--++(-135:1) node {\scalefont{0.7}$\circ $};
	
	\draw[densely dotted] (0,0) --++(135:0.5) coordinate(NE)--++(135:0.5)coordinate(Y)	--++(-135:0.5)coordinate(NW)--++(-135:0.5) 	coordinate (X1)
	--++(-45:1) 		--++(45:1) 	;

			\draw[very thick,->](NW) to [out=-90,in=-90] (NE)--++(90:1.2);;;

	\draw[densely dotted]  (X1)--++(135:0.5) coordinate(NE)--++(135:0.5) --++(-135:0.5)
	coordinate(NW)--++(-135:0.5) 	
	coordinate (X2)
	--++(-45:1) 		--++(45:1) 	;

			\draw[very thick](NE) to [out=-90,in=-90] (NW);

				\draw[densely dotted]   (X2)--++(135:0.5) coordinate(NE)--++(135:0.5) --++(-135:0.5)
	coordinate(NW)--++(-135:0.5) 	
	coordinate (X3)
	--++(-45:1) 		--++(45:1) 	;

				\draw[very thick](NE) to [out=-90,in=-90] (NW);

		\draw[densely dotted]  (X3) --++(135:0.5) coordinate(NE)--++(135:0.5) --++(-135:0.5)
	coordinate(NW)--++(-135:0.5) 	
	coordinate (X4)
	--++(-45:1) 		--++(45:1) 	;

				\draw[very thick,-<](NE) to [out=-90,in=-90] (NW)--++(90:1.2);;

	\draw[densely dotted]  (Y)--++(135:1) 	--++(-135:1) 	coordinate (Y2)
 	;	
	
		\path(Y)--++(-90:0.4)--++(135:0.5)--++(-135:0.5) node {\scalefont{0.7}$\circ $}
	--++(135:1)--++(-135:1) node {\scalefont{0.7}$\circ $}	--++(135:1)--++(-135:1) node {\scalefont{0.7}$\circ $};

	\path  (Y) --++(-135:0.5) coordinate(NE)--++(-135:0.5) --++(135:0.5)
	coordinate(NW)--++(135:0.5) 	
	coordinate (Y2);

				\draw[very thick](NE) to [out=90,in=90] (NW);

	\draw[densely dotted]  (Y2)--++(135:1) 	--++(-135:1) 	coordinate (Y3)
;

	\path  (Y2) --++(-135:0.5) coordinate(NE)--++(-135:0.5) --++(135:0.5)
	coordinate(NW)--++(135:0.5) 	
	coordinate (Y3); 

 		\draw[densely dotted]  (Y3)--++(135:1) 	--++(-135:1) 	coordinate (Y4)
; 

					\draw[very thick,densely dotted](NE) to [out=90,in=90] (NW);

	\path  (Y3) --++(-135:0.5) coordinate(NE)--++(-135:0.5) --++(135:0.5)
	coordinate(NW)--++(135:0.5) 	
	coordinate (Y4); 
					\draw[very thick](NE) to [out=90,in=90] (NW);

		\foreach \i in {0,1,2,3,4,5,6,7,8,9,10,11,12}
		{
			\path (origin2)--++(45:0.5*\i) coordinate (c\i); 
			\path (origin2)--++(135:0.5*\i)  coordinate (d\i); 
		}

		\end{tikzpicture} 
\]
\caption{The undecorated northern arcs.} \label{undecoratednorthernarcs}\end{figure}

In \cref{undecoratednorthernarcs,deg1north,deg0north,THREE,THREE2}, 
we   highlight  the  tiles through which the strand  ``wiggles" and the degree of the strand within a given tile (which we have already calculated in terms of the generators.
Note that, reading from left to right, the strand oscillates between being
 either the ``top" or ``bottom"
of a given generator tile.  
   Assume  the arc is undecorated, northern,  and clockwise-oriented. 
The degree contribution
 as the arc passes through
these tiles is given by
\begin{equation*}
+1 -1 +1-1+1-1\dots +1=1
\end{equation*}
(notice that the strand is locally either a clockwise northern arc or an anticlockwise southern arc at each step).  
If an undecorated northern arc is anti-clockwise oriented then it has degree 
\begin{equation*}0+\dots+0=0.
\end{equation*}
Thus the degrees match up.  
The degrees of undecorated southern arcs can be computed similarly. 
The degree of undecorated propagating strands can easily be seen to be zero as illustrated in \cref{THREE}.

 \begin{figure}[ht!]
 $$
$$
		\caption{Degree 1 decorated northern strands.  The pink highlights an odd number of boxes, the blue highlights an even number.  In this way, note that the blue section of the strand never contributes to the degree, regardless of the orientation.}
\label{deg1north}		\end{figure}

 \begin{figure}[ht!]
 $$
   %
$$
		\caption{Degree zero decorated northern strands.  }
\label{deg0north}		\end{figure}

The decorated northern arcs of degree 1 in type $(C_{n-1}, A_{n-2})$ are illustrated in  \cref{deg1north,deg0north,THREE,THREE2}.  
Swapping the orientation of the rightmost vertex from $\down$ to $\up$, we see that the pink strip now doesn't contribute to the degree either, and so these northern arc have degree zero. The decorated southern arcs can be dealt with in a similar way. 

 It is easy to see that any (decorated) propagating strand also has degree zero as it goes through an even number of tiles.

\begin{figure} [H]
$$
  $$	
\caption{The undecorated propagating strands, all of which are  of degree $0$.} 
\label{THREE}
\end{figure}

\begin{figure} [H]
$$
%
  $$	
\caption{The decorated propagating strands, all of which are  of degree $0$.} 
\label{THREE2}
\end{figure}

Finally, the strands in a digram of type $(D_n, A_{n-1})$ are the same as those in type $(C_{n-1}, A_{n-2})$ except that they contain at most one decoration and we only need to consider arcs which are flip-oriented. So the result holds for these as well.
\end{proof}

 \begin{cor}\label{nonnegativegrading}   
 The anti-spherical module  for ${\rm TL}_{(W,P)}(q)$ is non-negatively graded, that is 
 $$\deg(\varnothing \diag_{\mu} \la ) \geq 0 \qquad \mbox{for all $\la ,\mu \in {^PW}$}.$$Thus, the non-zero entries in the light leaves matrix 
$\Delta_{\la,\mu}$
are non-negative powers of $q$.
  \end{cor}
 
 \begin{proof}
  Any northern arc or propagating strand is non-negatively graded, thus it suffices to consider 
 the southern arcs (which can be negatively graded).  
 Now recall the coset diagram for $\varnothing$ in types $(W,P)=(A_n,A_{k}\times A_{n-k-1})$, $(C_n,A_{n-1})$, $(D_{n},A_{n-1})$ from \cref{section4}.   We see that any  southern arc must have rightmost vertex labelled by $\down$. So the southern arcs all have degree zero. This proves the result. 
 \end{proof}

\section{Factorisation of the light leaves matrix}
 \label{results2}

  We now explain how to construct the element $\diag_\mu\in {\rm TL}_W(q)$ for $\mu \in {^PW}$  via an efficient closed combinatorial algorithm, which has its origins in \cite{MR2955190,MR3544626,MR3286499,MR2813567}.   
  This allows us to enumerate the elements $\varnothing \diag_\mu \lambda$ and hence the paths  $\NPath (\la,\underline{\mu}) \subseteq \Path (\la,\underline{\mu})$ 
  and $ \BPath (\nu,\underline{\mu}) \subseteq  \Path (\nu,\underline{\mu})$ corresponding to  
  decomposition numbers and bases of simple modules for the Hecke category respectively (in other words, solving Libedinsky--Williamson's question for Hermitian symmetric pairs).

We have seen that every strongly fully commutative element $w$ in $W$ with $\underline{w} = \omega(\SSTT)$ for some $\SSTT\in {\rm Path}_{(W,P)}$ corresponds to a region $R(w)$ in the tiling of $W$. That region is determined by the path $\pi(w)$ walking along its boundary, starting with its northern boundary going East and coming back along its southern boundary going West. When $w=\mu \in {^PW}$, the region $R(\mu)$ has  a particularly simple shape.

\begin{lem}
We have that $w\in {^PW}$ if and only if the region $R(w)$ has the following property.
 In type $(A_{n-1}, A_{k-1}\times A_{n-k-1})$, the last $n$ steps of the path $\pi(w)$ are given by $((SW)^{n-k}, (NW)^{k})$. In types $(C_{n-1},A_{n-2})$ and $(D_n, A_{n-1})$, the last $n$ steps of the path $\pi(w)$ are given by $((SW)^n)$.
 \end{lem}

\begin{proof} Note that the minimal length coset representatives  $\mu \in {^PW}$ are characterised by the fact that every reduced expression for $\mu$ starts with $s\notin P$. From that characterisation, it is clear that if $\pi(w)$ has the stated form then $w\in {^PW}$. Examples of $\pi(\mu)$ for $\mu \in {^PW}$ are given in the first two pictures in \cref{resulting} and \cref{resulting2XXX}.  Now suppose that $\pi(w)$ does not have the stated form then  $\pi(w)$ would have two consecutive steps $\pi((n-i)') = SW$ and $\pi((n-i-1)') = NW$ (with $i\neq k$ in type $(A_{n-1}, A_{k-1}\times A_{n-k-1})$). But this would imply that there is a reduced expression for $w$ starting with some $s\in P$ and so $w\notin {^PW}$.  This is illustrated in the rightmost picture of \cref{resulting} and \cref{resulting2XXX} where we indeed observe that there is a reduced word for $w$ starting with ${\color{cyan} s}\in P$ in each case. 
\end{proof}

Therefore for $\mu \in {^PW}$, the region $R(\mu)$ is completely determined by its northern boundary.
In fact, we have the following natural correspondence between the northern boundary of $R(\mu)$ and the coset diagram of $\mu$. 

\begin{lem} Let $\mu\in {^PW}$.
The $i$-th vertex of the coset diagram of $\mu$  is labelled by $\up$, respectively  $\down$, if and only if the $i$-th step in $\pi(\mu)$ is given by $SE$, respectively  $NE$. In type $(C_{n-1}, A_{n-2})$ the first vertex is always labelled by $\circ$ and the first step is always $SE$. 
\end{lem}

\begin{proof}
We proceed by induction on $\ell(\mu)$. For $\ell(\mu) = 0$ we have $\mu = \varnothing$ and the result is clear from the description of the coset diagram for $\varnothing$ given in Section 5. Now assume that the result holds for $\lambda\in {^PW}$ and let $\mu = \la s_i > \la$. If $i=1, \ldots , n-1$ then the coset diagrams for $\la$ and $\la s_i$ only differ in position $i$ and $i+1$ and we have $\la = \ldots \up \down \ldots$ and $\la s_i = \ldots \down \up \ldots$. The corresponding paths $\pi(\la)$ and $\pi(\la s_i)$ are depicted in \cref{XXX1}. We see that the $i$-th and $(i+1)$-th steps in $\la$, respectively  $\la s_i$, are given by $SE, NE$,  respectively  $NE, SE$ and so the result follows by induction. For $i=1'$, the coset diagrams of $\la$ and $\la s_{1'}$ only differ in the second position and we have $\la = \circ \down \ldots$ and $\la s_{1'} = \circ \up \ldots$. The corresponding paths $\pi(\la)$ and $\pi(\la s_{1'})$ are depicted in \cref{XXX1}. We see that the second step in $\la$ is given by $NE$ while the second step in $\la s_{1'}$ is given by $SE$, and so the result follows by induction. Finally, if $i=1''$ then the coset diagrams for $\la$ and $\la s_{1''}$ only different in the first two positions and we have $\la = \down \down \ldots$ and $\la s_{1''} = \up \up \ldots$. The corresponding paths $\pi(\la)$ and $\pi(\la s_{1''})$ are depicted in \cref{XXX1}. We see that the first two steps in $\la$, respectively $\la s_{1''}$, are given by $NE,NE$, respectively  $SE, SE$, and so the result follows by induction. 
\end{proof}

  \begin{figure}[ht!]
 $$
  \begin{tikzpicture}[scale=1.3]
  	
	\begin{scope}
  \clip (0,0) (0,0)--++(45:4*0.5)--++(135:2*0.5)--++(-135:1*0.5)--++(135:1*0.5)--++(-135:1*0.5)--++(-135:1*0.5)--++(135:1*0.5)--++(135:1*0.5)--++(-135:1*0.5)--(0,0);

		\foreach \i in {0,1,2,3,4,5}
		{
			\path (0,0)--++(45:0.5*\i) coordinate (a\i); 
			\path (0,0)--++(135:0.5*\i)  coordinate (c\i); 
   \draw[   ] (a\i)--++(135:4);
   \draw[   ] (c\i)--++(45:4);   
 }
 \end{scope}

 \draw[   ] (0,0) (0,0)--++(45:4*0.5)--++(135:2*0.5)--++(-135:1*0.5)--++(135:1*0.5)--++(-135:1*0.5)--++(-135:1*0.5)--++(135:1*0.5)--++(135:1*0.5)--++(-135:1*0.5)--(0,0);

 \foreach \i in {0,1,2,3,4}
		{     \draw[   ,densely dotted] (a\i)--++(-45:0.3);
		 \draw[   ,densely dotted] (a\i)--++(45:0.3);
		
		}

 \foreach \i in {0,1,2,3,4,5}
		{  
   \draw[   ,densely dotted] (c\i)--++(-135:0.3);   
   \draw[   ,densely dotted] (c\i)--++(135:0.3);     
   			}

 \draw[very thick,->]  (135:2.5)--++(45:0.5) coordinate (X);
 \draw[very thick,->]  (X)--++(-45:0.5) coordinate (X);  
  \draw[very thick,->]  (X)--++(-45:0.5) coordinate (X);  
   \draw[very thick,->]  (X)--++(45:0.5) coordinate (Y);  
   
   \fill[gray!40](X)--++(45:0.5)--++(45:0.5)--++(-45:0.5)--++(-135:0.5)--++(135:0.5);
\path(X)--++(45:0.5)--++(45:0.25)--++(-45:0.25)   node {\scalefont{0.9}$i$};
   
      \draw[very thick,magenta,->]  (Y)--++(45:0.5) coordinate (X);  
            \draw[very thick,magenta,->]  (X)--++(-45:0.5) coordinate (X);  
      \draw[very thick,cyan,->]  (Y)--++(-45:0.5) coordinate (X);          \draw[very thick,cyan,->]  (X)--++(45:0.5) coordinate (X);    

      \draw[very thick,->]  (X)--++(45:0.5) coordinate (X);          \draw[very thick,->]  (X)--++(-45:0.5) coordinate (X);    \draw[very thick,->]  (X)--++(-45:0.5) coordinate (X);

\draw(0,-0.75) node {$ \la\color{gray}s_i\color{black}=\cdots \down\up \up\down \color{magenta}\down\up\color{black}\down\up\up\cdots $};
\draw(0,-1.25) node {$\color{white}s_i\color{black}\la =\cdots \down\up \up\down \color{cyan} \up\down\color{black}\down\up\up\cdots $};

\end{tikzpicture}
   \begin{tikzpicture}[scale=1.3]
  	
	\begin{scope}
  \clip (0,0) (0,0)--++(45:5*0.5)--++(135:1*0.5)--++(-135:1*0.5)
  --++(135:1*0.5)--++(-135:1*0.5)--++(135:1*0.5) 
  --++(-135:1*0.5)--++(-45:1*0.5)  --++(-135:1*0.5)--++(-45:1*0.5)
    --++(-135:1*0.5)--++(-45:1*0.5);

		\foreach \i in {0,1,2,3,4,5}
		{
			\path (0,0)--++(45:0.5*\i) coordinate (a\i); 
			\path (0,0)--++(135:0.5*\i)  coordinate (c\i); 
   \draw[   ] (a\i)--++(135:4);
   \draw[   ] (c\i)--++(45:4);   
 }
 \end{scope}

  \draw[   ] (0,0)--++(45:5*0.5)--++(135:1*0.5)--++(-135:1*0.5)
  --++(135:1*0.5)--++(-135:1*0.5)--++(135:1*0.5) 
  --++(-135:1*0.5)--++(-45:1*0.5)  --++(-135:1*0.5)--++(-45:1*0.5)
    --++(-135:1*0.5)--++(-45:1*0.5);

 \foreach \i in {0,1,2,3,4,5}
		{     \draw[   ,densely dotted] (a\i)--++(-45:0.3);\draw[   ,densely dotted] (a\i)--++(45:0.3);}

    \draw[   ,densely dotted] (c0)--++(-135:0.3);

\path (135:2.5)--++(45:0.5) coordinate (X);
\path (X)--++(-45:0.5) coordinate (X);  
 \path (X)--++(-45:0.5) coordinate (X);  
  \path (X)--++(45:0.5) coordinate (Y);  
   
   \fill[gray!40](X)--++(45:0.5)--++(45:0.5)--++(-45:0.5)--++(-135:0.5)--++(135:0.5);
\path(X)--++(45:0.5)--++(45:0.25)--++(-45:0.25)   node {\scalefont{0.9}$1'$};
   
      \path (Y)--++(45:0.5) coordinate (X);  
      
            \draw[very thick,magenta,<-]  (X)--++(135:0.5) ;
            \draw[very thick,magenta,->]  (X)--++(-45:0.5) coordinate (X);  
      \draw[very thick,cyan,->]  (Y)--++(-45:0.5) coordinate (X);          \draw[very thick,cyan,->]  (X)--++(45:0.5) coordinate (X);    

      \draw[very thick,->]  (X)--++(45:0.5) coordinate (X);          \draw[very thick,->]  (X)--++(-45:0.5) coordinate (X);  
        \draw[very thick,->]  (X)--++(45:0.5) coordinate (X);       \draw[very thick,->]  (X)--++(-45:0.5) coordinate (X);

\draw(0.5,-1.25) node {$ 
\color{white}s_{1'}\color{black}\la=\color{cyan} \circ   \down \color{black}\down\up \down\up \cdots $};
\draw(0.5,-0.75) node {$ \color{black}\la\color{gray}s_{1'}\color{black}= \color{magenta} \circ   \up \color{black}\down\up \down\up \cdots $};

\end{tikzpicture}
  \begin{tikzpicture}[scale=1.3]
  	
	\begin{scope}
  \clip (0,0) (0,0)--++(45:5*0.5)--++(135:1*0.5)--++(-135:1*0.5)
  --++(135:1*0.5)--++(-135:1*0.5)--++(135:1*0.5) 
  --++(-135:1*0.5)--++(-45:1*0.5)  --++(-135:1*0.5)--++(-45:1*0.5)
    --++(-135:1*0.5)--++(-45:1*0.5);

		\foreach \i in {0,1,2,3,4,5}
		{
			\path (0,0)--++(45:0.5*\i) coordinate (a\i); 
			\path (0,0)--++(135:0.5*\i)  coordinate (c\i); 
   \draw[   ] (a\i)--++(135:4);
   \draw[   ] (c\i)--++(45:4);   
 }
 \end{scope}

  \draw[   ] (0,0)--++(45:5*0.5)--++(135:1*0.5)--++(-135:1*0.5)
  --++(135:1*0.5)--++(-135:1*0.5)--++(135:1*0.5) 
  --++(-135:1*0.5)--++(-45:1*0.5)  --++(-135:1*0.5)--++(-45:1*0.5)
    --++(-135:1*0.5)--++(-45:1*0.5);

 \foreach \i in {0,1,2,3,4,5}
		{     \draw[   ,densely dotted] (a\i)--++(-45:0.3);\draw[   ,densely dotted] (a\i)--++(45:0.3);}

    \draw[   ,densely dotted] (c0)--++(-135:0.3);

\path (135:2.5)--++(45:0.5) coordinate (X);
\path (X)--++(-45:0.5) coordinate (X);  
 \path (X)--++(-45:0.5) coordinate (X);  
  \path (X)--++(45:0.5) coordinate (Y);  
   
   \fill[gray!40](X)--++(45:0.5)--++(45:0.5)--++(-45:0.5)--++(-135:0.5)--++(135:0.5);
\path(X)--++(45:0.5)--++(45:0.25)--++(-45:0.25)   node {\scalefont{0.9}$1''$};
   
      \path (Y)--++(45:0.5) coordinate (X);  
      
            \draw[very thick,magenta,<-]  (X)--++(135:0.5) ;            \draw[very thick,magenta,->]  (X)--++(-45:0.5) coordinate (X);

   \path (Y)--++(-45:0.5) coordinate (X);  
      
            \draw[very thick,cyan,<-]  (X)--++(-135:0.5) ;

               \draw[very thick,cyan,->]  (X)--++(45:0.5) coordinate (X);    

      \draw[very thick,->]  (X)--++(45:0.5) coordinate (X);          \draw[very thick,->]  (X)--++(-45:0.5) coordinate (X);  
        \draw[very thick,->]  (X)--++(45:0.5) coordinate (X);       \draw[very thick,->]  (X)--++(-45:0.5) coordinate (X);

\draw(0.5,-1.25) node {$ 
\color{white}s_{1''}\color{black}\la=\color{cyan} \down \down \color{black}\down\up \down\up \cdots $};
\draw(0.5,-0.75) node {$ \color{black}\la\color{gray}s_{1''}\color{black}= \color{magenta} \up   \up \color{black}\down\up \down\up \cdots $};

\end{tikzpicture}
 $$
 \caption{The effect of applying a reflection $\color{gray}s_i$ for $1\leq i \leq n-1$; the reflection $\color{gray}s_i\color{black}=\color{gray}s_{1'}$ (in type $C$); and  the reflection $\color{gray}s_i\color{black}=\color{gray}s_{1''}$ (in type $D$) respectively.
In each case we depict  the pair of weights $\la$ and $\la \color{gray}s_i$ and the corresponding northern edges of the paths $\pi(\la)$ and $\pi(\la \color{gray}s_i)$.
 We highlight in blue and pink the difference between  $\la$ and $\la \color{gray}s_i$ (both on the level of tiles and coset diagrams).
 }
 \label{XXX1}
 \end{figure}

Using this correspondence, we easily obtain the following closed combinatorial algorithm to construct $\diag_{\mu}$.

\begin{prop}\label{thething}
For each $\mu\in {^PW}$, the diagram $\diag_\mu\in {\rm TL}_W(q)$ can be constructed as follows.
Place the coset diagram $\mu$ on the northern boundary and $\varnothing$ on the southern boundary. Then
\begin{enumerate}[leftmargin=*]
\item Repeatedly connect neighbouring northern vertices (in the sense that they are next to each other or only have vertices already connected by an arc between them) labelled by $\down$ and $\up$ by a northern anti-clockwise arc.

\smallskip \noindent 
We are left with (in type $(C_{n-1}, A_{n-2})$ one vertex labelled by $\circ$ followed by) some vertices labelled by $\up$ followed by some vertices labelled by $\down$. 

\item 
\begin{enumerate}
\item In type $(A_{n-1}, A_{k-1}\times A_{n-k-1})$, draw undecorated propagating strands on all remaining vertices.
\item In type $(D_n, A_{n-1})$, starting from the left, connect neighbouring vertices labelled with $\up$'s with decorated northern arcs. Then draw a decorated propagating strand from the remaining vertex labelled by $\up$ (if it exists), and undecorated propagating strands  on all remaining vertices labelled by $\down$. 
\item In type $(C_{n-1}, A_{n-2})$, if $\mu = \varnothing$, then draw an undecorated propagating strand from each northern vertex. Otherwise, view the first label $\circ$ of $\mu$ as a $\up$, then follow exactly the same procedure as in type $(D_n, A_{n-1})$. 
\end{enumerate}
\end{enumerate}
Now there is a unique way of completing the diagram $\diag_\mu$ such that $\varnothing \diag_\mu \mu\in \mathbb{ODT}[W,P]$. 
 \end{prop}

 \begin{eg}
 A couple of illustrative large examples of the construction in  \cref{thething} are given in \cref{cups,cups2}.
The complete set of all $\diag_\mu$ for $\mu \in {^PW}$ for $(W,P)$ of type 
$(D_4 , A_3)$ 
and  $(C_3 , A_2)$
are given in  \cref{BruhatexampleD2}. In all examples, we will only picture the top of the diagram as this completely determines $\diag_\mu$.
   
  \begin{figure}[ht!]
 $$
 $$
 \caption{Two examples of the construction of $\diag_\mu$,  the former is of type $(A_8, A_4\times A_3)$ and the latter is of type $(C_8, A_7)$.
   }
 \label{cups}
 \end{figure}

  \begin{figure}[ht!]
 $$
  %
 $$
 \caption{An example  of the construction of $\diag_\mu$ in type $(D_{19},A_{18})$.   }
 \label{cups2}

 \end{figure}

  \end{eg}

\cref{thething} gives us an efficient way of finding the degree of $\varnothing \diag_\mu \lambda$ as the number of northern arcs whose rightmost vertex is labelled by $\down$. As this does not depend on the bottom of the diagram (as all southern arcs have degree zero), we will only depict the top of $\varnothing \diag_\mu \lambda$. Examples are given in \cref{kfkfkfkfkffff,firstcol}.

 \begin{figure}[ht!]
 $$%
$$

\caption{The construction of $\diag_\mu$ for all $\mu\in {^PW}$ in types  $(D_4 , A_3)$ 
and  $(C_3 , A_2)$. }
\label{BruhatexampleD2}
\end{figure}

 \begin{figure}[ht!]
 $$
 %
$$
\caption{The elements $\varnothing \diag_\mu \la$ from the first column of \cref{kfkfkfkfkffff} in order. }
\label{firstcol}
\end{figure}

\begin{figure}[H]
	%
	\caption{The   $q^{\deg({{\varnothing \diag_\mu \la}})}$   
		for $\la,\mu \in {{^P}W}$ for $(W,P)=  (A_3,A_1\times A_1)$.}  
	\label{kfkfkfkfkffff}
\end{figure}

  We have already seen in \cref{nonnegativegrading} that the entries in the light leaves matrix $\Delta_{\la \mu}$ are either zero or non-negative powers of $q$. In the simply-laced cases, we can say more.

\begin{thm}\label{simplylaced} Assume $(W,P)=(A_{n-1}, A_{n-2})$ or $(D_n, A_{n-1})$. Then for any $\la, \mu\in {^PW}$  with $\varnothing \diag_\mu \la \in \mathbb{ODT}[W,P]$ we have $\deg (\varnothing \diag_\mu \la) = 0$ if and only if $\la = \mu$. In particular, we have $\Delta_{\la \mu} = 1$ if and only if $\lambda = \mu$ and so the matrix of light leaves has a trivial factorisation 
 $$B = {\rm Id} \quad \mbox{and} \quad \Delta = N .$$
\end{thm}

\begin{proof}
Fix $\mu\in {^PW}$ and consider all $\la \in {^PW}$ with  $\varnothing \diag_\mu\la\in \mathbb{ODT}[W,P]$ and $\deg (\varnothing \diag_\mu\la) = 0$.  Note that this forces all undecorated northern arcs to be anti-clockwise, all undecorated propagating strands to be oriented, all northern decorated arcs to be labelled by two $\up$'s and all decorated propagating strands to be flip-oriented. This gives only one choice for $\lambda$, namely $\lambda = \mu$. 
\end{proof}

Thus, setting every basis element $\varnothing \diag_\mu \la$ for the anti-spherical module to be standard proves Theorem A in type $(W,P)=(A_{n-1}, A_{n-2})$ and $(D_n, A_{n-1})$.
We now state and proof Theorem $A$ in the only remaining type of Hermitian symmetric pair, namely $(C_{n}, A_{n-1})$.

\begin{thm}\label{bad62}
Let $\lambda, \mu \in {^{A_{n-1}}C_{n}}$. 
We say that $\varnothing \diag_{\mu} \la\in \mathbb{ODT}[C_n, A_{n-1}]$ is {\sf standard} if every decorated strand if flip-oriented.

We  define the matrices $N$ and $B$ in type $(W,P)=(C_n, A_{n-1})$ as follows:  
\begin{align} 
N_{ \la,\mu}  &= \left\{ \begin{array}{ll} q^{\deg (\varnothing \diag_{\mu} \la)} & \mbox{if $\varnothing \diag_\mu \la \in \mathbb{ODT}[C_n, A_{n-1}]$ is standard}\\
0 & \mbox{otherwise} \end{array}\right.
\\
B_{ \la,\mu}  &= \left\{ \begin{array}{ll} 1 & \mbox{if $\varnothing \diag_\mu \la\in \mathbb{ODT}[C_{n}, A_{n-1}]$  and $\deg (\varnothing \diag_\mu \la) =0$ }\\
0 & \mbox{otherwise.} \end{array}\right.
\end{align}
Then the matrix of light leaves in type $(C_n, A_{n-1})$ factorises as
$$\Delta = N B.$$
\end{thm}

\begin{proof}
Fix $\mu\in {^{A_{n-1}}C_n}$ and consider 
$$C_\mu = \{\lambda \in {^{A_{n-1}}C_n} \, : \, \varnothing \diag_\mu \la \in \mathbb{ODT}[C_n, A_{n-1}] \}.$$
Define an equivalence relation on $C_\mu$ by setting $\la \sim_\mu \eta$ for $\la, \eta \in C_\mu$ if $\varnothing \diag_\mu \la$ and $\varnothing \diag_\mu \eta$ differ only in one northern arc where the differing arcs are given by 
 \begin{equation}\label{simmu}
 \left\{ \begin{minipage}{1.2cm}
 \begin{tikzpicture}[scale=0.7]
 \draw[thick](-1,-0.155) node {$\boldsymbol\up$};
  \draw[very thick](0,0.1) node {$\boldsymbol\down$};
  \draw[very thick]  (-1,-0) to [out=-90,in=180] (-0.5,-0.75) to [out=0,in=-90] (0,0);
 \end{tikzpicture}
 \end{minipage}
 , \; 
  \begin{minipage}{1.2cm}
\begin{tikzpicture}[xscale=-0.7, yscale=0.7]
 \draw[thick](-1,-0.155) node {$\boldsymbol\up$};
  \draw[very thick](0,0.1) node {$\boldsymbol\down$};
  \draw[very thick]  (-1,-0) to [out=-90,in=180] (-0.5,-0.75) to [out=0,in=-90] (0,0);
 \end{tikzpicture}
 \end{minipage}
 \!  \right\} \, \mbox{or}
\,
\left\{  \begin{minipage}{1.2cm}
 \begin{tikzpicture}[scale=0.7]
 \draw[thick](-1,-0.155) node {$\boldsymbol\up$};
  \draw[very thick](0,0.1) node {$\boldsymbol\down$};
  \draw[very thick]  (-1,-0) to [out=-90,in=180] (-0.5,-0.75) to [out=0,in=-90] (0,0);
 \fill 	(-0.5,-0.75) 	 circle (4pt); \end{tikzpicture}
 \end{minipage}
 , 
  \begin{minipage}{1.2cm}
\begin{tikzpicture}[xscale=-0.7, yscale=0.7]
 \draw[thick](-1,-0.155) node {$\boldsymbol\up$};
  \draw[very thick](0,0.1) node {$\boldsymbol\down$};
  \draw[very thick]  (-1,-0) to [out=-90,in=180] (-0.5,-0.75) to [out=0,in=-90] (0,0);
 \fill 	(-0.5,-0.75) 	 circle (4pt); \end{tikzpicture}
 \end{minipage}
 \!  \right\}	\, \mbox{or}
\,
\left\{  \begin{minipage}{1.2cm}
 \begin{tikzpicture}[xscale=-0.7, yscale=0.7] 
   \draw[white,very thick](-1,0.1) node {$\boldsymbol\down$};

 \draw[thick](-1,-0.155) node {$\boldsymbol\up$};
 \draw[thick](0,-0.155) node {$\boldsymbol\up$};

  \draw[very thick]  (-1,-0) to [out=-90,in=180] (-0.5,-0.75) to [out=0,in=-90] (0,0);
 \fill 	(-0.5,-0.75) 	 circle (4pt); \end{tikzpicture}
 \end{minipage}
 , 
  \begin{minipage}{1.2cm}
\begin{tikzpicture} [scale=0.7]
  \draw[very thick](-1,0.1) node {$\boldsymbol\down$};

  \draw[very thick](0,0.1) node {$\boldsymbol\down$};
  \draw[very thick]  (-1,-0) to [out=-90,in=180] (-0.5,-0.75) to [out=0,in=-90] (0,0);
 \fill 	(-0.5,-0.75) 	 circle (4pt); \end{tikzpicture}
 \end{minipage}
 \!  \right\} \, \mbox{or} \,
 \left\{  \begin{minipage}{1.2cm}
 \begin{tikzpicture}[xscale=-0.7, yscale=0.7] 
   \draw[white,very thick](-1,0.1) node {$\boldsymbol\down$};

 \draw[thick](-1,-0.155) node {$\boldsymbol\up$};
 \draw[thick](0,0.1) node {$\boldsymbol\circ$};

  \draw[very thick]  (-1,-0) to [out=-90,in=180] (-0.5,-0.75) to [out=0,in=-90] (0,0);
 \fill 	(-0.5,-0.75) 	 circle (4pt); \end{tikzpicture}
 \end{minipage}
 , 
  \begin{minipage}{1.2cm}
\begin{tikzpicture} [scale=0.7]
  \draw[very thick](-1,0.1) node {$\boldsymbol\circ$};

  \draw[very thick](0,0.1) node {$\boldsymbol\down$};
  \draw[very thick]  (-1,-0) to [out=-90,in=180] (-0.5,-0.75) to [out=0,in=-90] (0,0);
 \fill 	(-0.5,-0.75) 	 circle (4pt); \end{tikzpicture}
 \end{minipage}
 \!  \right\} . 		
 \end{equation}
We extend $\sim_{\mu}$ by transitivity.     
 Note that, under this relation, each equivalence class contains a unique element $\nu$ satisfying $\deg (\varnothing \diag_\mu \nu) = 0$ (obtained by orienting all northern arcs with rightmost vertex labelled by $\up$). 
 
 Now, let $\lambda, \mu \in {^{A_{n-1}}C_n}$. We claim that 
 $$\Delta_{\la , \mu} = \sum_{\eta \in {^{A_{n-1}}C_n}} N_{\la , \eta} B_{\eta, \mu} = \left\{ \begin{array}{ll} N_{\la , \nu} B_{\nu, \mu}  & \mbox{if $\varnothing \diag_\mu \la \in \mathbb{ODT}[C_n, A_{n-1}]$} \\ 0 & \mbox{otherwise} \end{array}\right.$$where $\nu$ is the unique element satisfying $\la \sim_\mu \nu$ and $\deg (\varnothing \diag_\mu \nu) =0$.
 Suppose that $N_{\la, \eta} \neq 0$ and $B_{\eta , \mu}\neq 0$ for some $\eta$. This implies that $\varnothing \diag_\mu \eta\in \mathbb{ODT}[C_n, A_{n-1}]$ with $\deg (\varnothing \diag_\mu \eta ) = 0$, that is all northern arcs in $\varnothing \diag_\mu \eta$ have rightmost vertex labelled by $\up$. Thus we have that $\diag_{\eta}$ has the same arcs as $\diag_\mu$ but with beads corresponding to oriented decorated northern arcs not connected to the symbol $\circ$ in $\varnothing \diag_\mu \eta$ removed. Now, $N_{\la , \eta}\neq 0$ means that $\varnothing \diag_{\eta}\la\in \mathbb{ODT}[C_{n}, A_{n-1}]$ is standard. This implies that we also have $\varnothing \diag_\mu \la \in \mathbb{ODT}[C_n, A_{n-1}]$. Now we claim we have $\lambda \sim_\mu \eta$. The fact that the northern arcs  are related as in (\ref{simmu}) follows by definition. That the orientation of the propagating lines of $\varnothing \diag_\mu \la$ and $\varnothing \diag_\mu \eta$ coincide follows from the fact that the orientation on the southern boundary is given by $\varnothing$ in both cases and if $\diag_\mu$ contains a decorated propagating line then its northern label is given by the parity condition. 
 Now as $\eta \sim_\mu \la$ and $\deg (\varnothing \diag_\mu \eta) = 0$ implies that $\eta = \nu$. Thus we have shown that
 $$\sum_{\eta \in {^{A_{n-1}}C_n}} N_{\la , \eta}B_{\eta , \mu} = N_{\la , \nu} B_{\nu, \mu}.$$
 Finally note that, as $\nu \sim_{\mu} \la$, a decorated arc in $\diag_\mu$ is oriented  in $\varnothing \diag_\mu \la$ if and only if it is oriented in  $\varnothing \diag_\mu \nu$. Now, as noted above, $\diag_\nu$ is obtained from $\diag_\mu$ by removing all beads on these oriented decorated arcs but the position of the arcs are the same and so $\deg ( \varnothing \diag_\mu \la) = \deg (\varnothing \diag_\nu \la)$.
 Thus we get
 $$\Delta_{\la , \mu} = q^{\deg (\varnothing \diag_\mu \la)} = q^{\deg ( \varnothing \diag_\nu \la)}  = N_{\la , \nu}B_{\nu , \mu}.$$as required. 
  \end{proof}

\begin{rmk}
Note that the unique element $\nu$ described in the proof above is precisely the element satisfying $\varphi(\varnothing \diag_\mu \la) = \overline{\varnothing} \diag_{\overline{\nu}}\overline{\la}$ where $\varphi : {\rm TL}_{(C_n, A_{n-1})}(q) \rightarrow {\rm TL}_{(D_{n+1}, A_n)}(q)$ is the homomorphism defined in \cref{CtoD}.
\end{rmk}

\begin{eg}
There are 6 non-standard basis elements $\varnothing \diag_\mu \la$ in the anti-spherical module of type $(C_3,A_2)$   (this can be deduced from \cref{bad6,bad62}), these  are depicted in \cref{depicted}.

\begin{figure}[ht!]$$
\begin{tikzpicture}  [scale=0.85]
		
		\clip(4,1.82) rectangle ++ (3,-2.7);
		
		\path (4,1) coordinate (origin); 
	 		\path (origin)--++(0.5,0.5) coordinate (origin2);  
		\draw(origin2)--++(0:2.5); 
		\foreach \i in {1,2,3,4}
		{
			\path (origin2)--++(0:0.5*\i) coordinate (a\i); 
			\path (origin2)--++(0:0.5*\i)--++(-90:0.08) coordinate (c\i); 
		}
		
		\foreach \i in {1,2,3,4,5}
		{
			\path (origin2)--++(0:0.25*\i) --++(-90:0.5) coordinate (b\i); 
			\path (origin2)--++(0:0.25*\i) --++(-90:0.7) coordinate (d\i); 
		}

		\draw[  very thick](c2) to [out=-90,in=0] (b3) to [out=180,in=-90] (c1);

		\draw[fill=black](b3) circle (3pt);

		\draw[very   thick](c3)  --  (6,0.55);

		\draw[very   thick](c4)  --  (6.5,0.55);
		
		
		\draw(a1) circle (2pt);
		\path(a2) --++(90:-0.175) node  {  $  \up   $} ;
		\path(a3) --++(90:0.175) node  {  $  \down   $} ;
		\path(a4) --++(90:0.175) node  {  $ \down   $} ;

		\path (4,-1) coordinate (origin); 
		\path (origin)--++(0.5,0.5) coordinate (origin2);  
		\draw(origin2)--++(0:2.5); 
		\foreach \i in {1,2,3,4}
		{
			\path (origin2)--++(0:0.5*\i) coordinate (a\i); 
			\path (origin2)--++(0:0.5*\i)--++(90:0.08) coordinate (c\i); 
  }
		
		\foreach \i in {1,2,3,4,5,...,19}
		{
			\path (origin2)--++(0:0.25*\i) --++(90:0.5) coordinate (b\i); 
			\path (origin2)--++(0:0.25*\i) --++(90:0.65) coordinate (d\i); 
		}

		\draw[  very thick](c2) to [out=90,in=0] (b3) to [out=180,in=90] (c1);

		\draw[very   thick](c3)  --  (6,0.55);
	 	\draw[fill=black](b3) circle (3pt);
		
		\draw[very   thick](c4)  --  (6.5,0.55);

 	 	\path(a3) --++(90:0.175) node  {  $  \down   $} ;
		\path(a4) --++(90:0.175) node  {  $ \down   $} ;

	\draw(a1) circle (2pt);

		\draw(a1) circle (2pt);

		\path(a2) --++(90:0.175) node  {  $ \down   $} ;

		\fill   (6,0.55) circle (3pt);		
		
	\end{tikzpicture}  
 	\begin{tikzpicture}  [scale=0.85]
		
		\clip(4,1.82) rectangle ++ (3,-2.7);
		
		\path (4,1) coordinate (origin); 
		\path (origin)--++(0.5,0.5) coordinate (origin2);  
		\draw(origin2)--++(0:2.5); 
		\foreach \i in {1,2,3,4}
		{
			\path (origin2)--++(0:0.5*\i) coordinate (a\i); 
			\path (origin2)--++(0:0.5*\i)--++(-90:0.08) coordinate (c\i); 
 
		}
		
		\foreach \i in {1,2,3,4,5}
		{
			\path (origin2)--++(0:0.25*\i) --++(-90:0.5) coordinate (b\i); 
			\path (origin2)--++(0:0.25*\i) --++(-90:0.7) coordinate (d\i); 
		}

		\draw[  very thick](c2) to [out=-90,in=0] (b3) to [out=180,in=-90] (c1);

		\draw[fill=black](b3) circle (3pt);

		\draw[very   thick](c3)  --  (6,0.55);

		\draw[very   thick](c4)  --  (6.5,0.55);
		

 		\path(a2) --++(90:0.175) node  {  $  \down   $} ;
		\path(a3) --++(90:0.175) node  {  $  \down   $} ;
		\path(a4) --++(90:0.175) node  {  $ \down   $} ;
		
 		\draw(a1) circle (2pt);
		
		\path (4,-1) coordinate (origin); 
		\path (origin)--++(0.5,0.5) coordinate (origin2);  
		\draw(origin2)--++(0:2.5); 
		\foreach \i in {1,2,3,4}
		{
			\path (origin2)--++(0:0.5*\i) coordinate (a\i); 
			\path (origin2)--++(0:0.5*\i)--++(90:0.08) coordinate (c\i); 
  }
		
		\foreach \i in {1,2,3,4,5,...,19}
		{
			\path (origin2)--++(0:0.25*\i) --++(90:0.5) coordinate (b\i); 
			\path (origin2)--++(0:0.25*\i) --++(90:0.65) coordinate (d\i); 
		}

		\draw[  very thick](c2) to [out=90,in=0] (b3) to [out=180,in=90] (c1);

		\draw[very   thick](c3)  --  (6,0.55);
	 	\draw[fill=black](b3) circle (3pt);
		
		\draw[very   thick](c4)  --  (6.5,0.55);

 	 	\path(a3) --++(90:0.175) node  {  $  \down   $} ;
		\path(a4) --++(90:0.175) node  {  $ \down   $} ;
 		\path(a2) --++(90:0.175) node  {  $ \down   $} ;

				\draw(a1) circle (2pt);

		\fill   (6,0.55) circle (3pt);		
		
	\end{tikzpicture}  	
\begin{tikzpicture}  [scale=0.85]
		
		\clip(4,1.82) rectangle ++ (3,-2.7);
		
		\path (4,1) coordinate (origin); 
		\path (origin)--++(0.5,0.5) coordinate (origin2);  
		\draw(origin2)--++(0:2.5); 
		\foreach \i in {1,2,3,4}
		{
			\path (origin2)--++(0:0.5*\i) coordinate (a\i); 
			\path (origin2)--++(0:0.5*\i)--++(-90:0.08) coordinate (c\i); 
 
		}
		
		\foreach \i in {1,2,3,4,5,6,7,8,9}
		{
			\path (origin2)--++(0:0.25*\i) --++(-90:0.5) coordinate (b\i); 
			\path (origin2)--++(0:0.25*\i) --++(-90:0.7) coordinate (d\i); 
		}

		\draw[  very thick](c2) to [out=-90,in=0] (b3) to [out=180,in=-90] (c1); 
		\draw[  very thick](c4) to [out=-90,in=0] (b7) to [out=180,in=-90] (c3);

		\draw[fill=black](b3) circle (3pt);
		\draw[fill=black](b7) circle (3pt);		
 		
 		\path(a2) --++(90:-0.175) node  {  $  \up   $} ;
		\path(a4) --++(90:-0.175) node  {  $   \up  $} ;
		\path(a3) --++(90:0.175) node  {  $ \down   $} ;
		
 		\draw(a1) circle (2pt);
		
		\path (4,-1) coordinate (origin); 
		\path (origin)--++(0.5,0.5) coordinate (origin2);  
		\draw(origin2)--++(0:2.5); 
		\foreach \i in {1,2,3,4}
		{
			\path (origin2)--++(0:0.5*\i) coordinate (a\i); 
			\path (origin2)--++(0:0.5*\i)--++(90:0.08) coordinate (c\i); 
		  }
		
		\foreach \i in {1,2,3,4,5,...,19}
		{
			\path (origin2)--++(0:0.25*\i) --++(90:0.5) coordinate (b\i); 
			\path (origin2)--++(0:0.25*\i) --++(90:0.65) coordinate (d\i); 
		}

		\draw[  very thick](c2) to [out=90,in=0] (b3) to [out=180,in=90] (c1); 
		\draw[  very thick](c4) to [out=90,in=0] (b7) to [out=180,in=90] (c3);

 	 	\draw[fill=black](b3) circle (3pt);	 	\draw[fill=black](b7) circle (3pt);

 	 	\path(a3) --++(90:0.175) node  {  $  \down   $} ;
		\path(a4) --++(90:0.175) node  {  $ \down   $} ;
 		\path(a2) --++(90:0.175) node  {  $ \down   $} ;
 		
				\draw(a1) circle (2pt);

		 	\end{tikzpicture}   \begin{tikzpicture}  [scale=0.85]
		
		\clip(4,1.82) rectangle ++ (3,-2.7);
		
		\path (4,1) coordinate (origin); 
		\path (origin)--++(0.5,0.5) coordinate (origin2);  
		\draw(origin2)--++(0:2.5); 
		\foreach \i in {1,2,3,4}
		{
			\path (origin2)--++(0:0.5*\i) coordinate (a\i); 
			\path (origin2)--++(0:0.5*\i)--++(-90:0.08) coordinate (c\i); 
 
		}
		
		\foreach \i in {1,2,3,4,5,6,7,8,9}
		{
			\path (origin2)--++(0:0.25*\i) --++(-90:0.5) coordinate (b\i); 
			\path (origin2)--++(0:0.25*\i) --++(-90:0.7) coordinate (d\i); 
		}

		\draw[  very thick](c2) to [out=-90,in=0] (b3) to [out=180,in=-90] (c1); 
		\draw[  very thick](c4) to [out=-90,in=0] (b7) to [out=180,in=-90] (c3);

		\draw[fill=black](b3) circle (3pt);
		\draw[fill=black](b7) circle (3pt);

 		\path(a2) --++(90:-0.175) node  {  $  \up   $} ;
		\path(a3) --++(90:-0.175) node  {  $   \up  $} ;
		\path(a4) --++(90:0.175) node  {  $ \down   $} ;
		
 				\draw(a1) circle (2pt);

		\path (4,-1) coordinate (origin); 
		\path (origin)--++(0.5,0.5) coordinate (origin2);  
		\draw(origin2)--++(0:2.5); 
		\foreach \i in {1,2,3,4}
		{
			\path (origin2)--++(0:0.5*\i) coordinate (a\i); 
			\path (origin2)--++(0:0.5*\i)--++(90:0.08) coordinate (c\i); 
	  }
		
		\foreach \i in {1,2,3,4,5,...,19}
		{
			\path (origin2)--++(0:0.25*\i) --++(90:0.5) coordinate (b\i); 
			\path (origin2)--++(0:0.25*\i) --++(90:0.65) coordinate (d\i); 
		}

		\draw[  very thick](c2) to [out=90,in=0] (b3) to [out=180,in=90] (c1); 
		\draw[  very thick](c4) to [out=90,in=0] (b7) to [out=180,in=90] (c3);

 	 	\draw[fill=black](b3) circle (3pt);	 	\draw[fill=black](b7) circle (3pt);

 	 	\path(a3) --++(90:0.175) node  {  $  \down   $} ;
		\path(a4) --++(90:0.175) node  {  $ \down   $} ;
 		\path(a2) --++(90:0.175) node  {  $ \down   $} ;
 		
				\draw(a1) circle (2pt);

		 	\end{tikzpicture}   
\begin{tikzpicture}  [scale=0.85]
		
		\clip(4,1.82) rectangle ++ (3,-2.7);
		
		\path (4,1) coordinate (origin); 
		\path (origin)--++(0.5,0.5) coordinate (origin2);  
		\draw(origin2)--++(0:2.5); 
		\foreach \i in {1,2,3,4}
		{
			\path (origin2)--++(0:0.5*\i) coordinate (a\i); 
			\path (origin2)--++(0:0.5*\i)--++(-90:0.08) coordinate (c\i); 
 
		}
		
		\foreach \i in {1,2,3,4,5,6,7,8,9}
		{
			\path (origin2)--++(0:0.25*\i) --++(-90:0.5) coordinate (b\i); 
			\path (origin2)--++(0:0.25*\i) --++(-90:0.7) coordinate (d\i); 
		}

		\draw[  very thick](c2) to [out=-90,in=0] (b3) to [out=180,in=-90] (c1); 
		\draw[  very thick](c4) to [out=-90,in=0] (b7) to [out=180,in=-90] (c3);

		\draw[fill=black](b3) circle (3pt);
		\draw[fill=black](b7) circle (3pt);		
 
		\path(a2) --++(90:0.175) node  {  $  \down   $} ;
		\path(a4) --++(90:-0.175) node  {  $   \up  $} ;
		\path(a3) --++(90:0.175) node  {  $ \down   $} ;
		
		\draw(a1) circle (2pt);
		
		\path (4,-1) coordinate (origin); 
		\path (origin)--++(0.5,0.5) coordinate (origin2);  
		\draw(origin2)--++(0:2.5); 
		\foreach \i in {1,2,3,4}
		{
			\path (origin2)--++(0:0.5*\i) coordinate (a\i); 
			\path (origin2)--++(0:0.5*\i)--++(90:0.08) coordinate (c\i); 
	  }
		
		\foreach \i in {1,2,3,4,5,...,19}
		{
			\path (origin2)--++(0:0.25*\i) --++(90:0.5) coordinate (b\i); 
			\path (origin2)--++(0:0.25*\i) --++(90:0.65) coordinate (d\i); 
		}

		\draw[  very thick](c2) to [out=90,in=0] (b3) to [out=180,in=90] (c1); 
		\draw[  very thick](c4) to [out=90,in=0] (b7) to [out=180,in=90] (c3);

 	 	\draw[fill=black](b3) circle (3pt);	 	\draw[fill=black](b7) circle (3pt);

 	 	\path(a3) --++(90:0.175) node  {  $  \down   $} ;
		\path(a4) --++(90:0.175) node  {  $ \down   $} ;
 		\path(a2) --++(90:0.175) node  {  $ \down   $} ;
 		
		\draw(a1) circle (2pt);
		
		 	\end{tikzpicture}   %
	\begin{tikzpicture}  [scale=0.85]
		
		\clip(4,1.82) rectangle ++ (3,-2.7);
		
		\path (4,1) coordinate (origin); 
		\path (origin)--++(0.5,0.5) coordinate (origin2);  
		\draw(origin2)--++(0:2.5); 
		\foreach \i in {1,2,3,4}
		{
			\path (origin2)--++(0:0.5*\i) coordinate (a\i); 
			\path (origin2)--++(0:0.5*\i)--++(-90:0.08) coordinate (c\i); 
 
		}
		
		\foreach \i in {1,2,3,4,5,6,7,8,9}
		{
			\path (origin2)--++(0:0.25*\i) --++(-90:0.5) coordinate (b\i); 
			\path (origin2)--++(0:0.25*\i) --++(-90:0.7) coordinate (d\i); 
		}

		\draw[  very thick](c2) to [out=-90,in=0] (b3) to [out=180,in=-90] (c1); 
		\draw[  very thick](c4) to [out=-90,in=0] (b7) to [out=180,in=-90] (c3);

		\draw[fill=black](b3) circle (3pt);
		\draw[fill=black](b7) circle (3pt);		
%
%
%
		

 		\path(a2) --++(90:0.175) node  {  $  \down   $} ;
		\path(a3) --++(90:-0.175) node  {  $   \up  $} ;
		\path(a4) --++(90:0.175) node  {  $ \down   $} ;
		
 	\draw(a1) circle (2pt);
			
		\path (4,-1) coordinate (origin); 
		\path (origin)--++(0.5,0.5) coordinate (origin2);  
		\draw(origin2)--++(0:2.5); 
		\foreach \i in {1,2,3,4}
		{
			\path (origin2)--++(0:0.5*\i) coordinate (a\i); 
			\path (origin2)--++(0:0.5*\i)--++(90:0.08) coordinate (c\i); 
	  }
		
		\foreach \i in {1,2,3,4,5,...,19}
		{
			\path (origin2)--++(0:0.25*\i) --++(90:0.5) coordinate (b\i); 
			\path (origin2)--++(0:0.25*\i) --++(90:0.65) coordinate (d\i); 
		}

		\draw[  very thick](c2) to [out=90,in=0] (b3) to [out=180,in=90] (c1); 
		\draw[  very thick](c4) to [out=90,in=0] (b7) to [out=180,in=90] (c3);

 	 	\draw[fill=black](b3) circle (3pt);	 	\draw[fill=black](b7) circle (3pt);

 	 	\path(a3) --++(90:0.175) node  {  $  \down   $} ;
		\path(a4) --++(90:0.175) node  {  $ \down   $} ;
 		\path(a2) --++(90:0.175) node  {  $ \down   $} ;
 		
			\draw(a1) circle (2pt);
	
		 	\end{tikzpicture}  	
$$
\caption{The six non-standard diagrams of type $(C_3,A_2)$.   The first and third of these diagrams have degree zero.}
\label{depicted}
\end{figure}

\end{eg}

\begin{cor}
Under the parity specialisation map ${^{A_{n-1}}C_n} \rightarrow {^{A_n}D_{n+1}}\, : \, \lambda \mapsto \overline{\lambda}$ given in \cref{parity}  we have that
$$N^{(C_n, A_{n-1})}_{\la , \mu} = N^{(D_{n+1}, A_n)}_{\overline{\la}, \overline{\mu}}$$
\end{cor}

\begin{proof} 
We need to show that $\varnothing \diag_\mu \la\in \mathbb{ODT}[C_n, A_{n-1}]$ is standard if and only if $\overline{\varnothing} \diag_{\overline{\mu}} \overline{\la}\in \mathbb{ODT}[D_{n+1}, A_n]$. First note that if $\varnothing \diag_\mu \la\in \mathbb{ODT}[C_n, A_{n-1}]$ is not standard then $\overline{\varnothing} \diag_{\overline{\mu}} \overline{\la}\notin \mathbb{ODT}[D_{n+1}, A_n]$. Now assume that $\varnothing \diag_\mu \la\in \mathbb{ODT}[C_n, A_{n-1}]$ is standard. The result is trivial for $\mu = \varnothing$ so we assume that $\mu\neq \varnothing$. This implies that the strand coming out of the first northern vertex in $\diag_\mu$, call it $S$, is decorated. By definition, $\diag_{\overline{\mu}}$ is equal to $\diag_\mu$ if $S$ is flip-oriented in  $\overline{\varnothing} \diag_\mu \overline{\mu}$, and $\diag_{\overline{\mu}}$ is obtained from $\diag_\mu$ by removing the bead on $S$ if it is oriented.  Now, except possibly for $S$ the orientations of the propagating strands in $\overline{\varnothing} \diag_{\mu}\overline{\mu}$ and in $\overline{\varnothing} \diag_{\overline{\mu}} \overline{\la}$ coincide and each northern arc is oriented, respectively   flip-oriented, in the former if and only if it is in the latter. As $\overline{\la}$ and $\overline{\mu}$ both have an even number of $\up$ arrow, this implies that $S$ is oriented, respectively  flip-oriented, in the former, if and only if it is in the latter. Hence we have that $\overline{\varnothing} \diag_{\overline{\mu}} \overline{\la}\in \mathbb{ODT}[D_{n+1}, A_n]$.
\end{proof}

\appendix 
 \section{Singular Kazhdan--Lusztig theory}
 
 In this section, for any $\ctau\in S_W$, we calculate the singular Kazhdan--Lusztig polynomials ``lying on a $\ctau$-hyperplane" for simply-laced Hermitian symmetric pairs $(W,P)$.
This is a technical result that we will need in \cite{compan} and we therefore adopt the following notation from that paper.  We set
$\mathscr{P}_{(W,P)} = {^PW}$ and we let 
$$\mtau:=
\{
\mu  \in {^PW}
\mid   
 \mu\ctau   <\mu
\} \subset\mathscr{P}_{(W,P)} .
$$We believe the following proposition was first proven by Enright--Shelton, where any singular category $\mathcal O$ for a Hermitian 
 symmetric pair is proven to be Morita equivalent to a regular category $\mathcal O$ for a Hermitian 
 symmetric pair of smaller rank.  It is explicitly stated as 
 {\cite[Proposition 5.4]{compan}} in our companion paper, however we note that the proof offered there is simply 
 given by inspection and follows from the combinatorial  description of the cosets  (this combinatorics can be found explicitly recorded in many other places too, even if the bijection itself cannot, see for instance \cite[Appendix]{MR3363009}).  Thus while the following is a reference to a paper whose results depend on the current work, this is not a circular argument.  

 \newcommand{\TRNC}{\psi}
 \begin{prop} 
 Let $(W,P)$ be a simply laced Hermitian symmetric pair   and let $\ctau  \in  S_W$.
 There is an order preserving bijection 
$$\TRNC_\ctau: 	 	\mtau \to 	\mathscr{P}_{(W,P)^\ctau}$$where   $(W,P)^\ctau=(W^\ctau,P^\ctau)$ is defined by 
 \begin{itemize}[leftmargin=*]
   \item $(A_n,A_{k}\times A_{n-k-1})^\ctau = (A_{n-2},A_{k-1}\times A_{n-k-2})$; 
\item $(D_n, A_{n-1})^\ctau = (D_{n-2},A_{n-3})$; 
\item $(D_n, D_{n-1})^\ctau = (A_1,A_0)$; 
  \item $(E_6,D_5)^\ctau = (A_{5}, A_{4})$;
     \item $(E_7,E_6)^\ctau = (D_{6}, D_{5})$.

 \end{itemize} 
 For  $\ctau = {\color{cyan}s_i}$ in types $(W,P)=(A_n,A_{k}\times A_{n-k})$ and 
 $(D_n, A_{n-1})$, this map is given by deleting the pair of symbols in the $i$th and $(i+1)$th positions in the coset diagram (for $\ctau = {\color{cyan}s_{1''}}$, delete the pair of symbols in position $1$ and $2$).

  \end{prop}

The explicit element $ \TRNC  _\ctau(\la) $ in type $(D_n, D_{n-1})$ and exceptional types can be described in terms of tilings (see  {\cite[Proposition 5.4]{compan}}) or deduced directly from the Bruhat graphs without much effort.    

 \begin{figure}[ht!]
 $$
  \begin{tikzpicture} [scale=0.7]
\clip(-3.9,-1.1) rectangle (3.25,1.1);
\draw[rounded corners](-3.9,0) rectangle (3.25,1);
 
\path (0,-0.7)--++(135:0.5)--++(-135:0.5) coordinate (minus1);

\path (minus1)--++(135:0.5)--++(-135:0.5) coordinate (minus2);

\path (minus2)--++(135:0.5)--++(-135:0.5) coordinate (minus3);

\path (minus3)--++(135:0.5)--++(-135:0.5) coordinate (minus4);

\path (0,-0.7)--++(45:0.5)--++(-45:0.5) coordinate (plus1);

\path (plus1)--++(45:0.5)--++(-45:0.5) coordinate (plus2);

\path (plus2)--++(45:0.5)--++(-45:0.5) coordinate (plus3);

\draw[very thick](plus3)--(minus4);

\path(minus1)--++(45:0.4) coordinate (minus1NE);
\path(minus1)--++(-45:0.4) coordinate (minus1SE);

\path(minus4)--++(135:0.4) coordinate (minus1NW);
\path(minus4)--++(-135:0.4) coordinate (minus1SW);

\path(minus2)--++(135:0.4) coordinate (start);

\draw[rounded corners, thick] (start)--(minus1NE)--(minus1SE)--(minus1SW)--(minus1NW)--(start);

\path(plus3)--++(45:0.4) coordinate (minus1NE);
\path(plus3)--++(-45:0.4) coordinate (minus1SE);

\path(plus1)--++(135:0.4) coordinate (minus1NW);
\path(plus1)--++(-135:0.4) coordinate (minus1SW);

\path(plus2)--++(135:0.4) coordinate (start);

\draw[rounded corners, thick] (start)--(minus1NE)--(minus1SE)--(minus1SW)--(minus1NW)--(start);

\draw[very thick,fill=magenta](0,-0.7) circle (4pt);

\draw[very thick,  fill=darkgreen](minus1) circle (4pt);

\draw[very thick,  fill=orange](minus2) circle (4pt);

\draw[very thick,  fill=lime!80!black](minus3) circle (4pt);

\draw[very thick,  fill=violet](minus4) circle (4pt);

\draw[very thick,  fill=gray!80](plus1) circle (4pt);

\draw[very thick,  fill=cyan](plus2) circle (4pt);

\draw[very thick,  fill=pink](plus3) circle (4pt);

\path (0,0.4) coordinate (origin2);

\begin{scope}

     \foreach \i in {0,1,2,3,4,5,6,7,8,9,10,11,12}
{
\path (origin2)--++(45:0.5*\i) coordinate (c\i); 
\path (origin2)--++(135:0.5*\i)  coordinate (d\i); 
  }

\path(origin2)  ++(135:2.5)   ++(-135:2.5) coordinate(corner1);
\path(origin2)  ++(45:2)   ++(135:7) coordinate(corner2);
\path(origin2)  ++(45:2)   ++(-45:2) coordinate(corner4);
\path(origin2)  ++(135:2.5)   ++(45:6.5) coordinate(corner3);

\draw[thick] (origin2) --(corner1) --(corner4)--(origin2);

\clip(corner1)--(corner2)--++(90:0.3)--++(0:6.5)--(corner3)--(corner4)
--++(90:-0.3)--++(180:6.5) --(corner1);

 \path[name path=pathd1] (d1)--++(-90:7);   
 \path[name path=bottom] (corner1)--(corner4);   
 \path [name intersections={of = pathd1 and bottom}];
   \coordinate (A)  at (intersection-1);

        \path(A)--++(-90:0.1) node { $\up$ };

   \path[name path=pathd3] (d3)--++(-90:7);   
 \path[name path=bottom] (corner1)--(corner4);   
 \path [name intersections={of = pathd3 and bottom}];
   \coordinate (A)  at (intersection-1);

  \path(A)--++(90:-0.1) node { $\up$ };

  \path[name path=pathd5] (d5)--++(-90:7);   
 \path[name path=bottom] (corner1)--(corner4);   
 \path [name intersections={of = pathd5 and bottom}];
   \coordinate (A)  at (intersection-1);

  \path(A)--++(90:0.1) node { $\down$ };

 \path[name path=pathd7] (d7)--++(-90:7);   
 \path[name path=bottom] (corner1)--(corner4);   
 \path [name intersections={of = pathd7 and bottom}];
   \coordinate (A)  at (intersection-1);

     \path(A)--++(90:-0.1) node { $\up$ };

 \path[name path=pathd9] (d9)--++(-90:7);   
 \path[name path=bottom] (corner1)--(corner4);   
 \path [name intersections={of = pathd9 and bottom}];
   \coordinate (A)  at (intersection-1);

 \path(A)--++(-90:-0.1) node { $\down$ };

 \path[name path=pathc1] (c1)--++(-90:7);   
 \path[name path=bottom] (corner1)--(corner4);   
 \path [name intersections={of = pathc1 and bottom}];
   \coordinate (A)  at (intersection-1);
 \path(A)--++(-90:-0.1) node { $\down$ };

    \path[name path=pathc3] (c3)--++(-90:7);   
 \path[name path=bottom] (corner1)--(corner4);   
 \path [name intersections={of = pathc3 and bottom}];
   \coordinate (A)  at (intersection-1);

  \path(A)--++(-90:-0.1) node { $\color{cyan}\down$ };

 \path[name path=pathc5] (c5)--++(-90:7);   
 \path[name path=bottom] (corner1)--(corner4);   
 \path [name intersections={of = pathc5 and bottom}];
   \coordinate (A)  at (intersection-1);

 \path(A)--++(90:-0.1) node { $\color{cyan}\up$ };

 \path[name path=pathc7] (c7)--++(-90:7);   
 \path[name path=bottom] (corner1)--(corner4);   
 \path [name intersections={of = pathc7 and bottom}];
   \coordinate (A)  at (intersection-1);

   \path(A)--++(-90:0.1) node { $\up$ }; 

\clip(corner1)--(corner2)--(corner3)--(corner4)--(corner1);

%
%

\end{scope}

\end{tikzpicture}
\qquad\qquad
 \begin{tikzpicture} [scale=0.7]
\clip(-3.9,-1.1) rectangle (3.25,1.1);
\draw[rounded corners](-3.9,0) rectangle (1.75,1);
 
\path (0,-0.7)--++(135:0.5)--++(-135:0.5) coordinate (minus1);

\path (minus1)--++(135:0.5)--++(-135:0.5) coordinate (minus2);

\path (minus2)--++(135:0.5)--++(-135:0.5) coordinate (minus3);

\path (minus3)--++(135:0.5)--++(-135:0.5) coordinate (minus4);

\path (0,-0.7)--++(45:0.5)--++(-45:0.5) coordinate (plus1);
 
%
%
%
%
%

\draw[very thick](plus1)--(minus4);

\path(minus2)--++(45:0.4) coordinate (minus1NE);
\path(minus2)--++(-45:0.4) coordinate (minus1SE);

\path(minus4)--++(135:0.4) coordinate (minus1NW);
\path(minus4)--++(-135:0.4) coordinate (minus1SW);

\path(minus2)--++(135:0.4) coordinate (start);

\draw[rounded corners, thick] (start)--(minus1NE)--(minus1SE)--(minus1SW)--(minus1NW)--(start);

\path(plus1)--++(45:0.4) coordinate (minus1NE);
\path(plus1)--++(-45:0.4) coordinate (minus1SE);

\path(0,-0.7)--++(135:0.4) coordinate (minus1NW);
\path(0,-0.7)--++(-135:0.4) coordinate (minus1SW);

 \path(0,-0.7)--++(135:0.4)--++(0:0.2) coordinate (start);

\draw[rounded corners, thick] (start)--(minus1NE)--(minus1SE)--(minus1SW)--(minus1NW)--(start);

\draw[very thick,fill=magenta](0,-0.7) circle (4pt);

\draw[very thick,  fill=darkgreen](minus1) circle (4pt);

\draw[very thick,  fill=orange](minus2) circle (4pt);

\draw[very thick,  fill=lime!80!black](minus3) circle (4pt);

\draw[very thick,  fill=violet](minus4) circle (4pt);


\draw[very thick ] (plus1)  coordinate (hi) circle (4pt); 

\path(plus1)  --++(90:0.3) coordinate (top1X);
\path(plus1)  --++(-90:0.3) coordinate (bot1X);

{\begin{scope} \clip (plus1)   circle (4pt);  
\draw[ fill=pink,pink ]  (top1X)-- (plus1)  --++(30:0.2) ; 
 \draw[ fill=pink,pink ] (top1X)-- (plus1)  --++(150:0.2) ; 

\draw[ fill= gray,gray ]  (bot1X)-- (plus1) --++(30:0.2) ; 
 \draw[ fill=cyan,cyan ] (bot1X)-- (plus1) --++(150:0.2) ; 
\draw[very thick ] (plus1)  coordinate (hi) circle (4pt); 

\end{scope}}

%
%

\path (0,0.4) coordinate (origin2);

\begin{scope}

     \foreach \i in {0,1,2,3,4,5,6,7,8,9,10,11,12}
{
\path (origin2)--++(45:0.5*\i) coordinate (c\i); 
\path (origin2)--++(135:0.5*\i)  coordinate (d\i); 
  }

\path(origin2)  ++(135:2.5)   ++(-135:2.5) coordinate(corner1);
\path(origin2)  ++(45:2)   ++(135:7) coordinate(corner2);
\path(origin2)  ++(45:2)   ++(-45:2) coordinate(corner4);
\path(origin2)  ++(135:2.5)   ++(45:6.5) coordinate(corner3);

\draw[thick] (origin2) --(corner1) --++(0:4.9)--(origin2);

\clip(corner1)--(corner2)--++(90:0.3)--++(0:6.5)--(corner3)--(corner4)
--++(90:-0.3)--++(180:6.5) --(corner1);

 \path[name path=pathd1] (d1)--++(-90:7);   
 \path[name path=bottom] (corner1)--(corner4);   
 \path [name intersections={of = pathd1 and bottom}];
   \coordinate (A)  at (intersection-1);

        \path(A)--++(-90:0.1) node { $\up$ };

   \path[name path=pathd3] (d3)--++(-90:7);   
 \path[name path=bottom] (corner1)--(corner4);   
 \path [name intersections={of = pathd3 and bottom}];
   \coordinate (A)  at (intersection-1);

  \path(A)--++(90:-0.1) node { $\up$ };

  \path[name path=pathd5] (d5)--++(-90:7);   
 \path[name path=bottom] (corner1)--(corner4);   
 \path [name intersections={of = pathd5 and bottom}];
   \coordinate (A)  at (intersection-1);

  \path(A)--++(90:0.1) node { $\down$ };

 \path[name path=pathd7] (d7)--++(-90:7);   
 \path[name path=bottom] (corner1)--(corner4);   
 \path [name intersections={of = pathd7 and bottom}];
   \coordinate (A)  at (intersection-1);

     \path(A)--++(90:-0.1) node { $\up$ };

 \path[name path=pathd9] (d9)--++(-90:7);   
 \path[name path=bottom] (corner1)--(corner4);   
 \path [name intersections={of = pathd9 and bottom}];
   \coordinate (A)  at (intersection-1);

 \path(A)--++(-90:-0.1) node { $\down$ };

 \path[name path=pathc1] (c1)--++(-90:7);   
 \path[name path=bottom] (corner1)--(corner4);   
 \path [name intersections={of = pathc1 and bottom}];
   \coordinate (A)  at (intersection-1);
 \path(A)--++(-90:-0.1) node { $\down$ };

    \path[name path=pathc3] (c3)--++(-90:7);   
 \path[name path=bottom] (corner1)--(corner4);   
 \path [name intersections={of = pathc3 and bottom}];
   \coordinate (A)  at (intersection-1);
 
   \path(A)--++(-90:0.1) node { $\up$ }; 

\clip(corner1)--(corner2)--(corner3)--(corner4)--(corner1);

\end{scope}

\end{tikzpicture}
 $$
 \caption{The element $\la$ from \cref{typeAtiling-long} (with the vertices we will remove under the map $\TRNC _\ctau$ highlighted in blue)
and its image $\TRNC  _\ctau(\la)$.  
The tri-colouring of the rightmost node  of the Coxeter graph on the righthand-side is explained in \cite{compan} but is not important here.  
 }
 \label{contraction}
 \end{figure}

  \begin{figure}[ht!]
 $$
 \begin{tikzpicture}

		\path (4,1) coordinate (origin); 
		\path (origin)--++(0.5,0.5) coordinate (origin2);  
		\draw[thick](origin2)--++(0:5); 
		\foreach \i in {1,2,3,4,5,...,9}
		{
			\path (origin2)--++(0:0.5*\i) coordinate (a\i); 
			\path (origin2)--++(0:0.5*\i)--++(-90:0.08) coordinate (c\i); 
			}
		
		\foreach \i in {1,2,3,4,5,...,19}
		{
			\path (origin2)--++(0:0.25*\i) --++(-90:0.5) coordinate (b\i); 
			\path (origin2)--++(0:0.25*\i) --++(-90:0.7) coordinate (d\i); 
		}
		 \path(a1) --++(90:0.175) node  {  $  \down   $} ;
		\path(a3) --++(90:0.175) node  {  $  \down   $} ;
		\path(a2) --++(-90:0.175) node  {  $  \up   $} ;
		\path(a4) --++(-90:0.175) node  {  $  \up   $} ;
		\path(a5) --++(-90:0.175) node  {  $  \up  $} ;
		\path(a8) --++(90:-0.175) node  {  $ \color{cyan} \up  $} ;
		\path(a6) --++(90:0.175) node  {  $  \down  $} ;
		\path(a7) --++(90:0.175) node  {  $\color{cyan}  \down  $} ; 
		\path(a9) --++(-90:0.175) node  {  $  \up  $} ;
		
		\path(b15)--++(-90:0.2) coordinate (b15s);
		\draw[  very thick](c9) to [out=-90,in=0] (b15s) to [out=180,in=-90] (c6); 
		
		\draw[  very thick,cyan](c8) to [out=-90,in=0] (b15) to [out=180,in=-90] (c7); 
		\draw[  very thick](c2) to [out=-90,in=0] (b3) to [out=180,in=-90] (c1); 
		\draw[  very thick](c4) to [out=-90,in=0] (b7) to [out=180,in=-90] (c3); 		

		\draw[  very thick](c5) --++(-90:0.7);

\end{tikzpicture}
\qquad
 \begin{tikzpicture}

		\path (4,1) coordinate (origin); 
		\path (origin)--++(0.5,0.5) coordinate (origin2);  
		\draw[thick](origin2)--++(0:4); 
		\foreach \i in {1,2,3,4,5,...,7}
		{
			\path (origin2)--++(0:0.5*\i) coordinate (a\i); 
			\path (origin2)--++(0:0.5*\i)--++(-90:0.08) coordinate (c\i); 
			}
		
		\foreach \i in {1,2,3,4,5,...,19}
		{
			\path (origin2)--++(0:0.25*\i) --++(-90:0.5) coordinate (b\i); 
			\path (origin2)--++(0:0.25*\i) --++(-90:0.7) coordinate (d\i); 
		}
		 \path(a1) --++(90:0.175) node  {  $  \down   $} ;
		\path(a3) --++(90:0.175) node  {  $  \down   $} ;
		\path(a2) --++(-90:0.175) node  {  $  \up   $} ;
		\path(a4) --++(-90:0.175) node  {  $  \up   $} ;
		\path(a5) --++(-90:0.175) node  {  $  \up  $} ;
		\path(a6) --++(90:0.175) node  {  $  \down  $} ;
		\path(a7) --++(-90:0.175) node  {  $  \up  $} ;
		
		
		\draw[  very thick](c7) to [out=-90,in=0] (b13) to [out=180,in=-90] (c6); 
		\draw[  very thick](c2) to [out=-90,in=0] (b3) to [out=180,in=-90] (c1); 
		\draw[  very thick](c4) to [out=-90,in=0] (b7) to [out=180,in=-90] (c3); 		

		\draw[  very thick](c5) --++(-90:0.7);

\end{tikzpicture}
 $$
 \caption{The element $\mu$ from \cref{typeAtiling-long,contraction} with its associated diagram $\diag_\mu$ and their image under  $\TRNC_\ctau$. }
 \label{contraction2}
 \end{figure}

\begin{thm} Let $(W,P)$ be a simply-laced Hermitian symmetric pair, $\ctau \in S_W$ and $\la ,\mu \in \mathscr{P}_{(W,P)}^\ctau$.
Then we have a degree-preserving  bijection 
$$\Path(\la,\underline{\mu}) \to \Path(\TRNC  _\ctau(\la),\underline{\TRNC  _\ctau(\mu)}).$$
\end{thm}  

\begin{proof}
For types $(D_n, D_{n-1})^\ctau = (A_1,A_0)$, 
   $(E_6,D_5)^\ctau = (A_{5}, A_{4})$, and 
 $(E_7,E_6)^\ctau = (D_{6}, D_{5})$ the result can be checked directly by examining the light leaves matrices given in \cref{exceptional}. In types   $(A_n,A_{k}\times A_{n-k-1})$ with $\ctau  =  {\color{cyan} s_i} \in S_W$ arbitrary and in type 
$(D_n, A_{n-1})$ and for any $\ctau = {\color{cyan}s_i} \not = {\color{cyan} s_{1''}} $, note that $\la, \mu \in \mathscr{P}_{(W,P)}^\ctau$ if and only if the $i$-th and $(i+1)$-th vertex in their coset diagram are labelled by $\down$ $\up$. This implies that the
 diagrams  $\diag_{\SSTT} = {\varnothing  \diag_\mu \la }$ for $\SSTT\in \Path(\la,\underline{\mu})$ are precisely those which have an anti-clockwise oriented  northern  arc connecting $i$ and $i+1$. 
 In type   $(D_n, A_{n-1})$ with $\ctau =\color{cyan}s_{1''} $, note that $\la , \mu \in \mathscr{P}_{(W,P)}^\ctau$ if and only if the first and second vertex of their coset diagram are labelled by $\up \up$. This implies that the 
 diagrams  $\diag_\SSTT = {\varnothing  \diag_\mu \la }$ for  $\SSTT\in \Path(\la,\underline{\mu})$ are precisely those which have flip-oriented decorated northern arc connecting the first and second vertex, both labelled by $\up$. 
For the purposes of this proof, we will only consider the top half of the diagram $\varnothing  \diag_\mu \la$, for ease of exposition (the bottom half plays no significant role, see \cref{thething}). 
 
 In all cases, the northern arc identified above joins two adjacent vertices and has degree $0$. Now, if $(W,P) = (A_n, A_{k}\times A_{n-k-1})$, then removing this oriented arc produces an oriented Temperley--Lieb diagram of type $(A_{n-2}, A_{k-1}\times A_{n-k-2})$  of the same degree. If $(W,P) = (D_n, A_{n-1})$ and $\ctau = {\color{cyan} s_{1''}}$ then removing this flip-oriented decorated northern arc produces an oriented Temperley--Lieb diagram of type $(D_{n-2}, A_{n-3})$ of the same degree. Finally, if $(W,P) = (D_n, A_{n-1})$ and $\ctau = {\color{cyan} s_{i}}$ for $i\in \{1, \ldots , n-1\}$ then removing this oriented (undecorated) northern arc produces an oriented Temperley--Lieb diagram of the same degree for the isomorphic copy of $TL_{(D_{n-2}, A_{n-3})}(q)$ where each coset diagram has an odd number of $\up$ arrows (as discussed in \cref{isomorphic copies}).

This map is clearly a bijection, as starting from any oriented Temperley--Lieb diagram for the anti-spherical module in rank $n-2$, we can insert the degree zero arc in the position corresponding to $\ctau$ to obtain an oriented Temperley--Lieb diagram for the anti-spherical module in rank $n$ and these maps are clearly inverse to each other.

\end{proof}

         \bibliographystyle{amsalpha}   
\bibliography{master}

 \end{document}